\documentclass[10pt]{amsart}
\usepackage{graphicx} 
\usepackage{amssymb}
\usepackage{geometry}
\usepackage{tikz-cd}
\tikzset{%
    symbol/.style={%
        draw=none,
        every to/.append style={%
            edge node={node [sloped, allow upside down, auto=false]{$#1$}}}
    }
}

\newtheorem{theorem}{Theorem}[section]

\theoremstyle{proposition}
\newtheorem{proposition}[theorem]{Proposition}

\theoremstyle{corollary}
\newtheorem{corollary}[theorem]{Corollary}

\theoremstyle{definition}
\newtheorem{definition}[theorem]{Definition}
\newtheorem{example}[theorem]{Example}

\newtheorem{warning}[theorem]{Warning}
\newtheorem{observation}[theorem]{Observation}

\newtheorem{construction}[theorem]{Construction}

\theoremstyle{remark}
\newtheorem{remark}[theorem]{Remark}
\theoremstyle{question}

\numberwithin{equation}{section}

\usepackage{euscript}
\usepackage{mathrsfs}
\usepackage{hyperref}

\usepackage{slashed}

\usepackage{aliascnt}
\usepackage{hyperref, cleveref}

\usepackage{adjustbox}
\usepackage{tikz-cd}
\tikzset{%
    symbol/.style={%
        draw=none,
        every to/.append style={%
            edge node={node [sloped, allow upside down, auto=false]{$#1$}}}
    }
}

\newcommand\scalemath[2]{\scalebox{#1}{\mbox{\ensuremath{\displaystyle #2}}}}

\newcommand{\verteq}{\rotatebox{90}{$\,=$}}
\newcommand{\equalto}[2]{\underset{\scriptstyle\overset{\mkern4mu\verteq}{#2}}{#1}}
\newcommand{\vertin}{\rotatebox{90}{$\,\in$}}
\newcommand{\invert}[2]{\underset{\scriptstyle\overset{\mkern4mu\vertin}{#2}}{#1}}

\usepackage{graphicx} 
\usepackage{tikz-cd}
\usepackage{amssymb,amsmath,amsthm}
\usepackage{euscript}
\usepackage{hyperref}
\usepackage{graphicx} 
\usepackage{tikz-cd}
\usepackage{amssymb,amsmath,amsthm}
\usepackage{euscript}
\usepackage{hyperref}
\numberwithin{equation}{section}

\usepackage{euscript}
\usepackage{mathrsfs}
\usepackage{hyperref}

\usepackage{slashed}
\usepackage{aliascnt}
\usepackage{hyperref, cleveref}

\usepackage{adjustbox}
\usepackage{tikz-cd}
\tikzset{%
    symbol/.style={%
        draw=none,
        every to/.append style={%
            edge node={node [sloped, allow upside down, auto=false]{$#1$}}}
    }
}

\usepackage{graphicx} 
\usepackage{tikz-cd}
\usepackage{amssymb,amsmath,amsthm}
\usepackage{euscript}
\usepackage{hyperref}

\usepackage{aliascnt}

\newcommand{\Map}{\underline{\mathsf{Map}}}

\newcommand{\Sect}{\mathbb{R}\underline{\mathsf{Sec}}}

\newcommand{\RS}{\mathbb{R}\underline{\mathrm{Sol}}_X}
\newcommand{\RSol}{\mathbb{R}\underline{\mathrm{Sol}}}
\newcommand{\Res}{\underline{\mathsf{Res}}}
\newcommand{\Spec}{\underline{\mathsf{Spec}}}
\newcommand{\Sp}{\underline{\mathrm{Spec}}_{\mathscr{D}}}

\newcommand{\DR}{\mathrm{DR}}
\newcommand{\Hod}{\mathrm{Hod}}

\newcommand{\po}{/\!/}
\newcommand{\AG}{[\mathbb{A}^1/\mathbb{G}_m]}
\newcommand{\Q}{\mathsf{QCoh}}

\newcommand{\IC}{\mathsf{IndCoh}}
\newcommand{\PC}{\mathsf{ProCoh}}

\newcommand{\C}{\mathsf{Crys}}
\newcommand{\cdga}{\mathbf{cdga}}
\newcommand{\DG}{\mathbf{dgMod}_{\D}}
\newcommand{\XG}{[X/\mathbb{G}_m]}
\usepackage{mathrsfs}

\newcommand{\PS}{\mathsf{PStk}}

\newcommand{\A}{\EuScript{A}}
\newcommand{\M}{\mathcal{M}}
\newcommand{\D}{\mathscr{D}}
\newcommand{\Char}{\mathrm{Char}}

\newcommand{\Jets}{\underline{\mathsf{Jet}}^{\infty}}
\newcommand{\JetX}{\mathrm{Jet}_X^{\infty}}
\newcommand{\J}{\mathsf{Jet}_{\mathrm{DR}}^{\infty}}
\newcommand{\JetY}{\mathrm{Jet}_Y^{\infty}}
\newcommand{\Rmu}{\mathbb{R}\mu}
\newcommand{\EQ}{Z}
\newcommand{\Free}{\mathsf{Free}}
\newcommand{\IAB}{0\to \mathcal{I}_{Z^{\infty}}\to \mathcal{O}_{J_X^{\infty}E}\to \mathcal{O}_{Z^{\infty}}\to 0}

\title{Derived Moduli Spaces of Non-linear PDEs: Singular Propagations}
\author{J. Kryczka}
\address{School of Mathematics, Harbin Institute of Technology (HIT), Harbin 150001, China \&
Beijing Institute of Mathematical Sciences and Applications (BIMSA), Huairou District, Beijing 101408, China.
}
\email{jkryczka@hit.edu.cn}

\author{A. Sheshmani}
\address{
Beijing Institute of Mathematical Sciences and Applications (BIMSA), Huairou District, Beijing 101408, China.
}
\email{artan@mit.edu}

\author{S.--T. Yau}
\address{Yau Mathematical Sciences Center, Tsinghua University, Haidian District, Beijing,
China\\
}
\email{styau@tsinghua.edu.cn}

\date{Submitted version on October 02, 2024, Revised: June 08, 2026}

\begin{document}

\begin{abstract}
We construct a sheaf theoretic and derived geometric machinery to study nonlinear partial differential equations and their singular supports. We establish a notion of derived microlocalization for solution spaces of non-linear equations and develop a formalism to pose and solve singular non-linear Cauchy problems globally. Using this approach we estimate the domains of propagation for the solutions of non-linear systems. It is achieved by exploiting the fact that one may greatly enrich and simplify the study of derived non-linear PDEs over a space $X$ by studying its derived linearization which is a module over the sheaf of functions on the $S^1$-equivariant derived loop stack $\mathcal{L}X$.
\end{abstract}
\maketitle
\tableofcontents

\section{Introduction}
This is the first in a series of papers devoted to studying various aspects of the theory of non-linear partial differential equations (NPDEs) within the framework of derived algebraic and analytic geometry. Here our attention is on developing a formalism that can be used to study singularities arising due to non-linear phenomenon. We study the general question of having well-posed \emph{global non-linear} Cauchy problem and various aspects  related to their characteristic Cauchy data.

Roughly speaking, such data can be understood as a Cauchy surface with
initial data on it for which the Cauchy problem is \emph{not} well-posed (see \S.\ref{ssec: Cauchy Problems}). This happens, for instance, if the problem fails to have existence or uniqueness. The existence of characteristic
Cauchy data is a general feature of many non-linear (field) equations. In particular, characteristic surfaces encode relevant information about wave-front propagation\footnote{In classical
field theory, the boundary of a disturbance in the field, which might be understood as a
wave-front, is a characteristic surface and thus carries
relevant information about field wave propagation.}. 

Within this framework, the `non-propagation' of a solution refers to the obstruction of its analytic continuation across a smooth boundary $\Sigma$ of its domain $\Omega$. In other words, a solution defined on $\Omega$ fails to propagate if it cannot be extended as a solution to any neighborhood beyond $\Sigma$. This typically indicates the presence of singularities e.g. in the solution, the initial datum or both. Thus the characteristic Cauchy problem and the presence of singularities and their propagation are intimately linked.

Microlocal sheaf theory \cite{KashiwaraSchapira1990} and $\D$-module theory (following, for example \cite{kashiwara1970algebraic,KashiwaraMicro,KashiwaraOshima1977}) make this link precise. The \emph{microsupport} (alias, \emph{singular support}) $\mathrm{SS}(\mathcal{F})$ of a sheaf $\mathcal{F},$ introduced in \cite{KS82,KS85}, describes the codirections of non propagation of sections $\mathcal{F}.$

Linear PDEs have well-known algebraic descriptions as modules over the sheaf $\D_X$ of differential operators, and the important application of microlocal sheaf theory to PDEs arise by taking 
$\mathcal{F}$ to be the (derived) sheaf of holomorphic solutions $R\mathcal{S}\mathrm{ol}_X(P):=R\mathcal{H}om_{\D_X}(\M_P,\mathcal{O}_X)$ associated to a linear differential operator of order $P(x,\partial)=\sum_{|\sigma|\leq m}a_{\sigma}(x)\partial_{\sigma}$ of order $\leq m$ with holomorphic coefficients, whose corresponding $\D$-module is $\mathcal{M}_P:=\mathscr{D}_X/\mathscr{D}_X\cdot P.$ 

An important theorem due to Kashiwara--Schapira \cite[Thm.~11.3.3]{KashiwaraSchapira1990} establishes the following relation between characteristic codirections and those of non-propagation:
\begin{equation}
    \label{eqn: LPDE}
\mathrm{SS}\big(R\mathcal{S}\mathrm{ol}_{X}(P)\big)=\mathrm{Char}(P)\subset T^*X,
\end{equation}
where $\mathrm{Char}(P)$ is a closed $\mathbb{C}^{\times}$-conic $\mathbb{C}$-analytic subset of $T^*X$ called the characteristic variety of $P$. Note the inclusion $\mathrm{SS}(R\mathcal{S}ol_X(P))\subset \mathrm{Char}(P)$ implies the Cauchy--Kowaleskaya--Kashiwara theorem (recalled in Sect.\ref{ssec: Cauchy Problems}).

Recall, if $\sigma(P)$ is the principal symbol of $P$ given by $\sigma(P)(x;\xi)=\sum_{|\sigma|=m}a_{\sigma}(x)\xi^{\sigma},$ for $(x;\xi)\in T^*X,$ the characteristic variety is $\mathrm{Char}(P)=\{(x;\xi)\in T^*X|\sigma(P)(x;\xi)=0\}$. It is well-known to be integrable \cite{Sato1973},\cite{Gabber1981} (also called coisotropic/involutive e.g, \cite[Def.~6.5.1]{KashiwaraSchapira1990}).

\subsubsection{}
In this paper we study \emph{non-linear} analogs of (\ref{eqn: LPDE}) and aim to develop a sheaf-theoretic microlocal framework to provide  coordinate-free and homotopy-invariant meanings to non-linear (derived) solution spaces and their characteristic varieties. 

Unlike the linear setting of $\D$-modules, which systematically employs homological techniques to obtain propagation estimates, no such unified framework exists in the classical literature for non-linear equations, although the importance of cohomological analysis for nonlinear PDEs was certainly anticipated \cite{Vinogradov2001}. 

It should be emphasized that singularities of solutions of \emph{non-linear} PDE might occur along characteristic surfaces and for large classes of nonlinear PDEs, after imposing certain boundedness and regularities of solutions (see §§ \ref{sssec: On the non-linear generalizations}), one has propagation of solution singularities in the non-characteristic codirections. However, it is not necessarily the case that arbitrary non-linear singularities travel along bicharacteristic curves in the classical sense.

\begin{remark}
\label{important remark} 
There are examples of the formation of additional singularities in solutions to nonlinear wave
equations (e.g. semilinear ones) arising from nonlinear interactions rather than propagated by a Hamiltonian flow (e.g. non-linear waves might interact producing new singular waves). 
The mechanism behind the formation
of extra singularities in the solutions to the certain PDEs e.g. 
$$\textbf{(a)}\begin{cases}
    \partial_tu-\partial_xu=0
    \\
    \partial_tv+\partial_xv=0
    \\
    \partial_tw=uv,
\end{cases}\hspace{2mm} or\hspace{2mm}\textbf{(b)}\begin{cases}
  \partial_t^2u_j-(\partial_1^2 +\partial_2^2) u_j=0,j=1,2,3,
  \\
  \partial_t^2v-(\partial_1^2 +\partial_2^2) v=u_1u_2u_3,
\end{cases}$$
is the creation
of extra singularities in taking nonlinear functions $F_P,$ associated to a non-linear operator $P$. A simple explanation for this fact is that wavefront set\footnote{Locus where the solution fails to be smooth.} $\mathrm{SS}\big(F(u)\big)$ can be strictly bigger than $\mathrm{SS}(u).$ Typically, one expects that elements of $\mathrm{SS}\big(F(u)\big)\backslash \mathrm{SS}(u)\cap \mathrm{Char}(F_P),$
where $\mathrm{Char}(F_P)$ is the `non-linear' characteristic variety, are the ones which are propagated.
\end{remark}

A known example of singularities not present in the linear case arise from the crossing of two or
more singularity-bearing characteristics. In this case, from each crossing point of singularities for the corresponding linear problem, there are solutions to nonlinear
problems with singularities concentrated along forward characteristics. This phenomena is known to appear in the context of shock waves; if two shock waves interact, the interaction may produce not a third weaker singularity, but rather a third shock.

These types of situations and those in
Remark \ref{important remark}, show why non-linear phenomena is more complicated than its linear counterpart. Moreover, it is perhaps surprising that there even exists an analog of (\ref{eqn: LPDE}) for non-linear PDEs, due to the following.

\begin{example}
\label{ex: Pierre}
Consider $\partial_zf-f^2=0$ on $X=\mathbb{C},$ with solution $f(z)=1/(z-c)$ for, $c\in \mathbb{C}.$ Its symbol is $\partial_z$ with characteristic variety $\{\xi=0\}$, where $(z;\xi)\in T^*X.$ Thus, if \emph{(\ref{eqn: LPDE})} was true, it would imply $f(z)$ is regular holomorphic on all $X$ since $Char$ is the zero-section $T_X^*X,$ which is false. 
\end{example}
What remedies the situation, and where derived geometry enters naturally in the non-linear setting, is that the usual characteristic variety $\mathrm{Char}$ is insufficient to detect non-linear singularities e.g. blow-up phenomenon of solutions. This is captured by cohomologies of the tangent complex and in the non-linear setting the derived characteristic variety i.e. $\D$-geometric characteristic variety is defined as the support of the tangent complex (cf. Def.\ref{definition: 1-Microchar}). In particular, it depends on linearization around a background solution, which as above might be singular.

\begin{example}
\label{obs: Pierre}
     Continuing with Ex.~\ref{ex: Pierre}, the linearized operator around $f(\EQ)=1/(z-c)$ is $P_f:=\partial_z-2/(z-c)id.$ The tangent complex is $\mathcal{O}\xrightarrow{P_f}\mathcal{O}$ in degrees $0,1.$ Since the singular support $SS(f)$ (wavefront, see Def.~\ref{defn: An WaveFront}) is $\{z=c\},$ one may compute $H^i(\mathcal{O}\xrightarrow{P_f}\mathcal{O})$ for $i=0,1,$ as $ker(P_f)$ and $coker(P_f)$, respectively. The modified (derived) characteristic variety is the union of the supports of these cohomologies. Simple computation shows it is $\{\xi=0\}\cup \{(c,\infty)\}$ where $\infty$ refers to a blow-up in the solution (specifically, we obtain the sky-scraper sheaf supported at $c$). Thus, while the $H^0$-term sees the classical contribution, it is a non-vanishing $H^1\neq 0$ capturing derived phenomena - it is non-zero and supported at $z=c$. Thus
    $\mathrm{SS}(f)\subset T_X^*X\cup \{(z=c\},$
    as required.
\end{example}
To consider more general non-linear PDEs, we work within a geometric framework of jet-space geometry \cite{Krasilshchik1986}.

It turns out that abstract properties of jet-bundles (e.g, \cite{Kock1979,Kock2010,Saunders1989}) permit a translation of this approach to the derived setting, where we may study global moduli spaces of solutions associated to systems of PDEs
\begin{equation}
    \label{eqn:NLPDE system}
    Z:\bigg\{\mathsf{F}_{A}\big(x,u,\frac{\partial u}{\partial x},\ldots, \frac{\partial^k u}{\partial x^k},\ldots,\frac{\partial^{|\sigma|}u}{\partial x_{\sigma}}\big)=0,\hspace{2mm} A=1,\ldots, N 
    \end{equation} 
determined by functions $\mathsf{F}_A$ of three distinct types: algebraic, analytic and `$\D$-geometric', with prescribed non-linearities (e.g. quasi-linearity, semi-linearity or fully non-linear etc.) imposed on vector valued functions $u$ with $n$-independent variables $x=(x_1,\ldots,x_n)$ and $m$-dependent variables $u=(u^1,\ldots,u^m).$ Standard multi-index notation $\sigma=(\sigma_1,\ldots\sigma_n)\in \mathbb{N}^n$ with \emph{order} $|\sigma|:=\sigma_1+\cdots+\sigma_n,$ is used and for simplicity the orders of each operator are taken to be the same.
\begin{remark}
\label{term: Settings}
We say that we are in the algebraic (resp. analytic) setting when the functions $\mathsf{F}_i$ are algebraic (resp. analytic). When the defining functions are smooth in variables $x,u$ but only polynomial in derivatives of orders $\geq 1,$ we say we are in the $\mathscr{D}$-geometric or almost-algebraic setting.\footnote{This was referred to as \emph{relatively--algebraic} in \cite{Paugam2022}.}
\end{remark}

The classical geometric approach \cite{Krasilshchik1986}studies (\ref{eqn:NLPDE system}) in terms of an associated infinite-dimensional space -- an appropriate sub-manifold $Z^{\infty}$ of the space of infinite jets -- obtained via the formal theory of PDEs (E. Cartan's tool of prolongation \cite{Kuranishi1957}) by successively differentiating the ideal of relations. Equivalently, it means the ideal $\mathcal{I}$ is stable under an action by a sheaf of differential operators $\D_X \bullet \hspace{.5mm} \mathcal{I}\subset \mathcal{I}$, defined through a canonical involutive distribution $\mathcal{C}$ on jets.

Function algebras $\mathcal{O}_{Z^{\infty}}$ may then be modeled by quotient algebras $\mathcal{B}$ of functions on jet bundles i.e. jet algebras $\mathcal{A}:=\mathcal{O}(\mathrm{Jets}^{\infty})$ by $\D_X$-ideal sheaves $\mathcal{I}$\footnote{Wherever necessary, we make the assumption they are prime in the differential sense.},
\begin{equation}
    \label{eqn:NLPDE SES}0\to \mathcal{I}_{Z^{\infty}}\rightarrow \mathcal{O}_{J_X^{\infty}E}\xrightarrow{\mathsf{p}}\mathcal{O}_{Z_{\infty}}\to 0.
\end{equation}
A ubiquitous scenario is when $\mathcal{B}$ is differentially generated (Def. \ref{defn: RelAlgNLPDE} (2) below) or when it is given by non-linear operators locally of Cauchy-Kovalevskaya type (\ref{eqn: D-CK Ideals}). From the point of view of moduli theory, in §\ref{sec: Non-linear Algebraic Analysis} we view sequences (\ref{eqn: SES}) up to $\D$-isomorphism.

Sequences (\ref{eqn:NLPDE SES}) give solution functors $\mathrm{Sol}_{\D}(\mathcal{I})$ and a $\D$-subscheme of the ambient $\D$-scheme $Spec_{\D}(\mathcal{A}_X)$ of algebraic $\infty$-jets. 
 A consequence of the Cartan--K\"ahler \cite{Cartan1945}\cite{Kahler}, and Cartan--Kuranishi \cite{Kuranishi1957} theorems, as well as integrability results due to Goldschmidt \cite{Gold1967,Goldschmidt1967}, Spencer \cite{Sp,Sp2} and Quillen \cite{Quillen1964}, for analytic systems\footnote{This is a non-trivial fact due to the formal theory of integrability in the analytic category. We do not treat work in the $C^{\infty}$-setting \cite{Lewy1957}.}, $\mathrm{Sol}_{\D}(\mathcal{I})\neq Spec_{\D}(\mathcal{A}_X) \neq \emptyset,$ and we obtain an embedding of $\D$-schemes
$\iota:Spec_{\D}(\mathcal{B})\hookrightarrow Spec_{\D}(\mathcal{A}).$

Derived enhancements $\mathrm{Sol}_{\D}(\mathcal{I})\hookrightarrow \mathbb{R}\mathrm{Sol}_{\D}(\mathcal{I})$ may be attributed to sequences (\ref{eqn:NLPDE SES}) which provide an appropriate setting to study systems (\ref{eqn:NLPDE system}) modulo symmetries. This occurs when studying moduli spaces of solutions in gauge theories of the form $[\![\mathbb{R}\mathrm{Sol}(\mathcal{B})/\mathcal{G}]\!]$ i.e. \emph{after} taking the quotient by the gauge group $\mathcal{G}$ of the theory (see \ref{BV-Example}). 

\begin{remark}
Any derived $\D$-scheme underlies a classical differential ideal. However, not every differential
ideal arises this way. It seems without integrability conditions, there is no reason to guarantee the existence of a derived $\D_X$-scheme enhancing a given classical differential ideal. It is known that holomorphic integrability is restricted by topological obstructions. Such obstructions lead to sources of  
examples of (analytic germs of)
differential ideals which do not admit derived enhancements. 
\end{remark}

The derived linearization is, roughly speaking, encoded by the tangent complex $\mathbb{T}_{\mathbb{R}\mathrm{Sol}(\mathcal{B})}$ and it turns out we can simultaneously enrich and simplify our study of non-linear PDEs (\ref{eqn:NLPDE SES}) by viewing $\mathbb{T}_{\mathsf{Sol}(\mathcal{B})}$ as a module not over a  non-commutative ring such as $\D_X$, but a commutative one - the ring of holomorphic functions $\mathcal{O}(\mathcal{L}X)$ on the $S^1$-equivariant derived loop stack $\mathcal{L}X$. From the perspective of derived geometry $\mathbb{T}_{\RS(\mathcal{B})}$ is computed as either an appropriate $\D$-module dual of the cotangent complex $\mathbb{L}_{\mathcal{Q}^{\bullet}}$ of a cofibrant resolution $\mathcal{Q}$ of $\mathcal{A}/\mathcal{I}=\mathcal{B}$ as a differential graded $\mathcal{A}-\D$-algebra, or via a canonical anti self-duality for sheaves over $X_{\DR}$ \cite{GR14}. 

This approach permits the study of the global virtual geometry of derived enhancements of well-known moduli spaces of solutions and their deformations from the perspective of derived PDE theory.
\begin{example}
Let $G$ be a reductive algebraic group. Put $\EQ:=\J([\mathrm{pt}/G]\times X])$. Then, $\RS\big(\J([pt/G]\times X)\big)$ coincides with the stack of $G$-bundles on $X$. If $\lambda \in\mathbb{A}^1$ is a (Hodge) deformation parameter corresponding \cite{Sim09},\cite{SimpsonHodgeFiltrationNonabelian}, we have a Hodge family $\RS^{\lambda}(\EQ)\to\mathbb{A}^1$ of deformations of solutions:
 \begin{equation}
 \label{eqn: Hodge Deformation}
    \begin{tikzcd} 
\RS\big([\mathrm{pt}/G]\times X_{\DR}\big)\simeq \mathrm{LocSys}_G(X) & \RS^{\lambda}(\EQ)\arrow[r,"\lambda=0",dashed]
    \arrow[l,"\lambda=1",dashed] & \mathrm{Higgs}_G(X).
    \end{tikzcd}
 \end{equation}
Our main construction \ref{cons: Solutions} describes derived moduli of solutions to $\EQ_1$ with coefficients $\mathbb{R}\mathrm{Sol}(\EQ_1,\EQ_2)$ which arise, for instance, in relation to various super-symmetric twistings \cite{ElliotYoo2018} e.g. Kapustin-Witten \cite{KapustinWitten2007} twists of $\mathcal{N}=4$ super Yang-Mills theory compactified along a Riemann surface $\Sigma$ whose moduli of solutions is a family of $1$-shifted symplectic stacks \cite{PTVV13} with two-parameter family of deformations over $\mathbb{C}[\lambda_1,\lambda_2],$
    $$\mathbb{R}\mathrm{Sol}_{\mathbb{C}}(\Sigma_{\mathrm{Hod}},[pt/G]\times X))_{\mathrm{Hod}}\simeq\mathrm{Maps}(\Sigma_{\mathrm{Hod}},BG)_{\mathrm{Hod}}.$$
For a fixed $\lambda_1\neq 0,\lambda_2=0,$ one recovers Hitchin's moduli space.
\end{example}
We prove singular support theorems 
for compact objects in the $\infty$-category $\mathsf{CAlg}(\D_X)$ of \emph{differential-graded commutative $\D$-algebras}, 
 in two related settings; via homotopical algebraic geometry over the sheaf $\mathscr{D}_X$  and via derived stacks fibered over the de Rham stack $X_{\DR}$ \cite{Sim09}.

To do so, we work in the homotopical setting of (stable) $\infty$-categories of commutative monoids in quasi-coherent and ind-coherent crystals on $X$ \cite{GR17b} \cite{GR14}, simply called \emph{derived $\D$-algebras}. Thus, a left (resp. right) derived $\D$-algebra is an object of
$$\mathsf{QCAlg}(X_{\DR}):=\mathsf{CAlg}\big(\mathsf{QCoh}(X_{\DR})\big)$$ (resp. $\mathsf{ICAlg}(X_{\DR}):=\mathsf{CAlg}\big(\mathsf{IndCoh}(X_{\DR})\big)$).

We now state the main results establishing a derived geometric jet-bundle formalism.

\subsection{Main results in Part I}
\label{ssec: Main results in Part I}
Let $X$ be a smooth projective scheme over $\mathbb{C}$. We will study the derived algebraic and analytic geometry of closed-subschemes of derived $\infty$-jet spaces over $X$, and their derived moduli stacks of solutions.

To state the main results, we now fix some notation and conventions, elaborated further in \S\S.\ref{ssec: Notations and Conventions}.
\begin{remark}[Notation]
\label{TerminologyConventions}
Unless otherwise explicitly stated, $X$ is always a smooth projective $\mathbb{C}$-scheme and all functors are implicitly derived, so e.g. $f^*,f_*,\otimes_{\mathcal{O}_X},\otimes_{\mathscr{D}_X}$ means $Lf^*,Rf_*,\otimes_{\mathcal{O}_X}^L,\otimes_{\mathscr{D}_X}^L$, and so on. All $\infty$-categories are denoted by $\mathsf{C}$ with mapping spaces $\mathrm{Maps}_{\mathsf{C}}.$ Inner-hom complexes are denoted by either of $Maps=\underline{Hom}$ so by Dold--Kan, the connective truncation $\tau^{\leq 0}$ of them agree with $\mathrm{Maps}_{\mathsf{C}}$. $\mathsf{PStk}_{\mathbb{C}}$ denotes the $\infty$-category of derived prestacks over $\mathbb{C}.$ Model categories are denoted $\mathbf{C}$ and their homotopy categories $Ho(\mathbf{C}).$ 
Given a symmetric monoidal model category $\mathbf{C}$ satisfying some assumptions \cite[Subsect~1.1]{CPTVV17}, the associated 
symmetric monoidal $\infty$-category $\mathsf{C}$ is obtained via homotopy coherent
nerve of the Dwyer--Kan localization $L(\mathbf{C})$ along weak equivalences. 

We always assume (local) geometricity for derived stacks, and by a 'derived affine $\mathbb{C}$-scheme $S=\mathsf{Spec}_{\mathbb{C}}(A)$', we mean one of finite-type, so $A$ is a connective commutative differential-graded algebra of finite presentation over $\mathbb{C}$. We denote by $\mathsf{dSt}_{\mathbb{C}}$ the (locally) geometric derived stacks over $\mathbb{C}.$ All \emph{derived} mapping stacks (inner hom) are denoted $\Map$. Let
 $Z\to S$ and $E\to Z$ be morphisms of prestacks. Weil--restrictions are denoted $\underline{\mathsf{Res}}_{Z/S}(E)$ or by the derived stack of sections \cite{KMMP}, denoted $\Sect_{S}(E/Z)$. Given a moduli functor $\mathscr{M}$, we reserve notation $\mathbb{R}\mathscr{M}$ to indicate derived enhancement--the main example being derived solutions in \eqref{RSolIntro} below.
\end{remark}

Via Weil restriction along the canonical map $q:=q_{\DR}^X:X\rightarrow X_{\DR},$
we study moduli stacks of flat sections which for $q^*\mathsf{Res}_{X/X_{\DR}}(E),$ give solutions of the `universal' (co-free) PDE, $\mathsf{Jets}_{X}^{\infty}(E):=q^*q_*E,$. 
We call the composite $q^*\circ q_*$ the \emph{jet-comonad}.

Derived Weil-restriction is the $\infty$-functor $q_{\DR,*}$ right adjoint to the base-change $\infty$-functor $q_{\DR}^*,$
\begin{equation}
\label{eqn: Derived Weil Adjunction}
q_{\DR}^*:\mathsf{PStk}_{X_{\DR}}\rightleftarrows\mathsf{PStk}_{X}:q_{\DR,*}=:\J,
\end{equation} 
where $\PS_{X}:=\PS_{/X}$ denotes the slice $\infty$-category, whose objects are called $X$-prestacks.

The main objects of study in this paper are as follows; see Def. \ref{definition: Derived NLPDE System Definition}.
\begin{definition}
Let $E\to X$ be an $X$-prestack. A \emph{generalized PDE} imposed on sections of $E\rightarrow X$ is a closed $X_{\DR}$-substack\footnote{In other words, a subcrystal of derived prestacks \cite{Kry2026}.} $Z\rightarrow \J(E)$. 
\end{definition}
Given a generalized PDE $\EQ\subset q_{\DR*}E,$ denote by $\EQ_X:=q^*\EQ\subset \mathsf{Jets}_X^{\infty}(E)$, the corresponding $X$-prestack. 
The moduli stack of solutions $\mathbb{R}\mathrm{Sol}(\EQ)$, is given by Weil-restriction (see (\ref{eqn:Derived solution stack})), along the canonical map $X_{\DR}\rightarrow \mathrm{pt}=\mathrm{Spec}(\mathbb{C}),$
\begin{equation}
\label{RSolIntro}
\RS(\EQ):=\Res_{X_{\DR}/\mathrm{pt}}(\EQ)\simeq \Map_{/X_{\DR}}(X_{\DR},\EQ).
\end{equation}

The following representability and descent results are central to establishing a derived jet-bundle formalism. They are used frequently in computations e.g. of higher homotopy groups $\pi_i$ of derived jet sheaves, and computing the Postnikov truncations of derived PDEs.
\begin{theorem}[cf. Propositions \ref{prop: Geometry of Derived Jets1}, \ref{prop: Geometry of Derived Jets2} and \ref{prop: Geometry of Derived Jets3}]
\label{kJetPropertiesIntro}
Fix $k\in \mathbb{Z}_{\geq 0}$ and let $S=\mathsf{Spec}_{\mathbb{C}}(A)$ be an affine derived $\mathbb{C}$-scheme.
\begin{enumerate}
    \item[(i.)] 
The functor $\underline{\mathsf{Jet}}^{k}(S):\mathsf{dAff}_{\mathbb{C}}^{\mathrm{op}}\rightarrow \mathsf{Spc},$ is corepresented by an object $\mathbb{L}\mathcal{J}^k(A)$. Similarly, $\underline{\mathsf{Jet}}^{\infty}$ is corepresentable by the homotopy-colimit of $\mathbb{L}\mathcal{J}^k(A)$.

    \item[(ii.)] The derived jet-functors $\underline{\mathsf{Jet}}^k(-),\underline{\mathsf{Jet}}^{\infty}(-)$ are compatible with Zariski localization and \'etale morphisms.    
\end{enumerate}
\end{theorem}
We establish functorialities of jet-comonads in this setting for a suitable sheaf-theory $\mathsf{Shv}(X)$ on $X$, which in applications we take to be $\mathsf{QCoh},\mathsf{IndCoh}$ \cite{QCoh,IndCoh}. In the sequel \cite[Prop.~5.3,5.5, and 5.7]{KSY2}, a more general statement is proven via the Barr--Beck--Lurie theorem \cite[Thm.4.7.3.5]{Lur17}.
\begin{proposition}[cf. Proposition \ref{prop: CoalgebraBody}]
\label{prop: CoalgebraIntro}
    Let $X,Y$ be smooth projective $\mathbb{C}$-schemes and put $q_{\DR}^X:X\to X_{\DR},q_{\DR}^Y:Y\to Y_{\DR}.$ :
\begin{enumerate}
\item There is an equivalence $\mathsf{Mod}_{q_{\DR,!}q_{\DR^!}}(\mathsf{Shv}(X))\simeq \mathsf{CoMod}_{q_{\DR}^*q_{\DR,*}}(\mathsf{Shv}(X)).$
\item For any morphism  $f:X\to Y$ there is functor 
$$f_{\mathrm{coAlg}}^*(-):\mathsf{CoAlg}_{q_{\DR}^{Y,*}q_{\DR,*}^Y}\big(\mathsf{Shv}(Y)\big)\to \mathsf{CoAlg}_{q_{\DR}^{X,*}q_{\DR,*}^X}\big(\mathsf{Shv}(X)\big).$$
\end{enumerate}
\end{proposition}
While we will not need greater levels of generality, let us mention that the conditions on $X,Y$ being smooth can be relaxed; the result holds for more eventually coconnective derived schemes locally of finite type over $\mathbb{C}.$

In Sect.~\ref{sec: Derived dR-NLPDES} we prove several important results about derived jet functors in derived algebraic geometry.
\begin{theorem}[cf. Propositions \ref{prop: AffXsoAffXDR} and \ref{prop: Def but not laft proposition}]
\label{thm: Jet Properties Intro}
Let $X$ be a smooth projective
$\mathbb{C}$-scheme.
\begin{enumerate}
    \item[(i.)] 
    Let $Z_X$ be a derived $\mathbb{C}$-scheme affine over $X$. Then the de Rham $\infty$-jet stack $\J(Z_X):=q_{\DR,*}(Z_X)$ is affine over $X_{\DR}.$

    \item[(ii.)] Let $E\to X$ be a laft prestack. Suppose $E$ admits deformation theory relative to $X$. Then $\J(E) \to X_{\DR}$ admits deformation theory.
\end{enumerate}
\end{theorem}

In Thm. \ref{thm: Jet Properties Intro} (ii.), $\J(E)$ is generally not laft. Moreover, we could have stated $\J(E)\to X_{\DR}$ admits deformation theory \emph{relative} to $X_{\DR}$ since the deformation theory of $X_{\DR}$ is trivial \cite[Lem.~2.1.10]{CPTVV17}.
\begin{remark}[Notation]
Nevertheless, we will write\footnote{The same conventions are implemented in the sequel work \cite{KSY2}.} $\mathbb{L}_{Z'/X_{\DR}}$ explicitly to indicate an object $Z'$ lives over $X_{\DR}.$ This is because the associated crystal $Z_X:=q^*Z'\to X,q:X\to X_{\DR}$ will live over $X$ (see diagram \eqref{eqn: Jet pb}).
\end{remark}

We prove representability and establish finiteness hypothesis for derived stacks of solutions \eqref{RSolIntro} 
of a derived differential system $\EQ\to X_{\DR}.$
In Prop.\ref{Derived Sections are Derived Flat Sections of Jet Space} we prove an equivalence between the derived stack of flat sections $\RS$ of  $q_{\DR,*}(\EQ_X)$ over $X_{\DR}.$ and the derived stack of all sections $\Sect(\EQ_X/X)$ of $\EQ_X$ over $X.$

Following \cite[Def.2.1, Prop. 2.2]{toenvaquie2007}, we introduce the notion of being \emph{homotopically finitely presented over} $\mathscr{D}_X$ (or relative to $X_{\DR})$ via categorical compactness conditions. We prove a series of technical results concerning this property. 

Most importantly, 
we obtain a relation between $\D$-afp algebras $\A$ and corresponding finiteness hypothesis on spaces of flat sections $\RS(\EQ)$ associated with the corresponding $X_{\DR}$-spaces given by the relative spectrum $Z=\Spec_{X_{\DR}}(\A)$ (see \cite[Prop. 2.5.12, Cons. 2.5.13]{Lurie2018}). 

The property of being $\D$-afp is preserved under pullback by smooth morphisms of projective $\mathbb{C}$-schemes, which plays an important role in establishing the Cauchy-type theorem for $\D$-afp algebras (Thm. \ref{NonLinCKKThmIntro} below).

\begin{theorem}[cf. Propositions \ref{prop: RSol is affine}, \ref{Laft Descent}, \ref{prop: QCAlgDRpb} and \ref{prop: D-afp means RSolfinite}]
\label{AFPTheoremIntro}
Let $X$ be a smooth projective $\mathbb{C}$-scheme.
    \begin{enumerate}
\item[1.]
Consider $\EQ_X\to X$ as in Thm.\ref{thm: Jet Properties Intro} (i.).
Then $\RS(q_{\DR,*}Z_X)$, given by \eqref{RSolIntro} is affine.

\item[2.] 
Let $S$ be an eventually coconnective derived $\mathbb{C}$-scheme and $Z\rightarrow S_{\DR}$ a derived Artin stack. Then $\RS(\EQ)$ is a derived stack locally almost of finite type over $\mathbb{C}$.

\item[3.] Let $f:Y\to X$ be smooth morphism, with $\EuScript{A} \in \mathsf{CAlg}(\mathsf{Mod}_{\D_X})$ be a $\D_X$-afp commutative algebra. Then the induced $\D$-algebra by pullback $\EuScript{A}_Y:=f_{\DR}^{\mathsf{QCAlg},*}(\EuScript{A})\in \mathsf{CAlg}(\mathsf{Mod}_{\D_Y})$ is $\D_Y$-afp.

\item[4.]
    Suppose $\EuScript{A}\in \mathsf{QCAlg}(X_{\DR})$ is $\D_X$-afp and bounded. Then $\mathbb{R}\mathrm{Sol}_{\D}(\EuScript{A})$ is a derived stack locally of finite presentation over $\mathbb{C}$. Moreover, if $\A\simeq \mathrm{Sym}(\M)$ with $\M$ a compact generator of $\mathsf{QCoh}(\D_X),$ then $\A$ is compact and $\mathbb{R}\mathrm{Sol}_{\D}(\A)$ is representable by a derived affine $\mathbb{C}$-scheme.
    \end{enumerate}
\end{theorem}
Roughly speaking, we will then have existence of global (and dualizable) cotangent complexes $\mathbb{L}_{\mathbb{R}\mathrm{Sol}(\EQ)}$ and $\mathbb{L}_{Z}.$ 
With finiteness hypothesis and geometricity assumptions there exists global cotangent complexes which may be computed solution wise and which play the role of \emph{derived linearization/deformation complexes} of the derived differential system.

\subsubsection{Linearization}
We present a construction \ref{cons: Solutions} allowing for derived moduli spaces of solutions for arbitrary $\EQ\rightarrow X_{\DR}$ as well as more general spaces $\mathbb{R}\mathrm{Sol}(\EQ_1,\EQ_2)$ of solutions with coefficients. They come with natural universal families i.e. given a test affine scheme $T$ consider a $T$-point $u_T:T\to \RS(\EQ),$ corresponding to a solution, let $ev:\RS(\EQ)\times X_{\DR}\to Z$, be natural evaluation. Then there is a canonical morphism,
\begin{equation}
    \label{evu_T}
ev_{u_T}:T\times X_{\DR}\xrightarrow{u_T\times id}\RS(\EQ)\times X_{\DR}\xrightarrow{ev}Z.
\end{equation}

Following the computation of cotangent complexes of Weil restrictions given in \cite[\S.~19.1.4]{Lurie2018}, we compute the $\mathscr{D}_X$-(co)tangent complexes and their related linearizations along infinite jets of solutions.

\begin{theorem}[cf. Theorem \ref{thm: RSol is laft-def}, Proposition \ref{T and D-T equivalence} and Proposition \ref{Linearization D-Module Proposition}]
\label{theorem: TRSolIntro}
Let $X$ be a smooth proper $\mathbb{C}$-scheme of dimension $d$ with canonical map $q_{\DR}:X\to X_{\DR}.$ Then, the following statements hold:
\begin{itemize}
    \item 
Let $T$ be a test affine scheme and consider a $T$-point $u_T:T\to \RS(\EQ),$ i.e. a solution. Let $ev:\RS(\EQ)\times X_{\DR}\to Z$ be the canonical evaluation map.
There exists an object $\Theta(\EQ/X_{\DR})_{u_T}\simeq ev_{u_T}^*T_Z\otimes\omega_X[-2d]$ in sheaves on $T\times X_{\DR},$ where $ev_{u_T}$ is \eqref{evu_T}, such that
$$T_{\RS(\EQ)}\simeq (\mathrm{Id}\times p_{\DR}^X)_{*}(\Theta(\EQ/X_{\DR})).$$ 

\item Let $E$ be a laft-def prestack over $X,$ with $X_{\DR}$-space $Z:=q_{\DR,*}E.$ Let $u_T:T\times X\to E,$ be the $T$-parameterized section corresponding to $j_{\infty}(u_T):T\times X_{\DR}\to q_{\DR,*}E$ and put $u_T:T\times X\to q_{\DR,*}^XE\times_{X_{\DR}}X\xrightarrow{ev_{u_T}}E.$
Then, $$\mathbb{L}_{q_{\DR,*}E/X_{\DR},j_{\infty}(u_T)}\simeq (\mathrm{Id}_T\times q_{\DR}^X)_{\sharp}^{\mathrm{Pro}}\circ u_T^{\sharp}\mathbb{L}_{E/X}.$$
\end{itemize}
\end{theorem}

In $\D$-geometry it is necessary to impose supplementary finiteness conditions and for this one may work point-wise i.e with infinitesimally cartesian/cohesive derived stacks $E$ whose tangent complex $\mathbb{T}_xE$ exists and is perfect over $R$, for any point $x\in E(R)$, where $R$ is any $k$-field. Following \cite[Def. 2.1.5,2.1.6]{CPTVV17}, we call them \emph{formally good}. 
\begin{example}
An example to keep in mind concerns $(\lambda=1)$ deformations of $\RS^{\lambda}(\EQ)$ as in (\ref{eqn: Hodge Deformation}).
On the one hand, given a solution $\varphi=(P,A)$, with $P$ a principle $G$-bundle with flat connection $A,$  
$\mathbb{T}_{\varphi}\RS^{\lambda=1}(\EQ)\simeq \big(\Omega^{\bullet}(X,\mathfrak{g}_P)[1],d_A\big),$ is the de Rham complex with coefficients in $\mathfrak{g}_P.$ On the other hand, its singular support encodes infinitesimal symmetries $\sigma$ of the connection Example \ref{ex: Sing of LocSys}.
  
The tangent complex is the underlying cochain complex of a deformed dg Lie
algebra, whose cohomology depends on $A$. The connection $d+A$ is irreducible (resp. reducible) if $H^{-1}=0,$
as the stabilizer is trivial or discrete (resp. if there is nontrivial $H^{-1}$). Note that the cohomology 
group $H^{-1}$
is the Lie algebra of the centralizer of the holonomy group $\mathrm{Hol}(A)$ of $A.$ Thus $H^{-1}$
is trivial when the centralizer is discrete.
\end{example}
Finiteness conditions in $\mathscr{D}$-geometry are imposed solution-wise.
\begin{remark}
Looking to future works \cite{KSY2}, and applications to physical theories e.g. gauge theories, possibly coupled to fermions or a scalar theory, we note
the underlying graded vector space of fields is unchanged, no matter the choice of
solution when we linearize. What changes is the shifted $L_{\infty}$--brackets (see e.g. \cite{CG16,CG16b}) ). Hence the tangent complex at any solution
consists of the fixed graded vector space of fields but with a differential depending on the solution. 
\end{remark}

\subsection{Main results in Part II}
\label{ssec: Main results in Part II}
We now state the main results of this paper. They rely on the general formalism of (formal) derived algebraic geometry over $X_{\DR}$, when $X$ is a smooth projective $\mathbb{C}$-scheme developed in Part I.
Using this language, we may properly formulate and prove the abstract propagation of solution singularities result (cf. (\ref{eqn: LPDE})) and the related CKK-theorem.

The first result treats propagation and Cauchy problems for nonlinear PDEs with holomorphic coefficients in the complex domain where the initial data may have singularities along a regular hypersurface on the initial surface.\footnote{In the linear setting Y. Hamada \cite{Hamada} shows that by assuming the characteristic
surfaces issuing from the points of singularity in the initial data are simple, the singularity in the initial data is propagated along these surfaces. If the initial data have at most poles (resp.
essential singularities), the solution in general has at most poles (resp.
essential singularities) on these surfaces.} To reach it, we overcome the challenges of: (a) introducing $\D$-module microlocalization in the non-linear setting § \ref{sec: Algebraic D-Geometry}, (b) developing a geometric theory of solution spaces for generalized nonlinear PDEs in the presence of symmetries via their functor of points §\ref{sec: Derived D-NLPDES} and §\ref{sec: Derived dR-NLPDES} and (c) defining their cohomological singular supports, here the $1$-microcharacteristic variety $\mathsf{Ch}_{\D}^1(\mathcal{A})$ (§§\ref{sssec: D-Geometric Microcharacteristics}).

\begin{theorem}[cf. Theorem \ref{MainTheoremPropagationLinearizedBody}, Corollary \ref{NonlinPropCorollaryBody}]
\label{MainTheoremPropagationLinearizedIntro}
Let $M$ be a real analytic manifold with complexification $X$. Let $\mathcal{B}$ be a derived $\D$-algebra on $X$ which is $\D_X$-afp and bounded, with perfect tangent complex $\mathbb{T}_{\mathcal{B}}.$ 
Let $V\subset T^*X$ be a closed subset and $T_VT^*X$ the normal bundle. 
\begin{itemize}
    \item[1.] Then the $1$-micro-characteristic variety $\mathsf{Ch}_{\D,V}^1(\mathcal{B})$ is a conic subset of $\mathsf{Sol}(\mathcal{A})\times_X T_VT^*X.$
    \item[2.] If $V=T_X^*X,$ then for all (infinite jets) of solutions $\varphi$ of $\mathcal{S}ol_X(\mathcal{B}_X)$, by linearization (Prop. \ref{Linearization Lemma} and § \ref{sec: Derived Linearization and the Equivariant Loop Stack}) and considering base-change of microfunctions from $\mathcal{O}_X[\D_X]$-modules to $\mathcal{B}_X[\mathcal{E}_X]$-modules, one has that
\begin{equation}
    \label{eqn: NLPDE Propagation}
\mathrm{supp}\big(R\mathcal{S}ol_{\mathcal{B}_X[\mathcal{E}_X]}\big(_{\mathcal{B}}\mu(\mathbb{T}_{\mathcal{B}}^{\ell}),\mathrm{ind}_{\mathcal{B}_X[\mathcal{E}_X]}(\mathscr{C}_M)\big)\subset \mathrm{Char}\big(_{\mathcal{B}}\mu \mathbb{T}_{\mathcal{B},\varphi}^{\ell}\big)\cap T_M^*X.
\end{equation}
    \item[3.] In particular, for the linearization $\D$-module of a fixed $n$-coconnective truncation,
assume $\psi$ a $C^1$ function on $X$ is such that $d\psi(x_0)\notin \mathrm{Char}(\varphi^*\mathbb{T}_{\mathsf{Sol}(\mathcal{B})}).$ Then
$$\mathcal{E}xt_{\D_X}^j\big(\varphi^*\mathbb{T}_{\mathsf{Sol}(\mathcal{B})},R\Gamma_{\{\varphi\geq 0\}}(\mathcal{O}_X)\big)_{x_0}=0.$$
\end{itemize} 
\end{theorem}

It is necessary that $\mathbb{T}_{\mathcal{B}}$ may be written as a complex of free $\mathcal{B}[\D_X]$-modules of finite rank (Proposition \ref{HofpDAlgTangentComplex}), or that $\mathcal{B}$ is compact. In some cases, suffices to consider free algebras $\Free(\mathcal{M})$ on compact $\D$-modules for which pull-back along $\varphi$ presents $\mathbb{T}_{\mathsf{Sol}_X
(\mathcal{B})}$ as a compact $\D$-module. In particular from Proposition \ref{Tangent Good Filtration}, we have an induced good filtration (over its classical locus).

One then has a $\D$-algebraic analog \cite{Paugam2022} of Nirenberg's formulation of the classical real analytic non-linear Cauchy problem \cite{Nirenberg}.
Given a smooth projective $\mathbb{C}$-scheme or complex manifold $Y$, let $C_V(S)$ denote the normal cone to $S\subset T^*Y$ along a subset $V$ (see §§ \ref{ssec: Notations and Conventions}).
\begin{theorem}[Theorem \ref{theorem: DNLCKK}]
\label{MainTheoremCKKLinearizedIntro}
    Let $f:X\rightarrow Y$ be a morphism of complex analytic manifolds and $\mathcal{B}_Y$ a derived $\D$-algebra which is $\D_Y$-afp and bounded.
Suppose that its solution-wise characteristic variety satisfies
\begin{equation}
    \label{NCIntroTheoremCondition}
\mathring{\pi}_X^{-1}(p)\cap C_{T_Y^*Y}\big(\mathrm{Char}_{T_Y^*Y}^1(\varphi^*\mathbb{T}_{\mathcal{B}})\big)=\emptyset,
\end{equation}
for all $p\in T_Y^*Y$ with $\mathring{\pi}_X:\mathring{T}_X^*Y\times_Y T_Y^*Y\rightarrow T_Y^*Y$, where $\mathring{T}_X^*Y$ denotes the removal of the zero-section. Then the natural morphism
$$f^{-1}\mathbb{R}\underline{\mathrm{Sol}}_{Y}(\mathcal{B}_Y)\rightarrow \mathbb{R}\underline{\mathrm{Sol}}_X(Lf^*\mathcal{B}_Y),$$
is an equivalence of spaces.
\end{theorem}

Let $P(z,D)f=g,$ be an arbitrary system of linear differential equations with analytic coefficients. Theorems \ref{MainTheoremPropagationLinearizedIntro} and \ref{MainTheoremCKKLinearizedIntro} vastly extend the fact that there holds a theorem of existence and uniqueness similar to the usual Cauchy–Kovalevskaya theorem, concerning the existence and uniqueness of the solution in the space of convergent power series, satisfying certain initial conditions and the differential compatibility condition $Q(z,D)g=0.$ 
This is a \emph{solvability} condition (cf, \cite{kashiwara1970algebraic},\cite{Quillen1964}) and is given in terms of algebraic conditions on the $\mathscr{D}_X$-module $\mathcal{M}_P$ determined by $P(z,D).$

We prove a non-linear version in the context of geometry relative to de Rham stacks. 
Let $f:X\to Y$ be a morphism of smooth proper schemes over $\mathbb{C}$. We let $\mathsf{PStk}_{/Y_{\DR}}^{\D-\text{fin}}\subset \mathsf{PStk}_{/Y_{\DR}}$ denote $\D$-finitary $Y_{\DR}$-prestacks and $\mathsf{PStk}_{/Y_{\DR},f}^{\D-\text{fin}}\subset \mathsf{PStk}_{/Y_{\DR}}^{\D-\text{fin}},$
those for which $f$ is non-characteristic, in the sense of condition \eqref{NCIntroTheoremCondition} of Theorem \ref{MainTheoremCKKLinearizedIntro}.
\begin{theorem}[cf. Theorem \ref{thm: RSol is laft-def}, Theorem \ref{theorem: DAnCKK}]
\label{NonLinCKKThmIntro}

Let $f:X\rightarrow Y$ be a morphism of complex manifolds and consider the restriction of the induced morphism $f_{\DR}^*:\PS_{/Y_{\DR}}\rightarrow \PS_{/X_{\DR}}$ to $\PS_{/Y_{\DR},f}^{\D-\text{fin}}$ (denoted the same). Then the diagram
\begin{equation}
\label{eqn: CKK Commutes}
\begin{tikzcd}[column sep=6em]
\PS_{/Y_{\DR},f}^{\D-\text{fin}} \arrow[d,"f_{\DR}^{\star}"] \arrow[r,"\mathbb{R}\underline{\mathrm{Sol}}_Y(-)"] & \PS_{/Y_{\DR}}    \arrow[d,"f_{\DR}^{\star}"]
\\
\PS_{/X_{\DR}}\arrow[r,"\RS(-)"] & \PS_{/X_{\DR}}.
\end{tikzcd}
\end{equation}
is commutative.

\end{theorem}

The tools developed in this work are necessary to formulate similar results in wider contexts that treat singular non-linear Cauchy problems:
\begin{equation}
    \label{eqn: NLCauchyProblem}
\begin{cases}
    \frac{\partial^L}{\partial t^L}u^{\alpha}=\mathsf{F}\big(t,x,u^{\alpha},\partial_x^{\sigma}\partial_t^ju^{\alpha}\big),\hspace{2mm} |\sigma|+j\leq L,j<L
    \\
\frac{\partial^k}{\partial t^k}u^{\alpha}\big|_{t=0}=\Phi_k(x),\hspace{2mm} k=0,1,\ldots,m-1,
\end{cases}
\end{equation}
imposed on complex analytic functions $u^{\alpha}(t,x),\alpha=1,\ldots,m$ where  
for $x\in U\subset \mathbb{C}^n$ where $U$ is an open subset, and $t\in \mathbb{R}$, where $\mathsf{F}$ depends on the derivatives of $u$ with respect to $x,t$ of order $\leq L,$ and where we allow $\Phi_k(x)$ to be singular functions, for instance coming from sheaves of ramified
holomorphic functions, or as ramified sections of logarithmic type.  We refer to \S.~\ref{ssec: Cauchy Problems} for further motivations and \S.~\ref{sssec: D-Geometric Microcharacteristics} for details.  

\subsubsection{Microlocalization}
Using Theorem \ref{theorem: TRSolIntro}, we endow our complexes with certain
(homotopy-coherent) lift via the deformation to the normal bundle construction \cite[\S.~9.2.]{GR17b} of the
formal moduli problem $X\rightarrow X_{\DR}.$ From the perspective of non-abelian Hodge theory
\cite{Sim09,SiNonabelian}, this is given by a lift to the Hodge stack $X_{\Hod}$.

Linearizing around a solution gives a complex of sheaves over $\mathcal{L}X$ by Koszul duality \cite{BenZviNad}. Lifting to $X_{\Hod}$ and passing to the special fiber gives a complex of sheaves on the cotangent stack $\mathsf{T}^*X$.
The combination of these steps will be called (derived) \emph{micro-linearization}; the resulting sheaf can be thought of as the associated graded object to a good filtration \cite[Def.1.2.3]{kashiwara1970algebraic} on the linearization sheaf, denoted by 
$$\Rmu(\EuScript{T}
_{Z,\varphi}):=\Rmu\big(j_{\infty}(\varphi_T)_{\mathrm{IndCoh}}^!(\mathbb{T}_{\EQ/X_{\DR}})\big).$$
\begin{remark}
In this setting, $X$ is allowed to be more general than a smooth proejctive $\mathbb{C}$-scheme. For instance, it can be: a QCA stack in the sense of \cite{DriGai2013}, a quasi-smooth derived $\mathbb{C}$-scheme \cite{KhanRydh2025}, or an ind-inf-scheme \cite[Ch. 3]{GR17b}. In particular, $X$ may be a derived scheme of finite type or an ind-scheme built by gluing laft schemes along closed embeddings or quotients of laft spaces by Lie algebroid actions.
\end{remark}
It is convenient to use equivariant sheaves on $\mathcal{L}X$, which by the HKR theorem identify with ind-coherent sheaves on $\mathsf{T}_X[-1]$ \cite{BenZviNad}.

In this setting, ind-coherent (filtered) Koszul duality for stacks \cite{ChenFiltered} provides an equivalence
\begin{equation}
    \label{KoszulEquivIntro}
\kappa:\mathsf{Mod}^{fil}(\D_X)\rightarrow \big(\mathsf{IndCoh}(\mathbb{T}_X[-1])^{\mathbb{G}_m}\simeq \mathsf{QCoh}(\mathbb{T}_X^*)^{\mathbb{G}_m}\big).\end{equation}
Via \eqref{KoszulEquivIntro}, we view $\mathbb{T}_{\mathcal{B},\varphi}$ in terms of $\mathbb{G}_m$-equivariant sheaves $\widetilde{\mathbb{T}}_{\mathcal{B},\varphi}$ over $\mathbb{T}_X^*$. Roughly, it is the analog of taking the associated graded object to a good filtration (microlocalization).

Singular support in this context, a measure of local constancy, is understood more earnestly as a statement on the perfectness of the linearization complex an ind-coherent sheaf \cite{AriGai2015}. 
Then, “non-characteristic directions” refer to points $(z,\xi)$ not contained within the singular support. 
\begin{theorem}[cf. Theorem \ref{MainProposition}, Theorem \ref{thm: Main Theorem} and Corollary \ref{Main Theorem Corollary}]
Consider a non-linear PDE $\EQ$ over a (possibly derived and eventually coconnective) scheme $X$ locally almost of finite type. The microlocalization of its linearization $\Rmu(\mathsf{T}
_{\EQ,\varphi})$ along a solution $\varphi$ defines an object of $\IC(\mathsf{T}^*X),$  
with Koszul dual $\IC(\mathcal{L}X)^{S^1}$. Its singular support defines the characteristic variety $Ch_{\varphi}(\EQ)$. If $\EQ$ is $\D$-finitary and admits deformation theory relative to quasi-smooth $X$, so does $\RS(\EQ)$ and its corresponding derived linearization sheaf is locally constant in non-characteristic directions.
\end{theorem}

\subsection{Overview of the paper}
\label{ssec: Overview}

The paper is divided into two parts. The first establishes foundations for homotopical geometry of jet-spaces, PDEs and their stacks of solutions. The second studies applications to propagation of singularity and initial-value problems. The main results are summarized as follows: 

\begin{enumerate}
\item[\hypertarget{(\textit{Part I.})}{\textit{Part I.}}] (\S.\ref{sec: Algebraic D-Geometry}--\ref{sec: Derived dR-NLPDES}): We study non-linear PDEs in a coordinate free and algebro-geometric language by encoding them as differentially generated ideals in $\D_X$-algebras of $\infty$-jets of functions. We study their associated prestacks from the functor of points perspective, which naturally leads to considering derived enhancements inspired by works of To\"en–Vezzosi \cite{TV2} and Lurie \cite{Lur09,Lurie2018}, put on solid ground in \cite{dBPP18}. We systematically study derived PDEs imposed in derived jet-bundles $\Jets(S)$ (Def.\ref{DerivedInfiniteJets}), where $S$ is a derived $\mathbb{C}$-scheme of finite presentation or a proper geometric derived stack locally almost finitely presented over $\mathbb{C}$ (see Rmk.\ref{TerminologyConventions} for conventions).

We define derived analogs of formally-integrable derived non-linear PDEs in terms of their infinite-prolongations (Def.\ref{definition: Derived NLPDE System Definition}), and introduce their derived stacks of solutions (Def.\ref{Definition: Derived Stack of Solutions Definition}). We briefly discuss finite-order analogs (Def.\ref{def: S-derPDEk}).
\begin{remark}
Throughout the paper, we make a continued effort to discuss both relations to classical facts in PDE geometry \cite{Krasilshchik1986} as well as the derived \emph{analytic} analogs of many of the foundational constructions, following \cite{HolsteinPorta2025} (see also, \cite{LurieDAGV}, \cite{PortaYuDerivedGAGA2019,PortaYuPRIMS}). A systematic treatment will appear elsewhere (see e.g. \cite{KryQS2026}).
\end{remark}

\begin{itemize}
\item We prove the representability of finite and infinite jet functors $\J(E)$ for derived schemes and prestacks $E$ (Thm.\ref{kJetPropertiesIntro}). We discuss their (co)-monadic description in terms of de Rham prestacks $X_{\DR}$ and prove a Barr--Beck--Lurie theorem for their $\infty$-functorialities as coalgebras (Prop.\ref{prop: CoalgebraIntro}).

\item We study Hodge deformations $X_{\Hod},$ construct derived moduli stacks of solutions $\RS(\EQ)$ and establish their finiteness and representability (Thm.\ref{thm: Jet Properties Intro}). The main class of derived PDEs admitting good finiteness hypothesis and deformation theory are called \emph{$\D$-almost finitely presented} (hereafter, $\mathscr{D}$-afp, see Def. \ref{DAfpDefinition}).
We prove a series of results characterizing these objects $\EQ\to X_{\DR}$ over smooth projective $\mathbb{C}$-schemes $X$, and their derived solution stacks $\RS(\EQ)$ (Thm.\ref{AFPTheoremIntro}).
\end{itemize}

\item[\hypertarget{(\textit{Part II.})}{\textit{Part II.}}]
(\S.\ref{sec: Derived Non-linear Microlocal Analysis}--\ref{sec: Cauchy Problems and Propagation of Solution Singularities}): We develop a theory of singular support and propagation for solutions to non-linear systems within derived geometry. This extends microlocal techniques of algebraic analysis \cite{kashiwara1970algebraic} to the non-linear setting. In this context, polynomial-type non-linearities become tractable, and many classical constructions admit microlocal reinterpretations via bicharacteristics \cite{BonyLocal,BonySymbol}:
\begin{itemize}
    \item We prove that upon derived linearization of a $\D_X$-afp PDE, its derived solution stack satisfies a propagation of singularity theorem (Thm. \ref{MainTheoremPropagationLinearizedIntro}) and thus a non--characteristic Cauchy--Kowaleskaya--Kashiwara theorem (Thm. \ref{MainTheoremCKKLinearizedIntro}).
    
   \item The Cauchy problem is a local question, so we decompose $f:X\to Y$ into a closed-immersion and a smooth morphism. Treating both cases separately and using a universal GAGA property of de Rham stacks $X_{\DR}$ and analytification of derived mapping spaces \cite{HolsteinPorta2025}, we prove the main result of this paper which establishes an \emph{analytic} non-linear derived CKK-theorem (Thm. \ref{NonLinCKKThmIntro}).

\end{itemize}

\end{enumerate}
In the final part of the article (§\ref{sec: Derived Linearization and the Equivariant Loop Stack}--\ref{sec: Derived Microlocalization and the Hodge Stack}), we introduce derived microlocalization using the deformation to the normal cone construction and the Hodge stack $X_{\Hod}$ of the space of independent variables $X$. We apply these tools to $\D_X$-prestacks with $\D_X$-deformation theory and prove several results establishing formal properties.
We make a comparison of the Kashiwara--Schapira singular support $\mathrm{SS}$ and the singular support $\mathsf{SingSupp}(\EuScript{F})$ of an ind-coherent sheaf $\EuScript{F}\in \IC(S)$ on a quasi--smooth derived scheme $S$ defined by Arinkin--Gaitsgory \cite{AriGai2015}.

Even though our methods are abstract and technical, questions about the propagation of solutions (or (analytic) continuation or extendability) can help to argue non-existence of solutions of a given class e.g. holomorphic, distributional etc. in a given domain, due to possible singularity formation in the behaviour of the solution along the boundary of the domain (see §§\ref{sssec: On the non-linear generalizations}). These types of questions have clear and immediate applications to important long-standing problems in various branches of mathematics, physics and engineering which have resisted approaches through classical means and may benefit from a new circle of ideas. This paper can therefore be seen as laying the mathematical foundations for novel applications using homotopical algebro-geometric methods. 
\medskip

\subsubsection{Relation to other works.}
\label{sssec: Relation to other works}
Many constructions directly apply to the setting of Clausen-Sholze's non-archimedean derived analytic geometry Clausen-Sholze \cite{CS3}. This applies to sub-analytic topoi and to rigid analytic spaces over non-archimedean fields $K$, via an `analytic de Rham stack' introduced recently \cite{Cam}. Thus a flexible framework for both derived analytic and derived algebraic geometry, inclusive of Archimedean and non-Archimedean contexts may be found.
Crucially, one may use overconvergent analogs of the usual nilradical consisting of ideals of uniformly nilpotent or overconvergently close to zero elements together with naturally arising analytic de Rham functors.  

This assists in circumventing some technical issues while also permitting the use of spaces of generalized functions e.g. tempered distributions. Note we can not do this with usual sheaf-theoretic constructions, as these spaces do not form sheaves in the usual sense, but only a suitable (sub-)analytic sense.

From the perspective of PDEs this leads one to, roughly speaking, replace the formal neighbourhood of the diagonal, rings of formal power series and spaces of algebraic jets with overconvergent neighbourhoods, overconvergent series and overconvergent jets. 
These new objects play an important role in studying existence of solutions that is both an algebraic and analytic problem, consisting of first finding formal power series solutions and then proving their convergence \cite{KryQS2026}.

\subsection*{Acknowledgments}
J. Kryczka would like to thank Chris Brav, Juan Esteban Rodríguez Camargo, Harrison Chen and Frédéric Paugam for helpful discussions and explanations on specialized topics. The idea to formulate non-linear CKK-theorem via derived geometry originated with F. Paugam, we thank him for discussions. The first author is supported by the Postdoctoral International Exchange Program of Beijing Municipal Human Resources and Social Security Bureau. The second author is supported by grants from Beijing Institute of Mathematical Sciences and Applications (BIMSA), the Beijing NSF BJNSF-IS24005, and the China National Science Foundation (NSFC) NSFC-RFIS program W2432008. He would also like to thank China's National Program of Overseas High Level Talent for generous support. Finally, he would like to thank NSF AI Institute for Artificial Intelligence and Fundamental Interactions at Massachusetts Institute of Technology (MIT) which is funded by the US NSF grant under Cooperative Agreement PHY-2019786.

\subsection{Notation and background}
\label{ssec: Notations and Conventions}
We briefly collect the basic features of $\D$-module theory and derived algebraic geometry we will use frequently throughout the work.
\subsubsection*{$\D$-module theory}
By a `manifold' $M$ we mean a real analytic manifold with complexification $X.$ 
If $p:E\rightarrow M$ is a finite-dimensional vector bundle, denote the removal of its zero-section by $\mathring{p}:\mathring{E}\rightarrow M.$ For a symplectic manifold we always denote the Hamiltonian isomorphism by $H:TX\simeq T^*X.$ 
Our conventions on analytic $\D$-modules follows \cite{schneiders1994,kashiwara1970algebraic,schapira19941,Schapira19942}.
\begin{remark}[Notation]
\label{notate: D-Module Conventions}
Let $\mathrm{Mod}(\D_X)$ the abelian category of left $\D_X$-modules and $\mathrm{Mod}(\D_X^{op})$ that of right $\D_X$-modules. They are equivalent by the usual twist by the invertible sheaf of top-forms $\Omega_X^n.$ Given a (left)
$\D$-module $\mathcal{M}$ its sheaf of solutions is $R\mathcal{S}ol_X(\mathcal{M}):=R\mathcal{H}om_{\D_X}(\mathcal{M},\mathcal{O}_X),$ 
where
$R\mathcal{H}om_{\D_X}(-,-):\mathrm{D}^b(\D_X)^{op}\times \mathrm{D}^b(\D_X)\rightarrow \mathrm{D}^b(\mathbb{C}_X).$

Other sheaves we will make use of are: Sato hyperfunctions $\mathscr{B}_M$ \cite{Sato1973}, complex-valued real analytic functions $\mathscr{A}_M$, microfunctions $\mathscr{C}_M$ and microdifferential operators $\mathcal{E}_X.$ 
For a coherent $\D_X$-module $\mathcal{M}$ its microlocalization is
$\mu\mathcal{M}:=\mathcal{E}_X\otimes_{\pi^{-1}\D_X}\pi^{-1}\mathcal{M},$ for $\pi:T^*X\rightarrow X$, which is coherent as an $\mathcal{E}_X$-module.
The characteristic variety of $\mathcal{M}$ is the support of its microlocalization $\mathrm{Char}(\mathcal{M}):=supp\big(\mu \mathcal{M}\big).$
Details can be found in \cite{KashiwaraSchapira1990}.
\end{remark}

We treat non-linear PDEs modulo their symmetries in a coordinate free way and by mainly working with $(\infty,1)$-categories to avoid model-dependent arguments.

\subsubsection*{Derived algebraic geometry}

\begin{remark}[Notation]
\begin{itemize}
    \item $\mathsf{Spc}$ is the symmetric monoidal $(\infty,1)$-category of spaces and for a field $k$ of characteristic zero $\mathsf{Vect}_k$ is the stable symmetric monoidal $(\infty,1)$-category of unbounded cochain complexes.
    \item  $\PS_k$ is the category of prestacks i.e. functors $\mathsf{Fun}\big(\mathsf{CAlg}_k,\mathsf{Spc}),$ where $\mathsf{CAlg}_k:=\mathsf{CAlg}(\mathsf{Vect}_k)$ is the $(\infty,1)$-category of commutative algebras in $\mathsf{Vect}_k.$

    \item For any $(\infty,1)$-category $\mathsf{C}$ and an operad $\mathcal{O},$ we denote by 
$\mathcal{O}-\mathsf{Alg}[\mathsf{C}],$
the category of $\mathcal{O}$-algebras in $\mathsf{C},$ e.g. $\mathsf{CAlg}_{k}:=\mathcal{C}omm-\mathsf{Alg}[\mathsf{Vect}_k].$
\end{itemize}
\end{remark}
\begin{remark}[Notation]
Here and throughout for some finite index set $I$, we let $\{T_i:i\in I\}$ denote a set of indeterminates. Thus, $\mathbb{C}[T_i:i\in I]$ is the polynomial ring in these variables.
\end{remark}
A sequence $(f_1,\ldots,f_N)$ of elements of $A$ determines a homomorphism $\mathbb{C}[T_1,\ldots,T_N]\to A,$ given by $T_i\mapsto f_i,i=1,\ldots,N.$
The \emph{Koszul complex} 
$K_A(f_1,\ldots,f_N)=\otimes_i[A\xrightarrow{f_i}A],$ is quasi-isomorphic to a derived tensor product,
$$K_A(\big<f\big>)\xrightarrow{\simeq} A\otimes_{\mathbb{C}[T_1,\ldots,T_n]}^{\mathbb{L}}\mathbb{C}[T_1,\ldots,T_n]/\big<T_1,\ldots,T_n\big>.$$
\begin{remark} Derived geometrically, the Koszul complex can be seen as the dg-ring of functions on a certain derived sub-scheme of $\mathrm{Spec}_{\mathbb{C}}(A).$
\end{remark}
Let $A_{\bullet}$ be a simplicial commutative ring over $\mathbb{C}$, or equivalently, a connective commutative differential-graded $\mathbb{C}$-algebra.
Fix $\big<f\big>:=(f_1,\ldots,f_n)\in A_{\bullet},$ i.e. a collection of points of the underlying space of $A_{\bullet},$ and consider the homotopy-pushout square in $\mathsf{Ring}_{\mathbb{C}}^{\Delta},$
\begin{equation}
    \label{HoPushOut1}
\begin{tikzcd}
    \mathbb{C}[T_1,\ldots,T_n]\arrow[d]\arrow[r,"\mathsf{f}"] & A_{\bullet}\arrow[d]
    \\
    \mathbb{C}[T_1,\ldots,T_n]/\big<T\big>\arrow[r] & A_{\bullet}\po \big<f\big>.
\end{tikzcd}
\end{equation}
Note that 
\begin{equation}
    \label{Pi0}
\pi_0\big(A_{\bullet}\po \big<f\big>\big)\simeq \pi_0(A_{\bullet})/\big<f\big>,
\end{equation}
where we denote again by $f_i$ the connected component in $\pi_0(A_{\bullet}).$

The definitions of quasi-smooth derived $\mathbb{C}$-schemes, quasi-smooth morphism of derived $\mathbb{C}$-schemes, and their properties can be found in \cite[Sect.2]{AriGai2015}.
We recall the definition in a manner rendering the local descriptions explicit. See also \cite{KhanRydh2025}.
\begin{definition}
\label{QuasiSmoothClImm}
    Let $i:\mathsf{Z}\to \mathsf{X}$ be a closed-immersion in $\mathsf{dSch}_{\mathbb{C}}$ (i.e. $i^{\mathrm{cl}}:\mathsf{Z}^{\mathrm{cl}}\to \mathsf{X}^{\mathrm{cl}},$ is an ordinary closed-immersion of $\mathbb{C}$-schemes). It is said to be a \textit{quasi-smooth closed-immersion} if Zariski locally on $\mathsf{X}$, there exists a function $f:\mathsf{X}\to \mathbf{A}_{\mathbb{C}}^n,$ and a homotopy Cartesian square,
\begin{equation}
    \label{QSDiagram}
    \begin{tikzcd}
\mathsf{Z}\arrow[d]\arrow[r,"i"] & \mathsf{X}\arrow[d,"f"]
    \\
    \{0\}\arrow[r] & \mathbb{A}_{\mathbb{C}}^n,\end{tikzcd}\end{equation}
    where $\{0\}:=\mathrm{Spec}\big(\mathbb{C}[T_1,\ldots,T_n]/\big<T\big>\big)$ is the inclusion of the origin the $n$--dimensional affine space $\mathbb{A}_{\mathbb{C}}^n:=\mathrm{Spec}(\mathbb{C}[T_1,\ldots,T_n])$ over $\mathbb{C}.$
\end{definition}
This means that, locally on $i:\mathsf{Z}\to \mathsf{X}$ is of the form 
\begin{equation}
    \label{LocalQSIMModel}
\mathsf{Spec}_{\mathbb{C}}\big(A_{\bullet}\po\big<f\big>\big)\hookrightarrow \mathsf{Spec}_{\mathbb{C}}(A_{\bullet}),
\end{equation}
for some $f_1,\ldots,f_n\in A_{\bullet}.$

We study pre-stacks in derived algebraic geometry admitting deformation theory i.e. we work within the full subcategory
$$\mathrm{PreStk}_{\mathrm{laft}\text{-}\mathrm{def}}\subseteq \mathrm{PreStk}_{\mathrm{laft}},$$
of laft prestacks admitting deformation theory \cite{GR17b}. We do not repeat the full definition here, but instead mention that $Z$ is said to admit deformation theory if it is convergent and sends all push-out squares
\[
\begin{tikzcd}
S_1\arrow[d]\arrow[r] & S_2 \arrow[d]
\\
S_1'\arrow[r] & S_2'
\end{tikzcd}\hspace{2mm}\text{ to pull-back squares }\hspace{1mm} \begin{tikzcd}
Z(S_2')\simeq\underline{\mathsf{Map}}(S_2',Z)\arrow[d] \arrow[r] & Z(S_2)\simeq \underline{\mathsf{Map}}(S_2,Z)\arrow[d]
\\
Z(S_1')\simeq \underline{\mathsf{Map}}(S_1',Z)\arrow[r] & Z(S_1)\simeq\underline{\mathsf{Map}}(S_1,Z),
\end{tikzcd}
\]
where $S_1\rightarrow S_1'$ is a nilipotent embedding and where the pull-back is taken in $\mathrm{Spc}.$

\subsubsection{Pro-cotangents}
We study infinitesimal behaviour of $Z$ in the form of properties of maps $\underline{\mathsf{Map}}(S',Z)\rightarrow \underline{\mathsf{Map}}(S,Z)$ for nilpotent embeddings $S\rightarrow S'$ using a particular linear object- the \emph{pro-cotangent complex}.
Let $S=\mathsf{Spec}(A)$ and consider an $S$-point $g:S\rightarrow Z$ and a nilpotent embedding
$$S\rightarrow S[\EuScript{M}]:=\mathsf{Spec}(A\oplus \EuScript{M}),\hspace{2mm} \EuScript{M}:=\Gamma(S,\mathcal{I}),\hspace{2mm}\mathcal{I}\in \mathsf{QCoh}(S)^{\leq 0}.$$
One has that the functor
$$\mathsf{QCoh}(S)^{\leq 0}\rightarrow \mathrm{Spc},\hspace{2mm}\mathcal{I}\mapsto \underline{\mathsf{Map}}_{S/}(S[\EuScript{M}],Z),$$
is representable by $T_gZ$ in $\mathsf{Pro}\big(\mathsf{QCoh}(S)^-\big).$ More precisely,
$$T^*(\EQ)\in \underset{S\in\mathsf{dAff}_{/Z}}{\mathrm{lim}}\hspace{1mm}\mathrm{Pro}\big(\mathsf{QCoh}(S)^-\big),$$
where for all $g:S\rightarrow Z$ one has $g^{\sharp}\big(T^*(\EQ)\big)\in\mathrm{Pro}\big(\mathsf{QCoh}(S)^-\big)$ with values on $\EuScript{M}\in\mathsf{QCoh}(S)^{\leq 0}$ given by
$$g^{\sharp}\big(T^*(\EQ)\big)\simeq \underline{\mathsf{Map}}_{S/}\big(S[\EuScript{M}],Z\big).$$
\section{Microlocal analysis and sheaf propagation}
\label{sec: Microlocal Analysis and Sheaf Propagation}

Microlocal analysis \cite{Sato1973,KashiwaraMicro,KashiwaraSchapira1990} studies `local' phenomena occurring on manifolds $X$ via `global' phenomena in
their cotangent bundle. Most notably for us, the study of the singularities of solutions of a PDE on $X$ in terms of their wavefront set in $T^*X.$ In this direction the central object of study is the microsupport $\mathrm{SS}(\mathcal{F}^{\bullet})$ of a complex of sheaves that measures the codirections in which $\mathcal{F}^{\bullet}$ fails to be locally constant (the so-called `microlocal' directions). 
Local constancy of $\mathcal{F}^{\bullet}$ means $\mathrm{SS}(\mathcal{F}^{\bullet})$ is contained in the zero-section. Said differently, microsupports measure the directions where the propagation of sections of $\mathcal{F}^{\bullet}$ is obstructed.

There is an analogous notion of singular support $\mathsf{SS}(\mathcal{F}^{\bullet})$ of a sheaf in derived algebraic geometry, introduced by Arinkin-Gaitsgory \cite{AriGai2015}, that we study in §§ \ref{sssec: Singular Support Conditions}. It measures the failure of a sheaf to be a compact object, which may then be interpreted as a (non-)propagation statement. 
A comparison with Kashiwara-Schapira's microsupport is given in \S.\ref{sssec: Interpreting SS via Prop}.
 \subsection{Microlocal geometry}
 We follow \cite{KashiwaraSchapira1990} and recall basic definitions.
 
Let $f:X\rightarrow Y$ be a morphism of complex analytic manifolds, denote by $\pi:T^*X\rightarrow X$ the cotangent bundle with coordinates $(x;\xi)$ and consider the `microlocal' correspondence of $f$:
\begin{equation}
    \label{eqn: Microlocal correspondence}
    T^*X\xleftarrow{f_d}X\times_Y T^*Y\xrightarrow{f_{\pi}}T^*Y.
\end{equation}
Explicitly, $f_d\big(x;(f(x);\xi)\big)=(x;f^*\xi)$ and $f_{\pi}\big(x;(f(x);\xi)\big)=\big(f(x);\xi\big).$

\begin{definition}
    \label{Definition: NC}
    Let $f:X\rightarrow Y$ be a morphism of complex analytic manifolds and $\Lambda\subset T^*Y$ a closed conic subset. The map $f$ is \emph{non-characteristic for $\Lambda$} if 
    $f_{\pi}^{-1}(\Lambda)\cap T_X^*Y\subset X\times_Y T_Y^*Y,$ where $T_Y^*Y$ denotes the zero-section of $T^*Y.$
\end{definition}
Consider closed conic subsets $\Lambda_X\subset T^*X$ and $\Lambda_Y\subset T^*Y.$ If $f$ is proper on $\Lambda_X\cap T_X^*X$ then $f_{\pi}$ is proper on $f_d^{-1}(\Lambda_X).$ If $f_d$ is proper on $f_{\pi}^{-1}(\Lambda_Y),$ then $f$ is non-characteristic for $\Lambda_Y.$

Let $Y$ be a closed sub-manifold of $X$ with normal bundle $T_YX$, and conormal bundle $T_Y^*X.$ 
Here and throughout we denote by $C_{T_M^*X}(\Lambda)$ the Whitney cone along $T_M^*X$ to a close conic subset $\Lambda\subset T^*X$ (see \cite[Def. 4.1.1, Prop. 4.1.2]{KashiwaraSchapira1990}). It is the closed conic of $T_{T_M^*X}T^*X\simeq T^*T_M^*X$ describing the normal directions to $T_M^*X$ contained in $\Lambda.$

\subsubsection{Local description}
To be explicit, we now describe various bundles arising in microlocal geometry in local coordinates.

Let $f:N\rightarrow M$ be a morphism of real analytic manifolds for which $Y=N_{\mathbb{C}},X=M_{\mathbb{C}}$ and $f_{\mathbb{C}}:Y\rightarrow X$ are complexifications. Let $\pi_M:T_M^*X\rightarrow M$ be the conormal bundle to $M$ in $X$. The microlocal correspondence (\ref{eqn: Microlocal correspondence}) for $f$ is then
\begin{equation}
    \label{eqn: Microlocal correspondence complexified}
    T_N^*Y\xleftarrow{f_d}N\times_M T_M^*X\xrightarrow{f_{\pi}}T_M^*X.
\end{equation}

\begin{example}
Let $(x,y)=(x_1,\ldots,x_p,x_{p+1},\ldots,x_n)$ be coordinates in $M,$ with $(x;\xi)=(x,y,;\xi dx,\eta dy)\in T^*M$ and consider $N=\{y=0\}\hookrightarrow M.$
Then
$T_N^*M=\{(x,y;\xi,\eta)\in T^*M| y=\xi=0\}.$
Moreover, $$N\times_M T^*M=\{(x,y;\xi,\eta)\in T^*M| y=0\}.$$ There is an embedding $T^*N\hookrightarrow T_{T_N^*M}T^*M$ give in local coordinates $(x,v\partial_y,w\partial_{\xi},\eta dy)\in T_{T_N^*M}T^*M,$ by $(x,y;\xi,\eta)\mapsto (x,0;\xi,0).$
\end{example}

Let $Y$ be a closed submanifold of $X.$ Let $(x)=(x',x'')$ be local coordinates on $X$ such that $Y=\{x''=0\}.$
Then there are isomorphisms:
\begin{equation}
\label{eqn: Useful Hamiltonian Isomorphism}
T^*T_YX\xrightarrow{\simeq}T^*T_Y^*X\xrightarrow{\simeq}T_{T_Y^*X}T^*X\xrightarrow{\simeq}T^*X,
\end{equation}
given by
$(x',x'';\xi',\xi'')\leftrightarrow (x',\xi'';\xi',-x'')\leftrightarrow(x',\xi'';x'',\xi')\leftrightarrow (x',x'';\xi',\xi'').$

Locally, if $(z;\zeta)=(x+iy;\xi+i\eta)\in T^*X$ one has $T_M^*X=\{y=0,\xi=0\}$ if $M=\{y=0\}$ and isomorphisms
\begin{equation}
\label{eqn: Vertical Equalities}
   \invert{T^*T_MX} { \equalto{{(x,iy;\xi,i\eta)}}{(x,y;\xi,-\eta)}} \xrightarrow{\simeq}\invert{T^*T_M^*X} { \equalto{{(x,i\eta;\xi,-iy)}}{(x,\eta;\xi,-y)}} \xrightarrow{\simeq} \invert{T_{T_M^*X}T^*X} { \equalto{{(x,i\eta;iy,\xi)}}{(x,-\eta;y,\xi)}} \xrightarrow{\simeq}\invert{T^*X} { \equalto{{(z,;\zeta)}}{(z,y;\xi,-\eta)}} 
\end{equation}
where the bottom row are elements of the identified manifolds
$$T^*(T_MX^{\mathbb{R}})\xrightarrow{\simeq} T^*(T_M^*X^{\mathbb{R}})\xrightarrow{\simeq}T_{T_M^*X^{\mathbb{R}}}T^*X^{\mathbb{R}}\xrightarrow{\simeq}T^*X^{\mathbb{R}}.$$

\subsubsection{Microsupports}
Let $X$ be a complex analytic manifold. The microsupport is a tool to obtain global propagation results for (sections of) sheaves on $X.$
\begin{definition}
\label{Propagation Definition}
 A complex of coherent sheaves $\mathcal{F}^{\bullet}$ on $X$
\emph{propagates at $x_0\in X$ in the direction $p$} if for all $C^1$-functions $\varphi$ with $\varphi(x_0)=0$ and $d\varphi(x_0)=\xi_0$ the natural map
\begin{equation}
    \label{PropAtx}
\rho_j(x_0):\underset{U\ni x_0}{\varinjlim}\hspace{.5mm} H^j(U;\mathcal{F})\rightarrow \underset{U\ni x_0}{\varinjlim}\hspace{.5mm} H^j\big(U\cap \{\varphi<0\};\mathcal{F}\big),\end{equation}
is an isomorphism for all $j.$
\end{definition}

In other words, given a covector at a point understood as a hyperplane in the tangent space, we are studying whether a sheaf behaves locally constantly moving off this hypersurface, in the conormal direction. This interpretation generalizes the Cauchy-Kovalevskaya theorem for PDEs (see Theorem \ref{Thm: LinProp}). The isomorphism \eqref{PropAtx} stated propagation at $x_0\in X$ is equivalent to the condition
\begin{equation}
\label{eqn: ShvProp@x}
\big(R\Gamma_{\{\varphi\geq 0\}}(\mathcal{F}^{\bullet})\big)_{x_0}\simeq0.
\end{equation}

The microsupport is the collection of all points where there is no propagation i.e. the set of codirections of non-propagation, corresponding to points $p=(x_0;\xi_0)\in T^*X,$ where $\rho_j(x_0)$ is \emph{not} an isomorphism. Define the set of propagation of $\mathcal{F}^{\bullet},$
\begin{eqnarray*}
    P(\mathcal{F}^{\bullet}):=\{(x;\xi)\in T^*X |  &\forall& \text{open neighbourhoods } U \text{ of } x \text{ and } \varphi\in C^{\infty}(U)
    \\
    &\text{with}& \varphi(x)=0,d\varphi(x)=\xi,
    \text{ one has } (\ref{eqn: ShvProp@x})\}.
\end{eqnarray*}

\begin{definition}
\label{eqn: Sheaf Microsupport}
The \emph{microsupport of $\mathcal{F}^{\bullet}$} is the closure of the set of codirections of non-propagation,
$\mathrm{SS}(\mathcal{F}^{\bullet}):=\overline{T^*X\hspace{1mm}\backslash\hspace{1mm} P(\mathcal{F}^{\bullet})}.$
\end{definition}
The microsupport (\ref{eqn: Sheaf Microsupport}) is a coisotropic closed $\mathbb{C}^{\times}$-conic subset of $T^*M$
such that $$T_M^*M\cap \mathrm{SS}(F)=\pi_M\big(\mathrm{SS}(F)\big)=\mathrm{supp}(F),$$ and $\mathrm{SS}(F[i])=\mathrm{SS}(F)$ for $i\in \mathbb{Z}.$ Moreover, it satisfies the triangular inequality: for a triangle $F_1\rightarrow F_2\rightarrow F_3\xrightarrow{+1}$ in $\mathrm{D}^{\mathrm{b}}(\mathbf{k}_M)$ then $\mathrm{SS}(F_i)\subset \mathrm{SS}(F_j)\cup \mathrm{SS}(F_k)$ for $i,j,k\in \{1,2,3\}$ with $j\neq k.$

\begin{proposition}
\label{MicrosupportCriterionProp}
Fix $x_0\in M.$ Then $\xi \in P_{x_0}(\mathcal{F}^{\bullet})$ if and only if for every $\varphi\in C^{\infty}(U,\mathbb{C})$ for an open neighbourhood $U$ of $x_0$ such that $\varphi(x_0)=0$ and $d\varphi(x_0)=\xi$, the morphism 
$\mathcal{F}_{x_0}^{\bullet}\rightarrow \big(R\Gamma_{\{\varphi <0\}}(\mathcal{F}^{\bullet})\big)_{x_0},$
induced by the inclusion $\big(i_V:V\cap \{\varphi <0\}\hookrightarrow V\big)_{x_0\in V}$ is an isomorphism.
\end{proposition}
\begin{proof}
Apply stalk functor to 
$R\Gamma_{\{\varphi \geq 0\}}(F)\rightarrow R\Gamma_U(F)\rightarrow R\Gamma_{\{\varphi <0\}}(F)\xrightarrow {+1}.$
\end{proof}
Proposition \ref{MicrosupportCriterionProp} characterizes when a point belongs to $P(\mathcal{F}^{\bullet}).$ 

\begin{proposition}
\label{MicrosupportsClosedEmbeddings}
Let $F\in \mathrm{D}^{\mathrm{b}}(\mathbf{k}_M),G\in \mathrm{D}^b(\mathbf{k}_N).$ Then 
$\mathrm{SS}(F\boxtimes^L G)\subset \mathrm{SS}(F)\times \mathrm{SS}(G).$ If
$f:N\rightarrow M$ is proper on $\mathrm{supp}(F),$ then $Rf_!F\xrightarrow{\simeq} Rf_*F$ and 
$\mathrm{SS}(Rf_*F)\subset f_{\pi}f_d^{-1}\mathrm{SS}(F).$
Moreover, if $f$ is a closed embedding, $\mathrm{SS}(Rf_*F)=f_{\pi}f_d^{-1}\mathrm{SS}(F).$
\end{proposition}
Using Proposition \ref{MicrosupportsClosedEmbeddings}, we give an example showing that if a sheaf $\mathcal{F}$ propagates at $\big(x,p\circ df(x)\big)$ for each $x$ in $f^{-1}(y)$, then $Rf_*\mathcal{F}$ propagates at $(y;p).$
\begin{example}
\label{Lagrangian Relation Example}
Consider a smooth morphism $f:X\rightarrow Y$ and a Lagrangian sub-manifold $\Lambda\subset \overline{T^*X}\times T^*Y.$ For a subset $S\subset T^*X$ define $\Lambda(S)$ to be the reduction of $S\times \Lambda$ by the cositropic submanifold $D\times T^*Y$, with $D$ the diagonal in $T^*X\times\overline{T^*X}.$
Define
$\Lambda_f:=\{(x,\xi,y;p)\in \overline{T^*X}\times T^*Y| f(x)=y,\xi=p\circ df(x)\}.$ 
For $S\subset T^*X,$ we have $(y;p)\in \Lambda_f(S)$ if and only if there exists $x\in f^{-1}(y)$ such that $(x,p\circ df(x))\in S.$ If $f$ moreover proper on $\mathrm{supp}(\mathcal{F})$, then $\mathrm{SS}(Rf_*\mathcal{F})\subset \Lambda_f(\mathrm{SS}(\mathcal{F})\big).$ 
\end{example}
The next result is sometimes called the Petrowsky theorem for sheaves.
\begin{corollary}
\label{cor: Petrowsky}
Assume $G$ is cohomologically constructible, and $\mathrm{SS}(G)\cap \mathrm{SS}(F)\subset T_M^*M.$ Then the natural morphism $R\mathcal{H}om(G,k_M)\otimes F\to R\mathcal{H}om(G,F)$ is an isomorphism.
\end{corollary}

\subsection{$\D$-modules and propagation of solutions}
\label{ssec: Propagation of Solutions}

We now recall how microlocal sheaf theory applies to PDEs.
Microfunctions outside of the zero section of the conormal bundle $\pi_M:T_M^*X\rightarrow M$ describe the singular part of the flabby sheaf hyperfunctions i.e. hyperfunctions up to (non-singular) analytic functions as shown by the exact sequence of sheaves, exact at the level of global sections \cite[Prop. 11.5.2]{KashiwaraSchapira1990},
\begin{equation}
    \label{eqn: SESHyperfunction}
0\rightarrow \mathscr{A}_M\rightarrow \mathscr{B}_M\rightarrow \mathring{\pi}_{M*}\mathscr{C}_M\rightarrow 0.
\end{equation} 
The sheaf $\mathscr{C}_M$ is concentrated in degree $0$, thus is a conic sheaf on $T_M^*X$. By construction, there is an isomorphism $\mathscr{B}_M\simeq \pi_{M,*}\mathscr{C}_M\simeq\mathscr{C}_M|_{T_M^*M},$ and a natural \emph{specialization} map on global sections 
$\mathrm{sp}:\Gamma(M,\mathscr{B}_M)\to \Gamma(T_M^*X,\mathscr{C}_M).$ 
\begin{definition}
\label{defn: An WaveFront}
Let $u\in \mathscr{B}(M):=\Gamma(M,\mathscr{B}_M$) be a hyperfunction. The support of $\mathrm{sp}(u)$ in $T_M^*X,$ is
called the \emph{analytic wave front set} of $u$ and is denoted by $\mathrm{WF}(u),$ or $\mathrm{SS}(u)$ when no confusion arises.
\end{definition}

To a coherent $\D_X$-module $\mathcal{M}$ associate its characteristic variety $\mathrm{Char}(\mathcal{M}).$ It is a closed $\mathbb{C}^{\times}$-conic coisotropic $\mathbb{C}$-analytic subset of $T^*X.$ 
\begin{remark}
    Later we will use a variant called the $1$-microcharacteristic variety \cite{Laurent1985} (recalled in §\ref{ssec: Characteristic Varieties}).
    \end{remark}
Given a section $P(x,\partial_x)\in \D_X,$ its principal symbol $\sigma(P)$ is a holomorphic function on $T^*X$. If $P$ is of order $\leq m$ then $\sigma(P)$ is homogeneous of degree $m.$ The characteristic variety $\mathrm{Char}(\mathcal{M})$ is related to, and in many cases coincides with, the zero locus of this function.

\begin{example}
\label{Char Zero Locus}
 If $\mathcal{M}$ is generated by a single variable $u$, i.e. $\mathcal{M}\simeq\D_X/\mathcal{I}$ with $\mathcal{I}=\{P\in \D_X|P(z,\partial_z)u=0\},$
 then 
    $\mathrm{Char}(\mathcal{M})=\{p=(x;\xi)\in T^*X|\sigma(P)(x;\xi)=0\}.$ 
\end{example}

\begin{remark}
    If there exists a family $\{P_0,\ldots,P_N\}$ generating our ideal $\mathcal{I}$ with $P_j$ of order $\leq j$ then one has
    $\mathrm{Char}(\mathcal{M})\subset \bigcap_j\sigma_{j}(P_j)^{-1}(0),$ but in general \emph{not} an equality.
\end{remark}
Propagation of solutions for linear PDEs refers to sheaf propagation in the case when the sheaf $\mathcal{F}^{\bullet}$ is $R\mathcal{S}ol_{X}(\mathcal{M})$ for $\mathcal{M}\in \mathrm{Coh}(\D_X).$ 

 The complex of \emph{microfunction-solutions} associated to $\mathcal{M}$ is the object of the derived category of sheaves given by
 $$R\mathcal{S}ol_{\mathcal{E}_X}(\mathcal{M}):=R\mathcal{H}om_{\mathcal{E}_X}(\mu\mathcal{M},\mathscr{C}_M).$$

The following theorem is known as the \emph{propagation of solution singularity theorem} for microfunction solutions \cite[Prop.~11.5.4]{KashiwaraSchapira1990}.

\begin{theorem}
\label{Thm: LinProp}
Let $\mathcal{M}$ be a coherent $\mathcal{E}_X$-module. Then, 
$$\mathrm{SS}\big(RSol_{\mathcal{E}_X}(\mathcal{M})\big)\subset C_{T_M^*X}\big(\mathrm{Char}(\mathcal{M})\big)\subset T^*T^*X.$$
In particular $supp\big(RSol_{\mathcal{E}_X}(\mathcal{M})\big)\subset \mathrm{Char}(\mathcal{M})\cap T_M^*X$.
\end{theorem}
If $\mathcal{M}$ is a coherent $\D_X$-module, Theorem \ref{Thm: LinProp} recovers (\ref{eqn: LPDE}).

\begin{example}

Suppose that $P=P(x,\partial_x)\in \D_X^{\leq m}\subset \D_X$ is a differential operator of order $\leq m.$ By Theorem \ref{Thm: LinProp}, $P$ defines an isomorphism of the sheaf $\mathscr{C}_M$ on $T_M^*X\backslash \big\{\sigma_m(P)^{-1}(0)\big\}$ and if $u$ is a hyperfunction solution, its support is contained in the union of the support of $Pu$ and $\big\{\sigma_m(P)^{-1}(0)\big\}.$
\end{example}
The abstract propagation result for linear systems in Theorem \ref{Thm: LinProp} is related to local-constancy of the sheaf of solutions and the condition for propagation (\ref{eqn: ShvProp@x})
as follows.
\begin{proposition}
\label{ClassicalPropagation}
Let $U$ be an open subset of $T^*X$ and suppose $g\in \mathcal{O}(U)$ is such that $Im(g)$ vanishes on $T_M^*X.$
Then for all $p\in U\cap T_M^*X$ the bicharacteristic curve $\beta_p$ passing through $p$ of the involutive submanifold $\{Im(g)=0\}$ of $(T^*X){^\mathbb{R}}$ is contained in a neighbourhood of $p$ in $T_M^*X.$
Suppose $\mathcal{M}\in \mathsf{Coh}(\D_X)$ with $Im(g)\big|_{T_M^*X\cap U}=0.$
Assume that 
$\mathrm{Char}(\mathcal{M})\cap U\subset \{g=0\}.$
Then for all bicharacteristics $\beta$ of the manifold $\{Im(g)=0\}$, the sheaves
$R^j\mathcal{S}ol(\mathcal{M},\mathscr{C}_M)\big|_{\beta}=\mathcal{E}xt_{\D_X}^j\big(\mathcal{M},\mathscr{C}_M\big)\big|_{\beta},j\in\mathbb{N}$
are locally constant.
\end{proposition}
\begin{proof}
   See \cite[Prop.11.5.6]{KashiwaraSchapira1990}.
\end{proof}

In other words, for $p_0=(x_0,\xi_0)$ and the bicharacteristic curve $\beta_{p_0}$ as in Proposition \ref{ClassicalPropagation}, we have an inclusion $\mathrm{SS}\big(R\mathcal{S}ol_{\mathcal{E}_X}(\mathcal{M})\big)\subset C_{T_M^*X}\big(\{g^{-1}(0)\}\big),$ which implies the propagation statement \eqref{eqn: ShvProp@x} at $x_0$:
$$R\Gamma_{\{\beta_{p_0}\}}\big(R^j\mathcal{S}ol(\mathcal{M},\mathscr{C}_M)\big|_{\beta}\big)_{x_0}\simeq 0.$$
Results in this subsection apply to PDEs given by elliptic or hyperbolic operators, defining $\D$-modules of the corresponding type. By modifying characteristic varieties, one may encode the directions of (non)-propagation for their solutions. In particular, hyperbolicity appears as a geometric condition on $Char,$ as we now explain.
\subsubsection{Ellipticty and (micro)-hyperbolicity}
Let $M$ be a real analytic manifold with $X$ a complexification. A coherent $\D$-module $\M$ on $X$ is said to be \emph{elliptic} if
$\mathring{T}_M^*X\cap \mathrm{Char}(\mathcal{M})=\emptyset$, or equivalently,
\begin{equation}
\label{eqn: Elliptic D Module}
T_M^*X\cap \mathrm{Char}(\mathcal{M})\subset T_X^*X,
\end{equation}
Theorem \ref{Thm: LinProp} implies elliptic regularity, stating that the inclusion of analytic solutions in hyperfunction solutions is an isomorphism.
\begin{proposition}
Let $M$ be a real-analytic manifold with a complexification $X$ and consider a coherent $\D_X$-module $\M$ which is elliptic. Then the canonical map 
$R\mathcal{H}om_{\D_X}(\M,\mathscr{A}_M)\xrightarrow{\simeq} R\mathcal{H}om_{\D_X}(\M,\mathscr{B}_M),$ 
induced from \eqref{eqn: SESHyperfunction} is an isomorphism.
\end{proposition}
We introduce some terminology.
\begin{remark}
Say $\theta \in T_{T_M^*X}T^*X$ is microhyperbolic for a $\D_X$-module $\mathcal{M}$ if $\theta \notin C_{T_M^*X}\big(\mathrm{Char}(\mathcal{M})\big).$ Under $\theta \in T^*M\hookrightarrow T_{T_M^*X}T^*X$, one says $\theta$ is hyperbolic. If $\theta$ is microhyperbolic (resp. hyperbolic)
$\theta\notin \mathrm{SS}\big(R\mathcal{S}ol(\mathcal{M},\mathscr{C}_M)\big)$ 
(resp. $\theta\notin \mathrm{SS}\big(R\mathcal{S}ol_X(\mathcal{M},\mathcal{B})\big)$).
\end{remark}
The \emph{hyperbolic characteristic variety} of $\mathcal{M}$ along $M$, is
\begin{equation}
\label{eqn:Hyperbolic Char}
\mathrm{Char}_M^{hyp}(\mathcal{M}):=T^*M\cap C_{T_M^*X}\big(\mathrm{Char}(\mathcal{M})\big).
\end{equation}
Note $\theta \in T^*M$ is hyperbolic for $\mathcal{M}$ when $\theta  \notin \mathrm{Char}_M^{hyp}(\mathcal{M}).$

\begin{example}
\label{Hyperbolic example 1}
 If $z=x+iy\in X$ so that $M=\{z\in X|y=0\},$ and $\zeta=\xi+i\eta\in T^*X$ we get $T_M^*X=\{y=\xi=0\}.$ A direction
$p_0=(x_0;\theta_0)\in T^*M$ with $\theta_0\neq 0$ is hyperbolic if and only if there exist an open neighbourhood $U$ of $x_0$ in $M$ and open conic neighbourhood $\gamma$ of $\theta_0\in \mathbb{R}^n$ such that
$\sigma(P)(x;\theta+i\eta)\neq 0,$
for all $\eta \in \mathbb{R}^n,x\in U,\theta\in \gamma.$
\end{example}

\begin{remark}
\label{HyperbolicEllipticSubManRemark}

A submanifold $N\subset M$ is hyperbolic for $\mathcal{M}$ if 
$T_N^*M\cap \mathrm{Char}_M^{hyp}(\mathcal{M})\subset T_M^*M$. Obviously, any nonzero covector of $T_N^*M$ is hyperbolic for $\mathcal{M}.$
Similarly, a real sub-manifold $L\subset X$ is elliptic for $\mathcal{M}$ if $\mathrm{Char}(\mathcal{M})\cap T_L^*X\subset T_X^*X.$
\end{remark}
By Example \ref{Hyperbolic example 1} and Remark \ref{HyperbolicEllipticSubManRemark} it is clear what geometric hyperbolicity means.
\begin{example}
\label{Hyperbolic example 2}
Let $N\subset M$ be a real submanifold with complexification $Y\subset X$ and $\mathcal{M}=\D_X/\D_X\cdot P$ a coherent $\D_X$-module determined by $P$ of order $\leq k.$ Assume $(x)=(x_1,x')\in M$ such that $N=\{x_1=0\}$ and suppose that $\mathrm{Char}(\mathcal{M})\cap \mathring{T}^*X \cap \mathring{T}_M^*X$ is a regular involutive submanifold of $\mathring{T}_M^*X$ and $N$ is hyperbolic for $\mathcal{M}$. These assumptions translate to $\sigma(P)$ being a hyperbolic polynomial with respect to $N$ with real-constant multiplicities. In other words, $\sigma(P)(x;\tau,\eta')$ has $k$ real roots $\tau$ with constant multiplicites where $x,\eta'$ are real.
\end{example}

Hyperfunction and real-analytic solutions 
propagate in hyperbolic directions \cite{KashiwaraSchapira1990}.

\begin{proposition}
\label{Propagation for hyperbolic directions}
Suppose that $P\in\D_X$ with $\mathcal{M}:=\D_X/\D_X\cdot P\in\mathrm{Coh}(\D_X).$ Then
$$ \mathrm{SS}\big(R\mathcal{S}ol_X(\mathcal{M},\mathscr{A}_M)\big),\mathrm{SS}\big(R\mathcal{S}ol_X(\mathcal{M},\mathscr{B}_M)\big)\subset \mathrm{Char}_M^{hyp}(\mathcal{M}).$$
\end{proposition}
\begin{proof}
Follows from the isomorphisms,
$$R\Gamma_MR\mathcal{H}om_{\D_X}(\M,\mathcal{O}_X)\simeq R\mathcal{H}om_{\D_X}(\M,R\Gamma_M\mathcal{O}_X),$$
and $R\mathcal{H}om_{\D_X}(\M,\mathcal{O}_X)|_{M}\simeq R\mathcal{H}om_{\D_X}(\M,\mathcal{O}_X|_M),$ by applying \cite[Cor. 6.4.4]{KashiwaraSchapira1990}, i.e.
$$\mathrm{SS}\big(R\Gamma_M R\mathcal{H}om_{\D_X}(\M,\mathcal{O}_X)\big)\subset T^*M\cap C_{T_M^*X}\big(\mathrm{SS}(R\mathcal{H}om_{\D_X}(\M,\mathcal{O}_X)\big),$$
which agrees with (\ref{eqn:Hyperbolic Char}).
\end{proof}
We now elaborate how the inclusion in Theorem \ref{Thm: LinProp}, which is actually an equality c.f. (\ref{eqn: LPDE}), implies the well-posedness of the abstract Cauchy problem.  
\subsection{Cauchy problems}
\label{ssec: Cauchy Problems}
A
Cauchy problem is said to be \emph{well-posed} for a system $\mathcal{M}$ with initial value $\varphi$ in a space of functions $\mathscr{F}$ 
if on one hand there exists a solution to $\mathcal{M}$ around $\varphi$ which is unique and depends continuously on $\varphi$ in a prescribed manner, while on the other hand 
a solution of the system should remain in the same functional space $\mathscr{F}$ as the initial/boundary value.
Functional analysts may try to construct a sufficiently large space $\mathscr{F}$ in which the existence of a solution for a given type of equation (e.g. elliptic, hyperbolic, $\ldots$) is guaranteed, and to find a subspace in which uniqueness holds. In the case of distribution solutions, typical subspaces are Sobolev spaces. For example, Proposition \ref{Propagation for hyperbolic directions} implies the well-posedness of Cauchy problems for
hyperbolic systems in hyperfunctions.

The Cauchy-Kovalevskaya theorem for initial-value problems for (analytic) PDEs is formulated as follows. Let $X$ be an analytic subset of $\mathbb{C}^n$ with coordinates $\{x_i,\partial_i\}$ and suppose $\Sigma=\{x_1=0\}$ is a hypersurface, with $P\in \D_X^{\leq k},$ locally of the form $P(x,\partial_x)=\sum_{|\sigma|\leq k}a_{\sigma}(x)\partial_x^{\sigma}.$ 
One says $\Sigma$ is non-characteristic if the coefficient $a_{(k,0,\ldots,0)}$ of $\partial_{x_1}^k$ does not vanish.
\begin{definition}
Let $f:\Sigma\rightarrow X$ be a morphism and $\mathcal{M}\in \mathrm{Coh}(\D_X).$ Then $f$ is said to be \emph{non-characteristic} for $\mathcal{M}$ if it is so with respect to $\mathrm{Char}(\mathcal{M})$ i.e.
$f_{\pi}^{-1}\big(\mathrm{Char}(\mathcal{M})\big)\cap T_{\Sigma}^*X\subset \Sigma\times_X T_X^*X.$
\end{definition}
If $\sigma_k(P)$ is the principal symbol of a section $P\in \D_X$ with $\D$-module $\mathcal{M}_P=\D_X/\D_X\cdot P,$ the non-characteristic condition means
$\sigma_k(P)$ does not vanish on the conormal bundle to $\Sigma$ outside the zero-section i.e. it defines a non-vanishing holomorphic function $\sigma(P)(x;dx_1)\neq 0$ at the point $(x;dx_1).$ Thus for any $f\in\mathcal{O}_X$ defined on an open neighbourhood of $\Sigma$, and any $k$-tuple $(u_0,\ldots,u_{k-1})\in \mathcal{O}_{\Sigma}^{\oplus k},$ there exists a unique holomorphic solution $u\in \mathcal{O}_X,$ to the Cauchy problem (c.f equation \ref{eqn: NLCauchyProblem})
\begin{equation*}
\begin{cases}
    P(x,\partial_x)u=f,
    \\
     \partial_1^ju\big|_{\Sigma}=u_j,\hspace{2mm}j=0,1,\ldots,k-1.
\end{cases}
\end{equation*}
The classical Cauchy problem is solved if we can replace $u$ by $\gamma_{\Sigma}(u):=\big(u|_{\Sigma},\partial_1u|_{\Sigma},\ldots,\partial_1^{k-1}u|_{\Sigma}\big)$, the so-called first $(k-1)$-traces of $u.$ In its abstract form, this is a derived equivalence of solution sheaves \cite[Sect.2]{kashiwara1970algebraic}. 

\begin{theorem}
\label{Thm: LinCKK}
Let $f:\Sigma\rightarrow X$ be morphism of complex analytic spaces with $\mathcal{M}$ a coherent $\D_X$-module. If $f$ is non-characteristic for $\mathcal{M},$ its (derived) pullback $f^*\M$ is coherent and concentrated in degree zero with $\mathrm{Char}(f^*\M)=f_df_{\pi}^{-1}\mathrm{Char}(\M)$, and the canonical map
$$\textsc{ckk}_{f}(\mathcal{M},\mathcal{O}_Y):f^{-1}R\mathcal{S}ol_X(\mathcal{M})\xrightarrow{\simeq}R\mathcal{S}ol_{\Sigma}(f^*\mathcal{M}),$$
is an isomorphism of sheaves on $\Sigma$.
\end{theorem}
Theorem \ref{Thm: LinCKK} states it is possible to reconstruct solutions from boundary/initial conditions, globally and any reasonable real-analytic PDE has a unique solution depending
analytically on the given analytic initial condition. 

 It is convenient to introduce some terminology, following \cite[Def.2]{Paugam2022}. As we will see, this generalizes to provide a proper formulation of the derived non-linear Cauchy problem.
\begin{definition}
\label{definition: LinearCKKTrue}
Let $f:X \to Y$ be a morphism of complex manifolds and let $\M,\mathcal{N}$ be  coherent $\D_Y$-modules. The Cauchy-Kowaleskaya-Kashiwara theorem said to be true for the triple $(\M,\mathcal{N},f)$ if the natural morphism 
\begin{equation}
    \label{eqn: CKK-true}
\textsc{ckk}_f(\mathcal{M},\mathcal{N}):f^{-1}R\mathcal{H}om_{\D_Y}(\mathcal{M},\mathcal{N})\rightarrow R\mathcal{H}om_{\D_X}(f^*\mathcal{M},f^*\mathcal{N}),
\end{equation}
is an isomorphism.
\end{definition}

Morphism (\ref{eqn: CKK-true}) is an isomorphism, for example, when $f$ is non-characteristic for 
$$\mathrm{Char}(\mathcal{M})+_Y\mathrm{Char}(\mathcal{N}):=\big\{(x;\xi_1+\xi_2)|(x;\xi_1)\in \mathrm{Char}(\mathcal{M}),(x;\xi_2)\in \mathrm{Char}(\mathcal{N})\big\},$$ and if  $(\mathcal{M},\mathcal{N})$ satisfy 
\begin{equation}
\label{eqn: NC-pair}
\mathrm{Char}(\mathcal{M})\cap \mathrm{Char}(\mathcal{N})\subset T_Y^*Y.
\end{equation}
When (\ref{eqn: NC-pair}) holds one says $(\mathcal{M},\mathcal{N})$ form a \emph{non-characteristic pair}.

\subsection{On the non-linear generalizations}
\label{sssec: On the non-linear generalizations}

Propagation of solutions to nonlinear PDEs with holomorphic coefficients, and generalizations of Theorem \ref{Thm: LinProp} and Theorem \ref{Thm: LinCKK} need to be formulated in such a way that treats \emph{singular} situations and solutions with boundedness conditions.

Solutions may be singular along a non-characteristic hypersurface and the non-existence of singularities will imply propagation of solutions in the non-characteristic directions. In turn, the latter implies well-posedeness of Cauchy problems with data on the non-characteristic surface, which therefore may be singular data.
\begin{example}
Consider the KdV equation $u_{ttt}-6uu_x+u_x=0,$ in a variable $u$. The initial surface $\Sigma=\{t=0\}$ is non-characteristic while solutions are known to be of the form $u=t^{-2}2+g(x)t^2+h(x)t^4+\cdots,$ for functions $g,h.$
\end{example}
In this subsection we discuss through basic examples what must be considered in the non-linear setting to hope for a sheaf-theoretic non-linear analog of Zerner's propagation theorem (recalled in Prop.\ref{Zerner} below).

\subsubsection{Towards a nonlinear Zerner propagation theorem}

Consider $x=(x_1,\ldots,x_n)\in \mathbb{C}^n,t\in \mathbb{C}$ and for a fixed non-negative integer $m$, the set of multi-indices 
$I_m:=\{(j,\sigma)\in\mathbb{N}\times\mathbb{N}^n|j+|\sigma|\leq m,j<m\},$
with cardinality $N:=|I_m|.$
Represent the collection of dependent variables and their derivatives by $u_{[j,\sigma]}:=\{(u_{j,\sigma})\}_{(j,\sigma\in I_m})\in \mathbb{C}^N.$ 
Consider a holomorphic function $\mathsf{F}=\mathsf{F}(t,x,u_{[j,\sigma]})$ in
$$\big\{(t,x)\in \mathbb{C}\times\mathbb{C}^n| |t|<r_0,|x|<R_0\big\}\times \mathbb{C}^N,$$
with $r_0,R_0$ some positive constants. 

We consider
$D_{r_0}:=\{t\in \mathbb{C}: |t|<r_0\},B_{R_0}^n:=\{x\in \mathbb{C}^n: |x|<R_0\}.$
For $\epsilon>0,$ put $D_{r_0}(\epsilon):=\{t\in \mathring{D}_{r_0}: |arg t|<\epsilon\},$ with $\mathring{D}_{r_0}:=\{t\in \mathbb{C}\backslash \{0\}: |t|<r_0\},$ and then our PDEs are defined by holomorphic functions $\mathsf{F}$ defined in 
$D_{r_0}\times B_{R_0}^n\times \mathbb{C}^{N}.$
The norms we consider, on some open neighbourhood $\Omega$ of the origin i.e. of $x=0$ in $\mathbb{C}_{[x]}^n$, are always $|\!|u(t)|\!|_{\Omega}:=\underset{x \in \Omega}{\mathrm{sup}}\hspace{1mm} |u(t,x)|.$
\begin{remark}
    Kobayashi 
    \cite{Kobayashi} studied $\mathsf{F}$ which are holomorphic in $(t,x)$ but polynomial in $u_{[j,\sigma]}.$ This situation is encoded in the $\D$-algebraic formulation.
\end{remark}
The propagation problem is formulated as follows:
suppose $u(t,x)$ is a known solution which is holomorphic in $D_{r_0}(\theta>0)\times \{x\in \mathbb{C}^n| |x|<R_0\}.$ Then under what conditions can we we extend $u(t,x)$ as a holomorphic solution in a neighbourhood of the origin (i.e. $t\rightarrow 0$).

First, consider
\begin{equation}
    \label{eqn: First Order NLPDE}
    \frac{\partial u}{\partial t}=\mathsf{F}\big(t,x,u,\frac{\partial u}{\partial x}\big),
\end{equation}
in the complex domain $\mathbb{C}\times \mathbb{C}^n$ with coordinates $(t,x_1,\ldots,x_n).$
The structure of holomorphic solutions can be completely understood by the CK-theorem: any holomorphic function $\varphi(x)$ in a neighbourhood of a point $x_0$ has a holomorphic solution $u(t,x)$ in a neighbourhood of the origin $(0,0)\in \mathbb{C}\times\mathbb{C}^n$ such that $u(0,x)=\varphi(x).$ In this sense $u(t,x)$ is completely determined by $\varphi(x).$
Now, let us ask the question: what if there exists a singularity of the solution along the initial surface $\{t=0\}?$

One answer is to argue the non-existence of solutions via analytic continuation. Defining $U_+:=\{(t,x)|Re(t)>0\}$, one uses Zerner's Propagation Theorem \cite{Zerner} which states that any holomorphic solution on $U_+$ can be analytically extended to some neighbourhood of the origin. In other words, any holomorphic solution propagates analytically over noncharacteristic hypersurfaces. This implies there does not exist a solution with singularities only on the initial surface $\{t=0\}.$

Kashiwara-Schapira \cite{KashiwaraSchapira1990} give Zerner's theorem in the following form.

\begin{proposition}
\label{Zerner}
Let $\varphi$ be a real-valued $C^{\infty}$-function on a complex analytic manifold $X$ such that
$X_0:=\{x\in X|\varphi(x)=0\}$ is a smooth real hypersurface. Consider $X_+:=\{x\in X|\varphi(x)\geq 0\}\xrightarrow{j^+}X,$ and $X_{-}:=\{x\in X|\varphi(x)<0\}\xrightarrow{i_-}X.$
Let $\mathcal{M}$ be a coherent $\D_X$-module such that $\mathrm{Char}(\mathcal{M})\cap T_{X_0}^*(X)\subset T_X^*X.$
Then
$$\big[R\Gamma_{X_+}R\mathcal{S}ol_X(\mathcal{M})\big]_{X_0}\simeq 0,$$
if and only if the the morphism $$R\mathcal{S}ol_{X}(\mathcal{M})_{X_-^c}\xrightarrow{\simeq} R\Gamma_{X_-}(R\mathcal{S}ol_{X}(\mathcal{M})),$$
is an isomorphism.
\end{proposition}
Proposition \ref{Zerner} states any holomorphic solution to $\mathcal{M}$ on $X_-$ extends across $X_0$ as a holomorphic solution to $\mathcal{M}.$ 

\begin{remark}
A principal motivation for this paper was to establish a non-linear analog of Zerner's theorem in a suitable sheaf-theoretic framework for non-linear PDE. However, adapting Zerner's result to the non-linear setting is not so straightforward for several reasons e.g. growth conditions.\end{remark}

\begin{example}
\label{ex: Boundedness}
   Let $(t,x)\in \mathbb{C}^2$ and consider 
    $\frac{\partial u}{\partial t}=u\big(\frac{\partial u}{\partial x}\big)^m,$
    for some $m\in \mathbb{N}.$ The family of solutions $u(t,x)=(-1/m)^{1/m}(x+c)/t^{1/m}$ with arbitrary constants $c$ are holomorphic in $U_+$ by has a singularity on $\{t=0\}.$
    Notice the singularities on the initial surface of order $O(1)$ as $t\rightarrow 0,$ do not appear in the solutions, but do appear at $O(|t|^{-1/m})$ as $t\rightarrow 0.$ 
\end{example}
It is not enough, however, to simply study solutions $u=u(t,x)$ holomorphic in a neighbourhood of the origin
with the above prescribed growth conditions, that is being of order $O(|t|^{\lambda})$ for non-negative real numbers $\lambda,$ as $t\rightarrow 0,$ denoted $\mathcal{O}_{+}\subset \mathcal{O}_X.$
Indeed, they are insufficient to detect solutions to some singular non-linear PDEs e.g. those of the form,
$$(t\partial_t )^mu=F(t,x,(t\partial_t)^j\partial_x^{\sigma}u).$$
Instead, we can replace the condition $u(t,x)=O(|t|^a)$ as $t\rightarrow 0,$ for some $a>0,$ by the weaker condition:
$$u(t,x)\sim O(1/|\text{log } t|^a),\hspace{1mm} as \hspace{1mm} t\rightarrow 0,a>0.$$
This is necessary due to the following example concerning $\mathrm{Sol}(P,\mathcal{O}_+).$

\begin{example}
Consider $(t,x)\in \mathbb{C}^2,$ and the non-linear PDE,
$$t\cdot \frac{\partial u}{\partial t}=u\cdot \big(\frac{\partial u}{\partial x}\big)^m,m>0.$$
The family of singular solutions is given by
$$u(t,x)=(1/m)^{1/m}\frac{(x+\alpha)}{(c-\text{log } t)^{1/m}},\hspace{1mm} c,\alpha\in\mathbb{C}.$$
In this case $u \notin \mathrm{Sol}(P,\mathcal{O}_+).$ Instead, consider $$\mathcal{O}_{log}:=\{u(t,x)\in \mathcal{O}_X| \exists\hspace{1mm} a>0 \text{ so } \forall \hspace{1mm} \theta>0 \text{ max}_{|x|\leq R_0}|u(t,x)|=O(1/|\text{log }t|^a)\},$$
as $t\rightarrow 0$ in $D_{r_0}(\theta).$ Since $\mathcal{O}_+\subset \mathcal{O}_{log}$ is an inclusion of algebras, we have an inclusion of \emph{sets} of solutions,
\begin{equation}
\label{eqn: Temperate into logarithmic solutions}
\mathrm{Sol}(P,\mathcal{O}_+)\hookrightarrow \mathrm{Sol}(P,\mathcal{O}_{log}).
\end{equation}
These correspond to the sets of solutions $u=u(t,x)$ possessing temperate and logarithmic singularities, respectively.
Existence of formation of singularities i.e. of new singular solutions means that the inclusion (\ref{eqn: Temperate into logarithmic solutions}) is never an equality, and that there are no new singular solutions if it is an equality.\end{example}
Summarizing this discussion, a refined statement of the non-linear propagation problem asks: when solutions $u(t,x)$ holomorphic (or with other prescribed behaviour) are solutions in a neighbourhood of the origin with prescribed boundedness and growth conditions (as $t\rightarrow 0$)?

Treating these functional subtleties is outside the scope of this paper, but will be considered in a sequel work which studies, for example, the well-posedeness of singular initial-value problems of the type (\ref{eqn: NLCauchyProblem}) with growth conditions, and special cases of (a) non-linear elliptic equations (c.f. §§\ref{sssec: A Remark on Ellipticity}) and (b) quasilinear hyperbolic Cauchy systems. Note that linear hyperbolic Cauchy problems and propagation of solution singularities plays an important role in mathematical general relativity, and has been studied in $\D$-module approach in \cite{JubinSchapira2016}. The extension to quasi-linear systems would treat those of the form
\begin{equation}
    \label{eqn: QuasiLinHypCauchy}
\begin{cases}
    P(t,x,\partial_t,\partial_x,\partial_{t,x}^{\sigma}u)u(t,x)=\mathsf{F}(t,x,\partial_{t,x}^{\sigma}u),\hspace{1mm} |\sigma|\leq m-1,
    \\
\partial_t^ju(0,x)=\psi_j(x),\hspace{2mm} j=0,1,\ldots,m-1,
\end{cases}
\end{equation}
with $u(t,x)\in I\times \mathbb{R}^n$ for some interval $[0,T]$ with 
$$P=\partial_t^m+\sum_{j=0}^{m-1}\sum_{|\alpha|=m-j}Q_{\alpha}^{(j)}(t,x,\partial_{t,x}^{\sigma}u)\cdot\partial_x^{\alpha}\cdot \partial_t^j.$$
Notice that $\partial_{t,x}^{\sigma}u$ appearing in the above for $|\sigma|\leq m-1$ is a vector in $\mathbb{R}^{N}$ with $N$ given by the number of such $\sigma$ with length $\leq m-1.$ The functions $Q_{\alpha}^{(j)}$ and $\mathsf{F}$ should have specificed functional behaviour (e.g. Gevrey behaviour as in Construction \ref{cons: Gevrey}) with respect to $x$ and $y:=\partial_{t,x}^{\sigma}u.$ Systems (\ref{eqn: QuasiLinHypCauchy}) generalize the cases of weakly hyperbolic linear Cauchy problems ($m\geq 2$), as discussed in §§ \ref{ssec: Propagation of Solutions}.

We now study the differential-algebraic treatment of non-linear PDEs via methods of algebraic microlocal analysis.

\section{Algebraic $\D$-geometry of PDEs}
\label{sec: Algebraic D-Geometry}
Before introducing a derived enhancement of the non-linear analog of $\mathscr{D}$-module theory, treated in a coordinate free manner via the language of algebraic $\mathscr{D}$-geometry \cite{BD2004},\cite{Paugam2014}, we elaborate the classical approach in a scheme-theoretic language \cite{Malgrange2005}. 
\begin{remark}
    We do not go into detail about the \emph{local/functional} differential calculus on these spaces (so-called \emph{Secondary calculus} \cite{Vinogradov2001}), leaving this for the sequel \cite{KSY2}.\end{remark}

The category of affine $\mathscr{D}_X$-schemes over $X$ is the opposite category $\mathrm{Aff}_X(\mathscr{D}_X):=\mathrm{CAlg}(\mathscr{D}_X)^{\mathrm{op}},$ to commutative $\mathscr{D}_X$-algebras and a generic affine $\mathscr{D}_X$-scheme is denoted  $\mathrm{Spec}_{\mathscr{D}_X}(\mathcal{A}).$ 
We have the standard left/right equivalences and write them as $\mathcal{A}^r:=\omega_X\otimes_{\mathcal{O}_X}\mathcal{A}$ and $\mathcal{A}^{\ell}:=\omega_X^{\otimes -1}\otimes_{\mathcal{O}_X}\mathcal{A}$.
We briefly set up the classical theory of $\D_X$-schemes following ideas of \cite{SGA4I}.
\subsection{Non-linear algebraic analysis}
\label{sec: Non-linear Algebraic Analysis}
Let $X$ be a smooth proper $\mathbb{C}$-scheme of $dim_{\mathbb{C}}(X)=n,$ with sheaf of differential operators $\D_X.$ Fix a bundle $E\to X$ with sheaf of sections $\mathcal{E}.$ There is a space\footnote{Corresponding to a representable moduli functor \cite{KS1},\cite{KS2}.},
\begin{equation}
    \label{eqn: ModuliSpace}
\mathfrak{M}_{\D}(m,N,k),
\end{equation}
parameterizing systems
\begin{equation}
    \label{eqn: System}
    \mathsf{F}_1\big(x,u,u_{\sigma}^{\alpha}\big)=0,\ldots,\mathsf{F}_A\big(x,u,u_{\sigma}^{\alpha}\big)=0,\ldots,\mathsf{F}_N(x,u,u_{\sigma}^{\alpha})=0,
    \end{equation}
of $N$-non-linear differential operators $\mathsf{F}_i$ of orders 
$\mathrm{ord}(\mathsf{F}_i)=k_i,$ where $k:=(k_1,\ldots,k_N)\in \mathbb{N}^N,$ in $n$-independent variables $(x_1,\ldots,x_n)$, and $m$-dependent variables $u=(u^1,\ldots,u^m).$ 
The \emph{order} of (\ref{eqn: System}) is 
    $\mathrm{Ord}_{\D}(\mathcal{I}_X):=max\big\{\mathrm{ord}(\mathsf{F}_i)=k_i:i=1,\ldots,N\big\},$
but for simplicity we consider systems of \emph{pure order $k$} i.e. $k=\mathrm{ord}(\mathsf{F}_i)$ for all $i=1,\ldots,N.$ The integer $N$ is the $\D$-\emph{codimension}.

A point $[\mathcal{I}_X]$ in (\ref{eqn: ModuliSpace}) corresponds to a $\D_X$-ideal with $\mathrm{codim}_{\D_X}(\mathcal{I}_X)=N$ defined by a sequence\footnote{Up to $\D$-isomorphism: $[\mathcal{I}_X]$ denotes isomorphism classes of sequences viewed as equivalent if there are $\D_X$-isomorphism $\varphi:\mathcal{B}_X\rightarrow \mathcal{B}_X'
,$ for which $\mathcal{I}_X=ker(q)\simeq \mathcal{I}_X'=ker(q'),$ by restriction.}
\begin{equation}
    \label{eqn: SES}
   0\to  \mathcal{I}_X\rightarrow \mathcal{A}_X\xrightarrow{q}\mathcal{B}_X\to 0,
\end{equation}
where $\mathcal{A}_X$ is the commutative unital left $\D_X$-algebra of functions on an infinite jet scheme
$\mathcal{A}_X^{\ell}:=\mathcal{O}\big(\mathrm{Jet}^{\infty}(\mathcal{E})\big)$. Algebraic jet functors are characterized by the following universal property (see e.g, \cite{BD2004},\cite{Paugam2014}).
\begin{proposition}
\label{Jet-adjunction}
The functor forgetting the $\D$-action admits a left-adjoint, 
$$\mathrm{Jet}^{\infty}:\mathrm{CAlg}(\mathcal{O}_X)\rightleftarrows\mathrm{CAlg}(\D_X):\mathrm{For}_{\D_X}.$$
If $E$ is a formally smooth $X$-scheme $\mathrm{Jets}_X^{\infty}(E)$ is $\D_X$-formally smooth i.e. if $\mathcal{A}$ is formally smooth as a $\D$-algebra $For(\mathcal{A})$ is formally smooth as a $\mathcal{O}_X$-algebra.
\end{proposition}

If $\pi:E\rightarrow X$ is a locally ringed space $\pi_*\mathcal{O}_E$ is a $\mathcal{O}_X$-algebra and there is an isomorphism 
\begin{equation}
    \label{eqn: Jet-adjunction isomorphism} \mathrm{Hom}_{\mathrm{CAlg}(\D_X)}\big(\mathrm{Jet}^{\infty}(\mathcal{O}_E),\mathcal{A}\big)\cong \mathrm{Hom}_{\mathrm{CAlg}(\mathcal{O}_X)}\big(\mathcal{O}_E,\mathrm{For}_{\D}(\mathcal{A})\big).
    \end{equation}
For any section $s\in \Gamma(X,E)$ there exists a unique $\D_X$-algebra map 
$s_{\infty}^*:\mathrm{Jet}^{\infty}(\mathcal{O}_E)\rightarrow \mathcal{O}_X,$ with $s_{\infty}^*|_{\mathcal{O}_E}=s^*.$
The image of $s_{\infty}^*$ is generated by elements $\mathsf{P}\big(s^*(f)\big),$ for $\mathsf{P}\in \D_X.$ On generators
$\mathsf{P}\otimes f\in \D_X\otimes \mathcal{O}_E\subset \mathrm{Jets}^{\infty}(\mathcal{O}_E)$ we have $s_{\infty}^*\big(\mathsf{P}\otimes f\big)=\mathsf{P}\big(s^*(f)\big).$

The $\D_X$-algebra structure is provided by the lift of a differential operator $\mathsf{P}(x,\partial_x)=\sum_{\sigma}a_{\sigma}(x)\partial_x^{\sigma}$ to the jet bundle by the Cartan distribution
\begin{equation}
    \label{eqn: D-Action}
    \mathsf{P}\bullet f(x,u_{\sigma}^{\alpha}):=\mathcal{C}(\mathsf{P})(f)=\sum a^{\sigma}\mathcal{C}(\partial_x^{\sigma})(f),\hspace{1mm} \text{where }\hspace{1mm} \mathcal{C}(\partial_i):=D_i=\partial_i+\sum u_{\sigma+1_j}^{\alpha}\partial_{\sigma}^{\alpha}.
    \end{equation}

\begin{example}
\label{Nilpotent D-Algebras}
A \emph{split} $\D_X$-algebra is a commutative $\D_X$-algebra $\mathcal{A}$, together with an isomorphism of $\D$-modules
$\mathcal{A}\simeq\mathcal{I}\oplus \mathcal{O}_X,$ for some $\D_X$-module $\mathcal{I}.$
A split $\D_X$-algebra is \emph{nilpotent} if there exists some $n\in \mathbb{Z}_{>0}$ such that $\mu_{\mathcal{A}}^{\otimes n}|_{\mathcal{I}_{\mathcal{A}}}\equiv 0.$ That is, the $n$-fold product vanishes on $\mathcal{I}_{\mathcal{A}}.$
In this case, $\mathcal{I}_{\mathcal{A}}$ is the unique maximal $\D_X$-ideal in $\mathcal{A}.$ The category of nilpotent $\D_X$-algebras is $\mathrm{CAlg}_{X}^{nil}(\D_X).$
\end{example}
\subsubsection{Functor of points}
Sequences (\ref{eqn: SES}) have associated solution functors
$$\underline{Sol}_{\D}(\mathcal{I})\in Fun\big(\mathrm{CAlg}_X(\D_X)_{\mathcal{A}_X/},\mathrm{Set}\big).$$
The latter, also written referencing the quotient algebra $\mathcal{B}_X^{\ell}$ as opposed to $\mathcal{I}$ is
the representable functor defined on the category $\mathrm{CAlg}(\D_X)_{\mathcal{A}/}$ of commutative $\D_X-\mathcal{A}$-algebras with $\mathcal{R}$-points
\begin{equation}
    \label{eqn: Local solution D space}
    \underline{\mathrm{Sol}}_{\D}(\mathcal{B}):\mathrm{CAlg}(\D_X)_{\mathcal{A}/}\rightarrow \mathrm{Set},\hspace{5mm} \underline{\mathrm{Sol}}_{\D}(\mathcal{B})(\mathcal{R}):=\big\{\psi\in \mathcal{R}| \mathsf{F}(\psi)=0,\forall \mathsf{F}\in\mathcal{I}\big\}.
\end{equation}
Call (\ref{eqn: Local solution D space}) the \emph{$\D$-geometric (local) space of solutions}.
The functor of points philosophy gives a $\D$-Yoneda embedding given by
\begin{eqnarray}
    \label{eqn: D-Yoneda}
h_{\mathcal{A}}^{\D}:\mathrm{CAlg}(\D)\rightarrow \mathrm{Set},\hspace{3mm}\mathcal{B}\mapsto \mathrm{Hom}_{\mathcal{D}\mathrm{-Alg}}(\mathcal{A},\mathcal{B}),
\end{eqnarray}
where $\mathrm{Hom}_{\mathcal{D}\mathrm{-Alg}}$ is the set of commutative $\D_X$-algebra morphisms.
In particular, the space of holomorphic solutions is
\begin{equation}
    \label{eqn: Classical D-Solutions}
\mathrm{Sol}_{X}^{\D}(\mathcal{A}):=h_{\mathcal{A}}^{\D}(\mathcal{O}_X)=\mathrm{Hom}_{\mathcal{D}\mathrm{-Alg}}(\mathcal{A},\mathcal{O}_X),
\end{equation}
Here is an example in the real setting.
\begin{example}
Consider $\pi:\mathbb{R}\times\mathbb{R}\rightarrow \mathbb{R}$, with base coordinate $t$. A non-linear PDE, 
$F\big(t,\partial_t^i u\big)=0,$
on sections $u:\mathbb{R}\rightarrow \mathbb{R},$ with $F\in \mathbb{R}[t,\{u_i\}_{i\geq 0}]$ gives rise to 
$\mathrm{Sol}_{\D}(\mathcal{A}):=\{a\in \mathcal{A}|F(t,\partial_t^ia)=0\}.$
\end{example}

If $\mathscr{F}$ is another function space with a $\D_X$-algebra structure, write $\mathrm{Sol}_{X}^{\D}(\mathcal{A},\mathscr{F}),$ for $\mathscr{F}$-valued solutions. Note
we may not simply substitute spaces like hyperfunctions (\ref{eqn: SESHyperfunction}) into $\mathrm{Sol}_X^{\D}(\mathcal{A},\mathscr{B}_M)$, as they do not possess an algebra structure. However, one can modify their definition by imposing wavefront conditions \cite{Hormander}. 

\subsubsection{Microfunction solutions} The following result plays an important role the sheaf-theoretic microlocal analysis of singularity propagation for non-linear PDEs.
\begin{proposition}
\label{prop: HyperSol}
Let $M$ be a real-analytic manifold with complexification $X$.
Let $\gamma$ be a closed convex proper cone of $T_M^*X$. Then, for each open subset $U$,
$$\mathscr{B}_M^{\gamma}(U):=\{u\in \mathscr{B}_M(U)|\mathrm{SS}(u_1)\subset \gamma\},$$
defines a sheaf of $\mathbb{C}$-algebras with a compatible $\D$-algebra action.
Thus, for a $\D$-algebra the functor of hyperfunction solutions with $\gamma$-wavefront set $\mathrm{Sol}^{\D}(\mathcal{A},\mathscr{B}_M^{\gamma})$ is well-defined. In particular, there is a set of $\D$-algebra morphisms
$\mathrm{Hom}_{\mathcal{D}-\mathrm{Alg}}(\mathcal{A},\mathscr{B}_M^{\gamma}).$
\end{proposition}
\begin{proof}
Recall that the sheaf $\mathcal{D}b_M$ of distributions on $M$ is a subsheaf of $\mathscr{B}_M.$ Our hypothesis are a condition on the wavefront set of hyperfunction sections. Then, for $u_1,u_2 \in \mathscr{B}_M^{\gamma}(U)$ since $\mathrm{SS}(u_i)\subset \gamma,i=1,2$ we have a well defined product $u_1\cdot u_2\in \mathscr{B}_M^{\gamma}(U)$ since
$$\mathrm{SS}(u_1\cdot u_2)\subset\mathrm{SS}(u_1)+\mathrm{SS}(u_2)\subset \gamma+\gamma\subset \gamma,$$
but $u_1\cdot u_2$ is defined in the distribution sense if and only if\footnote{This condition ensures that no “singular directions” oppose, so the product is meaningful in the distributional sense.} $\mathrm{SS}(u_1)\cap \mathrm{SS}(u_2)^a=\emptyset.$ By our assumptions on $u_i\in \mathscr{B}_M^{\gamma}(U)$ this means $\gamma\cap \gamma^a=\emptyset,$ but $\gamma^a=-\gamma$ and since $\gamma$ is proper closed convex cone we have $\gamma\cap -\gamma=\emptyset,$ as required. 
The $\D$-action is well-posed since $\mathrm{SS}(Pu)\subset \mathrm{SS}(u)\cap\mathrm{Char}(P)\subset \gamma\cap \mathrm{Char}(P).$ This fact implies that $\mathscr{B}_M^{\gamma}$ is a $\D$-submodule of $\mathscr{B}_M.$
\end{proof}

\begin{definition}
    \normalfont 
    Let $M$ be a real analytic manifold, $X$ a complexification of $M$ and $\Omega$ an open subset of $X.$ One says that $\Omega$ is a \emph{tuboid with edge $M$,} if the Whitney normal cone $C_M(\Omega)\subset T_MX$ is a closed convex proper cone with non-empty interior at each $x\in M.$
\end{definition}
Thus Proposition \ref{prop: HyperSol} states that sections of $\mathscr{B}_M^{\gamma}$ are boundary values of sections of tuboids with edge $M$ with profile given by its polar cone.

\subsubsection{Gevrey solutions.}
We now construct another class of solution functors with functional-analytic behavior which is of practical use. For simplicity, consider $\mathbb{C}^2,$ with higher dimensional generalizations obtained in the obvious way. 

\begin{construction}
\label{cons: Gevrey}
Consider a polydisc $D_{\rho_1}\times D_{\rho_2}$ centered at $(0,0)\in\mathbb{C}^2$, where $D_{\rho_j}$ is the disc with center $0$ and radius $\rho_j>0.$ 
    Consider 
    $$u(t,x)=\sum_{j\geq 0}u_{j,*}(x)\frac{t^j}{j!}\in\mathcal{O}(D_{\rho_2})[\![t]\!].$$
    It is $s$-Gevrey if there exists $0<r_2<\rho_2$ and $C>0,K>0,$ such that 
    $$\underset{|x|\leq r_2}{\mathrm{sup}}\hspace{1mm} |u_{j,*}(x)|\leq C\cdot K^j\Gamma(1+(s+1)j),\hspace{1mm}\text{for each } j\geq 0.$$
The space of $s$-Gevrey functions is denoted $\mathcal{G}_{\rho_2}[\![t]\!]_s,$ which is naturally a presheaf of $\mathbb{C}$-algebras.
\end{construction}
The set of germs of analytic functions at the origin in $\mathbb{C}^2$ is then given by  $\bigcup_{\rho>0}\mathcal{G}_{\rho}[\![t]\!]_0.$ Any $0$-Gevrey function is convergent and for any $0<s<s'<\infty$ there is a filtration
    $$\mathcal{G}_{\rho_2}[\![t]\!]_0\subset \mathcal{G}_{\rho_2}[\![t]\!]_s\subset \mathcal{G}_{\rho_2}[\![t]\!]_{s'}\subset \mathcal{G}_{\rho_2}[\![t]\!].$$
    \begin{proposition}
    \label{prop: Gevrey solutions}
    For any $s>0$ the space of $s$-Gevrey functions is naturally a $\D_{\mathbb{C}^2}$-algebra and if $\mathcal{A}$ is a commutative $\D_{\mathbb{C}^2}$-algebra corresponding to a non-linear PDE, the space of $s$-Gevrey solutions is 
    $\mathrm{Sol}_{\D_{\mathbb{C}^2}}(\mathcal{A})^{s}:=\mathcal{H}\mathrm{om}_{\mathcal{D}-alg}(\mathcal{A},\mathcal{G}[\![t]\!]_s).$
    \end{proposition}
\begin{proof}
  Easily check that it is stable under the obvious action by $\partial_x,\partial_t$.
\end{proof}

    Let $\Sigma$ be an open sector with vertex $0\in\mathbb{C}.$ 
    \begin{definition}
    \label{definition: s-Gevrey}
    A holomorphic function $u(t,x)\in\mathcal{O}(\Sigma\times D_{\rho})$ is said to be \emph{s-Gevrey asymptotic to a formal series $\mathcal{O}(D_{\rho_2})[\![t]\!]$ on $\Sigma$} if:
    there exists an integer $0<r_2<min(\rho,\rho_2)$ such that for all proper $\Sigma'\subset \Sigma$ (compact) there exists constants $C>0$ and $K>0$ such that for every integer $J\geq 1$ and $t\in \Sigma':$
    $$\underset{|x|\leq r_2}{\mathrm{sup}}\hspace{1mm} \big|u(t,x)-\sum_{j=0}^{J-1}u_{j,*}(x)\frac{t^j}{j!}\big|\leq C\cdot K^J\cdot \Gamma(1+sJ)|t|^J.$$
    \end{definition}
For each $s>0,$ from Definition \ref{definition: s-Gevrey}, set
$$A_s(\Sigma,D_{\rho_2}):=\{u(t,x)| u \text{ is } s-\text{Gevrey asymptotic on }\Sigma \text{ to } D_{\rho_2}\}.$$
Spaces of $s$-Gevrey solutions are related by taking asymptotic series.
 \begin{proposition}
 Consider Proposition \ref{prop: Gevrey solutions} and the space of solutions $\mathrm{Sol}_{\D_{\mathbb{C}^2}}(\mathcal{A})^s.$ Then, there is a canonical morphism of algebraic $\D$-spaces sending $u=u(t,x)$ to its asymptotic series.
 \end{proposition}
 \begin{proof}
By construction, the map of interest is the canonically defined morphism $T_{s,\Sigma,D_{\rho_2}}:A_s(\Sigma,D_{\rho_2})\rightarrow \mathcal{O}(D_{\rho_2})[\![t]\!]_s,$
    that sends $u(t,x)$ to its asymptotic series. It is easy to verify it is a morphism of $\D_{\mathbb{C}^2}$-algebras and thus by functoriality, there is an induced map of solution spaces
    $\mathrm{Sol}_{\D}(\mathcal{A},T_{s,\Sigma,\rho_2}):\mathrm{Sol}_{\D_{\mathbb{C}^2}}(\mathcal{A},A_s)\rightarrow \mathrm{Sol}_{\D_{\mathbb{C}^2}}(\mathcal{A})^{s}.$
    \end{proof}

\subsubsection{Choosing a topology.}
To develop functional-analytic scheme theory for $\D$-spaces in what follows, we require a suitable $\D$-topology.
\begin{definition}
\label{definition: D-open immersion}
A \emph{$\D$-open immersion} $\mathcal{A}\hookrightarrow \mathcal{B}$ is a morphisms of $\D$-algebras such that:
\begin{itemize}
    \item[(i)] For every $\mathcal{C}\in \mathrm{CAlg}(\mathcal{D})$, the induced map $\mathrm{Sol}_{X}^{\D}(\mathcal{B},\mathcal{C})\hookrightarrow \mathrm{Sol}_X^{\D}(\mathcal{A},\mathcal{C}),$ is an injective map of sets;

    \item[(ii)] The base-change functor e.g. $-\otimes_{\mathcal{A}}\mathcal{B}:\mathrm{Mod}_{\D_X}(\mathcal{A})\rightarrow \mathrm{Mod}_{\D_X}(\mathcal{B}),$ commutes with finite limits;

    \item[(iii)] The functor $\mathrm{CAlg}(\D_X)_{\mathcal{A}/}\rightarrow \mathrm{Sets},$ given by $\mathcal{C}\mapsto \mathrm{Hom}_{\mathcal{D}-\mathrm{Alg}_{\mathcal{A}/}}(\mathcal{B},\mathcal{C})$ commutes with filtered colimits.
\end{itemize}
\end{definition}
These conditions will be elaborated later in (\ref{ssec: D-Schemes}).
Standard results about representability of functors by quotient algebras of polynomials, adapted to the $\D$-geometric setting give the following.
\begin{proposition}
\label{Jet Representability}
    Let $\mathcal{I}\rightarrow\mathcal{A}\rightarrow\mathcal{B}$ be a $\D$-smooth algebraic non-linear \textsc{pde} with $\D$-space of solutions $\mathsf{Sol}_{\D}(\mathcal{B}).$
    Then there is an isomorphism of affine $\D$-schemes $$\mathrm{Sol}_{\D}(\mathcal{I})\simeq Spec_{\D}(\mathcal{O}(\mathrm{Jets}^{\infty}(\mathcal{O}_E^{alg}))/\mathcal{I}\big).$$
\end{proposition}

We use other notation for the functor of points solution space,
$$\mathrm{Sol}_{\D}(\mathcal{I})=\mathrm{Sol}_{\D}(\mathsf{F}=0)\simeq\big\{(x,u_{\sigma}^{\alpha})|\D_X\bullet \mathsf{F}^A(x,u_{\sigma}^{\alpha})=0\big\}.$$
For example, if $\mathcal{A}=\mathrm{Sym}_{\mathcal{O}_X}(\D_X^k),$ consider the defining sequences for $\mathcal{B}$,
\begin{equation}
    \label{eqn: SES Example}
span\{\mathsf{P}_1,\ldots,\mathsf{P}_N\}\rightarrow \mathrm{Sym}_{\mathcal{O}_X}(\D_X^k)\rightarrow \mathcal{B},
\end{equation}
whose differential polynomials $\mathsf{P}_i(x,\partial_x),i=1,\ldots,N,$ are of fixed $\mathrm{ord}(\mathsf{P}_i):=d_i$ for $i=1,\ldots,N.$
The $\D$-space of solutions functor $\mathrm{Sol}_{\D_X}(\mathcal{I})$ is given by $$\mathrm{Sol}_{\D_X}(\mathcal{I})\simeq \{(f_1,\ldots,f_n)|\mathsf{P}_i(f_1,\ldots,f_n)=0\},$$ for some collection of functions $f_i\in\mathcal{O}_X.$ In this example, $\mathcal{A}$ corresponds to $k$-jets on the trivial bundle and in the more general situation, is the free $\D_X$-algebra on $\D_X^k\otimes_{\mathcal{O}_X}\mathcal{O}_E^{alg},$ whose elements are of the form 
$f=f(x,u,\partial_{\sigma}u^{\alpha}),|\sigma|\leq k.$

\begin{example}
\label{Differential Algebra Example}
For a (possibly infinite) family of variables $X_0,X_1,\ldots,X_i,\ldots$ consider $\mathbb{A}^1$ with coordinate $z$ and $\mathcal{V}=\mathbb{C}(\!(\EQ)\!)\cdot [X_0,X_1,\ldots,].$ It is acted upon by $\D_{\mathbb{A}^1}$ via the rule $\partial\bullet X_i:=X_{i+1}$ where $\partial_z$ acts in the usual way on $\mathbb{C}(\!(\EQ)\!).$ The $\D$-space determined by $\mathsf{P}\in \mathcal{V}$ is  $Spec_{\D_{\mathbb{A}^1}}(\mathcal{B})$ via
$\mathcal{I}_{\partial}<\mathsf{P}>\rightarrow \mathcal{V}\rightarrow \mathcal{B}.$ 
Allowing for coefficients in $\mathbb{C}$-algebras $A,$ put $A(\!(\EQ)\!)=A\otimes_{\mathbb{C}} \mathbb{C}(\!(\EQ)\!),$ so $\mathcal{V}_A:=A(\!(\EQ)\!)[X_0,X_1,\ldots],$ is a $\mathcal{V}$-algebra with compatible $\D_{\mathbb{A}^1}$-action.
The corresponding solution functor (c.f. (\ref{eqn: D-Yoneda}) and (\ref{eqn: Local solution D space})) in this setting is given by 
$\underline{\mathrm{Sol}}_{\D_{\mathbb{A}^1}}(A)=Hom_{\mathrm{CAlg}_{A(\!(\EQ)\!)}(\D_{\mathbb{A}^1})}(\mathcal{B}_A,A(\!(\EQ)\!)\big).$
In particular, a $\mathbb{C}$-point $s\in \underline{\mathrm{Sol}}_{\D_{\mathbb{A}^1}}(\mathsf{P})(\mathbb{C})$ corresponds to a map $\mathcal{B}\rightarrow \mathbb{C}(\!(\EQ)\!).$ 
\end{example}

The system $\mathsf{F}=0$ or equivalently the $\D_X$-ideal $\mathcal{I}_X$ is said to be non-trivial (resp. consistent) if 
$\mathrm{Sol}_{\D}(\mathcal{I})\neq Spec_{\D}(\mathcal{A}_X),$ (resp. $\mathrm{Sol}_{\D}(\mathcal{I})\neq \emptyset.$) The condition of consistency is related to the formal integrability and involutivity (see Definition \ref{defn: RelAlgNLPDE}).

A tuple of functions $\varphi=(\varphi^1,\ldots,\varphi^m)$ is a solution if $$\mathsf{F}^A(x,u_{\sigma}^{\alpha})|_{u=j_{\infty}\varphi(x)}=0\Leftrightarrow\big\{(x,u_{\sigma}^{\alpha})|u_{\sigma}^{\alpha}=j_{\infty}\varphi(x)=\varphi_{\sigma}^{\alpha}\big\}\subset \mathrm{Sol}_{\D}(\mathcal{I}).$$
In terms of morphisms of $\D_X$-algebras, it means (the dual of) its jet-extension factors through the structure map $q:\mathcal{A}_X^{\ell}\rightarrow \mathcal{B}_X^{\ell}$ i.e. we have a commutative diagram
\begin{equation}
\label{eqn: Point of non-local solution space factorization}
\begin{tikzcd}
   \mathcal{A}_X^{\ell}\arrow[dr,"j_{\infty}(\varphi)^*"] \arrow[r,"q"] & \mathcal{B}_X^{\ell}
   \\
   & \mathcal{O}_X\arrow[u]
\end{tikzcd}.
\end{equation}
This notion defines a sub-functor
\begin{equation}
    \label{eqn: Non-local solution space}
Sol_{\mathcal{I}}(X,E)=\{s\in \mathrm{Sect}(X,E)|j_{\infty}(s)^*L=0,\forall L\in \mathcal{I}\}\subset \mathrm{Sect}(X,E).
\end{equation}
Space (\ref{eqn: Non-local solution space}) consists of
$s=s(t,u)\in \underline{Sect}(X,E)(A)\subset \mathrm{Hom}(X\times \mathrm{Spec}(A),E),$ for which $L\circ \big(j_{\infty}s\big)(t,u)=0.$
So $s\in\underline{Sol}_{\mathcal{I}}(X,E)$ if and only if there is a natural factorization through the dual jet map $(j_{\infty}^*):\mathcal{A}\rightarrow \mathcal{O}_X$ as in (\ref{eqn: Point of non-local solution space factorization}).
\begin{remark}
A point in (\ref{eqn: Non-local solution space}) is called a solution section and if $\mathcal{B}$ is defined by $\Delta_{\mathrm{F}}=F\circ j_k(-)=0,$ they are those $s\in \mathrm{Sect}(X,E)$ for which $j^k(s):X\rightarrow \mathrm{Jet}^k(E)$ satisfies $\mathrm{Im}(j^k(s))\subset \mathcal{E}$ i.e. $j^k(s)\in Sect\big(X,\mathcal{E}\subset \mathrm{Jet}^k(E)\big).$
\end{remark}

\subsubsection{$\D_X$-Ideal Sheaves.}
\label{sssec: D-Ideal Sheaves}
An ideal $\mathcal{I}_X$ is said to be differentially generated if it is algebraically generated by
$\D_X\bullet (\mathsf{F}^A),A=1,\ldots,N,$
defined by
$\sum_{|\sigma|\geq 0}D_{\sigma}(\mathsf{F}),D_{\sigma}:=D_1^{\sigma_1}\circ\ldots D_n^{\sigma_n}.$

Differentially generated $\D_X$-ideals contain all differential consequences of their generators. It is possible to restrict to finite dimensional filtered subspaces $F^k\mathcal{A}$ for $k\geq 0$ (finite jets) but then the ideal will only be algebraic.
It is defined, for all $k\geq 0$ by
\begin{equation}
    \label{eqn: FiniteIdeal}
F^k\mathcal{I}:=\mathcal{I}\cap F^k\mathcal{A}^{\ell}\simeq Span\big\{D^{\sigma_i}(\mathsf{F}_i)|\mathrm{ord}(\mathsf{F}_i)+|\sigma_i|\leq k\big\}\subseteq F^k\mathcal{A}^{\ell}.
\end{equation}
Algebraic ideal (\ref{eqn: FiniteIdeal}) is understood as automatically containing all
integrability conditions up to order $k$.
\begin{definition}
\label{defn: RelAlgNLPDE}
A \emph{$\D$-algebraic non-linear PDE} is a sequence of commutative $\D_X$-algebras (\ref{eqn: SES}).
When $\mathcal{A}=\mathrm{Jet}^{\infty}(\mathcal{O}_E),$ for some bundle (or sheaf) $E$ over $X$, we say the $\D$-PDE is imposed on sections $\mathrm{Sect}(X,E),$ and further say that it is:
\begin{enumerate}
    \item \emph{Differentially generated}: if it is free as a $\D$-module and whose generic ideal
is of the form 
$$\mathcal{I}=\big\{\psi \in \mathcal{A}| \psi=\sum_{i=1,|\sigma|\geq 0}^N\psi_{i,\sigma}(x,u,\partial_x^{\tau}u) D_{\sigma}(\mathsf{F}_i),\psi_{i,\sigma}\in \mathcal{A}\big\},$$ via the action  (\ref{eqn: D-Action});

    \item \emph{Formally integrable} if for every $k$, we have $\mathcal{O}_X$-module isomorphisms 
$$(F^k\mathcal{B}^{\ell}\otimes\D_X^{\leq k}\otimes E^*)\cap \big(F^{k-1}\mathcal{A}^{\ell}\otimes\D_X^{\leq k-1}\otimes E^*)=F^{k-1}\mathcal{B}^{\ell}\otimes\D_X^{\leq k-1}\otimes E^*.$$
\end{enumerate}
\end{definition} 
We give one standard example.
\begin{example}
\label{ex: EvolutionaryExample}
Evolutionary equation $\{F:=u_t-f(x,t,u_i)=0\}$ with independent variables $(x,t)$ and dependent one $u$, with $u_i:=\partial_x^iu$ correspond to differentially generated $\D$-PDEs. In this case, the $\D$-action is generated by 
$D_x=\partial_x+\sum_i u_{i+1}\partial_{u_i},$ and $D_t=\partial_t+\sum_iD_x^i(f)\partial_{u_i}.$ For instance, considering a trivial bundle with fiber coordinate $u$ over a two dimensional base $(t,x)\in \mathbb{R}^2,$ a prime $\D_X$-ideal 
    $\mathcal{I}_{KdV}=Span\{\partial_t\bullet u-6u\partial_t\bullet u+\partial_x^3\bullet u\},$ via the usual $\D_X$-action gives the $\D_X$-algebra corresponding to the KdV equation.
\end{example}
Example \ref{ex: EvolutionaryExample} admit obvious generalizations e.g. 
$$\big\{\mathsf{F}_j:=u_t^j-f^j(x,t,\ldots,u^{\alpha},\ldots,u_x^{\alpha},\ldots)=0\big\},\hspace{1mm}j=1,\ldots,N,$$
with $\D$-action via operators
$$D_i=\partial_i+\sum_{j=1}^n\sum_{|\sigma|\geq 0}u_{\sigma i}^j\partial_{u_{\sigma}^j},\hspace{2mm} D_t=\partial_t+\sum_{j=1}^n\sum_{|\sigma|\geq 0}D_{\sigma}(f^j)\partial_{u_{\sigma}^j},i=1,\ldots,n.$$

 Integrability condition (2) coincides with the usual one. The following remark elaborates its meaning for linear systems.
\begin{remark}
Consider $\mathsf{P}\in \D_X^{N_0\times N_1}$ with $R=\mathrm{ord}(\mathsf{P}):=max\{\mathrm{ord}(\mathsf{P}_{ij})|i=1,\ldots,N_0,j=1,\ldots,N_1\}$. Formal integrability is the condition that for every integer $\ell\in \mathbb{Z},$
$$\big(\mathcal{D}^{1\times N_0}\cdot \mathsf{P})\cap F^{R+\ell}\mathcal{D}^{1\times N_1}=F^{\ell}\mathcal{D}^{1\times N_0}\cdot\mathsf{P}.$$
This means for every $\ell$ the system of order $R+\ell$ given by the left-hand side is the $\ell$'th prolongation $\mathsf{P}^{(\ell)}:=F^{\ell}\mathcal{D}^{1\times N_0}\cdot \mathsf{P}$ of the original system. This gives an infinite amount of conditions to check, but fundamental results of Goldschmidt and Spencer tell us that it can be replaced by a single condition
$$(\D_X^{1\times N_0}\cdot \mathsf{P})\cap F^{R+1}\D_X^{1\times N_1}=F^1\D_X^{1\times N_0}\cdot \mathsf{P},$$ if its symbol satisfies an additional condition\footnote{It is $2$-acyclic \cite{Krasilshchik1986}. In these terms, a PDE involutive if both formally integrable and its symbol is $n$-acyclic, for some $n.$}.
\end{remark}
A differential corollary of a differentially generated ideal $\mathcal{I}$ corresponding to a system $\{\mathsf{F}=0\}$ of order $k$ is a $\D_X$-ideal $\mathcal{I}_0\subset \mathcal{I}$ such that it is not a subset of any differentially generated ideal $\mathcal{I}'$ corresponding to some $\mathsf{F}'$ whose zero locus is a proper closed subset of that for $\mathsf{F}.$  

\subsection{Global theory: $\D$-Schemes}
\label{ssec: D-Schemes}
Arbitrary (not necessarily affine) $\D_X$-schemes $(Z,\psi)$ are built by gluing coverings by affine $X$-schemes $U_i$ for which the families $(U_i,\psi|_{U_i}:\mathcal{O}_{U_i}\rightarrow\mathcal{O}_{U_i}\otimes \pi_i^*\Omega_X)$ are affine $\D_X$-schemes. Introducing topologies, one may globalize further do define algebraic $\D$-spaces.

We recall properties of morphisms of $\D$-schemes defined in \cite{BD2004}.
A map $F:\mathcal{X}:=\mathrm{Spec}_{\D_X}(\mathcal{A})\rightarrow\EQ:=\mathrm{Spec}_{\D_X}(\mathcal{B})$ is \emph{formally} $\D$-\emph{smooth} if for all $\D_X$-algebras $\mathcal{C}$ and any nilpotent $\D_X$-ideal $\mathcal{J}\subset \mathcal{C},$ as well as every map
$\mathrm{Spec}_{\D_X}(\mathcal{C})\rightarrow \EQ,$ the canonical map
$$\mathrm{Hom}_{\D_X}(\mathrm{Spec}_{\D_X}(\mathcal{C}),\mathcal{X})\rightarrow \mathrm{Hom}\big(\mathrm{Spec}_{\D_X}(\mathcal{C}/\mathcal{J}),\mathcal{X}),$$
is surjective. 
If it is injective we say it is $\D_X$-\emph{unramified} and if bijective, it is $\D_X$-\emph{étale}. 
In particular,
$\mathcal{A}^{\ell}\in \textsc{CAlg}_X(\D_X)$ is $\D_X$-\emph{formally smooth} if: for any $\mathcal{C}^{\ell}$ and $\D$-ideal $\mathcal{I}\subset\mathcal{C}^{\ell}$ with $\mathcal{I}\cdot \mathcal{I}=0,$ one solves the lifting problem: 

\[
\begin{tikzcd}
& \mathcal{C}_X^{\ell}\arrow[d,"q"]
\\
\mathcal{A}_X^{\ell}\arrow[ur,"\exists",dashed] \arrow[r] & \mathcal{C}_X^{\ell}/\mathcal{I}
\end{tikzcd}
\]
For a $\D$-module $\mathcal{M}$, the $\D$-algebra
    $\mathcal{A}^{\ell}=\mathrm{Sym}_{\mathcal{O}_X}(\mathcal{M})$ is $\D$-formally smooth if and only if $\mathcal{M}$ is projective.

Fix a suitable topology $\tau$ on affine objects (it will be the Zariski $\D$-topology, introduced below).
\begin{definition}
\label{defn: (Formal) Algebraic D-Spaces}
\normalfont
An \emph{algebraic $\D_X$-space} is a presheaf $\mathcal{X}:\big(\mathrm{CAlg}_{\D_X}^{\mathrm{op}},\tau\big)\rightarrow \mathrm{Set},$ that is a sheaf for $\tau.$ An algebraic $\D_X$-space is said to be \emph{formal} if it is of the form $\mathcal{X}:\mathrm{CAlg}_X^{nilp}(\D_X)^{op}\rightarrow \mathrm{Set}.$
\end{definition}

Denote the corresponding category as $\mathrm{Space}_{\D_X,\tau}.$ The topology $\tau$ will be a $\D$-analog of the usual Zariski topology. Once equipped with it we can define $\D$-sheaves and $\D$-schemes. We do this now.

\subsubsection{} Any affine $\D$-space is viewed as a functor which is representable, $\mathrm{Spec}_{\D}(\mathcal{A})$ for some $\mathcal{A}\in \mathrm{CAlg}(\D_X).$ This defines a covariant assignment
$$\underline{\mathrm{Spec}}_{\mathscr{D}}(-):\mathrm{CAlg}_{X}(\D_X)^{op}\rightarrow \mathrm{Space}_{\D_X}.$$

To identify a category of $\D$-schemes, we follow the ideas of Grothendieck (e.g. \cite{SGA4I}) and first supply  $\mathrm{Space}_{\D_X}$ with a (pre)-topology and then define schemes in terms of functors which are sheaves for this topology and which can be covered by affine objects. These so-called $\D$-Zariski local sheaves,
$$\mathcal{X}\in\mathrm{Space}_{\D_X}^{Zar}:=\mathrm{Shv}\big(\mathrm{CAlg}_{X}(\D_X)^{op},\tau_{Zar}\big),$$
are such that: there exists a finite $\D$-covering 
$\{\mathcal{U}_i\rightarrow \mathcal{X}\}_{i\in I},$
where each $g_i:\mathcal{U}_i\rightarrow \mathcal{X}$ are $\D$-open immersions (c.f. (\ref{eqn: Classical D-Solutions}) items (i) -- (iii)) and the induced map
$$\coprod_{i\in I}\mathcal{U}_i\rightarrow \mathcal{X},$$
is a $\D_X$-epimorphism. The \emph{$\D$-Zariski pre-topology} on $\mathrm{CAlg}(\D)^{op}$, is the covering family defined for each $\D$-algebra $\mathcal{A}$ by:
\begin{equation}
\label{ZariskiCoverings}
\mathrm{Cov}_{Zar}(\mathcal{A}_X^{\ell}):=
\bigg\{
\text{Finite families }(\mathcal{A}\xrightarrow{g_i}\mathcal{B}_i)_{i\in I}\bigg| \begin{array}{c}

\text{there exists }(a_i)_{i\in I}\in \mathcal{A}^I,\text{with } (a_i)= \mathcal{A}
\\
\text{and an isomorphism of families}
\\
\{\mathcal{A}\xrightarrow{g_i}\mathcal{B}_i\}_{i\in I}\simeq \{\mathcal{A}\rightarrow\mathcal{A}_{a_i}\}_{i\in I},
\end{array}\bigg\}.
\end{equation}
We denote by $\mathcal{A}_{a_i}$ the localization at the element $a_i.$ Conditions \eqref{ZariskiCoverings} are well-defined by results of Ritt \cite{Ri}, which tell us the localization of a differential algebra at a multplicative subset is again a differential algebra.

The closed Zariski subsets of $\mathrm{Spec}_{\D}(\mathcal{A})$ are identified using the $\D$-ideals of $\mathcal{A}.$ 
Given a $\D$-ideal $\mathcal{I}$ of $\mathcal{A},$ denote by 
$$\underline{\mathrm{Spec}}_{\mathscr{D}}(\mathcal{I}\subset\mathcal{A}):\mathrm{CAlg}(\mathcal{D})\rightarrow \mathrm{Sets},$$
the algebraic $\D$-space defined by
$\underline{\mathrm{Spec}}_{\mathscr{D}}(\mathcal{I}\subset \mathcal{A})(\mathcal{R}_X^{\ell}):=\{\varphi\in \mathrm{Sol}_{\D}(\mathcal{A},\mathcal{R}_{X}^{\ell})|\varphi(\mathcal{I})=\mathcal{R}_{X}^{\ell}\}.$
It is a subset of $\mathcal{R}_X^{\ell}$-points of $\mathrm{Spec}_{\D}(\mathcal{A})(\mathcal{R}_X^{\ell}).$
It is straightforward to check if $\mathcal{I},\mathcal{I}'$ are $\D$-ideals of $\mathcal{A},$ then 
$$\underline{\mathrm{Spec}}_{\mathscr{D}}(\mathcal{I}\subset\mathcal{A})=\underline{\mathrm{Spec}}_{\mathscr{D}}(\mathcal{I}'\subset\mathcal{A}),$$
if and only if $\sqrt{\mathcal{I}}=\sqrt{\mathcal{I}'}.$
\begin{proposition}
    The functor $\underline{\mathrm{Spec}}_{\mathscr{D}}(\mathcal{I}\subset\mathcal{A})$ depends only on the radical of the $\D$-ideal $\mathcal{I}.$
\end{proposition}
We use the following notion of open immersion in the category of $\D$-spaces.
\begin{definition}
Let $\mathcal{F}\rightarrow \underline{\mathrm{Spec}}_{\mathscr{D}}(\mathcal{A})$ be a morphism of $\D$-spaces. It is an \emph{$\D$-open immersion} if it is isomorphic to a map of the form 
$\underline{\mathrm{Spec}}_{\mathscr{D}}(\mathcal{I}\subset \mathcal{A})\rightarrow \underline{\mathrm{Spec}}_{\mathscr{D}}(\mathcal{A}),$
for some $\D_X$-ideal $\mathcal{I}$.
\end{definition}
Such morphisms have the following characterization.
\begin{proposition}
    Consider a morphism $F:\mathcal{X}\rightarrow \mathcal{Y}$ of algebraic $\D$-spaes. It is a $\D$-open immersion if it is a monomorphism of functors and if: for every $\mathcal{A}\in \mathrm{CAlg}(\D_X),$ and morphism $\mathrm{Spec}_{\D}(\mathcal{A})\rightarrow \EQ,$ there exists a $\D$-ideal $\mathcal{I}$ of $\mathcal{A}$, such that
    \[\begin{tikzcd}
        \underline{\mathrm{Spec}}_{\mathscr{D}}(\mathcal{I}\subset\mathcal{A})\arrow[d]\arrow[r] & \mathrm{Spec}_{\D}(\mathcal{A})\arrow[d]
        \\
        \mathcal{X}\arrow[r,"F"] & \mathcal{Y}.
    \end{tikzcd}
    \]
    is a Cartesian diagram in $\mathrm{Space}_{\D_X}.$ Moreover,  $\D$-open immersions are stable under base change.
   \end{proposition}

\begin{proof}
 Consider a Cartesian diagram of $\D$-spaces
    \[
    \begin{tikzcd}
\mathcal{X}'\arrow[r,"F'"] \arrow[d] & \mathcal{Y}'\arrow[d]
\\
\mathcal{X}\arrow[r,"F"] & \mathcal{Y}.
    \end{tikzcd}
    \]
    Suppose that $F$ is a $\D$-open immersion. Then we must show $F'$ is $\D$-open as well. To this end, let $\mathcal{A}$ be a $\D$-algebra with a map $\Sp(\mathcal{A})\rightarrow \mathcal{Y}',$ and consider the pull--back
    \[
    \begin{tikzcd}
        \mathcal{X}'\times_{\mathcal{Y}'}\Sp(\mathcal{A})\arrow[d]\arrow[r] & \Sp(\mathcal{A})\arrow[d]
        \\
        \mathcal{X}'\arrow[r] & \mathcal{Y}'
    \end{tikzcd}.
    \]
By adjoining this to the diagram given, by our assumptions there exists $\mathcal{I}$ such that $\mathcal{X}'\times_{\mathcal{Y}'}\Sp(\mathcal{A})\simeq \Sp(\mathcal{I}\subset \mathcal{A}).$
\end{proof}

By definition $\mathcal{X}:\mathrm{CAlg}(\mathcal{D})\rightarrow \mathrm{Sets},$ is a Zariski $\D$-space, which is to say a sheaf for the $\D$-Zariski pre-topology \eqref{ZariskiCoverings} if for any $\mathcal{A}$ and finite family $(a_i)\in \mathcal{A}^I$ generating $\mathcal{A},$ one has the usual equalizer diagram 
$$\mathcal{X}(\mathcal{A})\rightarrow \prod_{i}\mathcal{X}(\mathcal{A}_{a_i})\rightrightarrows\prod_{i,j}\mathcal{X}(\mathcal{A}_{a_ia_j}).$$

\begin{proposition}
    Suppose that $\mathcal{X}\in \mathrm{Sch}_{\D_X}.$ Then, there exists a family $(\mathcal{A}_i)_{i\in I}$ of commutative $\D_X$-algebras such that $\{h_{\mathcal{A}_i}^{\D}\xrightarrow{u_i}\mathcal{X}\}$ is a $\D$-Zariski covering family.
\end{proposition}
\begin{proof}
We must show the Zariski pretopology is subcanonical, that each $u_i$ is a $\D$-open immersion and the induced map 
$$\coprod_{i\in I}h_{\mathcal{A}_i}^{\D}\rightarrow \mathcal{X},$$
is a $\D$-epimorphism.
The first claim is clear via standard arguments by noticing for any map $\mathcal{A}\rightarrow \mathcal{B}$ where $(b_i)$ is a finite family generating $\mathcal{B},$ one has 
$$h_{\mathcal{A}}^{\D}(\mathcal{B})\rightarrow \prod_{i\in I}h_{\mathcal{A}}^{\D}(\mathcal{B}_{b_i})\rightrightarrows \prod_{i,j}h_{\mathcal{A}}^{\D}(\mathcal{B}_{b_ib_j}),$$
which can of course be written as
$$\mathrm{Sol}_{\D}(\mathcal{A},\mathcal{B})\rightarrow \mathrm{Sol}_{\D}\big(\mathcal{A},\prod_{i}\mathcal{B}_{b_i}\big)\rightrightarrows \mathrm{Sol}_{\D}\big(\mathcal{A},\prod_{i,j}\mathcal{B}_{b_ib_j}\big).$$
Now proceed via standard arguments (e.g.\cite{SGA4I}).
For the second claim, it suffices to notice that there exists a canonical map 
$$\coprod_{i,j}\mathrm{Sol}_{\D}(\mathcal{A}_i)\times_{\mathcal{X}}\mathrm{Sol}_{\D}(\mathcal{A}_j)\rightrightarrows \coprod_{i}\mathrm{Sol}_{\D}(\mathcal{A}_i),$$
from the pull-back,
\[
\begin{tikzcd}
    & \arrow[dl] \mathrm{Sol}_{\D}(\mathcal{A}_i)\times_{\mathcal{X}}\mathrm{Sol}_{\D}(\mathcal{A}_j)\arrow[dr] & 
    \\
    \mathrm{Sol}_{\D}(\mathcal{A}_i)& & \mathrm{Sol}_{\D}(\mathcal{A}_j).
\end{tikzcd}
\]
Then by base change and using the $\D$-open immersions of $\D$-spaces, $\mathrm{Sol}_{\D}(\mathcal{A}_i)\rightarrow \mathcal{X}$ and $\mathrm{Sol}_{\D}(\mathcal{A}_j)\rightarrow \mathcal{X},$ one can show that natural morphism
$$\mathrm{colim}\hspace{1mm}\big(\coprod_{i,j}\mathrm{Sol}_{\D}(\mathcal{A}_i)\times_{\mathcal{X}}\mathrm{Sol}_{\D}(\mathcal{A}_j)\rightrightarrows \coprod_{i}\mathrm{Sol}_{\D}(\mathcal{A}_i)\big)\rightarrow \mathcal{X},$$
is an isomorphism.
\end{proof}

\begin{remark}
    \label{rmk: LocRingDSch}
A $\D$-scheme may be thought of as a ringed space $(\mathsf{X},\mathcal{O}_{\mathsf{X}}),$ together with an open covering $\cup_iV_i$ such that each $(V_i,\mathcal{O}_{\mathsf{X}}|_{V_i})$ is isomorphic to an affine $\D$-scheme. Thus, a morphism $F:(\mathsf{X},\mathcal{O}_{\mathsf{X}})\rightarrow (\mathsf{Y},\mathcal{O}_{\mathsf{Y}})$ is a morphism of topological spaces $|F|:|\mathsf{X}|\rightarrow |\mathsf{Y}|,$ with a morphism of sheaves of $\D$-algebras $F^{\sharp}:F^*\mathcal{O}_{\mathsf{Y}}\rightarrow \mathcal{O}_{\mathsf{X}}.$
\end{remark}
By Remark \ref{rmk: LocRingDSch}, by a closed $\D$-subscheme $\mathsf{U}\xrightarrow{i}\mathsf{X}$, we mean a $\D$-scheme such that $|i|:|\mathsf{U}|\subseteq |\mathsf{X}|$ is an inclusion of topological spaces whose image is a closed subset and whose sheaf map $i^{\sharp}:i^*\mathcal{O}_{\mathsf{X}}\rightarrow \mathcal{O}_{\mathsf{U}}$ is a surjective morphism of sheaves of $\D$-algebras.
Similarly, $\mathsf{U}$ is an open $\D$-subscheme if $|\mathsf{U}|\subseteq |\mathsf{X}|$ is an open subset such that for every open subset $\mathsf{V}$ of $\mathsf{U}$, one has $\mathcal{O}_{\mathsf{U}}(V)=\mathcal{O}_{\mathsf{X}}(\mathsf{W}\cap\mathsf{U}).$

Using this interpretation set the following.
\begin{remark}[Notation]
To each $\mathcal{A}_{X}^{\ell}\in \mathrm{CAlg}(\D_X),$ one has a $\D$-scheme viewed as a locally ringed space, $\big(\mathrm{Spec}_X(\mathcal{A}),\mathcal{O}_{\mathrm{Spec}}(\mathcal{A})\big),$ where we consider the underlying topological space with structure sheaf carrying a $\D$-action induced from that on $\mathcal{A}.$ Denote the object viewed this way by $\underline{\mathrm{Sp}}_{\D}(\mathcal{A}).$ 
\end{remark}
The following clarifies how this point of view relates to the previous one defined in terms of $\mathrm{Spec}_{\D}(\mathcal{A}),$ by their behaviour as objects of $\mathrm{Space}_{\D_X}$ viewed in terms of their $\D$-functor of points.
\begin{proposition}
\label{prop: D-Yoneda FF}
    Functor \emph{(\ref{eqn: D-Yoneda})} is fully-faithful and admits a left-adjoint $|-|$ characterized as follows. For each
 $\mathcal{A}\in \mathrm{CAlg}(\D_X),$ there is an isomorphism    $|h_{\mathcal{A}_X}^{\D}|\simeq \underline{\mathrm{Sp}}_{\D}(\mathcal{A}_X).$ Moreover, $h_{\mathcal{A}}^{\D}\simeq h_{\mathrm{Sp}_{\D}(\mathcal{A})},$ is an isomorphism of $\D$-spaces. Given a finite covering family of open immersions $\{\underline{\mathrm{Sp}}_{\D}(\mathcal{A}_i)\rightarrow \mathcal{X}\}$ with $\mathcal{X}$ a $\D$-scheme, then the induced map 
    $$\coprod_{i}\underline{\mathrm{Sp}}_{\D}(\mathcal{A}_i)\rightarrow \mathcal{X},$$
    is a $\D$-epimorphism of $\D$-spaces.
\end{proposition}
\begin{proof}
    Let $h:=h^{\D}|_{\mathrm{Sch}_{\D}}.$ Then, there are equivalences 
    \begin{equation*}
h_{\underline{\mathrm{Sp}}_{\D}(\mathcal{A})}(\mathcal{R}_X^{\ell})\simeq \mathrm{Hom}\big(\mathrm{Sp}_{\D}(\mathcal{R}_X^{\ell}),\mathrm{Sp}_{\D}(\mathcal{A})\big)
\simeq\mathrm{Hom}_{\mathcal{D}-\mathrm{Alg}}(\mathcal{A},\mathcal{O}(\mathrm{Sp}_{\D}(\mathcal{R})\big)
\simeq\mathrm{Hom}_{\mathcal{D}-\mathrm{Alg}}(\mathcal{A},\mathcal{R}).
\end{equation*}
    Secondly, it is enough to note by Proposition \ref{prop: D-Yoneda FF}, one has 
    $$|h_{\mathcal{A}}^{\D}|(\mathcal{R})=|\mathrm{Sol}_{\D}(\mathcal{A})|(\mathcal{R})\simeq \underset{\mathrm{Hom}_{\mathcal{D}-\mathrm{Alg}}(\mathcal{A},\mathcal{R})}{\mathrm{colim}}\hspace{1mm}\mathrm{Sp}_{\D}(\mathcal{R}).$$
To prove the final claim note that applying the $\D$-Yoneda functor (\ref{eqn: D-Yoneda}), one has an induced map 
$$\coprod_i h_{\mathrm{Sp}_{\D}(\mathcal{A}_i)}\rightarrow h_{\mathcal{X}}^{\D},$$
which follows via the isomorphism 
$h_{\coprod_{i}\mathrm{Sp}_{\D}(\mathcal{A}_i)}\simeq \coprod_{i}h_{\mathrm{Sp}_{\D}(\mathcal{A}_i)}.$
\end{proof}

\begin{remark}
A $\D$-algebraic variety is a $\D$-scheme which is compact, separated and reduced which moreover admits a covering by open sets, each isomorphic
to some $\D$-affine space.
It is possible to prove their characterization as follows: $\mathcal{V}$ is a $\D$-variety if it is a reduced $\D$-scheme of the form $(\mathcal{V},\mathcal{O}_{\mathcal{V}})$ where $\mathcal{V}$ is representable by some reduced $\D$-algebra $\mathcal{A}$ of finite type that is a quoient by a radical $\D$-ideal $\mathcal{I}.$
\end{remark}

\subsubsection{Cauchy-Kovalevskaya Equations}
Consider a locally closed embedding $Y\xrightarrow{\iota}X$ and let $(x,t):=(x_1,\ldots,x_n,t)$ be coordinates on $X$ such that $Y=\{t=0\},$ and let $P\in\D_X^{\leq k},$ of order $\leq k.$ If $Y$ is non-characteristic for $\mathcal{N}:=\D_X/\D_X\cdot P$ in a neighbourhood of $(x_0,0)$ in $Y$ using the Weierstrass preparation theorem, we may write it locally as
$P\big(x,t;\partial_x,\partial_t\big):=\sum_{0\leq j\leq k}P_j(x,t,\partial_x)\cdot \partial_t^j,$
where $P_j(x,t,\partial_x)$ are differential operators not depending on $\partial_t$ of orders $\leq k-j$ and $P_k(x,t)$ is a holomorphic function on $X$ satisfying $P_k(x_0,0)\neq 0.$
 Such operators are said to be of \emph{Cauchy-Kovalevskaya type} (CK-type).

For a given closed sub-manifold $i:Y\hookrightarrow X$, we have a sub-category
\begin{equation}
\label{eqn: NC-Cat}
\mathrm{Coh}_{f-\mathsf{NC}}(\mathcal{D}_Y):=\{\mathcal{N}\in \mathrm{Coh}(\mathcal{D}_Y)| f\text{ is }\mathsf{NC} \text{ for }\mathcal{N}\}\subset \mathrm{Coh}(\mathcal{D}_Y).
\end{equation}

\begin{proposition}
\label{CK-type and NC}
If $\mathcal{N}=\mathcal{D}_Y/\mathcal{D}_Y\cdot P$ is a coherent $\mathcal{D}_Y$-module with $P$ a Cauchy-Kovalevskaya type operator, then $\mathcal{N}\in \mathrm{Coh}_{f-\mathsf{NC}}(\mathcal{D}_Y).$ Conversely, every $\mathcal{N}\in \mathrm{Coh}_{f-\mathsf{NC}}(\mathcal{D}_Y)$ is locally of the form $\bigoplus_i \mathcal{D}/\mathcal{D}\cdot P_i$ of a quotient of a finite direct sum of modules $\mathcal{D}_Y/\mathcal{D}_Y\cdot P_i$ for $P_i$ of CK-type.
\end{proposition}
\begin{proof}
It is enough to work locally. Assume $x_1,\ldots,x_n$ are chosen local coordinates in $X$ with $Y=\{x_n=0\}.$ Then let $\D_X'$ be the sub-ring of sections not containing $\partial_{x_n}.$ Since $T_Y^*X=\{x_n=\xi'=0\},$ any cyclic $\D
$-module with $P$ of CK-type belongs to (\ref{eqn: NC-Cat}). Conversely, if $y=(x',0)\in Y,$ the non-characteristic hypothesis on $\M$ imlies each element of the stalk $\M_y$ is annihilated by a germ of a CK-operator. This implies the result since $\M$ is coherent,.
\end{proof}

We may reformulate this notion for $\D$-algebraic nonlinear PDEs. To this end, consider $\IAB,$ over a complex analytic manifold $X$ and let $p:Spec_{\D_X}(\mathcal{A})\rightarrow X$ be the corresponding $\D$-space projection (see Proposition \ref{Jet Representability} below). Suppose there is a coordinate system $x=(x_1,\ldots,x_n)$ on $X$ among which we have a distinguished coordinate $\tau$ defining a sub-manifold $Y=\{\tau=0\}$ of $X$ for which we can find coordinates $u^{\alpha}=(u^1,\ldots,u^m)$ in $E$ (rather the fibers over $x$) such that our defining $\D_X$-ideal is given by
\begin{equation}
    \label{eqn: D-CK Ideals}
\mathcal{I}=span\big\{\psi^1,\cdots\psi^m\big| \psi^{\alpha}=\mathcal{C}(\partial_{\tau})^{\sigma_{\alpha}}u^{\alpha}+\mathsf{Q}_{\alpha},\hspace{2mm}\alpha=1,\ldots,m \big\}.
\end{equation}
where the terms $\mathsf{Q}$ are independent of jets of $u^{\alpha}$ of order $\geq \sigma_{\alpha}$ in the direction $\partial_{\tau}.$
\begin{definition}
A $\D_X$-ideal $\mathcal{I}$ which is generated around its pre-image $p^{-1}(x)$ of the form (\ref{eqn: D-CK Ideals}) is said to be $\D_X$-\emph{Cauchy-Kovaleskaya}. A $\D$-algebraic non-linear \textsc{pde}, $\IAB$, over $X$ is $\D_X$-Cauchy-Kovalevskaya if $\mathcal{I}$ is so.
\end{definition}
In Proposition \ref{D-CK-NC} of §\ref{sssec: D-Geometric Microcharacteristics} we will provide an analog of Proposition \ref{CK-type and NC} for $\D$-CK operators.

\begin{proposition}
If $\eta:X\rightarrow C$ is fibered over a one-dimensional manifold $C,$ and $\mathcal{I}$ is a $\D_{X}$-ideal which is $\D$-Cauchy-Kovalevskaya, then $\mathcal{O}_{Z^{\infty}}$ is a commutative $\D_{X/C}$-algebra, where
$\D_{X/C}:=\{P\in\D_X | \big[P,\eta^{-1}\mathcal{O}_c\big]=0\big\}.$
\end{proposition}

\subsection{Non-linear Microlocal Analysis}
\label{ssec: Non-linear Microlocal Analysis}
In this subsection we develop the necessary ingredients to discuss microlocalization as it pertains to $\D$-algebraic ideal sheaves. These tools are necessary to formulate and prove the affine non-linear propagation theorem (with polynomial non-linearities).

\subsubsection{Module categories and tangent spaces}
\label{ModulCatsTangents} We study sheaves (more generally, complexes) on $\D$-spaces with natural differential structure, encoded by the action of a sheaf of relative differential operators.

\begin{definition}
\label{definition: Total DOS}
Consider an $X$-flat $\D$-space $\EQ$, with structure map $\pi_{\infty}:Z\rightarrow X.$ The sheaf of \emph{total differential operators} on $Z$ is
$\D_{\EQ}:=\mathcal{O}_{\EQ}\otimes_{\pi_{\infty}^{-1}\mathcal{O}_X}\pi_{\infty}^{-1}\D_X.$
\end{definition}
When $Z=\mathrm{Spec}_{\D_X}(\A)$ is an affine $\D$-space e.g. $\A=\mathcal{O}_{J_X^{\infty}E}$, we may identify this sheaf with $\A[\D_X]:=\A \otimes_{\mathcal{O}_X}\D_X.$
There are natural symmetric monoidal closed categories of sheaves of left/right $\mathcal{A}[\D_X]$-modules,
$$\big(\mathrm{Mod}(\mathcal{A}[\D_X]),\otimes_{\mathcal{A}},\mathcal{H}\mathrm{om}_{\mathcal{A}}\big),\hspace{1mm} \big(\mathrm{Mod}(\mathcal{A}[\mathcal{D}]^{op}),\otimes_{\mathcal{A}^r},\mathcal{H}om_{\mathcal{A}^r}\big),$$
with monoidal units $\mathcal{A}$ and $\mathcal{A}^r,$ respectively. 
The appropriate duality operation coming from $\mathcal{A}[\D_X]$-linearity\footnote{Note this is \emph{not} $\mathcal{H}\mathrm{om}_{\mathcal{A}}(\mathcal{M},\mathcal{A}).$}; \emph{local/inner duality} is given by the functor
\begin{equation}
    \label{eqn: Local duality}
    (-)^{\circ}:\mathrm{Mod}(\mathcal{A}[\D_X])\rightarrow \mathrm{Mod}(\mathcal{A}[\D_X])^{\mathrm{op}},\hspace{3mm}
    \mathcal{M}^{\circ}:=\mathcal{H}\mathrm{om}_{\mathcal{A}[\D_X]}\big(\mathcal{M},\mathcal{A}[\D_X]\big),
    \end{equation}
but notice this produces a right $\mathcal{A}[\D_X]$-module.
\begin{definition}
\label{Defin: Inner Verdier Duality}
The functor of \emph{inner Verdier duality} is $D_{\mathcal{A}}^{ver}:=(-)^{\ell}\circ (-)^{\circ},$ i.e.
\begin{equation}
    \label{eqn: Inner duality}
    D_{\mathcal{A}}^{ver}:\mathrm{Mod}\big(\mathcal{A}[\D_X]\big)^{op}\rightarrow \mathrm{Mod}\big(\mathcal{A}[\D_X]\big),\hspace{2mm} (\mathcal{M}^{\circ})^{\ell}:=\mathcal{H}\mathrm{om}_{\mathcal{A}[\D_X]}(\mathcal{M},\mathcal{A}[\D_X])^{\ell}.
\end{equation}
\end{definition}

The canonical projection $p:\mathrm{Spec}_{\D_X}(\mathcal{A})\rightarrow X,$
gives a natural sequence
\begin{equation}
    \label{eqn: Kahler Sequence}
    0\rightarrow p^*\Omega_{E/X}^1\rightarrow \Omega_{\EQ/X}^1\rightarrow \Omega_{\EQ/E}^1\rightarrow 0,
\end{equation}
where $\Omega_{\mathcal{A}}^1$ is obtained by the usual K\"ahler construction, adapted to the $\D$-geometric setting \cite{BD2004}.

It represents the functor of $\D$-linear derivations whose corresponding tangent sheaf is denoted $\Theta_{\mathcal{A}}.$ 
From universal constructions, there is a universal de Rham differential,
$d:\mathcal{A}\rightarrow \Omega_{\mathcal{A}}^1,$ often called the `vertical derivative.' The corresponding de Rham algebra $\Omega_{A}^*:=\mathrm{Sym}\big(\Omega_{\mathcal{A}}^1[1]\big),$ and its $\D$-module de Rham complex give a sheafified version of the variational bi-complex of PDEs \cite{KSY2}.

When $\mathcal{A}=\mathcal{O}(\mathrm{Jets}^{\infty}E),$ then there we have an isomorphism
$\Omega_{\mathcal{A}/X}^1\simeq \mathcal{A}[\D_X]\otimes p_{\infty}^*\Omega_{E/X}^1,$ and an isomorphism of 
$\mathcal{A}^r[\D_X^{\mathrm{op}}]$-modules,
$\Theta_{\mathcal{A}}\rightarrow \big(\pi_{\infty,0}^*\Theta_{E/X}\big)\otimes\mathcal{A}^r[\D_X^{\mathrm{op}}].$
In particular, when $E=X\times M$ then $\mathcal{D}er_{\D_X}(\mathcal{A})\simeq \Theta_M\otimes_{\mathcal{O}_M}\mathcal{A}.$

\begin{definition}
\label{D smooth definition}
A $\D$-algebra is said to be \emph{$\D$-smooth} if $\Omega_{\mathcal{A}}^1$ is a projective $\mathcal{A}$-module of finite $\mathcal{A}[\mathcal{D}]$-presentation and $\mathcal{A}$ is such that there exits a $\mathcal{O}_X$-module of finite type $\mathcal{M}$ and an ideal $\mathcal{K}\subset \mathrm{Sym}_{\mathcal{O}_X}(\mathcal{M})$ such that there exists a surjection $\mathrm{Jet}^{\infty}\big(\mathrm{Sym}_{\mathcal{O}_X}(\mathcal{M})/\mathcal{K}\big)\rightarrow \mathcal{A}$ whose kernel is $\D_X$-finitely generated. A $\D$-scheme is of \emph{finite-type} if of the form $Spec_{\D}\big(\mathrm{Sym}(\mathcal{D}\otimes \mathcal{E})/\mathcal{I}\big),$ for a coherent sheaf $\mathcal{F}$ and $\D$-ideal $\mathcal{I}.$
\end{definition}

\begin{proposition}
\label{CModules in D geometry lemma}
Let $\mathcal{A}=\mathcal{O}(\mathrm{Jet}^{\infty}(E)\big)$ with $\pi_{\infty}:Spec_{\D}(\mathcal{A})\rightarrow X.$ Then
$\mathcal{A}\subseteq \mathcal{A}[\D_X]$ is a sub-algebra and $\Theta_X$ is a Lie sub-algebra.
Moreover, we have an associative unital algebra morphism $\D_X\rightarrow \mathcal{A}[\D_X] P \mapsto 1_{\mathcal{A}}\otimes P.$
If $\mathcal{L}_1,\mathcal{L}_2$ are vector bundles we have an $\mathcal{A}$-module isomorphism identifying $\mathcal{A}\otimes_{\mathcal{O}_X}\mathrm{Diff}(\mathcal{L}_1,\mathcal{L}_2)$ with total differential operators acting from $\pi_{\infty}^*\mathcal{L}_1\rightarrow \pi_{\infty}^*\mathcal{L}_2.$ 
\end{proposition}
There are obvious embeddings $\mathcal{A}\hookrightarrow \mathcal{A}[\D_X], a\mapsto a\otimes 1_{\mathcal{O}_X}$ and $\Theta_X\hookrightarrow \mathcal{A}[\D_X],\theta \mapsto 1_{\mathcal{A}}\otimes \theta$ which are morphisms of associative unital algebras and Lie algebras, respectively.

\subsubsection{Filtrations.}
Consider the sheaf of rings given by Definition \ref{definition: Total DOS}, in the case of the jet-scheme i.e. $\A[\D].$ There is a canonical double filtration combining the standard order filtrations on $\mathcal{A}$ and on $\D_X.$ The total filtration, defined for every open subset $V$ of $X$ is denoted by 
$$\mathrm{Fl}_{\mathrm{Tot}}^p\big(\mathcal{A}[\D_X](V)\big):=\bigcup_{p=k+\ell}\mathrm{Fl}^k\mathcal{A}(V)\otimes_{\mathcal{O}_X(V)}\mathrm{Fl}^{\ell}\D_{X}(V).$$
Put $\mathrm{Fl}_{\mathrm{Tot}}\mathcal{A}[\D_X]:=\lbrace{\mathrm{Fl}_{k,\ell}\mathcal{A}[\D_X]\rbrace}_{k,\ell\in\mathbb{Z}\times\mathbb{Z}},$ and interpret 
sections of $\mathrm{Fl}^{k,\ell}\mathcal{A}[\D_X]$ as differential operators of order $\leq \ell$ with coefficients in $k$-jets.

For a section $P\in\Gamma(U,\D_X),$ we may consider $\widehat{P}\in \Gamma(U,\mathcal{A}[\D_X]),$ and call it the \emph{lift}. It is obtained in coordinates, roughly speaking, by replacing partial derivatives by total derivatives. Operators between bundles or coherent sheaves also can be lifted in the following sense (multi-differential operators are defined similarly). That is, for bundles $\mathcal{E},\mathcal{F}$ there is an isomorphism
$$\gamma:\mathcal{A}_X^{\ell}\otimes_{\mathcal{O}_X}\D_X(\mathcal{E},\mathcal{F})\simeq \mathcal{A}^{\ell}[\mathcal{D}](p_{\infty}^*\mathcal{E},p_{\infty}^*\mathcal{F}),$$
defined by 
$(f\otimes\mathsf{Q})(e):=f\cdot \big(\mathsf{Q}(e)\circ p_{\infty}\big),e\in \mathcal{E},\text{ with }\gamma:\big(f(x,u_{\sigma}^{\alpha})\otimes\mathsf{Q}\big)\mapsto (f(x,u_{\sigma}^{\alpha})\cdot_{\mathcal{A}} \hat{\mathcal{Q}}),$ where the lift $\mathsf{Q}\in \D_X(\mathcal{E},\mathcal{F})$ is defined by
$$\hat{\mathsf{Q}}:p_{\infty}^*(\mathcal{E})\rightarrow p_{\infty}^*(\mathcal{F}),\hspace{1mm} (j_{\infty}\varphi)^*\big(\hat{\mathsf{Q}}s_{\infty}\big)=\mathsf{Q}\big((j_{\infty}\varphi)^*s_{\infty}\big),$$
where $s_{\infty}\in \mathrm{Sect}(Jet_X^{\infty}E,p_{\infty}^*\mathcal{E})$ and $\varphi\in \mathrm{Sect}(X,E).$ See Proposition \ref{Lifting-type lemma} below.  

We use the lift of the Hamiltonian isomorphism induced by the symplectic structure on $T^*X.$

\begin{construction}[Lifted Hamiltonian Isomorphism]
\label{cons: Hamiltonian Lift}
    \normalfont For $j:U\hookrightarrow X$ an open subset recall that
$\xi_i=\sigma_1(\partial_i)=\partial_i\text{ mod }\mathcal{O}_U$
so $\mathrm{Gr}\D_U\simeq \mathcal{O}_U[\xi_1,\ldots,\xi_n].$ By Proposition \ref{Embedding of T^*X} below, the associated graded $gr^{k}\mathcal{A}^{\ell}[\D_X]|_U\simeq \bigoplus_{|\alpha|=k}(j^{-1}\mathcal{A})^{\ell}\Pi^{\alpha},$ (i.e. total symbols) is a polynomial ring $\mathcal{A}^{\ell}\big[\Pi_1,\ldots,\Pi_n],\hspace{1mm}\Pi_i=\sigma_i^{\mathcal{C}}(\mathcal{C}(\partial_i)),$ with $\sigma^{\mathcal{C}}$ the symbol map. 
The induced Poisson bracket is
$\{F,G\}_{\mathcal{C}}:=\partial_{\Pi_i}FD_i(G)-\partial_{\Pi_i}GD_i(F).$

If $P$ is a microdifferential operator on $U\subset T^*X,$ with characteristic variety $\mathrm{Char}(P)$, and $V$ is a closed conic of $T^*X,$ the bicharacteristic curves of $P$ are the integral curves of the vector field 
$H_{\sigma(P)}=\partial_{\xi_j}(\sigma(P))\partial_{x_j}-\partial_{x_j}\sigma(P)\partial_{\xi_j}.$

Denote the horizontal lift of this vector field as
$\widehat{H}_{\sigma(P)}.$ For a function $\varphi$ on an open subset $U$ of $T^*X$, we denote the image of the vector field $H(\varphi)$ on $T^*X$ into the bundle $T_VT^*X$ by $\overline{H}(\varphi),$ and the
horizontal lift as $\widehat{\overline{H}(\varphi)}.$ 
\end{construction}

\subsubsection{$\D$-Algebraic Microlocalization.}
\label{ssec: D-Algebraic Micro}
Consider a $\D$-geometric PDE, $\EQ=\mathsf{Spec}_{\D_X}(\mathcal{B})$, induced by $\mathcal{I}\rightarrow \mathcal{A}\rightarrow\mathcal{B}.$ There is an induced filtration on $\mathcal{B}$, determined by $\mathcal{A}$. Consider the projection $p_{\infty}:\EQ\rightarrow X$ and the space $T^*(X,\EQ):=\pi_{\infty}^*(T^*X)$,  defined via the pull-back diagram
\[
\begin{tikzcd}
T^*(X,\EQ)\arrow[d] \arrow[r] & T^*X\arrow[d,"\pi^X"]
\\
\EQ\arrow[r,"p_{\infty}"] & X
\end{tikzcd}
\]

The induced projection $\pi_{\infty}^{\EQ}:T^*(X,\EQ)\rightarrow \EQ,$ will play the role of $\pi^{X}:T^*X\rightarrow X$ and we set 
$\pi_{\infty}^X:=p_{\infty}\circ \pi_{\infty}^{\EQ}:T^*(X,\EQ)\rightarrow X.$ Observe that 
$T^*(X,\EQ)\simeq p_{\infty}^*T^*X\simeq \EQ\times_X T^*X,$
and
\begin{equation}
\label{eqn: OT*(X,Z)}
\mathcal{O}_{T^*(X,\EQ)}:=\mathcal{O}_{p_{\infty}^*T^*X}\simeq \mathcal{O}_{\EQ}\otimes_{\mathcal{O}_X}\mathcal{O}_{T^*X}.
\end{equation}
Ignoring the filtration on $\EQ$, the sheaf of commutative algebras (\ref{eqn: OT*(X,Z)}) has the induced filtration 
$\mathrm{Fl}^k\mathcal{O}_{T^*(X,\EQ)}:=\mathcal{O}_{\EQ}\otimes_{\mathcal{O}_X}\mathrm{Fl}^k\mathcal{O}_{T^*X}.$ 
\begin{remark}
Note the filtration on $Z$ is the one coming from the $\D$-space inclusion $i:\EQ\hookrightarrow \mathrm{Jet}^{\infty}(\mathcal{O}_E^{\mathrm{alg}}),$ by putting $\mathrm{Fl}^k\mathcal{O}_{\EQ}:=\mathrm{Fl}^k\big(\mathcal{O}_{\mathrm{Jet}^{\infty}(\mathcal{O}_E^{\mathrm{alg}})}|_{\EQ}\big).$
\end{remark}
The following is clear.
\begin{proposition}
\label{Embedding of T^*X}
There is an identification
$\mathrm{Spec}_X\big(gr\mathcal{A}[\D_X]\big)\simeq T^*Spec_{\D_X}(\mathcal{A}),$ and there exists an embedding 
$T^*X\hookrightarrow T^*Spec_{\D_X}(\mathcal{A}),$ which at the level of algebras is a morphism of Poisson algebras.
\end{proposition}
The following characterizes and serves as the definition of the characteristic variety in the non-linear setting. The proof of the following and the integrability of the non-linear characeristic variety can be found in \cite[\S.~3]{KS1}.
\begin{proposition}
\label{prop: Char inclusions}
Let $\EQ$ be an $X$-flat algebraic $\D_X$-space. Then:
\begin{itemize}
 \item 
There is an induced good filtration $F_{ord}^k(\D_{\EQ}):=\{\D_{\EQ}^{\leq k}|k\geq 0\},$ and an isomorphism of graded algebras 
$Gr^{F_{ord}}(\D_{\EQ})\simeq\mathcal{O}_{T^*(X,\EQ)};$
\item The module $\Omega_{\EQ/X}^1$ is canonically filtered via 
$\mathrm{Fl}^k\Omega_{\EQ/X}^1=\big\{\mathcal{O}_{\EQ}\otimes_{\mathcal{O}(F^k\EQ)}\Omega_{F^k\EQ/X}^1\big\}_{k\geq 0},$
and generated by $\D_{\EQ}$ and $\mathcal{O}_{\EQ}\otimes_{\mathcal{O}_{E}^{alg}}\Omega_{E/X}^1.$
\end{itemize}
There is a well-defined characteristic variety 
$$\mathcal{C}har_{\D}^1(\EQ):=supp\big(Gr(\Omega_{\EQ/X}^1)\big),$$ given by the support of a $Gr(\D_{\EQ})$-module.
Moreover, if $f:\EQ\rightarrow \mathcal{Z}$ is a $\D_X$-isomorphism between $\D_X$-smooth schemes there is an isomorphism in $f^*\mathcal{O}_{\mathcal{Z}}[\D_X]$-modules, 
    $f^*\Omega_{\mathcal{Z}/X}^1\simeq\Omega_{\EQ[k]/X}^1.$ In particular, the characteristic varieties are invariant under $\D$-isomorphism
    $\mathcal{C}har_{\D}(\mathcal{Z})\simeq \mathcal{C}har_{\D}(\EQ).$
\end{proposition}

For a vector bundle 
$E\rightarrow X$ and $\mathcal{F}$ a coherent sheaf on $X$, denote by $p_{\infty}:\mathrm{Jet}^{\infty}(\mathcal{O}_E^{\mathrm{alg}})\rightarrow X,$ the canonical projection from the algebraic infinite-jet $\D_X$-scheme. 
For any $U\subset X$, with canonical open chart $p_{\infty}^{-1}(U)$ of $\mathrm{Jet}^{\infty}\big(\mathcal{O}_E^{\mathrm{alg}}\big),$ then 
$$\Gamma\big(p_{\infty}^{-1}(U),p_{\infty}^*\mathcal{F}\big)\simeq \Gamma(U,\mathcal{F})\otimes_{\mathcal{O}_X|_U}\mathcal{O}\big(\mathrm{Jet}^{\infty}(\mathcal{O}_E^{\mathrm{alg}})|_U.$$

Denote the pull-back by $p_{\infty}^*\mathcal{F}\cong \mathcal{F}\times_X\mathrm{Jet}^{\infty}(\mathcal{O}_E^{\mathrm{alg}})\rightarrow \mathrm{Jet}^{\infty}(\mathcal{O}_E^{\mathrm{alg}}).$ 
\begin{proposition}
Let $\eta$ be the bundle map for $\mathcal{F}$ such that $\eta_*\mathcal{F}$ is a coherent $\mathcal{O}_X$-module. Then the interpretation of $\mathrm{Jet}^{\infty}(\eta_*\mathcal{F})\times_X\mathrm{Jet}^{\infty}(\mathcal{O}_E^{\mathrm{alg}})$ as the bundle of algebraic infinite horizontal jets of $p_{\infty}^*\mathcal{F}$ is valid, denoted by $\overline{\mathrm{Jet}}^{\infty}\big(p_{\infty}^*\mathcal{F}\big).$
\end{proposition}

For $P\in \Gamma(U,\D_X),$ take its lift $\widehat{P}\in \Gamma\big(p_{\infty}^{-1}(U),\mathcal{A}[\D_X]\big)$ denote by
$$\widetilde{\mathfrak{s}}_k(\widehat{P}):=(\pi_{\infty}^X)^{-1}\circ \mathcal{C}\mathfrak{s}_k(\widehat{P})\in\Gamma\big((\pi_{\infty}^X)^{-1}(U),\mathcal{O}_{T^*(X,\EQ)}(k)\big).$$
Coordinates on $T^*(X,\EQ)$, say over $\pi_{\infty}^{X,^{-1}}(U),$ are of the form $(u_{\sigma}^{\alpha},x,\xi).$  We can consider $\Omega\subset \pi_X^{-1}(U)$ in $T^*X$ and pullback
\[
\begin{tikzcd}
T_{\Omega}^*(X,\EQ)\arrow[d]\arrow[r]& \Omega\subset T^*X
\arrow[d]
\\
\EQ\arrow[r,"p_{\infty}"] & U\subset X
\end{tikzcd}
\]
Set
$\Omega_{\infty}:=\EQ\times_X \Omega\subset 
\big(\pi_{\infty}^X\big)^{-1}(U)\subset T^*(X,\EQ),$
and
$$\Gamma\big(\Omega_{\infty},\mathcal{O}_{T^*(X,\EQ)}(\star)\big):=\prod_{p\in\mathbb{Z}}\Gamma\big(\Omega_{\infty},\mathcal{O}_{T^*(X,\EQ)}(p)\big).$$
Functions on $\Omega_{\infty}$ are of the form $g=g(x_0,u_{\sigma}^{\alpha}(x_0),\xi_0)$ for a canonical coordinate chart $(x_0,\xi_0)$ of $\Omega\subset \pi_X^{-1}(U)$ in $T^*X.$

Consider coherent sheaves $\mathcal{F},\mathcal{G}$ of $\mathcal{O}_X$-modules. Let $\pi_{\infty}:\EQ\hookrightarrow \mathrm{Spec}_{\D_X}(\mathcal{A})\rightarrow X$ be the projection, with $\mathcal{A}=\mathcal{O}\big(\mathrm{Jet}^{\infty}(\mathcal{O}_E)\big),$ and consider the pullback diagrams
\[
\begin{tikzcd}
\pi_{\infty}^{\EQ,*}\mathcal{F}\arrow[d]\arrow[r]& \mathcal{F}\arrow[d]
\\
T^*(X,\EQ)\arrow[r,"\pi_{\infty}^{\EQ}"] & X
\end{tikzcd}\hspace{3mm}\text{ and }\hspace{3mm}
\begin{tikzcd}
\pi_{\infty}^{\EQ,*}\mathcal{G}\arrow[d]\arrow[r]& \mathcal{G}\arrow[d]
\\
T^*(X,\EQ)\arrow[r,"\pi_{\infty}^{\EQ}"] & X
\end{tikzcd}
\]

\begin{proposition}
\label{Lifting-type lemma}
If $\mathcal{A}$ is $\D$-smooth, there is a canonical isomorphism
$$
\mathcal{H}\mathrm{om}_{\mathcal{A}[\D_X]}\big(\mathrm{ind}_{\mathcal{A}[\D_X]}p_{\infty}^*(\mathcal{F}),\mathrm{ind}_{\mathcal{A}[\D_X]}p_{\infty}^*(\mathcal{G})\big)\rightarrow \mathcal{H}\mathrm{om}_{\mathcal{A}[\D_X]}\big(\mathrm{ind}_{\mathcal{A}[\D_X]}p_{\infty}^*(\mathcal{F}),\mathcal{A}[\D_X]\big)\otimes_{\mathcal{A}[\D_X]}\mathrm{ind}_{\mathcal{A}[\D_X]}^rp_{\infty}^*(\mathcal{G}),$$
of $\mathcal{A}[\D_X]$-modules compatible with filtrations.
\end{proposition}
Objects $F$ in Proposition \ref{Lifting-type lemma} are total differential operators from $p_{\infty}^*(\mathcal{F})$ to $p_{\infty}^*(\mathcal{G})$ and if $S:=\sigma(F)$ is its corresponding symbol, it is understood as a homomorphism
\begin{equation}
    \label{eqn: Microlocal pullback operator}   
F_S:\big(\pi_{\infty}^{\EQ}\big)^*\mathcal{F}\rightarrow \big(\pi_{\infty}^{\EQ}\big)^*\mathcal{G},
\end{equation}
of $\mathcal{O}_{T^*(X,\EQ)}$-modules. Then one might call an operator, understood as a morphism (\ref{eqn: Microlocal pullback operator}) \emph{$\mathcal{C}$-elliptic} if this map is an isomorphism. The operator of universal linearization of a non-linear PDE is a $\mathcal{C}$-differential operator, thus $\mathcal{C}$-ellipticity characterizes ellipticity of infinite-prolongations of non-linear PDEs.

\begin{remark}
We will provide a geometric characterization of ellipticity for nonlinear PDEs in §§ \ref{sssec: A Remark on Ellipticity}, below which is phrased in terms of the microlocal geometry of the derived characteristic varieties.
\end{remark}

\section{Derived non-linear PDEs via $\D$-geometry}
\label{sec: Derived D-NLPDES}
Let $X$ be a smooth projective $\mathbb{C}$-scheme. In this section we define derived (pre)stacks of solutions $\mathbb{R}\mathrm{Sol}_X(-)$ for generalized systems of PDEs $Z\subset q_{\DR,*}E$ via their homotopical functor of points, directly generalizing the functors (\ref{eqn: Classical D-Solutions}). 

Roughly speaking, they will be given by accessible $\infty$-functors (Definition \ref{Definition: D-PreStack})
$$\underline{\mathbb{R}\mathrm{Sol}}_{\D}:\mathsf{CAlg}_X(\D_X)\rightarrow \mathsf{Spc},$$
from an $(\infty,1)$-category of $X$-local sheaves of derived $\D$-algebras to that of spaces, defined as the homotopy localization of simplicial sets $L(\mathbf{sSet}).$ One can define $\mathsf{CAlg}_X(\D_X)$ as a homotopy localization of an associated model structure along its weak equivalences. This model structure was proven in 
\cite{dBPP18}, as well as \cite[Cor.~8.5]{PavSch}, which establishes the global model categorical approach in terms of algebras over an operad in $\D$-modules (left $\D$-modules in presheaves valued in complexes). In particular, the model structure is inherited from the local projective model structure of dg-$\D$-modules. A related stable model category of filtered $\D$-modules was proven in \cite{GwPa18}.
\begin{remark}
Alternatively, 
apply \cite[Prop. 7.1.4.11]{Lur17} to establish the cofibrantly generated combinatorial model structure on the category of commutative connective dg-algebras over $\mathscr{D}_X$, denoted $\cdga_{\D_X}$, for which we have an equivalence
$N_{dg}\big((\cdga_{\D_X})^{ cof}\big)[W^{-1}]\simeq \mathsf{CAlg}_X(\D_X),$
between the dg nerve of the sub-category of their cofibrant objects localized with respect to weak equivalences.
\end{remark}

We can model pre-stacks structured over these sites of homotopical algebras with differential structure via the local projective model structure on the category of simplicial presheaves on a Grothendieck site  \cite{Bla01}, further generalized in \cite[Thm. 3.4.1]{TV02}.

In the current paper and its sequel (see  \cite[Sect. 5-7]{KSY2}),  we prefer to work $\infty$-categorically to avoid model dependent arguments. Since $X$ is smooth and projective, we make use of \cite[Prop. 2.4.4, 4.4.2 and 4.7.3]{GR14} identifying $\mathsf{QCoh}(X_{\DR})$ with the derived category of the heart of its $t$-structure, as well as \cite[Lem. 1.1.6]{Gai11} which provides an equivalence of ind-coherent and quasi-coherent sheaves.
 
\subsubsection{}
We establish global geometric objects in derived geometry by gluing. We isolate a class with good deformation theory properties. This becomes essential in the study of graded mixed de Rham algebras, in the sense of \cite{CPTVV17}, which is the content of the sequel \cite{KSY2}.

Following \cite[Def.~2.1.1]{CPTVV17}, let $F\in \mathrm{FdStk}_k$ be any \emph{formal derived stack}. Recall that an object $F\in \mathrm{dSt}_k$ is \emph{nilcomplete} if for any $B\in\mathsf{CAlg}_k^{\leq 0},$ the natural map $F(B)\to \underset{k}{lim} F(B_{\leq k})$ is an equivalence, where $B_{\leq k}$ is the $k$-th Postnikov truncation of $B$. It is \emph{infinitesimally cohesive} if it sends pushouts of affine derived schemes to pull-backs, along square-zero extensions.
Moreover, $F\in \mathrm{FdSt}_k$ is \emph{almost affine} (resp. \emph{affine}) if $F_{red}:=i_!i^*(F)$ is an affine derived $k$-scheme (resp. almost affine) and there exists a cotangent complex $\mathbb{L}_F$ such that for each $\mathsf{Spec}_k(R)$-point, the $R$-dgmodule $\mathbb{L}_{F,u}$ is coherent and cohomologcally bounded.

In particular, if $F\in \mathrm{FdSt}_k^{aff}$, is an affine formal derived stack, there exists a global quasi-coherent cotangent complex \cite[Def.~1.4.1.5]{TV2}, such that for each $u:\mathsf{Spec}_k(R)\to F$, there is an equivalence of $R$-dgmodules,
\begin{equation}
    \label{eqn: Almostaffinecotangent}
u_{\mathsf{QCoh}}^*(\mathbb{L}_F)\simeq\mathbb{L}_{F,u}.
\end{equation}

\subsection{Functor of points derived $\D$-spaces}
In this subsection, we suppose that $X$ is a smooth projective variety over $\mathbb{C}.$ More generally, it is assumed to be $\D$-affine \cite{HT}.
\begin{remark}
    A smooth affine algebraic variety is $\D$-affine. Some non-affine varieties are $\D$-affine for example, projective spaces (see e.g. \cite[Thm.~1.6.5]{HT}).
\end{remark}

Let $\cdga_{\D_X}:=Comm\big(\mathsf{QCoh}(X_{\DR})^{\leq 0}\big)\simeq\mathrm{Comm}\big(\mathsf{Mod}_{\D_X}^{\leq 0}),$ be the category of commutative monoids in connective $\D_X$-modules, via the well-known equivalence between quasi-coherent sheaves on $X_{\DR}$ and complexes of left $\D_X$-modules on $X$ (see e.g, \cite{GR14}).

We call them \emph{derived $\D_X$-algebras}, considered with the model structure established in \cite{dBPP18}. Namely, let $W_{\cdga_{\D}}$ denote the distinguished class of morphisms provided by the weak-equivalences consisting of quasi-isomorphisms. We assume moreover that such derived $\D$-algebras are quasi-coherent over $X$ as $\mathcal{O}_X$-complexes by forgetting the connection. 

Set
$$\mathbf{dAff}_{\D_X}:=\big(\cdga\big)^{op},$$
as the \emph{affine derived $\mathcal{D}$-schemes over $X.$}
Objects $\mathsf{X}\in \mathbf{dAff}_{\D_X}$ may be intuitively understood as triples $(X,\mathsf{X},\mathcal{A}^{\bullet})$ where $X$ is a smooth complex analytic manifold, $\mathcal{A}^{\bullet}:=(\mathcal{A}^{\bullet},d_{\mathcal{A}})$ is a sheaf of differential graded commutative $\D_X$-algebras on $X$ with cohomologies $H_{\D}^i(\mathcal{A})$ for $i\leq 0,$ and $\mathsf{X}=\mathsf{Spec}_{\D_X}(\mathcal{A}^{\bullet})$ is an affine derived $\D_X$-scheme over $X$, such that  $H_{\D}^0(\mathcal{A})\in \mathrm{CAlg}_X(\D_X)$ (it is classical) and $H_{\D}^i(\mathcal{A}^{\bullet})\in \DG\big(H_{\D}^0(\mathcal{A}^{\bullet})\big).$
In particular we have a $\D_X$-linear multiplication 
$$m_0:H_{\D}^0(\mathcal{A}^{\bullet})\otimes_{\mathcal{O}_X}H_{\D}^0(\mathcal{A}^{\bullet})\rightarrow H_{\D}^0(\mathcal{A}^{\bullet}),$$ where $H_{\D}^i(\mathcal{A}^{\bullet})\in \DG\big(H_{\D}^0(\mathcal{A})\big).$ More generally, for $\mathcal{M}^{\bullet}\in \DG\big(\mathcal{A}^{\bullet}\big)$ with module multiplication $\eta,$ one has induced $\D_X$-linear morphisms
$$\psi^k:H_{\D}^k(\mathcal{A}^{\bullet})\otimes H_{\D}^0(\mathcal{M}^{\bullet})\rightarrow H_{\D}^k(\mathcal{M}^{\bullet}), $$
which are $H_{\D}^0(\mathcal{A}^{\bullet})$-bilinear.
Indeed, the $\D_X$-action commutes with taking cohomology classes since the differentials defining $H_{\D}^{\bullet}$ are already $\D_X$-linear.
If $X$ is smooth of dimension $n,$ the multiplication $m_0$ defines a commutative $\mathbb{C}$-algebra structure on de Rham cohomology sheaves $h\big(H_{\D}^0(\mathcal{A}^{\bullet})\big):=\Omega_X^n\otimes_{\D_X}H_{\D}^0(\mathcal{A}^{\bullet}),$ with $\Omega_X^n$ the sheaf of top forms.

This is a consequence of a more general statement relating algebraic structures in the $\mathcal{D}$-module setting to those in the $\mathbb{C}_X$-module setting via de Rham functors. In particular, local or $*$-operations denoted $P^*(-,-)$ on solution spaces of \textsc{nlpdes} corresponding to $\D$-algebras induce ordinary operations on pre-stacks of solutions e.g. de Rham cohomologies.

\begin{proposition}[\cite{BD2004},\cite{Paugam2014}]
\label{Lemma for * operations giving those in de Rham cohomologies}
The $*$-operations $\mathcal{P}_I^*$ induce usual operations in de Rham cohomology. In particular, a
$\mathsf{cdga}_{\D_X}$-resolution  $\epsilon:\mathcal{B}^{\bullet}\rightarrow \mathcal{O}_X,$ homotopically $\D_X$-flat, induces and equivalence
$$P_I^*\big(\{\mathcal{L}_i\},\mathcal{M}\big)\rightarrow \mathbb{R}\mathcal{H}om\big(\bigotimes_{i\in I}^{L}DR^r(\mathcal{L}_i),DR^r(\mathcal{M})\big).$$
\end{proposition}
Proposition \ref{Lemma for * operations giving those in de Rham cohomologies} extends to configuration spaces of points $\mathrm{Ran}(X)$ and is useful in practice, for instance, in describing for some complex of sheaves $\EuScript{F}^{\bullet}\in \mathsf{Shv}^!(\mathrm{Ran}(X))$ a canonical resolution $\widetilde{\EuScript{F}}.$ 

This is relevant for us since for any $\EuScript{M}\in \D_{Ran_X}^{\mathrm{op}}\mathrm{-mod},$ $DR(\EuScript{M})$ constitutes such a sheaf by the Riemann-Hilbert correspondence \cite{Kashiwara1984} and therefore admits such a resolution\footnote{Used in practical computations, for instance for chiral homology Section 3.1 \emph{op.cit.}} $\widetilde{DR}(\EuScript{M}).$

In particular, consider $\epsilon_{\EuScript{P}}:\EuScript{P}^{\bullet}\rightarrow \mathcal{O}_X$ a chosen $\D_X$-resolution of $\mathcal{O}_X$ with $\EuScript{P}^{>0}=0$ and with each $\EuScript{P}^i$ a module which is $\D_X$-flat. Then, $\EuScript{P}^{\otimes |I|}$ is $\D_{X^I}$-flat.

\begin{example}
\label{example: Resolutions for Curve}
Let $X$ be a projective curve. Then $\EuScript{P}^{\bullet}$ can be taken as the $2$-term resolution of $\mathcal{O}_X$ with 
$\EuScript{P}^0:=\mathrm{Sym}_{\mathcal{O}_X}^{>0}(\D_X),\EuScript{P}^{-1}:=\mathrm{Ker}(\epsilon_{\EuScript{P}}),$
where $\epsilon_{\EuScript{P}}(1_{\D_X})\mapsto 1_{\mathcal{O}_X}$ on the generator $1_{\D_X}\in\D_X\subset \EuScript{P}^0.$ Generally,
for $k>0$ one has
$$\epsilon_{\EuScript{P}}\big(\prod_{i=1}^kP_i(\partial_{z,i})\big)=\prod_{i=1}^k\big(P_i(\partial_{z,i})\cdot 1\big)\in\mathcal{O}_X, \prod_{i=1}^kP_i(\partial_{z,i})\in \mathrm{Sym}_{\mathcal{O}_X}^k(\D_X).$$
\end{example}
Example \ref{example: Resolutions for Curve} is adapted to higher dimensions.

\begin{example}
    \normalfont
Let $X$ be of dimension $n>1$ and let $\EuScript{P}^{\bullet}$ be a $\D_X$-resolution of $\mathcal{O}_X$ such that 
$\EuScript{P}^{\bullet}\xrightarrow{\simeq}\mathrm{Sym}^{>0}(\EuScript{T}^{\bullet}),$
as graded objects of $\mathrm{CAlg}_{\D_X},$ where $\EuScript{T}^{\bullet}$ is an $X$-locally projective graded $\D_X$-module in degrees $\leq 0.$
If $X$ is moreover quasi-projective, $\EuScript{T}^{\bullet}$ can be chosen as a direct sum
$\EuScript{T}^{\bullet}\simeq\bigoplus_{\alpha}\EuScript{N}_{\alpha}^{\bullet},$ where each $\D_X$-module $\EuScript{N}_{\alpha}$ is isomorphic to one of the form $\D_X\otimes\EuScript{L}_{\alpha}$ for a line bundle $\EuScript{L}_{\alpha}$ e.g.  $\EuScript{L}_{\alpha}$ is a  negative power of some ample line bundle.
\end{example}

The tautological embedding $i:\mathrm{Aff}_X^{cl}(\D_X)\hookrightarrow \mathbf{dAff}_{\D_X},$ admits an adjoint $(-)^{cl}$, given by taking $H_{\D}^0$.
We say $\mathcal{A}^{\bullet}$ is eventually co-connective if $H_{\D}^i\equiv 0$ for all $i\ll n$ for \emph{some} $n<0$ and denote their totality by $^{\ll}\mathbf{dAff}_{\D_X}.$ 
\subsubsection{Solutions with coefficients}
\label{sssec: Solutions with Coefficients}
We now describe the \emph{local solution functor}, extending the $\D_X$-module one $R\mathcal{H}om_{\D_X}(-,\mathcal{O}_X)$ to PDEs in total derivatives \cite{Krasilshchik1986}. This sheafifies the important class of PDEs appearing in jet-space geometry called $\mathcal{C}$-differential equations.\footnote{As pointed out to us by J. Krasil'shchik, most of the $\mathcal{C}$-differential equations one encounters are linear, and it may be interesting to try and find non-linear examples. We thank him for this interesting remark.}

Given a $\D_X$-algebra $\mathcal{A}$, we want to define solution functors with coefficients in $\mathcal{A}$.
Namely, we may introduce the $\mathbb{C}_X$-vector space of $\mathcal{A}^{\ell}[\mathcal{D}_X]$-solutions with values in an $\mathcal{A}[\mathcal{D}_X]$-module function space $\mathcal{F}$, as
\begin{equation}
\label{eqn: Local Solutions}
Sol_{\mathcal{A}}(-,\mathcal{F}):=\mathcal{H}\mathrm{om}_{\mathcal{A}[\mathcal{D}_X]}(-,\mathcal{F}).
\end{equation}
Objects $Sol_{\mathcal{A}}(\mathcal{N},\mathcal{F})$ have the structure of $\mathcal{A}^{\ell}$-vector spaces. There
is a rather direct analogy with usual $\D$-module solution functors (c.f. \ref{notate: D-Module Conventions}) by putting 
$\mathsf{D}:=\mathcal{A}[\mathcal{D}_X].$ Consider $\mathsf{P}\in \mathsf{D}^{q\times p}.$ Then $\mathcal{M}_{\mathsf{P}}:=\mathsf{D}^p/\mathsf{D}^q\cdot \mathsf{P}$ is a finitely presented left $\mathsf{D}$-module. If $\mathcal{F}$ is another left $\mathcal{A}[\mathcal{D}_X]$-module, then there is an isomorphism 
$$ker(\mathsf{P}-)=\{g=(g_1,\ldots,g_p)\in \mathcal{F}^p|\mathsf{P}(g)=0\}\simeq \mathcal{H}om_{\mathsf{D}}(\mathcal{M}_{\mathsf{P}},\mathcal{F}).$$

Functors (\ref{eqn: Local Solutions}) also describe solution spaces to equations for unknown variables $\varphi^1,\ldots,\varphi^{N_1}$ determined by $N_0$ operators in total derivatives
\begin{equation}
    \label{eqn: Linear CDiff Eqn}
\sum_{\ell=1}^{N_1}\Box_{ij}\varphi^{j}=0,\hspace{2mm} i=1,\ldots,N_0,\hspace{2mm} \Box_{ij}=\sum_{|\sigma|\geq 0}A_{\sigma}^{ij}(x,u_{\sigma})D^{\sigma},
\end{equation}
where $D^{\sigma}=\big(D_{x_1}\big)^{\sigma_1}\circ\cdots\circ \big(D_{x_n}\big)^{\sigma_n}$ and $D_{x_i}$ as in equation (\ref{eqn: D-Action}).
\begin{remark}
More generally, consider a space of variables $(x_1,\ldots,x_n,u_{\sigma}^{\alpha},v_{J}^B)$ for $\alpha=1,\ldots,m$ and $B=1,\ldots,r$ and the geometric object cut-out by the following complex analytic function in these arguments
$F(x_1,\ldots,x_n,u_{\sigma}^{\alpha},\ldots,v_{J}^B,\ldots)=0.$
They are understood in a more geometric language as sub-manifolds 
$\mathcal{E}_h^{\infty}\subset \mathrm{Jets}_h^{\infty}(n,m;r),$
in spaces of \emph{horizontal jets} in $n$ independent, $m$ dependent and $r$ horizontal dependent variables \cite{Krasilshchik1986},\cite{Vinogradov2001}; equations \eqref{eqn: Linear CDiff Eqn} encompass many classical objects e.g. 
horizontal vector fields, the horizontal de Rham differential $\overline{d}$,
the Spencer differential $\overline{S}$ and the universal linearizations. 
\end{remark}

The notions in §§ \ref{sssec: Solutions with Coefficients} have the following analogs in derived $\D_X$-geometry. Denote by $\DG(\mathcal{A})$ the category of $\mathcal{A}-$modules in complexes of $\D_X$-modules and recall $\mathcal{A}[\mathcal{D}_X]$ denotes $\mathcal{A}\otimes\mathcal{D}_X$.
\begin{definition}
\label{D Finite Type}
An eventually co-connective derived affine $\D_X$-scheme $\mathsf{X}=\mathsf{Spec}_{\D}(\mathcal{A}^{\bullet})$ is said to be:
\begin{itemize}
    \item \emph{$\D$-finite type}: if $H_{\D}^0(\mathcal{A})$ is a commutative $\D_X$-algebra of finite-type as in Definition \ref{D smooth definition} and a $\D$-algebraic PDE, with $H_{\D}^{-i}(\mathcal{A})\in\DG\big(H_{\D}^0(\mathcal{A})\big)$  finitely-generated $H_{\D}^0(\mathcal{A})[\mathcal{D}_X]$-modules for all $i$;
    
    \item \emph{Finitely $\D$-presented}: if $H_{\D}^0(\mathcal{A})$ is a finitely generated $\D_X$-algebra that is a $\D$-algebraic PDE with $H_{\D}^i(\mathcal{A})$ are finitely presented $H_{\D}^0(\mathcal{A})[\mathcal{D}_X]$-modules, for each $i.$ 
\end{itemize}

\end{definition}
\begin{remark}
Being $\D$-finitely generated, we may consider a local resolution of $H_{\D}^{-i}(\mathcal{A})$ by
$0\leftarrow H_{\D}^{-i}(\mathcal{A})\leftarrow \big(H_{\D}^0(\mathcal{A})\otimes^!\mathcal{D}_X\big)^{\oplus r_i},r_i\in\mathbb{Z}.$
Similarly, for those of finite $H_{\D}^0(\mathcal{A})[\mathcal{D}_x]$-presentation,
$$0\leftarrow H_{\D}^{-i}(\mathcal{A})\leftarrow \big(H_{\D}^0(\mathcal{A})\otimes^!\mathcal{D}_X\big)^{\oplus r_i}\leftarrow \big(H_{\D}^0(\mathcal{A})\otimes^!\mathcal{D}_X\big)^{\oplus s_i}.$$
\end{remark}

The derived functors of \eqref{eqn: Local Solutions} give higher $\mathcal{A}^{\bullet}[\mathcal{D}_X]$-solution functors 
$$R\mathcal{S}ol_{\mathcal{A}[\D]}(-,-):Ho\big(\DG(\mathcal{A})\big)^{op}\times Ho\big(\DG(\mathcal{A})\big)\rightarrow Ho\big(\mathcal{A}^{\ell}-\mathbf{Vect}_{\mathbf{k}_X}\big),$$
such that $H^0\big(R\mathcal{S}ol_{\mathcal{A}[\D]}\big)=Sol_{\mathcal{A}}.$

\subsubsection{Free-forgetful adjunctions}
\label{FreeForgetfulAdjunctions}
Consider free-algebra functor from complexes of $\D$-modules to derived $\D$-algebras, possessing an adjoint $U$ forgetting the algebra structure,
\begin{equation}
\label{eqn: Free-forgetful algebra}
\mathsf{Free}_{\D}:\DG\rightleftarrows \cdga_{\D_X}:U.
\end{equation}
Let $\mathcal{A}\in \cdga_{\D_X}.$ There is an analogous adjunction
\begin{equation}
    \begin{tikzcd}
    \label{eqn: Free-forget AD-Mod/Alg Adjunction}
\mathsf{Free}_{\mathcal{A}}:\DG(\mathcal{A})\arrow[r,shift left=.5ex] & Comm\big(\DG(\mathcal{A})\big)\simeq \cdga_{\D_X,\mathcal{A}/}:\arrow[l,shift left=.5ex]\mathrm{For}.
\end{tikzcd}
\end{equation}

\begin{example}
\label{Free Example}
    \normalfont Let $\mathcal{M}\in \DG$ be a compact object and consider its image under (\ref{eqn: Free-forgetful algebra}). One has that $\mathsf{Free}_{\D}(\mathcal{M})^r=\mathsf{Free}_{\D}(\mathcal{M})\otimes \omega_X$ is a compact right $\D$-algebra. 
\end{example}
Adjunction (\ref{eqn: Free-forget AD-Mod/Alg Adjunction}) induces a similar adjunction of $\infty$-categories, denoted
$$\mathsf{Free}_{\mathcal{A}}:\mathsf{Mod}_{\D_X}(\mathcal{A})\rightleftarrows \mathsf{CAlg}_{\mathcal{A}/}(\D_X):\mathsf{For}.$$

\subsubsection{Local Verdier duality}
A compact $\D_X$-algebra $\mathcal{A}$ that is written as a filtered colimit of strict finite $I$-cell objects is equivalently viewed as a filtered colimit of compact $\D_X$-modules via the forgetful functors \eqref{eqn: Free-forgetful algebra}, \eqref{eqn: Free-forget AD-Mod/Alg Adjunction}. 
Compact objects in $\DG$ correspond with perfect $\D_X$-modules. In what follows, we fix $\mathcal{A}$ and understand it to be given by such retracts with underlying perfect $\D$-module. A more detailed discussion of compactness is given in Subsect.\ref{ssec: Compactness}.
 
\begin{proposition}
\label{Biduality}
Suppose that $X$ is a affine. 
For perfect $\mathcal{A}[\mathcal{D}_X]$-modules, the  bi-duality isomorphism hold $\mathcal{M}\xrightarrow{\sim}(\mathcal{M}^{\circ})^{\circ}$ in $\mathrm{Ho}\big(\DG(\mathcal{A})\big).$
\end{proposition}
\begin{proof}
It suffices to construct the map $\mathcal{M}\rightarrow \big(\mathcal{M}^{\circ}\big)^{\circ},$ where $(-)^{\circ}:=\mathbb{R}\mathcal{H}\mathrm{om}_{\mathcal{A}[\mathcal{D}_X]}\big(-,\mathcal{A}[\mathcal{D}_X]\big),$ and noting that
$$\big(\mathcal{M}^{\circ}\big)^{\circ}\cong \mathbb{R}\mathcal{H}\mathrm{om}_{\mathcal{A}[\mathcal{D}_X]^{\mathrm{op}}}\big(\mathbb{R}\mathcal{H}\mathrm{om}_{\mathcal{A}[\mathcal{D}_X]}(\mathcal{M},\mathcal{A}[\mathcal{D}_X]),\mathcal{A}[\mathcal{D}_X]\big).$$
This is clear by viewing $\mathbb{R}\mathcal{H}\mathrm{om}_{\mathcal{A}[\mathcal{D}_X]}(\mathcal{M},\mathcal{A}[\mathcal{D}_X])$ and $\mathcal{A}\otimes\mathcal{D}_X$ as complexes of $\mathcal{A}[\mathcal{D}_X^{\mathrm{op}}]$-modules via the right multiplication of $\mathcal{A}[\mathcal{D}_X]$ on $\mathcal{A}[\mathcal{D}_X]$ and by the left $\mathcal{A}[\mathcal{D}_X]$-action on the right hand side arising from the tautological left $\mathcal{A}^{\ell}[\mathcal{D}_X]$-action on itself. 
One uses the homotopically finitely $\D$-presentation assumption on $\mathcal{A}^{\bullet}$ to reduce the statement to that of coherent $\D$-module complexes $\mathcal{M}$.
\end{proof}

\begin{corollary}
If $\mathcal{A}$ is $\D_X$-smooth (ungraded) algebra, it is in particular a perfect $\mathcal{A}$-module. Then, natural bi-duality isomorphisms hold for $\Theta_{\mathcal{A}}$, therefore $\Omega_{\mathcal{A}}^1$.
\end{corollary}
\begin{proof}
By definition $\Theta\cong \mathcal{H}\mathrm{om}_{\mathcal{A}[\mathcal{D}_X]}\big(\Omega_{\mathcal{A}}^1,\mathcal{A}[\mathcal{D}_X]\big),$ thus there are isomorphisms,
$$\Omega_{\mathcal{A}}^1\simeq \mathcal{H}om_{\mathcal{A}^r\otimes_{\omega_X}\mathcal{D}^{op}}(\Theta_{\mathcal{A}},(\mathcal{A}\otimes_{\mathcal{O}_X}\omega_X)\otimes_{\omega_X}\mathcal{D}^{op})\simeq \mathcal{H}om_{\mathcal{A}^r[\mathcal{D}^{op}]}(\Theta_{\mathcal{A}},\mathcal{A}^r[\mathcal{D}^{op}]).$$
\end{proof}
\begin{remark}
Local duality operation $(-)^{\circ}$ is defined via the functor isomorphism,
$$\mathcal{H}\mathrm{om}_{\mathcal{A}}\big(-,(\mathcal{M}^{\circ})^{\ell}\big)\simeq\mathcal{H}\mathrm{om}_{\IC(RanX_{\DR})}(-\otimes^{\star}\mathcal{M},\mathcal{A}),$$
for the $\otimes^{\star}$-tensor product (c.f. (\ref{ssec: Formal D-Moduli from Chiral Koszul Duality})) for sheaves over $Ran(X_{\DR})$, but we do not use this full theory here.
\end{remark}

\subsubsection{Derived $\D$-prestacks}
Categories of $\D$-algebras of finite type and almost finite presentation define full $\infty$-subcategories 
$$\mathsf{CAlg}_{\mathrm{f.t}}(\D_X),\mathsf{CAlg}_{\mathrm{afp}}(\D_X)\hookrightarrow \mathsf{CAlg}_X(\D_X).$$
An example of a locally $\D$-free $\D$-scheme is one of the form $Spec_{\mathcal{D}}(\mathsf{Free}(\mathcal{D}_X\otimes\mathcal{E}))$ for some locally free sheaf $\mathcal{E}.$
\begin{definition}
\label{Definition: D-PreStack}
A \emph{derived $\D$-prestack} is a functor of $\infty$-categories
$\mathsf{X}:\mathsf{CAlg}_{\D_X}^{\leq 0}\rightarrow\mathsf{Spc}.$ The derived $\D$-prestack assocated to a derived affine $\D$-scheme is denoted $\underline{\mathsf{Spec}}_{\D_X}(\mathcal{A}^{\bullet}):\mathsf{CAlg}_{\D_X}^{\leq 0}\rightarrow \mathsf{Spc}.$
\end{definition}
The following remark establishes this notion via a diffrent route, which is of independent interest, for example, in relation to a notion of `derived chiral deformation theory.'
\begin{remark}
In §\ref{ssec: Formal D-Moduli from Chiral Koszul Duality} consider the restriction of functors $\mathsf{Fun}\big(\mathsf{Chiral}(X)^{op},\mathsf{Spc}),$ to the full-subcategory of commutative chiral Lie algebras. By the equivalence of $(\infty,1)$-categories (\ref{eqn: CAlg and CommChiral}) we arrive at the definition. In particular, it states that $\D_X$-prestacks $\mathsf{X}$ are those chiral functors $Y$ which are equivalent to their own restriction along the inclusion 
$\mathsf{Chiral}(X)_{comm}\hookrightarrow \mathsf{Chiral}(X).$
Alternatively, from Proposition \ref{prop: D-Prestack Proposition} define them in terms of factorization algebras situated on the main diagonal $X\subset \mathrm{Ran}(X).$
\end{remark}
The $\infty$-category of $\D$-prestacks will be denoted by $\PS_X(\D_X).$ By the $\infty$-categorical $\D_X$-Yoneda embedding (c.f. the proof of Proposition \ref{prop: Chiral Moduli Proposition}) we write a derived affine $\D$-scheme via its $\D$-prestack of points by $\underline{h}_{\mathcal{A}^{\bullet}}:=\underline{\mathsf{Spec}}_{\D_X}(\mathcal{A}).$

The adjunction $\big(i,(-)^{\mathrm{cl}}\big)$
extends to $\D_X$-pre stacks and their homotopy categories; given
$\mathcal{X}:\mathsf{CAlg}(\D_X)\rightarrow \mathsf{Spc}$, a higher $\D_X$-pre stack (or an étale stack in the suitable topology), one says that a \emph{derived enhancement} or a \emph{homotopical thickening}, for $\mathcal{X}$ is a derived $\D_X$-pre stack $\mathbb{R}\mathcal{X}$ with an open embedding 
\begin{equation}
    \label{eqn: DerivedEnhancement}
j:\mathcal{X}\rightarrow \mathbb{R}\mathcal{X},
\end{equation}
such that  $\big(\mathbb{R}\mathcal{X}\big)^{\mathrm{cl}}\simeq \mathcal{X}$ is an equivalence. The functor of classical truncation $(-)^{\mathrm{cl}}$ on representable derived $\D$-pre-stacks is given by $\mathbb{R}\underline{\mathrm{\mathsf{Spec}}}_{\mathcal{D}}(\mathcal{A})^{\mathrm{cl}}\cong Spec_{\D_X}\big(H_{\D}^0(\mathcal{A})\big).$

A derived $(-1)$-Artin $\D$-stack will be understood as an affine derived $\D$-scheme i.e. an object in the essential image of the evident $\D$-Yoneda embedding (c.f. \ref{eqn: D-Yoneda})
$$h_{\mathbf{dAff}_{\D_X}}^{\mathcal{D}}:\mathbf{dAff}_{\D_X} \hookrightarrow \PS_{\D_X}.$$
Thus, there exists a derived $\D_X$-algebra $\mathcal{A}$ for which $(h_{\mathbf{dAff}_{\D_X}}^{\mathcal{D}})^{op}(\mathcal{A})\simeq\underline{\mathsf{Spec}}_{\D_X}(\mathcal{A}).$
A derived $0$-Artin $\D$-stack might be defined as a colimit of a groupoid object:
$$ \cdots \rightarrow \mathsf{X}_1\rightrightarrows \mathsf{X}_0\rightarrow \big[\mathsf{X}_0/\mathsf{X}_1\big],$$
where $\mathsf{X}_0$ and $\mathsf{X}_1$ are $(-1)$-geometric and the top map in the double arrow is a submersion. We do not elaborate more here, as we do not require this specificity in this work.

Other classes of derived affine $\D$ schemes can be naturally defined by implementing finiteness hypothesis. For example, an object $\mathsf{X}\in \mathbf{dAff}_{\D_X}$ is holonomic if the underlying $\D$-algebra is holonomic as a $\D$-module. A holonomic prestack $\mathcal{X}$ is one determined by its values on holonomic derived $\D$-schemes via Kan extension along the embedding $\mathbf{dAff}_{\D_X}^{hol}\hookrightarrow \mathbf{dAff}_{\D_X}$ i.e. 
$$\PS_X^{hol}(\D_X)\xrightarrow{\simeq}\mathsf{Fun}\big(Comm(\DG^{hol})^{op},\mathsf{Spc}\big).$$
A more useful finiteness condition for us is given in §§ \ref{ssec: D-Fredholm}.
\subsubsection{Model categorical interludes}
\label{sssec: Model Categorical Interlude}
Let us provide a description in terms of model categories and their homotopy categories \cite{Hirschhorn2003}. Specifically, we describe a combinatorial left proper model structure localized along a proper set of moprhisms $S$, such that an $S$-local (fibrant) presheaf is exactly a presheaf that sees homotopy equivalences between test-affines as equivalences. We introduce derived $\D$-prestacks as simplicial presheaves which are homotopy invariant with respect to weak-equivalences of test affines e.g. derived affine $\D$-schemes. 
Let $Q$ denote the canonical cofibrant replacement functor on $\cdga_{\D_X},$ established in \cite[Thm.34, Thm.38]{Diss}.
\begin{proposition}[\cite{Kry}]
Let $S_0\subset W[\D]$ be a set of representatives of weak-equivalences and put $S:=h^{\D}[S_0].$ The left-Bousfield localisation 
$$\mathbf{dPSt}_{\D_X}:=L_{h^{\D}}[W_{\D}]\big(\mathbf{sPShv}\big(\mathbf{CAlg}(\D_X-\mathbf{Mod})^{\leq 0}\big)^{proj}\big),$$
is a combinatorial left proper simplicial model category whose cofibrations are level-wise and whose fibrant objects are those simplicial presheaves $\EuScript{F}$ such that for every $f\in S_0$, we have $\EuScript{F}(f)$ is a weak-equivalence of simplicial sets.
\end{proposition}
We do not give the full details, but mention the core aspects instead.

We form Bousfield localisation of simplicial presheaves with respect to the class of morphisms 
$h^{\D}[W_{\mathcal{D}}]=\{ h_f^{\mathcal{D}} | f \text{ is a weak equivalence} \},$
with $h^{\D}$ the Yoneda embedding and $W_{\mathcal{D}}$ the weak-equivalences.
More precisely, there is a local projective model structure on the category of simplicial presheaves on a Grothendieck site \cite{Bla01}, further generalized in \cite[Theorem 3.4.1]{TV02}.

The natural $\D$-geometric homotopy Yoneda embedding
\begin{equation}
\label{eqn: Model D Yoneda embedding}
    \mathbb{R}\underline{h}^{\mathcal{D}}:=Ho(h^{\D}):Ho\big(\mathbf{dAff}_{\D_X}\big)\rightarrow Ho\big(\mathbf{dPSt}_{\D_X}\big).
    \end{equation}
Since our objects $\EuScript{F}$ as $S$-local, i.e. fibrant with respect to original projective model structure, and satisfy for each $s\in S$, (e.g. $s=h_f^{\D},f\in S_0)$, the induced map on representatives $\mathrm{Maps}(h^{\D}(c),\EuScript{F})\simeq \mathrm{Maps}(h^{\D}(c'),\EuScript{F})$ is a weak-equivalence. These spaces are $\EuScript{F}(c),$ then $\EuScript{F}(c')\simeq \EuScript{F}(c)$ when $f:c\to c'\in S_0;$ in other words, we take the fibrant replacement and set
$\EuScript{F}(C)\simeq\underset{f:C\to C':f\in W_{\D}}{\mathrm{holim}}\hspace{.5mm} \EuScript{F}(C').$

Use the simpicial fibrant replacement
$\mathrm{R}(\mathcal{A})_{\bullet}$ given by sending $[n]\mapsto \mathrm{R}(\mathcal{A})_{[n]}:=\mathcal{A}\otimes\Omega_{\mathrm{poly}}^*(\Delta^{n})\in \cdga_{\D_X},$
with $\mathcal{A}\in \cdga_{\D_X},$ we have, by denoting the canonical cofibrant replacement functor $Q(\mathcal{A})\rightarrow \mathcal{A},$ an assignment
\begin{equation}
    \underline{\mathbf{Sol}}_{\D_X}:\cdga_{\D_X}\rightarrow \mathbf{dPSt}_{\D_X},\hspace{2mm}\mathcal{A}\mapsto \underline{\mathbf{Sol}}_{\D_X}(\mathcal{A}):=Maps_{\cdga_{\D_X}}(\mathcal{A},-).
\end{equation}
Precisely, $\mathrm{R}(-)_{\bullet}$ is the fibrant replacement in the Reedy model structure on $\mathrm{Funct}(\Delta^{op},\cdga_{\D_X})$ \cite{Hirschhorn2003}.

In conformity with equation (\ref{eqn: Classical D-Solutions}), we have
\begin{equation}
    \label{eqn: Pre-Derived non-linear solution functor}
    \mathbf{Sol}_{\D_X}(-):=\underline{\mathbf{Sol}}_{\D_X}(-,\mathcal{O}_X):\cdga_{\D_X}\rightarrow SSets,
\end{equation}
with $ \mathbf{Sol}_{\D_X}(\mathcal{A}):=Maps_{\cdga_{\D_X}}\big(\mathcal{A},\mathcal{O}_X\big)_{\bullet}.$
By general properties of homotopy categories of model categories we have derived enhancements of the functors (\ref{eqn: Pre-Derived non-linear solution functor}), 
\begin{equation}
    \label{eqn: Derived non-linear solution functor}
    \mathbb{R}\mathbf{Sol}_{\D_X}(-):Ho\big(\cdga_{\D_X}\big)\rightarrow Ho\big(SSets\big),
    \end{equation}
    where $\mathbb{R}\mathbf{Sol}_{\D_X}(\mathcal{A}):=\mathbb{R}\underline{\mathcal{H}\mathrm{om}}_{\cdga_{\D_X}}\big(\mathcal{A},\mathrm{R}(\mathcal{O}_X)\big)\simeq Maps\big(Q(\mathcal{A}),\mathrm{R}(\mathcal{O}_X)\big).$
    More generally, in the slice category $\big(\mathbf{cdga}_{\D_X})_{\mathcal{A}/},$
     compute $\mathrm{Maps}_{\big(\mathbf{cdga}_{\D_X})_{\mathcal{A}/}}\big(\mathcal{B},\mathcal{C}\big)$ as the space of maps $Q\mathcal{B}\to \mathrm{R}_{\bullet}(\mathcal{C}).$

     More generally, following \cite{TV02} we have 
     $$\mathbb{R}\underline{\mathbf{Spec}}_{\D}(\mathcal{M}^{\bullet}):Ho\big(\mathbf{dgMod}_{\D}(\mathcal{A})\big)^{op}\to Ho\big(Fun\big(\mathbf{dgMod}_{\D}(\mathcal{A}),\mathbf{SSet})\big)\big).$$
\begin{remark}
    There is a natural inclusion $\mathrm{Sol}_{\mathcal{D}}(\mathcal{A})\rightarrow \mathbb{R}\mathbf{Sol}_{\mathcal{D}}(\mathcal{A}),$ which might be thought of as the inclusion of the critical space into its homotopy tubular neighborhood as a derived enhancement (\ref{eqn: DerivedEnhancement}). 
\end{remark}
In the Appendix, we discuss two examples coming from physics. See \cite{KSY2} for further details. 

\begin{example}
\label{Derived Differential Algebra NLPDE Example}
Consider Example \ref{Differential Algebra Example}. One way to produce a derived enhancement of $Sol_{\D_{\mathbb{A}^1}}(\mathsf{P})$ is by observing that $\mathbb{R}\mathbf{Sol}_{\D_{\mathbb{A}^1}}(\mathsf{P})$ can be defined by sending an object $A^{\bullet}\in \mathbf{cdga}_k$ to 
the simplicial set of maps of differential graded commutative $A^{\bullet}(\!(\EQ)\!)$-algebras with compatible action by derivations. 
\end{example}
\subsection{$\D$-Geometric tangent and cotangent complexes}
\label{sssec: D-Geom Tangent and Cotangent Complexes}
 The cotangent complex $\mathbb{L}_{\mathcal{A}}$ as well as a relative cotangent $\D$-complex $\mathbb{L}_f$ to a morphism of $\D$-schemes is constructed as follows.
Let $\mathcal{A}$ be a derived $\D$-algebra. Set $\mathbb{L}_{\mathcal{A}}$ to be the left-derived object of the $\D$-geometric K\"ahler construction (cf.  \ref{ModulCatsTangents}). Namely, $\mathbb{L}_{\mathcal{A}}:=\Omega_{Q\mathcal{A}}\in Ho(\DG(\mathcal{A})).$
It has the obvious properties: there is a canonical map $\mathbb{L}_{\mathcal{A}}\rightarrow \Omega_{\mathcal{A}}$ which is an equivalence if $\mathcal{A}$ is cofibrant, and if $f:\mathcal{A}\rightarrow \mathcal{B}$ is a weak-equivalence of cofibrant algebras $df:Lf^*\Omega_{\mathcal{A}}\rightarrow \Omega_{\mathcal{B}}$ is a quasi-isomorphism. The canonical replacement map $c:Q\mathcal{A}\rightarrow \mathcal{A}$ induces a functor (denoted the same) $c:\DG(Q\mathcal{A})\rightarrow \DG(\mathcal{A})$ that provides an equivalence $$Ho(c):Ho(\DG(Q\mathcal{A}))\rightarrow Ho(\DG(\mathcal{A})).$$ Then $\mathbb{L}_{\mathcal{A}}:=Ho(c)\big(\Omega_{Q\mathcal{A}}\big).$ 

In the relative setting, consider $f:\mathcal{A}\to \mathcal{B}$ and a $\mathcal{B}[\D_X]$-dg module $\mathcal{M}.$ There is an object, the $\D$-square zero extension $\mathcal{B}[\mathcal{M}]$. Its natural $\D$-algebra structure is described in \cite{BD2004}. 
In particular, 
$\mathcal{B}[\mathcal{M}]$ is naturally an $\mathcal{A}[\D_X]$-dg algebra under $\mathcal{B}$ with repsect to the $\D$-action 
$$P\bullet (\beta,m):=\big(P\bullet \beta,P\bullet m\big),P\in \D_X,\beta\in \mathcal{B},m\in\mathcal{M}.$$
We use the standard multiplicative structure $(\beta_1,m_1)\cdot (\beta_2,m_2):=(\beta_1\beta_2,\beta_1\cdot m_2+\beta_2\cdot m_1),\beta_i\in \mathcal{B},m_i\in \mathcal{M},i=1,2.$

Following \cite[Def.~1.4.1.4]{TV2}, the space of $\D$-derivations is
\begin{equation}
    \label{D_X-Derivations}
\underline{\mathbf{Der}}_{\D}(\mathcal{B},\mathcal{M}):=\mathrm{Maps}_{\mathrm{Comm}\big(\mathbf{Mod}_{\D_X}(\mathcal{A})_{/\mathcal{B}}\big)}(\mathcal{B},\mathcal{B}[\mathcal{M}]).
\end{equation}
We make one brief remark concerning   \eqref{D_X-Derivations}.
\begin{remark}
\label{HoDDer}
    Let $\mathcal{A}\to\mathcal{B}$ be a morphism of derived $\D$-algebras. The functor of $\D$-derivations acts by 
    $\underline{\mathbf{Der}}_{\D}(\mathcal{B},-):Ho(\mathbf{Mod}_{\D}(\mathcal{B})\big)\to Ho(\mathbf{SSet}).$ It preserves the weak-equivalences e.g. between cofibrant objects and therefore defines an element of the model category of functors
    $\mathrm{Funct}\big(\mathbf{Mod}_{\D}(\mathcal{B}),\mathbf{SSet}\big),$ so that for each $\mathcal{M}$, the object \eqref{D_X-Derivations} is of $\mathbf{SSet},$ not merely $Ho(\mathbf{SSet}).$
\end{remark}
We have the following representability result.
\begin{proposition}
\label{Prop: RelCot}
    Let $f:\mathcal{A}\to\mathcal{B}$ be a morphism of derived $\D$-algebras as before. There is an object $\mathbb{L}_{f}$ with the structure of a $\mathcal{B}[\D_X]$-dg-module and $\mathcal{A}\otimes_{\mathcal{O}_X}\D_X$-linear derivation $d\in \underline{\mathbf{Der}}_{\mathcal{A}\otimes \D_X}(\mathcal{B},\mathbb{L}_f)$ such that 
    $$\mathrm{Maps}_{\mathcal{B}\otimes \D_X-\mathbf{dgMod}}(\mathbb{L}_f,\mathcal{M})\simeq \underline{\mathbf{Der}}_{\mathcal{A}\otimes\D_X}(\mathcal{B},\mathcal{M}),$$
   is an equivalence of spaces.
\end{proposition}
\begin{proof}
    Let $\mathcal{Q}:=Q(\mathcal{B})$ be a cofibrant replacement of $\mathcal{B}$ and put $\mathbb{L}_F:=\Omega_{\mathcal{Q}/\mathcal{A}}^1\otimes_{\mathcal{Q}}^L\mathcal{B}.$
    Since $\mathcal{Q}$ is cofibrant, $\Omega^1_{\mathcal{Q}/\mathcal{A}}$ is a $\mathcal{Q}[\D_X]$-dg module and there exists $d_{\mathcal{Q}}:\mathcal{Q}\to \Omega_{\mathcal{Q}/\mathcal{A}}^1$ which is both $\mathcal{A}$-linear and $\D$-linear. Then
    $\mathcal{B}\simeq Q\mathcal{B}\otimes_{Q\mathcal{B}}\mathcal{B}\xrightarrow{d\otimes id}\Omega_{\mathcal{Q}/\mathcal{A}}\otimes_{Q\mathcal{B}}\mathcal{B}=\mathbb{L}_f.$
    Moreover, the natural map
    $$\underline{\mathbf{Der}}_{\mathcal{A}\otimes_{\mathcal{O}_X}\D_X}\big(\mathcal{B},\mathcal{M}\big)\xrightarrow{\simeq}\underline{\mathbf{Der}}_{\mathcal{A}\otimes_{\mathcal{O}_X}\D_X}\big(\mathcal{Q},\mathcal{M}\big),$$
    is an isomorphism in $Ho(\mathbf{SSet}).$ By remark \ref{HoDDer}, it lifts to an equivalence in $\mathbf{SSet},$ and we have equivalences of simplicial sets $$\underline{\mathbf{Der}}_{\mathcal{A}\otimes_{\mathcal{O}_X}\D_X}\big(\mathcal{B},\mathcal{M}\big)\simeq \mathrm{Map}_{\mathcal{Q}\otimes_{\mathcal{O}_X}\D_X}(\Omega_{\mathcal{Q}/\mathcal{A}}^1,\mathcal{M})\simeq \mathrm{Map}_{\mathcal{B}\otimes \D_X}\big(\Omega_{\mathcal{Q}/\mathcal{A}}^1\otimes_{Q\mathcal{B}}^L\mathcal{B},\mathcal{M}\big).$$
    Thus
$\mathrm{Map}_{\mathcal{B}\otimes\D_X}\big(\mathbb{L}_f,\mathcal{M}\big)\simeq \mathrm{Map}_{\mathcal{Q}\otimes\D_X}\big(\Omega_{\mathcal{Q}/\mathcal{A}}^1,\mathcal{M}\big).$
\end{proof}
\begin{remark}
Alternatively, one might like to work with stabilizations of $\infty$-categories and the suspension-loop space adjunctions $(\Sigma_+^{\infty},\Omega^{\infty}).$ In the latter case, the essential image of $\Omega^{\infty}$ is interpreted as forming square-zero extensions, with $\Sigma_+^{\infty}:\mathsf{CAlg}_{X}(\D_X)_{/\mathcal{A}}\rightarrow \mathsf{Mod}_{\D_X}(\mathcal{A}),$ by identifying $\mathcal{A}$-module complexes with suspension.
For instance if $f:\mathsf{X}\rightarrow \mathsf{Y}$ is a morphism of affine derived $\D_X$-spaces over $X$, the relative cotangent complex of $f$ is determined by the equivalence $\mathbb{L}_f\simeq \Sigma_+^{\infty}\big(\mathcal{O}_{\mathsf{X}}\bigsqcup_{f^*\mathcal{O}_{\mathsf{Y}}}\mathcal{O}_{\mathsf{X}}\big),$
of $\mathcal{O}_{\mathsf{X}}[\mathcal{D}_X]$-modules.
\end{remark}

\begin{proposition}
\label{Relative D-cotangent Complex}
Let $f:\mathcal{X}\rightarrow \mathsf{Y}$ be a morphism of affine derived $\D_X$-spaces over $X.$
The relative cotangent complex of $f$ is given by the equivalence $\mathbb{L}_f\simeq cofib\big(f^*\mathbb{L}_{\mathcal{X}}\rightarrow \mathbb{L}_{\mathcal{Y}}\big)$ of $\mathcal{O}_{\mathcal{X}}[\mathcal{D}_X]$-modules.
Moreover, given a diagram $\mathcal{X}\xrightarrow{f}\mathcal{Y}\xrightarrow{g}\mathcal{Z}$ we have a cofiber sequence
$$f^*\mathbb{L}_g\rightarrow \mathbb{L}_{g\circ f}\rightarrow \mathbb{L}_f.$$
\end{proposition}
\begin{proof}
    Let $\mathcal{O}_{\mathcal{Y}}=\mathcal{B}$ and $\mathcal{O}_{\mathcal{X}}=\mathcal{A}.$
    The cofiber sequence in $\mathcal{A}[\D_X]$-dg modules,
    $\mathcal{A}\otimes_{\mathcal{B}}^L\mathbb{L}_{\mathcal{B}}\to \mathbb{L}_{\mathcal{A}}\to \mathbb{L}_f,$
    is clear, where $\mathbb{L}_f$ is given by Proposition \ref{Prop: RelCot}. 
For the second claim, set $\mathcal{C}:=\mathcal{O}_{\mathcal{Z}}$. We use the above cofiber sequence, and the analogous cofiber sequences
$$\mathcal{B}\otimes_{\mathcal{C}}^L\mathbb{L}_{\mathcal{C}}\to\mathbb{L}_{\mathcal{B}}\to \mathbb{L}_{\mathcal{B}/\mathcal{C}},\hspace{2mm}\mathcal{A}\otimes_{\mathcal{C}}^L\mathbb{L}_{\mathcal{C}}\to\mathbb{L}_{\mathcal{A}}\to \mathbb{L}_{\mathcal{A}/\mathcal{C}},$$
induced from $g$ and from $g\circ f,$ respectively. Then, since $f^*\mathbb{L}_g\simeq \mathcal{A}\otimes_{\mathcal{B}}^L\mathbb{L}_{\mathcal{B}/\mathcal{C}}$ and $\mathbb{L}_{g\circ f}\simeq \mathbb{L}_{\mathcal{A}/\mathcal{C}},$ the result follows from the sequence
$\mathcal{A}\otimes_{\mathcal{B}}^L\mathbb{L}_{\mathcal{B}/\mathcal{C}}\to\mathbb{L}_{\mathcal{A}/\mathcal{C}}\to\mathbb{L}_{\mathcal{A}/\mathcal{B}}.$
\end{proof}

\subsubsection{Covering families}
In this subsection we introduce descent along $\D$-étale covers. Starting with the notion of $\D$-prestacks, this defines a a suitable notion of stack in $\D$-geometry. 

Consider a finite family $\{\mathcal{A}_{\alpha}\}$ of derived $\D$-algebras. 
Then, note the equivalences
$$\bigsqcup_{\alpha}\mathbf{Spec}_{\D}(\mathcal{A}_{\alpha})\cong \mathbf{Spec}_{\D}\big(\prod_{\alpha}\mathcal{A}_{\alpha}\big),$$
and for a family of morphisms $f_{\alpha}:\mathcal{A}_{\alpha}\rightarrow \mathcal{A},$ in $\cdga_{\D_X}$ to a fixed $\mathcal{A}\in \cdga_{\D_X},$ one has that
$$\bigsqcup_{\alpha} f_{\alpha}:\bigsqcup_{\alpha}\mathbf{Spec}_{\D}(\mathcal{A}_{\alpha})\rightarrow \mathbf{Spec}_{\D}(\mathcal{A}).$$
A family $\big\{f_i:\mathbf{Spec}_{\D}(\mathcal{A}_i)\rightarrow \mathbf{Spec}_{\D}(\mathcal{A})\big\}_{i\in I},$ is said to be a $\D$-\emph{étale-covering} if $f_i:\mathcal{A}\rightarrow \mathcal{A}_i$ is $\D$-étale for each $i\in I$, which is to say that:
\begin{enumerate}
    \item for each morphism we have a fibre sequence 
$$\mathbb{L}_{\mathcal{A}}\otimes_{\mathcal{A}}^{\mathbb{L}}\mathcal{A}_i\rightarrow \mathbb{L}_{\mathcal{A}_i}\rightarrow \mathbb{L}_{\mathcal{A}_i/ \mathcal{A}},$$
and subsequent isomorphisms 
$\mathbb{L}_{\mathcal{A}}\otimes_{\mathcal{A}}^{\mathbb{L}}\mathcal{A}_i\xrightarrow{\simeq}\mathbb{L}_{\mathcal{A}_i},$
in $Ho\big(\mathsf{Mod}_{\mathcal{D}}(\mathcal{A}_i)\big)$. In other words, $\mathbb{L}_{\mathcal{A}_i/\mathcal{A}}\simeq 0$ for every $i\in I;$

\item 

there exists some finite subset $J\subset I$ such that $\{f_j: j\in J\}$ determine conservative functors
$\mathcal{A}_i\otimes_{\mathcal{A}}^{\mathrm{L}}-:Ho\big(\DG(\mathcal{A})\big)\rightarrow Ho\big(\DG(\mathcal{A}_i)\big).$
\end{enumerate}
Note that (2) is more explicitly stated: for every $n<0,$ we have isomorphism $$H_{\D}^n(\mathcal{A})\otimes H_{\D}^0(\mathcal{A}_i)\rightarrow H_{\D}^n(\mathcal{A}_i),$$ with $H_{\D}^0(f_i):H_{\D}^0(\mathcal{A})\rightarrow H_{\D}^0(\mathcal{A}_i)$ is $\D$-étale as in §\ref{sec: Algebraic D-Geometry}. We remark this is well-define by the convergence of the Tor-spectral sequence in $\D$-geometry \cite{Tor}.
\begin{proposition}
Given a $\D$-étale morphism $\mathcal{A}\rightarrow \mathcal{B},$ for every $\mathcal{R}\in \mathrm{CAlg}_X(\D_X)$, there is a pull-back diagram
\[
\begin{tikzcd}
\mathrm{Maps}(\mathcal{B},\mathcal{R})\arrow[d]\arrow[r] & \mathrm{Maps}\big(H_{\D}^0(\mathcal{B}),H_{\D}^0(\mathcal{R})\big)\arrow[d]
\\
\mathrm{Maps}(\mathcal{A},\mathcal{R})\arrow[r] & \mathrm{Maps}\big(H_{\D}^0(\mathcal{A}),H
_{\mathcal{D}}^0(\mathcal{R})\big).
\end{tikzcd}
\]
The $\D$-étale topology on $\mathsf{CAlg}_{X}(\D_X)^{c,op}$ is subcanonical.
\end{proposition}
\begin{proof}
One just shows that the Cech co-nerve of a surjective $\D$-étale map is a limit diagram in $\mathsf{CAlg}_X(\D_X).$
\end{proof}
The derived functor of points is therefore written as
$$\underline{\mathsf{Spec}}_{\mathcal{D}}(\mathcal{A})(-):=Maps_{\mathsf{CAlg}(\D_X)^c}(\mathcal{A},-):\mathsf{CAlg}_X(\D_X)^{c}\rightarrow \mathsf{Spc},$$
which is a derived $\D_X$-pre stack.

Restriction along this functor is given by a right Kan-extension thus for any $\mathcal{Y}\in \mathsf{dStk}_X(\D_X),$ one has
\begin{equation}
    \label{MapsDStkDX}
\mathrm{Maps}_{\mathsf{dStk}_X(\D_X)}(-,\mathcal{Y}):\mathsf{dStk}_X(\D_X)^{op}\rightarrow \mathsf{Spc}.\end{equation}
The functor \eqref{MapsDStkDX}
preserves limits, and is extended from $\mathsf{dAff}_X(\D_X)^{op}.$ Write
$$\mathrm{Maps}_{\mathsf{dStk}_X(\D_X)}(\mathcal{X},\mathcal{Y})\simeq \underset{\mathsf{Spec}_{\D}(\mathcal{A})\rightarrow \mathcal{X}}{\mathrm{holim}} \hspace{1mm} Maps(\mathsf{Spec}_{\D}(\mathcal{A}),\mathcal{Y}).$$

 Proposition \ref{Relative D-cotangent Complex} can be understood dually. In particular, the tangent complex is the derived local Verdier dual,
\begin{equation}
    \label{eqn:Tangent D complex}
\mathbb{T}_{\mathcal{X}}:=\mathsf{Maps}_{/\mathcal{X}}\big(\mathbb{L}_{\mathcal{X}},\mathcal{O}_{\mathcal{X}}\otimes_{\mathcal{O}_X}\mathcal{D}_X\big).
\end{equation}
Notice in (\ref{eqn:Tangent D complex}) we consider the tensor product over $X.$
Suppose that the cotangent complex $\mathbb{L}_{\mathcal{A}}$ is a perfect, therefore dualizable $\mathcal{A}[\mathcal{D}_X]$-module, defined as usual by $\mathbb{L}_{\mathcal{A}}:=\Omega_{Q\mathcal{A}}^1\otimes_{Q\mathcal{A}}\mathcal{A},$ we then put $\mathbb{T}_{\mathcal{A}}:=\mathbb{L}_{\mathcal{A}}^{\circ},$ as the derived local Verdier dual with 
$$\mathbb{L}_{\mathcal{A}}^r=\mathbb{L}_{\mathcal{A}}\otimes \omega_X,\hspace{5mm} \mathbb{L}_{\mathcal{A}}^{\ell}=\mathbb{L}_{\mathcal{A}}\otimes \omega_X^{-1}.$$
A $\D$-geometric version of quasi--smoothness arises when the cotangent complex is a two-term complex of $\mathcal{A}[\D_X]$-modules. See Definition \ref{QuasiSmoothDerivedPDE} for a related class of derived PDEs of finite jet-order (in particular, they are not $\D$-geometric).
\begin{definition}
    \label{Quasi-smooth}
   A derived $\D$-algebra is said to be $\D$\emph{-quasi-smooth} if $\mathbb{L}_{\mathcal{A}}$ is quasi-isomorphic to a $2$-term perfect complex (in $\DG(\mathcal{A})$) of degrees $[-1,0].$ It is quasi-smooth if pull-back along a solution $s:X\rightarrow Spec_{\D_X}(\mathcal{A})$ presents $s^*\mathbb{L}_{\mathcal{A}}$ as a $2$-term perfect complex in $\DG.$
\end{definition}
Both $\D$-quasi-smooth and quasi-smooth objects arise in practice, for instance in differential geometry and the theory of Fredholm operators. We return to this topic in a wider context in § \ref{ssec: D-Fredholm}.
\begin{example}
\label{Non-lin elliptic Example}
\normalfont

Let $\mathsf{P}:\Gamma(M,F)\rightarrow \Gamma(M,E)$ be a non-linear elliptic operator on sections of vector bundles (or coherent sheaves $\mathcal{F},\mathcal{E}$) over a compact manifold $M.$ Equivalently, it is a morphism of sheaves $\mathsf{F}_{\mathsf{P}}:Jet^k(\mathcal{F})\rightarrow \mathcal{E},$ and $\mathsf{Sol}_{\mathsf{P}}\simeq Ker(\mathsf{F}_{\mathsf{P}})$, is presented as a zero-locus. Let $s\in \mathsf{Sol}_{\mathsf{P}}$ i.e. $s\in \mathsf{P}^{-1}(0)$ be a solution giving a $2$-term complex $lin_s(\mathsf{P}):\Gamma(s^*TM)\rightarrow \Gamma(E).$ Often, this $2$-term tangent complex is Fredholm but if $lin_s(\mathsf{P})$ is not surjective one uses a reduction by local Kuranishi models and inverse function theorems to produce a quasi-smooth complex that is quasi-isomorphic to $lin_s(\mathsf{P}).$
\end{example}
Treating Example \ref{Non-lin elliptic Example} and similar situations arising in differential geometry are outside the scope of the current paper, as they pose their own set of unique challenges to overcome.\footnote{One should adapt elliptic bootstrapping
methods, Sobolev completions, regularity theory and temperate growth conditions to the derived differential and symplectic setting.}
\begin{example}
\label{Free Tangent Cotangent Example}
Considering \ref{Free Example}  yields $\mathbb{L}_{\mathsf{Free}_{\D}(\mathcal{M})}\simeq \mathsf{Free}_{\D}(\mathcal{M})\otimes\mathcal{M}$ which is naturally an object of $\mathsf{Free}_{\D}(\mathcal{M})-\DG.$ Consider $\mathsf{Free}_{\D}^0(\mathcal{M})$ by omiting the $0$-th component. Then the (derived) de Rham differential is a morphism $d_{\DR}:\mathsf{Free}_{\D}^0(\mathcal{M})\rightarrow \mathsf{Free}_{\D}(\mathcal{M})\otimes\mathcal{M}.$
\end{example}

\subsubsection{Classical tangent complex and linearizations}
Let $f:Spec_{\D_X}(\mathcal{A})\rightarrow Spec_{\D_X}(\mathcal{B})$ be a morphism of $(X,\mathcal{D}_X)$-affine schemes corresponding to jet-schemes of vector bundles $\pi:E\rightarrow X$ and $\eta:F\rightarrow X$ with $\mathcal{A}=\mathrm{Jet}^{\infty}(\mathcal{O}_E), \mathcal{B}=\mathrm{Jet}^{\infty}(\mathcal{O}_F)$ i.e. 
\begin{equation}
\begin{tikzcd}
Spec_{\mathcal{D}}(\mathcal{A})\arrow[d,"\pi_{\infty}"] \arrow[rr,"f"] & &  Spec_{\mathcal{D}}(\mathcal{B})\arrow[d,"\eta_{\infty}"]
\\
E\arrow[r,"\pi"] & X &F\arrow[l,"\eta"]
\end{tikzcd}.
\end{equation}
The tangent map to $f$ is naturally interpreted as the operator of universal linearization \cite{Krasilshchik1986}, associated to the system determined by $f.$ It determines an operator in total derivatives $\Box_f:\mathrm{Jet}^{\infty}(\mathcal{O}_E)\rightarrow \mathrm{Jet}^{\infty}(\mathcal{O}_F),$
whose tangent map is $df\in\mathrm{Hom}_{\mathcal{A}[\mathcal{D}_X]}\big(\Theta_{\mathcal{A}},f^*(\Theta_{\mathcal{B}})\big).$
\begin{proposition}
\label{Linearization Lemma}
The central de Rham cohomology $h(df),$ coincides with the operator of universal linearization associated to $\Box.$ That is,
$\ell_{\Box_f}=h(df):\pi_{\infty}^*\Theta_{E/X}\rightarrow \Box^*(\tilde{\pi}_{\infty})^*\Theta_{F/X}=\pi_{\infty}^*\Theta_{E/X},$
where $\tilde{\pi}:F\rightarrow X$ is the bundle for $F.$
\end{proposition}
This description agrees with the analytic notion of linearization of a $k$-th order operator $\mathsf{F}(x,\partial^k u)$ at a solution $\varphi$, given formally by $lin_{\varphi}(\mathsf{F})=\frac{d}{dt}\mathsf{F}\big(x,\partial^k(u+t\varphi)\big)|_{t=0}.$

\begin{remark}
For $\EQ\hookrightarrow \mathsf{Jets}_X^k(E)$, defined by $\mathsf{F}_{\mathsf{P}}$, the situation is summarized by the diagram:
\begin{equation}
    \label{eqn: Linearization Diagram}
    \adjustbox{scale=.85}{
\begin{tikzcd}
    \EQ\arrow[d] & \arrow[l] T_{\EQ/X}\arrow[d] & \arrow[l] j_k(s)^*T_{\EQ/X}\arrow[d] \simeq lin_s(\EQ)\arrow[d]
    \\
    \mathrm{Jets}_X^k(E)\arrow[d,"\pi_{k,0}"]& \arrow[l] T_{\mathrm{Jets}_X^k(E)/X}\arrow[d]& \arrow[l] j_k(s)^*T_{\mathrm{Jets}_X^k(E)/X} \arrow[d,"\pi_{k,0}'"]
    \\
    E\arrow[d,"\pi"]& \arrow[l] T_{E/X} &  \arrow[l] s^*T_{E/X}\arrow[d,"\pi'"]
    \\
  \arrow[bend left=65]{uuu}{j_k(s)} X \arrow[bend right=45,swap]{u}{s} & & \arrow[bend left=45]{u}{\eta} X \arrow[bend right=65, swap]{uuu}{j_k(\eta)} 
\end{tikzcd}}
\end{equation}
Solutions of $lins_s(\EQ)$ are vertical vector fields e.g. sections $\eta:X\rightarrow s^*T_{E/X}$
such that $j_k(\eta)=v\circ j_k(s)$ for some vertical symmetry $v$ of $\EQ.$ 
If $(x^i,u^{\alpha})$ are fibered coordinates of $E,$ $\eta=\eta^{\alpha}\partial_{u^{\alpha}}$ is a solution of $lin_s(\EQ)$ if and only if diagram (\ref{eqn: Linearization Diagram}) commutes. Thus, $j_k(\eta)(\EQ)$ gives a point in $T_{\EQ/X}$ for every $z\in X$ around $q=j_k(s)(\EQ)\in j_k(s)(X)\subset \EQ.$ 
Flows of vertical vector fields transform
solutions e.g $j_k(s)(X)$ of $\EQ$ into new solutions of $\EQ$, while solutions of $lin_s(\EQ)$ provide initial conditions for the flows.
\end{remark}
    
In the algebraic setting one makes sense of formal linearizations as well.

\begin{example}
\label{Formal Linearization Example}
    \normalfont 
Let $\mathcal{V}$ be the freely generated $\D$-algebra as in Example \ref{Differential Algebra Example}, and consider the algebraic PDE corresponding to $\mathsf{F}\in \mathcal{V}$.
If $s$ is a formal solution i.e $s=s(z)\in k(\!(z)\!),$ the linearisation of the equation $\mathsf{P}$ at $s$ is a linear differential operator on $k(\!(z)\!)$. 
In practice, one can introduce some linearization parameter $\epsilon$ and then put
$lin_{s}(\mathsf{P}):D_1^{\circ}\rightarrow D_1^{\circ},$ as the operator defined by
$lin_s(\mathsf{P})\big(f(z)\big):=\mathsf{F}_{\mathsf{P}}(z,s(z)+\epsilon f(z), \partial s(z)+\epsilon \partial f(z),\ldots \big)=0\hspace{2mm} mod\hspace{1mm} \epsilon^2.$

Recalling the derived $\D_{\mathbb{A}^1}$-space of solutions $\mathbb{R}\mathbf{Sol}_{\D_{\mathbb{A}^1}}(\mathsf{P})$ as in Example \ref{Derived Differential Algebra NLPDE Example}, it is $\D$-quasi-smooth since in degrees $[0,1],$ we have
$$\mathbb{T}_{\mathbb{R}\mathbf{Sol}_{\D_{\mathbb{A}^1}}(\mathsf{P}),s}\simeq D_1^{\circ}\xrightarrow{lin_s(\mathsf{P})}D_1^{\circ},$$
over the punctured disk. 
\end{example}

The following Subsection discusses an example of independent interest providing further evidence for the necessity of derived techniques in PDE theory (c.f. Example \ref{ex: Pierre} and \ref{obs: Pierre}). It can be skipped on a first reading.

\subsubsection{Cohomology of tangent $\D$-Complexes and singular integrals}
Let $D_1$ be the disc in $\mathbb{C}$ with coordinate variable $z$. Consider the non-linear equation $$\mathsf{P}_{\mathsf{F}}:=\mathsf{G}(z,u(z),\mathsf{F})=0,\hspace{1mm}\text{ with } \mathsf{G}(z,u(z),\mathsf{F}):=f(z)-z\partial_z u(z)+\mathsf{F}(z,\partial u),$$ with $\mathsf{F}$ a polynomial in $z.$ 
A family of solutions is given by $t\mapsto s_t(z):=tz+\mathsf{F}(t).$ One has

$$\mathbb{T}_{\mathbb{R}\mathbf{Sol}_{\D_{\mathbb{A}^1}},s_t}\simeq \mathcal{O}(\!(z)\!)\xrightarrow{lin_{s_t}(\mathsf{G})}\mathcal{O}(\!(z)\!),$$ where $lin_{s_t}(\mathsf{G}):=u(z)-\big(z+\partial_z\mathsf{F}(t)\big)\partial_zu(z).$ Then $H^i(\mathbb{T}_{\mathbb{R}\mathbf{Sol}_{\D_{\mathbb{A}^1}},s_t}\big)\neq 0,$ for $i>0$ if and only if $t$ is a root of $\partial_z\mathsf{F}.$

The presence of higher tangent cohomology has important meaning for the \emph{classical} geometry of the underlying equation e.g. in relation to singular integrals; the envelope of the generic solution is a singular integral given parametrically by
$$x(\tau)=-f'(\tau),\hspace{2mm} u(\tau)=-\tau f'(\tau)+f(\tau).$$
The parameter $\tau$ is the value of the derivative $u'$ and the singular integral 
arises as the solution of
the (overdetermined) system 
$\big\{\mathsf{P}_{\mathsf{F}}=0\big\}\cap \big\{Sep_{\mathsf{P}_{\mathsf{F}}}=f'(u')+x=0\big\},$
where $Sep$ is the separant. 
Note all points on the singular integral are irregular singularities.

Another example appears over the disc as well, by considering 
$\mathsf{P}_{\mathsf{F}}=\{(u')^2-Cu^n=0\},$ for a constant $C$ and for some value $n\geq 0.$ For example, if $n=3,$ 
all points $(x, 0, 0)$ are algebraic singularities, but all other points on the
corresponding algebraic differential equation in $Jets^1$ are regular. The singular
integral is given by the solution $u_0(x)=0$ and algebraic singularities form the graph of the first prolongation
of $u_0$. Such singular integrals $u_0(x)=0$ can be seen as coming from isolated points for which tangent cohomology of linearizations at general solutions $s_t(z)$ do not vanish. Let us describe these points for a simple case of $\mathsf{P}_{\mathsf{F}}=(u')^2-4u=0$ and note a generic solution ($\mathbb{A}^1$-family) is $s_t(z)=(z+t)^2.$ Then, the linearization is 
$lin_{s_t}(\mathsf{P}_{\mathsf{F}})=\{(z+t)u'-u=0\},$ and one verifies that 
$$H_{\D}^{i}(\mathbb{T}_{\mathbb{R}Sol(\mathsf{P}_{\mathsf{F}},s_t})\neq 0,i>0,$$
precisely at the points $t=0.$
\subsubsection{Derived Solutions for $\D$-algebraic PDEs}
\label{sssec: Derived Solutions for RelAlgNLPDEs}
We may now describe the derived analogs of solution functors $\mathsf{Sol}_{\mathcal{D}}(-)$ to a given $\D$-algebraic PDE $\mathcal{I}\rightarrow\mathcal{A}\rightarrow \mathcal{B},$
viewed as an object of $\mathsf{CAlg}(\D_X)_{\mathcal{A}/},$ by reflecting we are in a slice model topos. That is, we must describe $\tau$-local homotopy classes of functors
\begin{equation}
    \label{eqn:Slice D pre stacks}
    \widetilde{\mathsf{dAff}}_{\tau}(\D_X)_{/\underline{h}_{\mathcal{X}}^{\mathcal{D}_X}}:=Loc_{H_{\tau}}\big(\PS_X(\D_X)\big)_{/\underline{h}_{\mathcal{X}}^{\mathcal{D}_X}},
    \end{equation}
for some fibrant $\mathcal{X}\in\mathsf{dAff}(\D_X).$ The notation $Loc_{H_{\tau}}$ in equation (\ref{eqn:Slice D pre stacks}) means we are localizing along $\tau$-hypercovers, for $\tau$ the class of $\D$-étale covering families introduced above in §\ref{sssec: D-Geom Tangent and Cotangent Complexes}, and $\underline{h}_{\mathcal{X}}^{\mathcal{D}}$ is the homotopy Yoneda embedding (\ref{eqn: Model D Yoneda embedding}) for derived $\D$-algebras.

The discussion is simplified without loss of generality by assuming $\mathcal{A}^{\bullet}\in\cdga_{\D_X}$ is cofibrant and $\mathcal{X}=\mathsf{Spec}_{\D_X}(\mathcal{A}^{\bullet}).$ Consider now the slice model category of commutative $\mathcal{A}$-algebras,
$\big(\mathbf{dAff}_{\D_X,/\mathcal{X}} \big)^{\mathrm{op}}=(\cdga_{\D_X})_{\mathcal{A}/}.$ Recall, for instance the adjunctions (\ref{eqn: Free-forget AD-Mod/Alg Adjunction}), where we identified this latter object with the model category $Comm\big((\DG(\mathcal{A})\big).$ The model pre-topology $\tau$ on $\mathbf{dAff}_{\D_X}$ induces in a natural way, a topology denoted again by $\tau$, on $(\mathbf{dAff}_{\D_X})_{/\mathcal{X}}.$ The category of derived pre-stacks on this slice category is denoted 
$\mathbf{dPSt}(\mathcal{X}).$

There is a natural equivalence of categories
\begin{equation}
\label{eqn: Equivalence of categories for H localizations in slice}
\big(\widetilde{\mathbf{dAff}_{\D_X}}(\mathcal{X})_{\tau}\big):=Loc_{H_{\tau}}\big(\mathbf{dPSt}(\mathcal{X})\big)
\simeq Loc_{H_{\tau}}\big(\mathbf{dPSt}\big)_{/h_{\mathcal{X}}^{\mathcal{D}}}.
\end{equation}
The category on the right-hand side of \eqref{eqn: Equivalence of categories for H localizations in slice} is exactly $(\widetilde{\mathbf{dAff}_{\D_X}}_{\tau})_{/h_{\mathcal{X}}^{\mathcal{D}_X}}.$
We have a natural map
$h_{\mathcal{X}}^{\mathcal{D}_X}\rightarrow \underline{h}_{\mathcal{X}}^{\mathcal{D}_X},$ and this yields a Quillen adjunction
\begin{equation}
    \label{eqn: Quillen adjunction for slice D stacks}
    L\colon \widetilde{\mathbf{dAff}}_{\tau}(\D_X)_{h_{\mathcal{X}}^{\mathcal{D}_X}}\rightleftarrows \widetilde{\mathbf{dAff}}_{\tau}(\D_X)_{\underline{h}_{\mathcal{X}}^{\mathcal{D}_X}}\colon R
    \end{equation}
with the functor $L(F\rightarrow h_{\mathcal{X}}^{\mathcal{D}}):=F\rightarrow \underline{h}_{\mathcal{X}}^{\mathcal{D}}$ via the above natural map. One can verify the right-adjoint functor $R$ in (\ref{eqn: Quillen adjunction for slice D stacks}) is simply given by forming the pull-back of functors via the canonical diagram
\[
\begin{tikzcd}
 F\times_{\underline{h}_{\mathbf{X}}^{\mathcal{D}}}h_{\mathcal{X}}^{\mathcal{D}}\arrow[d]\arrow[r] & F\arrow[d]
 \\
 h_{\mathcal{X}}^{\mathcal{D}}\arrow[r] & \underline{h}_{\mathcal{X}}^{\mathcal{D}}
 \end{tikzcd}
\]

\begin{proposition}
\label{The result for interpreting DCPS}
The Quillen adjunction (\ref{eqn: Quillen adjunction for slice D stacks}) induces a Quillen equivalence $$
\widetilde{\mathbf{dAff}_{\D_X}}(\mathcal{X})_{\tau}\simeq \widetilde{\mathbf{dAff}_{\D_X}}_{/\underline{h}_{\mathbf{X}}^{\mathcal{D}_X}}.$$
\end{proposition}
\begin{proof}
Let $\mathcal{F}\rightarrow \underline{h}_{\mathbf{X}}^{\mathcal{D}}$ be a fibrant object and consider $\mathbf{Y}\in\mathbf{dAff}(\D_X)_{/\mathbf{X}}.$
Since there are obvious maps $h_{\mathbf{X}}^{\mathcal{D}}\rightarrow \underline{h}_{\mathbf{X}}^{\mathcal{D}}$, we may form the homotopy pull-back $\big(\mathcal{F}\times_{\underline{h}_{\mathbf{X}}^{\mathcal{D}}}^h h_{\mathbf{X}}^{\mathcal{D}}\big)(\mathbf{Y}).$
Furthermore, by base-change there is a canonical functor
$\mathrm{Ho}\big(\widetilde{\mathbf{dAff}}_{\tau}(\D_X)_{/\underline{h}_{\mathbf{X}}^{\mathcal{D}}}\big)\rightarrow  \mathrm{Ho}\big(\widetilde{\mathbf{dAff}}_{\tau}(\D_X)_{/h_{\mathbf{X}}^{\mathcal{D}}}\big).$
This is conservative since if $\mathbf{Y}$ is fibrant, the morphism $h_{\mathbf{X}}^{\mathcal{D}}(\mathbf{Y})\rightarrow \underline{h}_{\mathbf{X}}^{\mathcal{D}}(\mathbf{Y})$ is always surjective (up to homotopy).
We must show that the forgetful functor,
$\mathrm{Ho}\big(\widetilde{\mathbf{dAff}}_{\mathcal{D}_X,\tau}(\mathbf{X})\big)\rightarrow \mathrm{Ho} \big(\widetilde{\mathbf{dAff}}_{\tau}(\D_X)_{/\underline{h}_{\mathbf{X}}^{\mathcal{D}}}\big)$
is fully faithful.
To this end, we employ the Yoneda lemma for the slice site $\mathbf{dAff}(\D_X)_{/\mathbf{Y}}.$ Namely, for $\mathbf{Y},\mathbf{Z}$ there are equivalences
$\mathrm{Maps}_{\widetilde{\mathbf{dAff}}_{\mathcal{D}_X,\tau}(\mathbf{X})}\big(\mathbf{Spec}_{\D}(\mathcal{B}),\mathbf{Spec}_{\D}(\mathcal{R})\big)\cong \mathrm{Maps}_{\mathbf{dAff}(\D_X)_{/\mathbf{X}}}\big(\mathbf{Y},\mathbf{Z}\big),$
affine (derived) $\D$-schemes $\mathbf{Y}\cong\textbf{Spec}_{\mathcal{D}}(\mathcal{B}),$ and $\mathbf{Z}\cong \textbf{Spec}_{\mathcal{D}}(\mathcal{R}).$
Yoneda lemma for $\mathbf{dAff}(\D_X)$ provides an equivalence

$$\mathrm{Maps}_{\widetilde{\mathbf{dAff}}_{\tau}(\D_X)_{/\underline{h}_{\mathbf{X}}^{\mathcal{D}}}}\big(\mathbf{Spec}_{\D}(\mathcal{B}),\mathbf{Spec}_{\D}(\mathcal{R})\big)\cong \mathrm{Maps}_{\mathbf{dAff}(\D_X)}\big(\mathbf{Y},\mathbf{Z}\big).$$
Moreover, both the sequence
\begin{equation}
\label{eqn:fibration seq 1}
   \mathrm{Maps}_{\widetilde{\mathbf{dAff}}_{\mathcal{D}_X,\tau}(\mathbf{X})}\big(\mathbf{Spec}_{\D}(\mathcal{B}),\mathbf{Spec}_{\D}(\mathcal{R})\big)\rightarrow \mathrm{Maps}_{\mathbf{dAff}(\D_X)}\big(\mathbf{Y},\mathbf{Z}\big)\rightarrow \mathrm{Maps}_{\mathbf{dAff}(\D_X)}\big(\mathbf{Z},\mathbf{X}\big), 
\end{equation}
as well as
\begin{equation}
\label{eqn:fibration seq 2}
    \text{Maps}_{\widetilde{\mathbf{dAff}}_{\tau}(\D_X)_{/\underline{h}_{\mathbf{X}}^{\mathcal{D}}}}\big(\mathbf{Spec}_{\D}(\mathcal{B}),\mathbf{Spec}_{\D}(\mathcal{R})\big)\rightarrow \mathrm{Maps}_{\mathbf{dAff}(\D_X)}\big(\mathbf{Y},\mathbf{Z}\big)\rightarrow \mathrm{Maps}_{\mathbf{dAff}(\D_X)}\big(\mathbf{Z},\mathbf{X}\big),
\end{equation}
are fibration sequences.
Together, they imply the weak-equivalence of simplicial sets
\begin{equation}
     \mathrm{Maps}_{\widetilde{\mathbf{dAff}}_{\mathcal{D}_X,\tau}(\mathbf{X})}\big(\mathbf{Spec}_{\D}(\mathcal{B}),\mathbf{Spec}_{\D}(\mathcal{R})\big)\cong  \mathrm{Maps}_{\widetilde{\mathbf{dAff}}_{\tau}(\D_X)_{/\underline{h}_{\mathbf{X}}^{\mathcal{D}}}}\big(\mathbf{Spec}_{\D}(\mathcal{B}),\mathbf{Spec}_{\D}(\mathcal{R})\big).
\end{equation}
In other words, one has that the functor
$\mathrm{Ho}\big(\widetilde{\mathbf{dAff}}_{\mathcal{D}_X,\tau}(\mathbf{X})\big)\rightarrow \mathrm{Ho} \big(\widetilde{\mathbf{dAff}}_{\tau}(\D_X)_{/\underline{h}_{\mathbf{X}}^{\mathcal{D}}}\big)$
is fully faithful when restricted to representable derived $\D_X$-stacks. However, since an object on the left hand side is a colimit of representable stacks and the derived pull-back functor $\mathrm{Ho}\big(\widetilde{\mathbf{dAff}}_{\tau}(\D_X)_{/\underline{h}_{\mathbf{X}}^{\mathcal{D}}}\big)\rightarrow  \mathrm{Ho}\big(\widetilde{\mathbf{dAff}}_{\tau}(\D_X)_{/h_{\mathbf{X}}^{\mathcal{D}}}\big)$ commutes with homotopy colimits, since homotopy pullbacks of simplicial sets do so, we obtain that the required forgetful functor is actually fully faithful on the entire category, not just on the representable derived stacks.
\end{proof}
Proposition \ref{The result for interpreting DCPS} characterizes the homotopy category
$Ho\big(\widetilde{\mathbf{dAff}_{\D_X}}_{/\underline{h}_{\mathcal{X}}^{\mathcal{D}}}\big),$
as functors
\begin{equation}
    \label{eqn:Slice functors}
    \mathcal{X}:(\cdga_{\D_X})_{\mathcal{A}/}\rightarrow SSets.
    \end{equation}
In particular, construction of (\ref{eqn:Slice functors}) applies to any object $\mathcal{I}\rightarrow \mathcal{A}\rightarrow \mathcal{B}.$
For instance, $\mathbb{R}\mathbf{Sol}_{\mathcal{D}}(\mathcal{A}_{EL})$ as in
\ref{Derived EL PDE Example} is an object of
$Ho\big(\widetilde{\mathbf{dAff}_{\D_X}}/\underline{h}_{\mathsf{Spec}_{\D}(\mathcal{A})}^{\mathcal{D}}\big),$ by Proposition \ref{The result for interpreting DCPS}.
\subsection{Compactness}
\label{ssec: Compactness}
The appropriate non-linear algebraic theory of coherent $\D$-modules are those which are homotopically finitely $\D$-presented. They will be those compact objects in the $\infty$-category $\mathsf{CAlg}_X(\D_X)$, or equivalently in $Ho(\cdga_{\D_X}).$ We will use the characterization of such objects to obtain the appropriate derived version of coherent $\mathcal{A}^{\bullet}[\mathcal{D}_X]$-modules when $\mathcal{A}^{\bullet}$ is such a homotopically finitely presented algebra. 

Compact differential graded $\D_X$-modules are perfect $\D_X$-modules, and in some constructions, it is necessary to fix an algebra $\mathcal{A}$ and understand it to be given by retracts of such objects with underlying perfect $\D$-module. The details are spelled out in the following two subsections.

\subsubsection{Cellular Objects}
\label{sssec: Cellular Objects}
A finite-cell object in $\mathsf{CAlg}_X(\D_X)$ is 
a sequence 
$$\mathcal{B}_0=\iota\mathcal{O}_X\rightarrow \mathcal{B}_1^{\bullet}\rightarrow\cdots\rightarrow\mathcal{B}_{n-1}^{\bullet}\rightarrow \mathcal{B}_n^{\bullet}=\mathcal{B},$$
for some $0\leq n<\infty.$ Let us say that a cellular object is \emph{quasi-finite} if there exists a countable sequence
\begin{equation}
    \label{eqn: quasi-finite cell}
\mathcal{B}_0\rightarrow\cdots\rightarrow \mathcal{B}_{n-1}\rightarrow \mathcal{B}_n\rightarrow\ldots\hspace{2mm}\text{ such that } \mathcal{B}\simeq \underset{i\in I}{\mathrm{colim}}\hspace{1mm}\mathcal{B}_i.
\end{equation}

A derived $\D_X$-algebra $\mathcal{A}^{\bullet}$ is finitely $\D$-presented if it is a retract of a finite cell object, such that there are
homotopy coCartesian squares 
\begin{equation}
    \label{eqn: Hopushout squares}
\begin{tikzcd}
\Free(\mathcal{M}_i^{\bullet})\arrow[r,"c_i"] \arrow[d,"\alpha_i"] & \Free(\mathcal{M}_{i+1}^{\bullet})\arrow[d,"\alpha_{i+1}"]
    \\
    \mathcal{B}_i^{\bullet}\arrow[r,"\beta_i"] & \mathcal{B}_{i+1}^{\bullet}
\end{tikzcd},
\end{equation}
with $\mathcal{M}_i^{\bullet}$ a compact object in the $\infty$-category $\mathsf{Mod}(\D_X)$ i.e. a complex with bounded coherent cohomology. The morphism $c_i:\Free(\mathcal{M}_i^{\bullet})\rightarrow \Free(\mathcal{M}_{i+1}^{\bullet})$ is a cofibration.

Given a pushout \eqref{HoPushOut1}, by Proposition \ref{Relative D-cotangent Complex}, we have 
$$\mathbb{L}_{\mathcal{B}_{i+1}/\mathcal{B}_i} \;\simeq\; \mathcal{B}_{i+1} \otimes_{\Free(\mathcal{M}_i)} \mathbb{L}_{\Free(\mathcal{M}_{i+1})/\Free(\mathcal{M}_i)}.$$
Since $c_i$ is a cofibration, $\mathbb{L}_{\Free(\mathcal{M}_{i+1})/\Free(\mathcal{M}_i)} \;\simeq\; \Free(\mathcal{M}_{i+1}) \otimes_{\mathcal{O}_X} \operatorname{cofib}(\mathcal{M}_i\to \mathcal{M}_{i+1}).$ The description of cotangent complex for a $\D$-finitely presented algebra is given in Proposition \ref{HofpDAlgTangentComplex}.
\begin{remark}[Terminology]
\normalfont 
For later use, we also say that a finitely $\D$-presented algebra is \emph{pair-wise transversal} if each $\mathcal{M}_i$ in (\ref{eqn: Hopushout squares}) are a non-characteristic pair (\ref{eqn: NC-pair}).
\end{remark}
A cofibration is a retract of a (transfinite) composition of pushouts of generating cofibrations $\Free(S^{n-1})\rightarrow \Free({D^n})$ with $S^{n-1}$ and $D^n$ the sphere and disk complexes with object $\mathcal{D}_X.$ For instance, $S^n$ has $\D_X$ in degree $n$, and is zero everywhere else.

\begin{proposition}
\label{Retract proposition}
For every homotopy push-out diagram
\[
\begin{tikzcd}
    \mathcal{A}\arrow[r] & \mathcal{B}
    \\
    \mathcal{A}'\arrow[u]\arrow[r] & \mathcal{B}'\arrow[u]
\end{tikzcd}
\]
in $\mathsf{CAlg}_X(\D_X),$ we have a corresponding pull-pack diagram 
\[
\begin{tikzcd}
    Maps(\mathcal{A},-)\arrow[d]&\arrow[l] Maps(\mathcal{B},-)\arrow[d]
    \\
    Maps(\mathcal{A}',-)& \arrow[l] Maps(\mathcal{B}',-)
\end{tikzcd}
\]
Given a ceullar retract \eqref{eqn: quasi-finite cell} $Maps(\mathcal{B},-)\rightarrow\ldots \rightarrow Maps(\mathcal{B}_{k+1},-)\rightarrow Maps(\mathcal{B}_{k},-)\rightarrow \mathcal{O}_X,$ is a retract as well.
\end{proposition}
We will also characterize $\D$-algebras which are almost finitely $\D$-presented. 
\begin{proposition}
\label{Ho-D-fp Proposition}
An object $\mathcal{A}$ is almost finitely $\D$-presented if it is a retract of a (quasi-finite) cell object \emph{(\ref{eqn: quasi-finite cell})} such that $\mathsf{Spec}_{\D_X}(\mathcal{A})\big(\underset{i\in I}{\mathrm{colim}} \mathcal{A}_i^{\bullet}\big)\simeq \underset{i\in I}{\mathrm{colim}}\hspace{1mm}\mathsf{Spec}_{\D_X}(\mathcal{A}^{\bullet})(\mathcal{A}_i^{\bullet}),$ for any filtered diagram $\{\mathcal{A}_i\}$ of eventually coconnective $\D$-algebras.
\end{proposition}

Cofibrations in $\mathsf{Mod}(\mathcal{A}[\mathcal{D}_X])$ are themselves retracts of induced $I$-cells. In particular, they are retracts of transfinite compositions of pushouts of generating cofibrations, under the adjunction (\ref{eqn: Free-forget AD-Mod/Alg Adjunction}). So, they are pushouts of
$$id_A\otimes i_n:\mathsf{Free}_{\mathcal{A}}(S^{n-1})\rightarrow \mathsf{Free}_{\mathcal{A}}(D^n),$$
along a morphism $f:\mathsf{Free}_{\mathcal{A}}(S^{n-1})\rightarrow \mathcal{B},$ by
\[
\begin{tikzcd}
    \mathsf{Free}_{\mathcal{A}}(S^{n-1})\arrow[d,"id\otimes i_n"] \arrow[r,"f"] & \mathcal{B}\arrow[d,"incl"]
    \\
    \mathsf{Free}_{\mathcal{A}}(D^n)\arrow[r] &\mathcal{B}\oplus \mathsf{Free}_{\mathcal{A}}(S^n).
\end{tikzcd}
\]

\begin{proposition}
\label{Compact D algebras}
A derived $\D$-algebra is compact if it is a retract of a strict finite cellular object in $Ho\big(\cdga_{\D_X}\big)$ whose underlying $\D$-module is compact.
\end{proposition}
\begin{proof}
This follows from the fact that the forgetful functor $U$ in the adjunction (\ref{eqn: Free-forgetful algebra}) commutes with filtered colimits. That is, for a $\D$-module $\mathcal{M}$,
$$\mathrm{Hom}_{\cdga_{\D_X}}\big(\mathsf{Free}_{\D}\mathcal{M},\underset{i}{\mathrm{colim}}\hspace{1mm} \mathcal{A}(i)\big)\cong \underset{i}{\mathrm{colim}}\hspace{1mm} \mathrm{Hom}_{\cdga_{\D_X}}\big(\mathsf{Free}_{\D}(\mathcal{M}),\mathcal{A}(i)\big),$$
for a filtered diagram of $\mathcal{A}=\{\mathcal{A}(i)\}$ of $\D$-algebras. 
Indeed,
\begin{eqnarray*}
\mathrm{Hom}_{\cdga_{\D_X}}\big(\mathsf{Free}_{\D}(\mathcal{M}),\underset{i}{\mathrm{colim}}\hspace{1mm} \mathcal{A}(i)\big)&\cong&\mathrm{Hom}_{\DG}\big(\mathcal{M},U(\underset{i}{\mathrm{colim}}\hspace{1mm}(\mathcal{A}(i))\big)
\\
&=& \mathrm{Hom}_{\DG}\big(\mathcal{M},\underset{i}{\mathrm{colim}}\hspace{1mm} U(\mathcal{A}(i))\big)
\\
&\cong& \underset{i}{\mathrm{colim}}\hspace{1mm}\mathrm{Hom}_{\DG}\big(\mathcal{M},U(\mathcal{A}(i))\big)
\\
&\cong&
\underset{i}{\mathrm{colim}}\hspace{1mm} \mathrm{Hom}_{\cdga_{\D_X}}\big(\mathsf{Free}_{\D}\mathcal{M},\mathcal{A}(i)\big).
\end{eqnarray*}
Then if $R_{\bullet}(\mathcal{A})$ is the 
the Reedy fibrant replacement of the constant simplicial object with values in such a filtered diagram $\mathcal{A}$ in the model category structure on simplicial objects, it follows since $\mathrm{colim}_{q\in Q}R_{\bullet}\big(\underline{\mathcal{A}}(q)\big)$ is Reedy fibrant object in simplicial objects in $\mathbf{cdga}_{\D_X}$ and we obtain, for any ($\omega$-small)
$\mathcal{B}^{\bullet}\in\mathsf{cdga}_{\D_X}^{\leq 0},$ that
\begin{eqnarray*}
\mathrm{Hocolim}_{q\in Q}\mathsf{Maps}\big(\mathcal{B},\underline{\mathcal{A}}(q)\big) &\cong& \mathrm{colim}\mathsf{Maps}\big(\mathcal{B},\underline{\mathcal{A}}(q)\big)
\\
&\cong& \mathcal{H}\mathrm{om}_{\mathbf{cdga}_{\D_X}}\big(Q_{\mathcal{D}}\mathcal{B},\mathrm{colim}R_{\bullet}\underline{\mathcal{A}}(q)\big)
\\
&\cong& \mathsf{Maps}\big(\mathcal{B},\mathrm{colim}\underline{\mathcal{A}}(q)\big).
\end{eqnarray*}
As filtered colimits of simplicial sets preserve homotopy pull-backs so any finite cell objects is also homotopically finitely presented, as they are constructed from domains and codomains of $I$ by iterated homotopy push-outs. This implies that any retract of a finite cell object is homotopically finitely presented. 

Conversely, let $\mathcal{A}$ be finitely presented in $\mathrm{Ho}\big(\mathbf{cdga}_{\D_X}^{\leq 0}\big)$. The small object argument tells us that $\mathcal{A}$ is equivalent to an $I$-cell complex and the compactness of the domains and codomains of $I$ tell us this complex is the filtered colimit of its finite sub $I$-cell complexes. Consequently, $\mathcal{A}$ is equivalent to a filtered colimit of strict finite $I$-cell objects and we write it as $\mathrm{Colim}_i\mathcal{A}_i$, with $\mathcal{A}_i$ a filtered diagram of finite cell objects. Maps in $\mathrm{Ho}\big(\mathbf{cdga}_{\D_X}^{\leq 0}\big)$ from $\mathcal{A}$ are equivalent to a  colimit over maps from $\mathcal{A}\rightarrow \mathcal{A}_i$. Thus $1_{\mathcal{A}}$, the identity map, factors through some $\mathcal{A}_i$. In other words, $\mathcal{A}$ is a retract of some $\mathcal{A}_i$ in $\mathrm{Ho}\big(\mathbf{cdga}_{\D_X}^{\leq 0}\big).$

The underlying $\D_X$-module is compact since the forgetful functor commutes with filtered colimits, so a compact algebra that is written as a filtered colimit of strict finite $I$-cell objects is equivalently viewed as a filtered colimit of compact $\D_X$-modules.

\end{proof}
\begin{proposition}
\label{HofpDAlgTangentComplex}
 A homotopically finite $\D$-algebra $\mathcal{B}^{\bullet}$ has a $\D$-finite-presentation tangent complex. Moreover, over its classical locus it is a complex of $H_{\D}^0(\mathcal{B})[\D]$-modules of finite rank.
\end{proposition}

\begin{proof}
Let us start by explicitly describing all relative cotangent $\D_X$-modules $\mathbb{L}_{\mathcal{B}_k/\mathcal{B}_{k-1}}.$
It is easy to see from the inductive definitions of each level, that on one hand we have
$$\mathbb{L}_{\mathcal{B}_{k-1}/\mathrm{Sym}_{\mathcal{B}_{k-1}}(\mathcal{M}_{-k}[k-1])}\simeq \mathcal{M}_{-k}^{\bullet}[k],$$
while by base-change we have that
\begin{equation}
\label{eqn: Relative Cotangent k-1}
\mathbb{L}_{\mathcal{B}_k/\mathcal{B}_{k-1}}\simeq \mathcal{B}_k^{\bullet}\otimes_{\mathcal{B}_{k-1}}^L\mathcal{M}_{-k}^{\bullet}[k].
\end{equation}

In this case, one has
\begin{equation}
\label{eqn: CotComplex}
    \mathbb{L}_{\mathcal{B}^{\bullet}}|_{\mathsf{Spec}(H^0(\mathcal{B}^{\bullet}))}\simeq 0\rightarrow \mathcal{L}^{-n}\rightarrow \mathcal{L}^{1-n}\rightarrow\ldots\rightarrow\mathcal{L}^{-1}\rightarrow \mathcal{L}^0\rightarrow 0,
\end{equation}
represented as a bounded (above by $0$) complex of free $H_{\D}^0(\mathcal{B}^{\bullet})[\mathcal{D}_X]$-modules, with terms
$$\mathcal{L}^{-k}=H_{\D}^{-k}\big(\mathbb{L}_{\mathcal{B}_k/\mathcal{B}_{k-1}}\big),$$
in degree $-k.$
By dualizing, namely applying $\mathbb{R}\mathcal{H}\mathrm{om}_{\mathcal{B}[\mathcal{D}_X]}(-,\mathcal{B}[\mathcal{D}_X])\otimes\omega_X^{-1},$ the local Verdier duality operation (\ref{Defin: Inner Verdier Duality})
we have
$$\mathbb{T}_{\mathcal{B}^{\bullet}}^{\ell}\simeq \mathbb{R}\mathcal{H}om_{\mathcal{B}[\mathcal{D}]}\big(\mathbb{L}_{\mathcal{B}},\mathcal{B}[\mathcal{D}_X]\big)|_{\mathsf{Spec}(H^0\mathcal{B})},$$
is has amplitude $[0,n],$ and whose terms are again free $H_{\D}^0(\mathcal{B}^{\bullet})[\mathcal{D}_X]$-modules.
\end{proof}

\begin{proposition}
 \label{Tangent Good Filtration}
Consider $\mathcal{B}^{\bullet}$ as above. Then its tangent complex admits a good filtration as a $\D$-module, compatible with the natural $H_{\D}^0(\mathcal{B}^{\bullet})$-action.
\end{proposition}
\begin{proof}
In particular, for free algebras (see \autoref{Free Tangent Cotangent Example}), $\mathbb{L}\simeq \mathrm{Sym}^*(\mathcal{M}^{\bullet})\otimes\mathcal{M},$ and one uses the fact that locally on $X$ coherent $\D$-modules $\mathcal{M}^{\bullet}$ admit good filtrations. Thus, we have the obvious one on $\mathbb{L}$ induced from $\mathcal{M}.$
Generally, since  
$\mathbb{L}_{\mathcal{B}_{k-1}/\mathrm{Sym}_{\mathcal{B}_{k-1}}(\mathcal{M}_{-k}[k-1])}\simeq \mathcal{M}_{-k}^{\bullet}[k],$
where $\mathcal{M}_{-k}^{\bullet}$ is coherent, thus carries a good filtration, so do the relative cotangents. Similar reasoning endows terms (\ref{eqn: Relative Cotangent k-1}) with a good filtration, and similarly for (\ref{eqn: CotComplex}).
For $\mathbb{T}_{\mathcal{B}}^{\ell}$, the induced filtration is the obvious dual one,
$$Fl^m(\mathcal{L}_{-k}^{\circ}):=\{\psi \in \mathbb{R}\mathcal{H}om(\mathcal{L}^{-j},\mathcal{B}[\mathcal{D}])|\psi(Fl^r\mathcal{L}^{-j})\subset Fl^{r+m}(\mathcal{B}[\mathcal{D}])\big\}.$$
\end{proof}

\subsection{Connectivity and finiteness for $\D$-spaces}
In this subsection, we follow \cite{GR17a}, and collect the $\D$-geometric analogs of standard notions e.g. $n$-coconnectivity, arising in derived algebraic geometry.

\subsubsection{Truncation}
For each $n\geq 0$, there is a 
$$\mathsf{CAlg}^{-n}(\D_X):=\{\mathcal{A}^{\bullet}\in\mathsf{cdga}_{\D_X}^{\leq 0}| H_{\D}^i(\mathcal{A}^{\bullet})=0, \forall i\leq -n\},$$ 
and a canonical embedding of this category into arbitrary derived $\D$-algebras admitting a left-adjoint, referred to as  (cohomological) truncation, 
\begin{equation}
\label{eqn:Cohomological truncation}
    \tau^{-n }:\mathsf{CAlg}_X(\D_X)\rightarrow \mathsf{CAlg}_X^{ -n}(\D_X).
    \end{equation}
An object $\mathcal{X}\in\mathsf{dAff}(\D_X)$ is $n$-\emph{co-connective} if $\mathcal{X}$ is representable, with corresponding derived $\D$-algebra in the essential image of (\ref{eqn:Cohomological truncation}) i.e. $\mathcal{X}=\mathbf{Spec}_{\D_X}(\mathcal{A}^{\bullet})$ with $\mathcal{A}^{\bullet}\in \mathsf{CAlg}_X^{-n}(\D_X)$. At the level of opposite categories, we have a full-subcategory
$\mathsf{dAff}^{n}(\D_X)\subset \mathsf{dAff}(\D_X),$
and natural embeddings
\begin{equation}
\label{eqn:Embedding for D affines} 
i_{\leq n}:\mathsf{dAff}^{n}(\D_X) \rightarrow \mathsf{dAff}(\D_X)
\end{equation}
admitting right-adjoint functors
\begin{equation}
\label{eqn:Restriction on affines functor}
(-)^{\leq n}:\mathsf{dAff}(\D_X) \rightarrow \mathsf{dAff}^{ n}(\D_X),\hspace{2mm} \mathbf{Spec}_{\D_X}(\mathcal{A})\mapsto\mathbf{Spec}_{\D_X}\big(\tau^{ -n}\mathcal{A}\big).
\end{equation}
Note that for $n=0$ the category $\mathsf{dAff}^{0}(\D_X)$ coincides with $\mathrm{Aff}_X(\D_X)$, the $1$-category of classical $\D_X$-affine schemes. In particular, algebras are concentrated in cohomological degree $0,$ supplied with the trivial differential.

Natural extensions of various free-forgetful adjunctions e.g. (\ref{eqn: Free-forgetful algebra}), which can be shown to naturally preserve $\infty$-subcategories of $n$-coconnective objects, give rise to a natural functor
$$\mathsf{CAlg}_X^{-n}(\D_X)\rightarrow \mathsf{Mod}(\D_X)\rightarrow \mathsf{QCoh}(X),$$
by forgetting the $\D$-algebra structure followed by forgetting the $\D$-module structure.
The essential image of such a functor are $n$-co connective quasi-coherent sheaves on $X.$

Moreover, one has a natural diagram
\[
\begin{tikzcd}
    \mathsf{CAlg}_X(\D_X)\arrow[d] \arrow[r,"(\ref{eqn:Cohomological truncation})"]& \mathsf{CAlg}_X^{-n}(\D_X)\arrow[d]
    \\
    \mathsf{Mod}(\D_X)\arrow[r] & \mathsf{Mod}^{-n}(\D_X)
\end{tikzcd}
\]
whose bottom functor is a similar truncation.
\begin{remark}
\label{remark: k-truncated}
    There are other useful notions e.g. of $k$-truncated $\D$-prestacks $$\PS(\D_X)_{\leq k}:=\mathsf{Fun}\big(\mathsf{Aff}(\D_X)^{op},\mathsf{Spc}_{\leq k}\big),$$ where $\mathsf{Spc}_{\leq k}\subset \mathsf{Spc}$ is a full subcategory of $k$-truncated spaces i.e. $S$ such that each connected component $S'$ of $S$ satisfies $\pi_i(S')=0,i>k,$ but we do not make explicit use of them here. 
\end{remark}
The functor (\ref{eqn:Restriction on affines functor}) extends to
$res_{n}:\PS(\D_X)\rightarrow \PS^{ n}(\D_X),\hspace{1mm} \mathcal{X}\mapsto \mathcal{X}^{\leq n},$
and such a restriction admits a fully faithful left adjoint given by left Kan extension along (\ref{eqn:Embedding for D affines}):
\begin{equation}
    \label{eqn: LeftKan for n-coconn D Pre stacks}
Kan_{i_{\leq n}}^{L}:\PS^{ n}(\D_X)\rightarrow \PS(\D_X).
\end{equation}
A $\D_X$-pre stack $\mathcal{X}$ is said to be \emph{$n$-co-connective} if and only if it lies in the essential image of (\ref{eqn: LeftKan for n-coconn D Pre stacks}). We may unambiguously write
$\PS^n(\D_X):=\mathsf{Fun}\big(\mathsf{dAff}^n(\D_X)^{op},\mathsf{Spc}\big).$

By Remark \ref{remark: k-truncated}, one may define subcategories of $n$-coconnective $k$-truncated prestacks for various $n,k:$ $\mathcal{X}\in \PS^n(\D_X)$ is $k$-truncated if the corresponding $\infty$-hom groupid it determines is an $(n+k)$-groupoid. In particular, for $n=0,k=0$ we recover pre-sheaves of sets on classical affine $\D$-schemes i.e. ordinary classical $\D$-spaces.

Consider the category of $\D$-finite type schemes (\ref{D Finite Type}) which are additionally $n$-coconnective. There is a corresponding embedding, 
$$\mathrm{incl}_{f.t.}^n:\mathsf{dAff}_{\mathrm{f.t}}^{ n}(\D_X):=\mathsf{dAff}_{\mathrm{f.t}}^{<\!<}(\D_X)\cap \mathsf{dAff}^{n}(\D_X)\hookrightarrow \mathsf{dAff}^n(\D_X).$$

An $n$-coconnective $\D$-pre stack $\mathcal{X}$, is said to be \emph{locally of finite type} if it is a left Kan extension of its own restriction along the embedding $\mathrm{incl}_{\mathrm{f.t}}^n.$

\subsubsection{$\D$-finite presentation}
We now define the finite-presentation condition for morphisms of $\D$-algebras. For applications, we apply this to objects $\mathcal{B}$ of the slice $\mathsf{CAlg}_{X}(\D_X)_{\mathcal{O}_{J^{\infty}(E)}/},$ for some infinite-jet $\D_X$-algebra.
\begin{definition}
    \label{DAfpDefinition}
Let $n\geq 0$ and consider a morphism $\mathcal{A}^{\bullet}\rightarrow \mathcal{B}^{\bullet}$ in $\mathsf{CAlg}_X(\D_X)^{\leq -n}.$ We say that $\mathcal{B}^{\bullet}$ is $\D_X$ \emph{finitely presented} over $\mathcal{A}$ (as an $\mathcal{A}^{\bullet}-\mathcal{D}_X$-algebra via $f$) if it is a compact object in $$\mathsf{CAlg}\big(\mathsf{Mod}_{\D_X}(\mathcal{A})^{\leq -n}\big)\simeq \mathsf{CAlg}_X(\D_X)_{\mathcal{A}/}^{\leq -n}.$$
It is $\D_X$-\emph{almost finitely presented} if for every $n\in \mathbb{Z},$ the induced map 
$$\tau^{\leq -n}(f):\tau^{\leq -n}(\mathcal{A})\rightarrow \tau^{\leq -n}(\mathcal{B}),$$
presents $\tau^{\leq -n}(\mathcal{B})$ as a finitely presented $\tau^{\leq -n}(\mathcal{A})$-algebra.
Similarly, $\mathcal{A}\in \mathsf{CAlg}_X(\D_X)$ is \emph{$\D_X$-almost finitely presented} if for every $n$ the object $\tau^{\leq -n}(\mathcal{A})$ is compact in $\mathsf{CAlg}_X(\D_X)^{\leq -n}.$
Further, $\mathcal{A}$ is \emph{bounded} if it is eventually coconnective.
\end{definition}

Thus, a morphism of
 $\D_X$-stacks $f:\mathcal{X}\rightarrow \EQ$ is $\D$-finitely presented if for $\mathsf{Spec}_{\D_X}(\mathcal{A})$ and a map $\gamma:\mathsf{Spec}_{\D_X}(\mathcal{A})\rightarrow \EQ$, together with a filtered system $\{\mathcal{B}_i\}_{i\in I}$ of derived $\D$-algebras under $\mathcal{A}$, certain mapping spaces commute with homotopy colimits i.e.
$$\mathrm{Maps}_{\mathsf{dStk}(\D_X)/\mathsf{Spec}_{\D}(\mathcal{A})}\big(\Spec_{\D}(\underset{i\in I}{\mathrm{colim}}\mathcal{B}_i),\mathcal{X}\times_{\EQ}^h\mathsf{Spec}_{\D}(\mathcal{A})\big),$$
and
$$\underset{i\in I}{\mathrm{colim}}\hspace{1mm}\mathrm{Maps}_{\mathsf{dStk}(\D_X)/\mathsf{Spec}_{\D}(\mathcal{A})}\big(\mathsf{Spec}_{\D}(\mathcal{B}_i),\mathcal{X}\times_{\EQ}^h\mathsf{Spec}_{\D}(\mathcal{A})\big),$$
are naturally equivalent.

In turn we have a full-subcategory spanned by such objects
\begin{equation}
\label{eqn: Finite type n coconnective D pre stacks}
\PS^{ n}_{lft}(\D_X):=\mathsf{Fun}\big(\mathsf{dAff}_{\mathrm{f.t}}^{ n}(\D_X)^{op},\mathsf{Spc}\big)\hookrightarrow \PS^{ n}(\D_X).
\end{equation}
Left Kan extension along $\mathrm{incl}_{\mathrm{f.t}}^n$, with the appropriate notion of restriction provides a mutually adjoint pair of $\infty$-functors
\begin{equation}
    \label{eqn:Lft extension and restriction adjunction D Pre stacks}
    Kan_{\mathrm{incl}_{\mathrm{f.t}}^n}^{L}\colon \PS_{\mathrm{lft}}^{n}(\D_X)\rightleftarrows \PS^{n}(\D_X)\colon \mathrm{res}.
    \end{equation}
It follows from properties of being a left-adjoint, that for each $n$, the restriction
\eqref{eqn:Lft extension and restriction adjunction D Pre stacks}
commutes with colimits and with finite limits. Moreover, if $\mathcal{X}\in\mathsf{dAff}^{ n}(\D_X),$ then it is contained in $\mathsf{dAff}_{\mathrm{f.t}}^{ n}(\D_X)$ if and only if the derived $\D_X$-pre stack it represents is in $\PS_{\mathrm{lft}}^{ n}(\D_X).$

\section{Derived non-linear PDEs via de Rham spaces}
\label{sec: Derived dR-NLPDES}
Let us now elaborate the relative $X_{\DR}$-space point of view, and how the geometry of non-linear \textsc{pde} fits into this context.

For every $X\in \PS,$ there is a canonical morphism $q_{\DR}:X\rightarrow X_{\DR}$ inducing a functor on slice $\infty$-categories
\begin{equation}
    \label{eqn: PreStk dR Pushforward}
q_{\DR,*}:\PS_{/X}\rightarrow \PS_{/X_{\DR}},
\end{equation}
admitting an adjoint functor given by
\begin{equation}
    \label{eqn: PreStk dR Pullback}
q_{\DR}^*:\PS_{/X_{\DR}}\rightarrow \PS
_{/X}.
\end{equation}
If $f:X\rightarrow Y$ is a morphism of prestacks there is a corresponding morphism $f_{\DR}:X_{\DR}\rightarrow Y_{\DR}$ and an induced functor $f_{\DR}^*$ of prestacks, of which equation (\ref{eqn: PreStk dR Pushforward}) is a special case, over $Y_{\DR}$ to prestacks over $X_{\DR}$ given by precomposition with $f.$

\begin{observation}
\label{obs: Artin}
Let $X$ be a derived Artin stack. Then $q_{\DR}$ is proper so $q_{\DR,*}$ is left adjoint to the symmetric monoidal functor $q_{\DR}^!:\IC(X_{\DR})\rightarrow \IC(X).$ One has
that $q_{\DR,*}^{\IC}$ is op-lax symmetric monoidal and in particular, $q_{\DR,*}\omega_X\in \mathsf{CoCommcoAlg}\big(\IC(X_{\DR})\big),$ is a co-commutative coalgebra. 
\end{observation}
If $Z$ is a $D_X$-stack, $T^*Z$ exists in pro-$\D$-modules i.e. $x:\mathsf{Spec}_{\D_X}(\A
)\to Z$ gives a pro-object $T_x^*Z$ in $\A$-modules in $\D$-complexes.
Similarly, $Z'\to X_{\DR}$ gives $T^*(Z'/X_{\DR})\in \mathrm{Pro}\big(\mathsf{QCoh}(X_{\DR})\big)$ and $x
':S\to Z'$ defines a fiber $T^*_{x'}(Z'/X_{\DR})\in \mathrm{Pro}(\mathsf{QCoh}(S)),$ with a compatible crystalline structure over $X_{\DR}.$
Somehwat informally, let 
$$\Phi:\mathsf{QCoh}(X_{\DR})\xrightarrow{\simeq}\\D_X-\mathsf{Mod},$$
denote the equivalence $\infty$-categories \cite{GR14}. Then, 
$$\mathrm{Pro}(\Phi):\mathrm{Pro}(\mathsf{QCoh}(X_{\DR}))\simeq\mathrm{Pro}(\\D_X-\mathsf{Mod}),$$
is an equivalence as well. This induces an equivalence of $\A$-module categories,
$$\A-\mathrm{Mod}\big(\mathsf{QCoh}(X_{\DR})\big)\simeq\Phi(\A)-\mathrm{Mod}(\D_X-\mathsf{Mod}),$$
and subsequently, 
$$\mathrm{Pro}\big(\A-\mathrm{Mod}\big(\mathsf{QCoh}(X_{\DR})\big)\big)\simeq\mathrm{Pro}\big(\Phi(\A)-\mathrm{Mod}(\D_X-\mathsf{Mod})\big).$$

Nilcompleteness is preserved under this equivalence. The following appears as \cite[Prop.~4.9]{Kry2026}, which we recall.
\begin{proposition}
\label{NilCompleteness}
Let $Z\in \mathrm{PStk}_{/X_{\DR}}^{\mathrm{\mathrm{laft-def}}}.$ Then its associated $\D_X$-prestack $\EuScript{F}_{\EQ}$ is nilcomplete i.e. for any connective $A\in \mathsf{CAlg}(D_X)^{\leq 0},$ there is a canonical equivalence 
$$\EuScript{F}_{\EQ}(A)\rightarrow \varprojlim_n \EuScript{F}_{\EQ}(\tau_{\D}^{\leq n}A).$$
\end{proposition}
\begin{proof}
Let $A$ be such that $T=\mathsf{Spec}_{X_{\DR}}A\to X_{\DR}.$ We must show there are equivalences,
$$\Map_{/X_{\DR}}(T,Z)\simeq \varprojlim_n\Map_{/X_{\DR}}(\mathsf{Spec}_{X_{\DR}}\tau_{\D}^{\leq n}A,Z).$$ 
Since $\tau^{\leq n}\mathsf{Spec}_{X_{\DR}}(A)\simeq \mathsf{Spec}_{X_{\DR}}(\tau_{\D}^{\leq n}A),$ by hypothesis on $Z$, the map 
    $$\Map_{/X_{\DR}}(T,Z)\to \varprojlim_n\Map_{/X_{\DR}}(\tau^{\leq n}T,Z)$$ is an equivalence. Thus, from the definition of $\EuScript{F}_{\EQ}(A)$ we have
    \begin{eqnarray*}
    \Map_{/X_{\DR}}(\mathsf{Spec}_{X_{\DR}}A,Z)
        &\simeq&\varprojlim_n\Map_{/X_{\DR}}(\tau^{\leq n}\mathsf{Spec}_{X_{\DR}}A,Z)
        \\
        &\simeq& \varprojlim_n\Map_{/X_{\DR}}\big(\mathsf{Spec}_{X_{\DR}}(\tau_{\D}^{\leq n}A),Z\big)
        \\
        &\simeq& \varprojlim_n \EuScript{F}_{\EQ}(\tau_{\D}^{\leq n}A).
    \end{eqnarray*}
    Thus $\EuScript{F}_{\EQ}$ is nilcomplete in the $\D$-sense.
\end{proof}

Under the equivalence between $Z\to \EuScript{F}_{\EQ}$, deformation theory is also preserved. To this end, recall the pointwise almost affine deformation complex (\ref{eqn: Almostaffinecotangent}).

\begin{proposition}
Consider $Z\in \mathrm{PStk}_{X_{\DR}}^{\mathrm{\mathrm{laft-def}}}.$ Then, $\mathbb{L}_{\EuScript{F}_{\EQ}}$ exists. Furthermore, for each $\varphi:\mathsf{Spec}_{X_{\DR}}(A)\to Z$, it is the cofiber 
$$\mathbb{L}_{\EuScript{F}_{\EQ},\varphi}\to \mathbb{L}_{\mathsf{Spec}_{X_{\DR}}(A)/X_{\DR}}\to \mathbb{L}_{\mathsf{Spec}_{X_{\DR}}A/Z}.$$
\end{proposition}
\begin{proof}
By hypothesis $\mathbb{L}_{\EQ/X_{\DR}}\in \mathsf{QCoh}(\EQ)$ exists. Then, for each $A\in \mathsf{QCAlg}(X_{\DR})$ with point $\varphi:\mathsf{Spec}_{X_{\DR}}A\to Z,$ the deformation theory of $\EuScript{F}_{\EQ}$ at $\varphi$ is controlled by the space of lifts of $\varphi$ to square-zero extension of $A\oplus M$ for some connective module $M$. i.e. 
$$Lifts_{\varphi}(A\oplus M)\simeq\{\varphi':\mathsf{Spec}_{X_{\DR}}(A\oplus M)\to Z,\text{ over }X_{\DR} \text{ extending }\varphi\}.$$
This space of lifts is controlled by $\varphi^*\mathbb{L}_{\EQ/X_{\DR}}$ i.e. $\mathrm{Maps}_{A-Mod}(\varphi^*\mathbb{L}_{\EQ/X_{\DR}},M).$ Then, the cotangent space of the $\D_X$-prestack of points of $\EuScript{F}_{\EQ}$ at $\varphi$ is the $A$-module controlling deformations of $\varphi:\mathsf{Spec}_{X_{\DR}}(A)\to Z$, denoted $\mathbb{L}_{\EuScript{F}_{\EQ},\varphi}.$ Thus, there is an equivalence 
$$\mathbb{L}_{\EuScript{F}_{\EQ},\varphi}\simeq \varphi^*\mathbb{L}_{\EQ/X_{\DR}}\in \mathrm{Mod}_{A}\big(\mathsf{QCoh}(X_{\DR})\big).$$
Applying Proposition \ref{Relative D-cotangent Complex} to the morphisms $\mathsf{Spec}_{X_{\DR}}\xrightarrow{\varphi}\EQ\xrightarrow{\pi}X_{\DR},$ we have
$$\varphi^*(\mathbb{L}_{\EQ/X_{\DR}})\to \mathbb{L}_{\mathsf{Spec}_{X_{\DR}}A/X_{\DR}}\to \mathbb{L}_{\mathsf{Spec}_{X_{\DR}}A/Z}\xrightarrow{+1},$$
which by the above, presents the deformation complex as a cofiber at $\varphi$,
$$\mathbb{L}_{\EuScript{F}_{\EQ},\varphi}\to \mathbb{L}_{\mathsf{Spec}_{X_{\DR}}(A)/X_{\DR}}\to \mathbb{L}_{\mathsf{Spec}_{X_{\DR}}A/Z}.$$

\end{proof}
Note if $\EQ\to X_{\DR}$ is relatively affine, so $Z\simeq\mathsf{Spec}_{X_{\DR}}(\A)$ for some $\A\in \mathsf{QCAlg}(X_{\DR}),$ then 
$$\mathsf{QCoh}(\mathsf{Spec}_{X_{\DR}}(\A))\simeq \mathsf{Mod}_{\A}\big(\mathsf{QCoh}(X_{\DR})\big)\simeq \A-\mathsf{Mod}\big(\D_X-\mathsf{Mod}(X)^{\leq 0}).$$
If $f:Z=\mathsf{Spec}_{X_{\DR}}(\A)\to Z'=\mathsf{Spec}_{X_{\DR}}(\A'),$ or equivalently, $f^{\sharp}:\A'\to \A'$ is a morphism in $\mathsf{QCAlg}(X_{\DR}),$ there is an induced functor 
$$f^*:\mathsf{Mod}_{\A'}\big(\mathsf{QCoh}(X_{\DR})\big)\to \mathsf{Mod}_{\A'}\big(\mathsf{QCoh}(X_{\DR})\big).$$
For not necessarily relatively $X_{\DR}$-affines, the relative-spec adjunction gives an equivalence 
$$\mathsf{QCoh}(\EQ)\simeq\underset{\mathsf{Spec}_{X_{\DR}}(\A)\to Z}{\mathrm{lim}}\mathsf{QCoh}\big(\mathsf{Spec}_{X_{\DR}}(\A)\big),$$
via the $*$-pullback, in the $\infty$-category of presentable stable $\infty$-categories with colimit preserving functors.

It is convenient to fix some notation.
\begin{remark}[Notation]
The underlying affine scheme associated to a given affine $\D$-scheme is denoted $\mathsf{Spec}_X(For^{\ell}A),$ and $\mathsf{Spec}_{X_{\DR}}(A)$ is used for the $\D$-perspective.
\end{remark}

\begin{proposition}
\label{prop: Co-conntruncations}
    Let $Z\to X$ be affine over $X.$ Then, at any level of coconnective truncation, $Z$ is a finite product of affine $D$-schemes which are possibly infinite products of affine $D$-schemes 
    $$U_i\simeq\mathsf{Spec}_{X_{\DR}}(A_i),A_i\simeq \mathrm{Sym}^{\otimes^!}(\mathrm{ind}^{\ell}E_i),$$
    for a vector bundle $E_i$ on $X.$
\end{proposition}
\begin{proof}
    Since $(\mathsf{QCoh}(X_{\DR}),\otimes)$ is locally presentable stable and compactly generated $\infty$-category, generated by $\mathrm{ind}^{\ell}E,$ for $E\in \mathrm{Coh}(X),$ thus any $A\in \mathsf{QCAlg}(X_{\DR})$ admits a cellular presentation as a filtered colimit of free commutative algebras generated by objects $\mathrm{ind}^{\ell}E.$ Then $\tau^{\geq -n}A$ is obtained by attaching finitely many cells via iterated homotopy pushouts, built on $\mathrm{ind}^{\ell}E_i[-k_i]$ for $0<k<n$. In particular, 
    $$\tau^{\geq -n}A\simeq \pi_0(A)\otimes Sym\big(\bigoplus_i \mathrm{ind}^{\ell}E_i[-k_i]\big),$$
    for vector bundles or coherent sheaves $E_i\in \mathrm{Coh}(X).$
    Since coproducts correspond to tensor products (equiv. products are fiber products), a finite direct sum of $E_i$'s induces a finite tensor product of commutative algebras, thus a finite product of affine $D$-schemes of the form $\mathsf{Spec}_{X_{\DR}}\big(\mathrm{Sym}(\mathrm{ind}^{\ell}E_i)\big).$
\end{proof}

We mainly focus on the jet-construction applied to classical objects, but note the following general result.
\begin{proposition}
\label{prop: AffXsoAffXDR}
    Suppose $\EQ_X\to X$ is affine over $X$. Then its derived Weil restriction $q_{\DR*}(Z_X)$ is affine over $X_{\DR}.$
    \end{proposition}
    \begin{proof}
By assumption, let $R_0\in \mathsf{QCAlg}(X)^{\leq 0},$ be a connective quasi-coherent commutative algebra on $X,$ for which $Z_X\simeq \mathsf{Spec}_X(R_0).$ Then, 
consider $For^{\ell}:\mathsf{QCoh}(X_{\DR})^{\leq 0}\to \mathsf{QCoh}(X)^{\leq 0},$ and its induced functor on commutative algebra objects,
$$\mathsf{For}_{\mathrm{Comm}}^{\ell}:\mathsf{QCAlg}(X_{\DR})^{\leq 0}\to \mathsf{QCAlg}(X)^{\leq 0}.$$
Recall that $\mathsf{QCAlg}(X_{\DR})^{\leq 0}\simeq \mathrm{CAlg}\big(D-\mathrm{mod}(X))$ and similarly $\mathsf{QCAlg}(X)^{\leq 0}\simeq \mathrm{CAlg}(\mathsf{QCoh}(X)^{\leq 0}).$ Then, since for any $t:T\to X_{\DR},$ we have by the universal property of $q_{\DR*},$
\begin{eqnarray*}
\mathrm{Maps}_{\mathsf{PStk}_{X_{\DR}}}(T\times X_{\DR},q_{\DR*}\mathsf{Spec}_{X}(R_0)\big)&\simeq& \mathrm{Maps}((\mathsf{For}_{\mathrm{Comm}}^{\ell})^L(R_0),t_*\mathcal{O}_T)
\\ 
&\simeq& \mathrm{Maps}(T\times X_{\DR},\mathsf{Spec}_{X_{\DR}}(A_0)\big),
\end{eqnarray*}
for $A_0\simeq (\mathsf{For}_{\mathrm{Comm}}^{\ell})^{L}(R_0),$ where we have the left-adjoint to the forgetful functor above. Thus, 
$$q_{\DR*}(\mathsf{Spec}_X(R_0)\to X)\simeq (\mathsf{Spec}_{X_{\DR}}(A_0)\to X_{\DR}),$$
so that $q_{\DR*}\EQ_X$ is affine over $X_{\DR}.$
    \end{proof}
As an example, if $R_0=\mathrm{Sym}(E)$ for some $E\in \mathsf{QCoh}(X)$, then using the left-adjoint property one has $A_0\simeq \mathrm{Sym}^!(q_{\DR*}^{\mathsf{QCoh}}E),$ where $q_{\DR*}^{\mathsf{QCoh}}:\mathsf{QCoh}(X)\to \mathsf{QCoh}(X_{\DR}),$ is the functor of induction. 

\begin{proposition}
\label{D-PreStk equiv X dR PreStk}
Let $E\rightarrow X$ be an object of $\PS_{/X},$ and denote its restriction $q_{\DR,*}(E)\in \PS_{/X_{\DR}}$ by $\J(E).$
There is an identification of pre-stacks,
$\Sect_{X_{\DR}}(E/X)\simeq\J(E).$
\end{proposition}
\begin{proof}
    One can check this identification holds between objects of $\PS_{X_{\DR}}$, since for $\mathsf{Res}_{X/X_{\DR}},$ we have
for $U\in\PS_{/X_{\DR}}$ one has
$$\Map_{/X_{\DR}}\big(U,\mathsf{Res}_{X/X_{\DR}}E)\big)=\Map_{/X}(U\times_{X_{\DR}}X,E).$$
However, by the $(q_{\DR}^*,q_{\DR,*})$-adjunction,
\begin{equation}
    \label{eqn: dR-Jets as Mapping Space}
\Map_{/X}(U\times_{X_{\DR}}X,E)=\Map_{/X}(q_{\DR}^*U,E)\simeq \Map_{/X_{\DR}}(U,q_{\DR,*}E).
\end{equation}
This last expression is precisely the value of
$\Sect_{X_{\DR}}(E/X)(U\rightarrow X_{\DR}).$
One identifies the derived stack over $X_{\DR}$ 
    $$\Sect_{X_{\DR}}(E/X)\simeq \Map_{X_{\DR}}(X,E)\times_{\Map_{X_{\DR}}(X,X)}^h X_{\DR},$$
    and computes its $(t:T\rightarrow X_{\DR})$-points.
    It explicitly gives
    \begin{eqnarray*}
\Sect_{X_{\DR}}(E/X)(T)&\simeq&\Map_{/X_{\DR}}\big(T,\Map_{X_{\DR}}(X,E)\times_{\Map_{X_{\DR}}(X,X)}^h X_{\DR}\big)
\\
&\simeq& \Map_{/X_{\DR}}\big(T,
\Map_{X_{\DR}}(X,E)\big)\times_{\Map_{/X_{\DR}}(T,\Map_{X_{\DR}}(X,X))}^h \{t\}
\\
&\simeq& 
\Map_{/X_{\DR}}(T\times_{X_{\DR}}^h X,E)\times_{\Map_{/X_{\DR}}(T\times_{X_{\DR}}^h X,X)}^h\Map_{/X_{\DR}}(T,X_{\DR}).
    \end{eqnarray*}
\end{proof}
When $X$ is a complex analytic manifold, the pre-stack $\J(E)$ is called the $\D_X$-jet prestack of $E.$

\begin{proposition}
\label{prop: Def but not laft proposition}
For an object $E\in \PS_{/X}^{\mathrm{laft}},$ the corresponding $\D_X$-jet prestack $\J(E)$ is equivalent to the relative prestack
 $\J(\mathcal{E})^{\mathrm{Crys}}\rightarrow X$ with a flat connection along $X$ (i.e. crystal of spaces). If $E$ is laft and admits deformation theory, then $\J(E)$ also admits deformation theory.
\end{proposition}
The identification comes via the pull-back diagram in prestacks:
\begin{equation}
\label{eqn: Jet pb}
\begin{tikzcd}
\JetX(E)\simeq X \times_{X_{\DR}} q_{\DR,*}(E)\arrow[d] \arrow[r] & \J(E) \arrow[d]
\\
X\arrow[r,"q_{\DR}"] & X_{\DR}
\end{tikzcd}
\end{equation}

\begin{warning}
\label{warn: Jets not laft}
In Proposition \ref{prop: Def but not laft proposition}, even if $E$ is laft, in general $\J(E)$ will \emph{not} be. Thus, additional finiteness requirements later will be imposed. See also Corollary \ref{cor: Laft to DAfp}.
\end{warning}

As a suitable adjoint, $\J(-)$ has several desirable properties, which for our purposes are indicated in the following particular case.
\begin{proposition}
    Functor \emph{(\ref{eqn: PreStk dR Pushforward})} preserves fibre products.
\end{proposition}

There is a natural $\infty$-comonad corresponding to the adjunction $(q_{\DR,*},q_{\DR}^*),$ i.e. the composition $q_{\DR}^*\big(\J(E)\big)= q_{\DR}^*\circ q_{\DR,*}(E)\in \PS_{/X},$ which defines
\begin{equation}
    \label{eqn: Jet functor}
\JetX(-):\PS_{/X}\rightarrow \PS_{/X}.
\end{equation}
\begin{definition}
\label{DerivedInfiniteJets}
Let $E\in \PS_X$ be an $X$-prestack. The \emph{infinite jet prestack} of $E$ is the image $\JetX(E)$ by the functor (\ref{eqn: Jet functor}). 
    \end{definition}

\begin{remark}
\normalfont Functor (\ref{eqn: PreStk dR Pullback}) can be viewed more generally as an assignment
$\PS_{/X_{\DR}}\rightarrow \mathsf{Fun}(\Delta^1,\PS),$ mapping 
$(\EQ\rightarrow X_{\DR})$ to $(Z\times_{X_{\DR}}^hX\rightarrow Z).$ Viewed this way, it preserves étale morphisms (descends to pre-stacks statisfying étale descent i.e. derived stacks) and preserves fiber products.
\end{remark}
The next result summarizes some results\footnote{Obtained in the first authors PhD thesis \cite{Kry} but not announced elsewhere.} concerning the representability of both infinite jet and finite jet functors in derived algebraic geometry.
\begin{proposition}
\label{prop: Geometry of Derived Jets1}
Consider $\mathrm{Jets}^{\infty}(-):\mathsf{PStk}_{/X}\rightarrow \mathsf{PStk}_{/X},$
 and set $\mathsf{Jets}_X^k$ to be the finite order truncations $p_{k-dR}^*p_{k-dR*}$. Let $X=\mathsf{Spec}_k(A)$ be an affine derived scheme.
 \begin{itemize}
     \item[(i)]
 For every $k\in \mathbb{Z}_{\geq 0},$ the functor $\underline{\mathsf{Jets}}_X^{k}:\mathsf{dAff}^{op}\rightarrow \mathsf{Spc},$ is corepresented by an object $L\mathcal{J}^k(A)$.
 \item[(ii)] If $X$ is a smooth classical $k$-scheme of finite type, so is $L\mathcal{J}^k(A),$ and $\pi_0\big(L\mathcal{J}^k(A)\big)\simeq \mathcal{J}^k\big(\pi_0(A)\big).$
 \item[(iii)] If $\{U_i\rightarrow X\}_{i}$ is a Zariski open cover of a derived scheme $X,$ then $L\mathcal{J}_{U_i}^k$ glue along these covers.
 \end{itemize}
\end{proposition}

\begin{proof}
    We need to show that 
$\underline{\mathsf{Jets}}_A^k\simeq h_{\mathbb{L}\mathcal{J}^k(A)}\simeq \mathrm{Maps}_{\mathsf{CAlg}_k}(\mathbb{L}\mathcal{J}^k(A),-).$
This is immediate from evaluation on $B$-points and using the fact that we have an adjunction
$\pi_0\simeq H^0:\mathsf{CAlg}_k^{\leq 0}\rightleftarrows CAlg_k:incl$
Namely,
\begin{eqnarray*}
    \underline{\mathsf{Jets}}_A^k(B)&\simeq&\mathrm{Maps}_{\mathsf{CAlg}_k}(A,B\otimes_k^{\mathbb{L}}k[\epsilon]/\epsilon^{k+1})\simeq \mathrm{Maps}_{\mathsf{CAlg}_k}(H^0(A),B\otimes_k^{\mathbb{L}}k[\epsilon]/\epsilon^{k+1}).
\end{eqnarray*}
Thus $\underline{\mathsf{Jets}}_A^k(B)\simeq  \mathrm{Hom}_{CAlg_k}(\mathcal{J}^k\big(H^0(A)\big),B).$

Similarly, there are equivalences
$$h_{\mathbb{L}\mathcal{J}^k(A)}(B)\simeq \mathrm{Maps}_{\mathsf{CAlg}_k}\big(\mathbb{L}\mathcal{J}^k(A),B\big)\simeq \mathrm{Hom}_{CAlg_k}\big(H^0\mathbb{L}\mathcal{J}^k(A),B\big)\simeq  \mathrm{Hom}_{CAlg_k}\big(\mathcal{J}^k\big(H^0(A)\big),B\big).$$
Now suppose that $X$ is not necessarily of finite type e.g. of the form
$A\simeq \underset{i\in I}{\mathrm{colim}} A_i$ for finitely presented $A_i,$ as a sifted colimit completion.
\begin{equation*}
\underline{\mathsf{Jets}}_A^k(B)\simeq \mathrm{Maps}_{\mathsf{CAlg}_k}\big(\underset{i\in I}{\mathrm{colim}} A_i,B\otimes_k^{\mathbb{L}}k[\epsilon]/\epsilon^{k+1}\big)
\simeq \underset{i\in I}{\mathrm{lim}}\mathrm{Maps}_{\mathsf{CAlg}_k}\big(A_i,B\otimes_k^{\mathbb{L}}k[\epsilon]/\epsilon^{k+1}\big)
\simeq \underset{i\in I}{\mathrm{lim}} \underline{\mathsf{Jets}}_{A_i}^k(B)
\simeq \underset{i\in I}{\mathrm{lim}}\hspace{.3mm} h_{\mathbb{L}\mathcal{J}^k(A_i)}(B).
\end{equation*}
Now, it follows using the first part of the claim, that the limit is equivalent to 
\begin{equation*}\underset{i\in I}{\mathrm{lim}}\mathrm{Maps}_{\mathsf{CAlg}_k}\big(\mathbb{L}\mathcal{J}^k(A_i),B\big)
\simeq \mathrm{Maps}_{\mathsf{CAlg}_k}\big(\underset{i\in I}{\mathrm{colim}} \mathbb{L}\mathcal{J}^k(A_i),B\big)
\simeq\mathsf{Maps}\big(\mathbb{L}\mathcal{J}^k\big(\underset{i\in I}{\mathrm{colim}} A_i\big),B\big)
\simeq \mathsf{Maps}\big(\mathbb{L}\mathcal{J}^k(A),B)
\end{equation*}
which by the Yoneda embedding is given by $h_{\mathbb{L}\mathcal{J}^k(A)}(B),$ since the jet functor is a left-adjoint therefore commutes with filtered colimits and the mapping spaces commute with taking homotopy limits.
Now, consider a Zariski open covering $\{f_i:U_i\rightarrow X\}$ and assume that $X$ is a derived affine scheme $\mathsf{Spec}_k(A).$
We will show the result by proving that for each morphism  $f_i:\mathsf{Spec}(A_i)\rightarrow \mathsf{Spec}(A)$ in the covering, since it is a $D^-$-étale morphism of affine derived scheme we have for every $k\geq 0$ a homotopy equivalence $\mathbb{L}\mathcal{J}_{\mathsf{Spec}(A_i)}^k\simeq \mathbb{L}\mathcal{J}_{\mathsf{Spec}(A)}^k\times_{\mathsf{Spec}(B)}^{h}\mathsf{Spec}(A_i).$ By abuse of notation let $f_i:A\rightarrow A_i$ denote the dual and since it is étale, we have that 
$f_i^*:\mathsf{Maps}(A_i,-)\rightarrow \mathsf{Maps}(A,-),$ is formally étale with the property that for all small extensions $\widetilde{B}\rightarrow B$ making the following
\[
\begin{tikzcd}
A_i\arrow[d,"f_i"] \arrow[r] & \widetilde{B}\arrow[d]
\\
A\arrow[ur,dashed] \arrow[r] & B
\end{tikzcd}
\]
commute,
there is an equivalence
$$\mathsf{Maps}(A_i,\widetilde{B})\rightarrow \mathsf{Maps}(A_i,B)\times_{\mathsf{Maps}(A,B)}^{\mathbb{L}}\mathsf{Maps}(A,\widetilde{B}).$$

We proceed by induction on $k$ and consider the maps 
$\gamma^{k,k-1}:k[\epsilon]/\epsilon^{k+1}\rightarrow k[\epsilon]/\epsilon^k.$

For the case $k=1,$ we have maps, for every derived $k$-algebra $B$,
$$\gamma^{1,0}(B):B\otimes_k^{\mathbb{L}}k[\epsilon]/\epsilon^2\rightarrow B\otimes_k^{\mathbb{L}}k[\epsilon]/\epsilon.$$
It is a small extension and thus we have an equivalence 
$$\mathsf{Maps}(A_i,B\otimes_k^{\mathbb{L}}k[\epsilon]/\epsilon^2)\rightarrow \mathsf{Maps}(A,B\otimes_k^{\mathbb{L}}k[\epsilon]/\epsilon^2)\times_{\mathsf{Maps}(A,B\otimes_k^{\mathbb{L}}k[\epsilon]/\epsilon^2)}^{\mathbb{L}}\mathsf{Maps}(A_i,B\otimes_k^{\mathbb{L}}k[\epsilon]/\epsilon).$$
Thus, 
$\mathbb{L}\mathcal{J}_{\mathsf{Spec}(A_i)}^1\simeq \mathbb{L}\mathcal{J}_{\mathsf{Spec}(A)}^1\times_{\mathbb{L}\mathcal{J}_{\mathsf{Spec}(A)}^0}^{\mathbb{L}}\mathbb{L}\mathcal{J}_{\mathsf{Spec}(A_i)}^0.$
Now, the induction step is obvious to see and one can show since the maps $B\otimes_k^{\mathbb{L}} k[\epsilon]/\epsilon^{k+1}\rightarrow B\otimes_k^{\mathbb{L}} k[\epsilon]/\epsilon^{k}$ are small extensions as well, there are equivalences
$$\mathsf{Maps}(A_i,B\otimes_k^{\mathbb{L}}k[\epsilon]/\epsilon^{k+1})\rightarrow \mathsf{Maps}(A,B\otimes_k^{\mathbb{L}}k[\epsilon]/\epsilon^{k+1})\times_{\mathsf{Maps}(A,B\otimes_k^{\mathbb{L}}k[\epsilon]/\epsilon{k})}^{\mathbb{L}}\mathsf{Maps}(A_i,B\otimes_k^{\mathbb{L}}k[\epsilon]/\epsilon^k).$$
Thus, we have homotopy equivalences
$$\mathbb{L}\mathcal{J}_{\mathsf{Spec}(A_i)}^k\rightarrow \mathbb{L}\mathcal{J}_{\mathsf{Spec}(A)}^k\times_{\mathbb{L}\mathcal{J}_{\mathsf{Spec}(A)}^{k-1}}^{\mathbb{L}}\mathbb{L}\mathcal{J}_{\mathsf{Spec}(A_i)}^{k-1},$$ for all $k\geq 0.$
\end{proof}
Let $X$ be a smooth classical $\mathbb{C}$-scheme. By Proposition \ref{prop: Geometry of Derived Jets1} (ii),  $L\mathcal{J}^k(X)\simeq J^k(X)$ is smooth and classical. In particular,
\begin{equation}
\label{StandardJets}
L\mathcal{J}^k\big(\mathbb{C}[x_1,\ldots,x_d]\big)\simeq J^k\big(\mathbb{C}[x_1,\ldots,x_d]\big)=\mathbb{C}\big[x_i^{(\sigma)}||\sigma|\leq k\big].
\end{equation}
We call \eqref{StandardJets} the \emph{standard $k$-jet $\mathbb{C}$-scheme} and refer to its elements as
 differential polynomials.

\begin{proposition}
\label{prop: Geometry of Derived Jets2}
Consider Proposition \ref{prop: Geometry of Derived Jets1}. Let $X=\mathsf{Spec}_k(A)$ be an affine derived scheme. Then there is an equivalence of objects $\underset{k}{\mathrm{colim}}\hspace{1mm} \mathbb{L}\mathcal{J}_{\mathsf{Spec}(A)}^k\simeq \mathbb{L}\mathcal{J}_{\mathsf{Spec}(A)}^{\infty},$
    that represents $\underline{\mathsf{Jets}}_{\mathsf{Spec}_k(A)}^{\infty}:\mathsf{dAff}^{op}\rightarrow \mathsf{Spc}.$
\end{proposition}
\begin{proof}
For the case of a finitely presented affine derived $k$-scheme the situation is clear from the proofs above. Consider then the case that $A\simeq \underset{i\in I}{\mathrm{colim}} A_i.$ Then
\begin{equation*}
    \underset{k}{\mathrm{colim}}\hspace{.3mm} \mathbb{L}\mathcal{J}_{A}^k\simeq \underset{k}{\mathrm{colim}}\hspace{.3mm}\big(\underset{i\in I}{\mathrm{colim}}\hspace{.3mm}\mathbb{L}\mathcal{J}^k(A_i)\big)
    \simeq \underset{i\in I}{\mathrm{colim}}\hspace{.3mm}\big(\underset{k}{\mathrm{colim}}\hspace{.3mm}\mathcal{J}^k(A_i)\big)
    \simeq\underset{i\in I}{\mathrm{colim}}\hspace{.3mm}\mathcal{J}^{\infty}(A_i)
    \simeq \mathbb{L}\mathcal{J}^{\infty}(A).
\end{equation*}
Then, 
\begin{equation*}
\mathsf{Maps}\big(\mathbb{L}\mathcal{J}^{\infty}(A),B\big)\simeq \mathsf{Maps}\big(\underset{k}{\mathrm{colim}}\hspace{1mm} \mathbb{L}\mathcal{J}^k(A_i),B\big)
\simeq \underset{k}{\mathrm{lim}}\hspace{1mm}\mathsf{Maps}\big(\mathbb{L}\mathcal{J}^k(A_i),B\big)
\simeq \underset{k}{\mathrm{lim}}\hspace{1mm}\mathsf{Maps}(A,B\otimes_k^{\mathbb{L}}k[\epsilon]/\epsilon^{k+1}\big),
\end{equation*}
which is precisely $\mathsf{Maps}(A,B[\![\epsilon]\!])
\simeq \underline{\mathsf{Jets}}_{\mathsf{Spec}(A)}^{\infty}(B),$ as required.
\end{proof}
The (homotopy)colimit in Proposition \ref{prop: Geometry of Derived Jets2} is over the jet-projection given for each $k\leq \ell\leq \infty,$ by 
$L\mathcal{J}^{\ell}(\mathsf{X})\to L\mathcal{J}^k(\mathsf{X}).$ Dually, we obtain 
$\pi_{\ell,k}^*:L\mathcal{J}^k(A_{\bullet})\to L\mathcal{J}^{\ell}(A_{\bullet}).$ 
They induce morphism of derived derivations (\cite[Def.~1.4.14]{TV2}), 
$$\mathrm{Der}_{\mathbb{C}}\big(L\mathcal{J}^{\ell}(A_{\bullet}),-\big)\to \mathrm{Der}_{\mathbb{C}}\big(L\mathcal{J}^{k}(A_{\bullet}),-\big).$$
By Proposition \ref{Prop: RelCot} there are induced morphisms on representatives $\mathbb{L}_{L\mathcal{J}^k(A_{\bullet})/\mathbb{C}}\to\mathbb{L}_{L\mathcal{J}^{\ell}(A_{\bullet})/\mathbb{C}}.$
We have an induced map
$$\mathbb{L}_{L\mathcal{J}^k(A_{\bullet})/\mathbb{C}}\otimes_{L\mathcal{J}^k(A_{\bullet})}^L\mathbb{L}_{L\mathcal{J}^{\ell}(A_{\bullet})/\mathbb{C}}\rightarrow \mathbb{L}_{L\mathcal{J}^{\ell}(A_{\bullet})/\mathbb{C}}.$$
The following globalizes to not necessarilly affine derived jet spaces.

\begin{proposition}
\label{prop: Geometry of Derived Jets3}
Consider $X\in \mathsf{dSch}_k,$  not necessarily affine. Then for each $k\geq 0$ there is a derived scheme $\mathbf{L}\mathcal{J}_X^k$ representing the functor
    $\underline{\mathsf{Jets}}_X^k:\mathsf{dSch}_k\rightarrow \mathsf{Spc}.$
Construction of $\underline{\mathsf{Jets}}_{X}^{\infty}$ is compatible with Zariski localization and for every étale morphism of derived $k$-schemes $f:X\rightarrow Y,$ there is a natural homotopy equivalence 
    $\mathbf{L}\mathcal{J}_X^{\infty}\xrightarrow{\simeq}\mathbf{L}\mathcal{J}_Y^{\infty}\times_Y^{\mathbb{L}}X.$
    \end{proposition}
    \begin{proof}
    For a derived scheme $\mathcal{X}$, put
$L\mathcal{J}^k(\mathcal{O}_{\mathcal{X}}^{\bullet}):=\mathcal{O}_{L\mathcal{J}^k(\mathcal{X})}^{\bullet}.$
For any morphism $f:\mathcal{X}\to \mathcal{Y}$ there is an induced morphism $L\mathcal{J}^k(f):L\mathcal{J}^k(\mathcal{X})\to L\mathcal{J}^k(\mathcal{Y}),$ and thus
$$\mathbb{L}_{L\mathcal{J}^k(\mathcal{Y})/\mathbb{C}}\otimes_{\mathcal{O}_{L\mathcal{J}^k(\mathcal{Y})}}^L\mathcal{O}_{L\mathcal{J}^k(\mathcal{X})}\rightarrow \mathbb{L}_{L\mathcal{J}^k(\mathcal{X})/\mathbb{C}}\rightarrow \mathbb{L}_{L\mathcal{J}^k(\mathcal{X})/L\mathcal{J}^k(\mathcal{Y})}\xrightarrow{+1}.$$
Each derived $k$-scheme we have the jet-tower $\cdots\rightarrow\mathbb{L}\mathcal{J}^k(X)\rightarrow\cdots \rightarrow \mathbb{L}\mathcal{J}^0(X)\rightarrow 0,$
with structure maps $\pi^{k,\ell}:\mathbb{L}\mathcal{J}^k(X)\rightarrow \mathbb{L}\mathcal{J}^{\ell}(X)$ for $k\geq \ell.$ 
Note:
$$\mathsf{dAff}_{/\mathbb{L}\mathcal{J}^0(X)}\rightleftarrows \mathsf{QCAlg}^{\leq 0}(\mathcal{O}_{\mathbb{L}\mathcal{J}^0(X)}).$$
Then $\mathbb{L}\mathcal{J}^k(X)\simeq \Spec_{\mathbb{L}\mathcal{J}^0X}\big((\pi_0^{k})_*^{\mathsf{QCoh}}\mathcal{O}_{\mathbb{L}\mathcal{J}^k(X)}\big),$ and
$$\Spec_{\mathbb{L}\mathcal{J}^0X}\bigg(\underset{k}{\mathrm{colim}}\hspace{1mm}\big((\pi_0^{k})_*^{\mathsf{QCoh}}\mathcal{O}_{\mathbb{L}\mathcal{J}^k(X)}\big)\bigg)\simeq \underset{k}{\mathrm{lim}}\hspace{1mm} \Spec_{\mathbb{L}\mathcal{J}^0X}\big((\pi_0^{k})_*^{\mathsf{QCoh}}\mathcal{O}_{\mathbb{L}\mathcal{J}^k(X)}\big).$$
This homotopy limit, computed in $\mathsf{dSch}_k$ exists and thus our functor is indeed representable by a derived $k$-scheme. We have used the fact that the relative derived spectrum functor is a right-adjoint \cite[Prop. 2.5.12, Cons. 2.5.13]{Lurie2018}.

 Finally, note that
 \begin{eqnarray*}
\mathrm{Maps}_{\mathsf{dSch}_k}\big(\mathsf{Spec}_k(A),\mathbf{L}\mathcal{J}^{\infty}(X)\big)
&\simeq&
\mathrm{Maps}_{\mathsf{dSch}_k}\big(\mathsf{Spec}_k(A),\underset{k}{\mathrm{lim}}\hspace{1mm} \mathbf{L}\mathcal{J}^k(X)\big)
\\
&\simeq& \underset{k}{\mathrm{lim}}\hspace{1mm}\mathrm{Maps}_{\mathsf{dSch}_k}\big(\mathsf{Spec}_k(A),\mathbf{L}\mathcal{J}^k(X)\big)
\\
&\simeq& \underset{k}{\mathrm{lim}}\hspace{1mm}\mathrm{Maps}_{\mathsf{dSch}_k}\big(\mathsf{Spec}_k A[\epsilon]/\epsilon^{k+1},X\big).
 \end{eqnarray*}
When $X$ is a derived scheme so is 
 $\underset{k}{\mathrm{lim}}\hspace{1mm}\mathbf{L}\mathcal{J}^k(X),$ 
 and $\mathbf{L}\mathcal{J}^{\infty}(X)$, as a homotopy limit along affine morphisms $\pi^{k,k-1}$ of derived $k$-schemes.

\end{proof}
In addition to Propositions \ref{prop: Geometry of Derived Jets1},\ref{prop: Geometry of Derived Jets2},\ref{prop: Geometry of Derived Jets3}, we have the fundamental property that the $\infty$-jet comonad is a right-adjoint $\infty$-functor.
\begin{proposition}
\label{FormalDisk prop}
    There is an equivalence of endofunctors 
    $(q_{\DR})^*\circ(q_{\DR})_!(-)\simeq X\times_{X_{\DR}}^hX\times_X(-),$
    on $\PS_{/X},$ with right-adjoint
   \emph{(\ref{eqn: Jet functor})}.
\end{proposition}
\begin{proof}
For a relative prestack $\EuScript{F}\in \PS_{/X}$ it suffices to notice 
\begin{eqnarray*}
\Map_{/X}\big(\EuScript{F},\JetX(E)\big)&\simeq& \Map_{/X}(\EuScript{F},q_{\DR}^*q_{\DR,*}(E)\big)
\\
&=&\Map_{/X}\big(q_{\DR,!}\EuScript{F},q_{\DR,*}E\big)
\\  
&\simeq& \Map_{/X}(q_{\DR}^*q_{\DR,!}\EuScript{F},E)
\\
&\simeq& \Map_{/X}(X\times_{X_{\DR}}X\times_X\EuScript{F},E).
\end{eqnarray*}
\end{proof}
For any section $s:X\rightarrow E$ there is an analog of jet-prolongation $j_{\infty}(s)$, given as a section of $\JetX(E)$ defined by the following composition. Let $\mathbf{1}_X:X\rightarrow X$ be the tautological bundle over $X$ given by the initial object in $\PS_{/X}$ with corresponding jet prestack $\JetX(X)$. 

Then $j_{\infty}(s)$ is given as the composition
\begin{equation}
    \label{eqn: Jet prolongation}
j_{\infty}(s):X\xrightarrow{\simeq} \JetX(X)\xrightarrow{\JetX(s)}\JetX(E).
\end{equation}
The following result is a consequence of formal categorical properties.
\begin{proposition}
\label{ComonadicJet}
There is a natural transformation of endofunctors 
$q_{\DR}^*\circ q_{\DR,*}\simeq q_{\DR}^*\circ q_{\DR,*}\big(q_{\DR}^*\circ q_{\DR,*}\big).$
In particular, for any $E\rightarrow X\in \PS_{/X},$ we have a canonical morphism of relative prestacks
$\JetX(E)\rightarrow\JetX\big(\JetX(E)\big).$
\end{proposition}
The image of a morphism $\mathsf{F}:\JetX(E)\rightarrow \EuScript{F}$ under (\ref{eqn: Jet functor}),
$$\JetX\mathsf{F}:\JetX(\JetX(E))\rightarrow \JetX\EuScript{F}.$$
Post-composing with the co-algebra co-multiplication
\begin{equation}
    \label{eqn: Comultiplication}
\Delta_{E}:\JetX(E)\rightarrow \JetX(\JetX(E)),
\end{equation}
gives
\begin{equation}
    \label{eqn: Differential Morphism Prolongation}
    \mathsf{F}^{\infty}:=\JetX(\mathsf{F})\circ \Delta_{E}:\JetX(E)\rightarrow \JetX(\JetX(E))\rightarrow \JetX\EuScript{F}.
\end{equation}
Morphism (\ref{eqn: Differential Morphism Prolongation}) is called the \emph{comonadic extension} and can be defined for derived PDEs via pull-backs in $\PS_{/X}$:
\begin{equation}
    \label{eqn: Infinite prolongation}
    \begin{tikzcd}[row sep=large, column sep = large]
        & \EQ_X^{\infty}\arrow[d]\arrow[r,"i_{\infty}"] & \JetX(E)\arrow[d]
        \\
        \EQ_X\arrow[r,"\rho"] & \JetX \EQ\arrow[r,"\JetX(i_X)"] & \JetX\JetX(E)
    \end{tikzcd}
\end{equation}
where $\rho$ is the induced $\JetX(E)$-coalgebra map.

\subsubsection{Jet coalgebra functoriality}
\label{ssec: Jet coalgebra functoriality}
Consider $q^*q_*$ as a comonad acting on $\mathsf{QCoh}(X).$ Considering the adjoint functors, we have that $q_!q^!$ is the \emph{monad of differential operators}, adjoint to $q^*q_*.$ 
Barr-Beck-Lurie gives
$$R:\mathsf{C}\leftrightarrows \mathsf{Mod}_{q_{\DR}^*q_{\DR*}}(\mathsf{Shv}(X)):L,$$
such that $R$ satisfies 
$$\mathrm{Forget}_{q_{\DR}^*\circ q_{\DR,*}}\circ R(-)\simeq q_{\DR}^*(-).$$
We calculate $L$ on objects $E \in \mathsf{Mod}_{q^*q_*}(\mathsf{QCoh}(X)),$ as a colimit of a simplicial object $E\simeq \underset{[n]\in \Delta^{op}}{\mathrm{colim}}(q^*q_*)^nE.$ 

Since $L$ is a left-adjoint it commutes with colimits. Thus it suffices to know $L$ on free $q^*q_*$-modules. However, since $L\circ \mathrm{Free}_{q^*q_*}(-)\simeq q_{\DR,*}(-)$, one has 
$$LE\simeq  \underset{[n]\in \Delta^{op}}{\mathrm{colim}}q_{\DR,*}\big(q_{\DR}^*\circ q_{\DR,*}\big)^nE.$$
The following summarizes the main features of these functors we will use.

\begin{proposition}
\label{prop: CoalgebraBody}
   Let $X,Y$ be smooth projective $\mathbb{C}$-schemes and put $q_{\DR}^X:X\to X_{\DR},q_{\DR}^Y:Y\to Y_{\DR}.$ :
\begin{enumerate}
\item There is an equivalence $\mathsf{Mod}_{q_{\DR,!}q_{\DR^!}}(\mathsf{Shv}(X))\simeq \mathsf{CoMod}_{q_{\DR}^*q_{\DR,*}}(\mathsf{Shv}(X)).$
\item For any morphism  $f:X\to Y$ there is functor 
$$f_{\mathrm{coAlg}}^*(-):\mathsf{CoAlg}_{q_{\DR}^{Y,*}q_{\DR,*}^Y}\big(\mathsf{Shv}(Y)\big)\to \mathsf{CoAlg}_{q_{\DR}^{X,*}q_{\DR,*}^X}\big(\mathsf{Shv}(X)\big).$$
\end{enumerate}
\end{proposition}
\begin{proof}
For (1), the funtor from left to right sends $E$ to $q_{\DR}^*\underset{[n]\in \Delta^{op}}{\mathrm{colim}}q_{\DR,!}\big(q_{\DR}^!q_{\DR,!})^nE,$ and since $q_{\DR}^*$ commutes with coloimts the result follows. From right to left, $q_{\DR}^!\underset{[n]\in \Delta^{op}}{\mathrm{lim}}q_{\DR,*}(q_{\DR}^*q_{\DR,*})^nE$, then use $q_{\DR}^!$ commutes with limits. Now, use that $q_{\DR,!}\simeq q_{\DR,*}.$

For (2), note that $f:X\to Y$ induces a commutative diagram
\[
\begin{tikzcd}
X\arrow[d,"q_{\DR}^X"]\arrow[r,"f"] & Y\arrow[d,"q_{\DR}^Y"]
\\
X_{\DR}\arrow[r,"f_{\DR}"] & Y_{\DR},
\end{tikzcd}
\]
and the desired functor is readily checked to be 
$$q_{\DR}^{X,*}f_{\DR}^*\underset{[n]\in \Delta^{op}}{\mathrm{lim}}q_{\DR,*}^Y\big(q_{\DR}^{Y,*}q_{\DR,*}^Y)^n.$$
Indeed, using the commutation of the diagram, we get 
$$f^*q_{\DR}^{Y^*}\underset{[n]\in \Delta^{op}}{\mathrm{lim}}q_{\DR,*}^Y(q_{\DR}^{Y,*}q_{\DR,*}^Y)^n.$$

\end{proof}

\subsection{Prestacks of solutions and derived non-linear PDEs}
\label{ssec: Prestacks of Solutions and DNLPDES} 

Let $E\rightarrow X$ and $\JetX(E)$ be as in § \ref{sec: Derived dR-NLPDES}. A morphism of prestacks $s:X\rightarrow \JetX(E)$ is a flat section if it is obtained as a pull-back from a map $s':X_{\DR}\rightarrow \J(E)$ via (\ref{eqn: Jet pb}).
This is captured by the \emph{pre-stack of flat sections $X\rightarrow \JetX(E)$} given by
\begin{equation}
\label{eqn: Flat Sections}
\Sect^{\nabla}\big(X,\JetX(E)\big):=\underline{\mathsf{Maps}}_{X_{\DR}}\big(X_{\DR},\J(E)\big),
\end{equation}
understood as an assignment $\Sect^{\nabla}\big(X,\JetX(E)\big):\mathsf{dAff}\rightarrow \mathsf{Spc},$ whose value on test affine spaces $U$ is
\begin{equation*}
\Sect^{\nabla}\big(X,\JetX(E)\big)(U)=\underline{\mathsf{Maps}}_{X_{\DR}}\big(X_{\DR},\J(E)\big)(U)
=\Map_{/X_{\DR}}\big(U\times X_{\DR},\J(E)\big).
\end{equation*}
For each $U\in\mathsf{dAff}$ an element $s_U\in \Sect^{\nabla}\big(X,\JetX(E)\big)$ is understood as a $U$-parameterized solution.
\begin{remark}[Notation]
If $E\rightarrow X$ is a laft--def prestack over $X,$ we define the pre-stack of homotopy cofree solution sections to the homotopy cofree non-linear PDE $\J(E),$ formed by adjunction (\ref{eqn: PreStk dR Pushforward}). In terms of the stack of flat sections \eqref{eqn: Flat Sections}, we denote it by $\Sect^{\nabla-\text{cofree}}(X,E):=\Sect^{\nabla}(X,\JetX(E)\big),$ to emphasize it depends only on $X,E.$
\end{remark}
\begin{proposition}
 \label{Weil Proposition}
Let $E\in \PS_{/X}^{\mathrm{laft-def}}.$ The prestack of flat-cofree sections is determined by 
$$\mathsf{Maps}\big(U,\Sect^{\nabla-\mathrm{cofree}}(X,E)\big)=\mathsf{Maps}_{/X_{\DR}}\big(U\times_{pt}X_{\DR},q_{\DR,*}E\big),$$
for each affine test scheme $U$.
\end{proposition}
\begin{proof}
Take
$E\xrightarrow{\pi} X\xrightarrow{q_{\DR}}X_{\DR},$ and identify the $\J(E)$ with the derived moduli of flat sections of $\pi$ i.e. sections of $\pi$ over $X_{\DR}$,
$\Sect_{X_{\DR}}(E/X)\simeq \J(E),$
via Proposition \ref{D-PreStk equiv X dR PreStk}. Then, the identification follows from using the universal property of $\JetX(E).$
\end{proof}

\begin{proposition}
\label{Prop: Serves to identify horizontal sections}
The canonical map $q_{\DR}:X\rightarrow X_{\DR}$ induces a natural transformation via pullback,
$$\underline{\mathsf{Maps}}_{X_{\DR}}\big(X_{\DR},-\big)\circ q_{\DR,*}(-)\rightarrow \underline{\mathsf{Maps}}_{X}\big(X,-\big)\circ q_{\DR}^*\circ q_{\DR,*}(-),$$
of functors $\PS_{/X}\rightarrow \mathsf{Spc}.$
\end{proposition}
\begin{proof}
By universal constructions related to $\J(E)$ and to $\mathbb{R}\mathsf{Maps}$
one may show that 
$$\mathbb{R}\mathrm{Sol}(\J E)\hookrightarrow \Sect_{pt}(\mathrm{Jets}_X^{\infty}E/X),$$
is a naturally defined 
closed embedding of sheaves of spaces on $X.$
Indeed, on the one hand we see by their universal properties there is an induced natural transformation by pull-back, such that for each $T$-derived stack,
$$\Map_{/X_{\DR}}(T\times_{pt}^hX_{\DR}, q_{\DR,*}(-)\big)\Rightarrow \Map_{/X}\big(T\times X, q_{\DR}^*q_{\DR,*}(-)\big):\PS_{/X}\rightarrow \mathsf{Spc}.$$

\end{proof}
Proposition \ref{Prop: Serves to identify horizontal sections} identifies sections $s:X\rightarrow \JetX(E)$ which are horizontal for the canonical flat connection. 
The derived moduli space $\mathbb{R}\mathrm{Sect}$ of sections appearing above is given for $E\in \PS_{/X}$  
\begin{equation}
    \label{eqn: Derived Sections}
\Sect_(X,E):=\underline{\mathsf{Maps}}(X,E)\times_{\underline{\mathsf{Maps}}(X,X)} \big\{\mathbf{1}\big\}\simeq \underline{\mathsf{Maps}}_{/X}(\mathbf{1},\pi),
\end{equation}
taken in the slice site $\PS_{/X}.$

More generally given $E\rightarrow X\xrightarrow{f} S$ with $S,E$ (possible derived) Artin stacks then the derived space of sections $\Sect_S(E/X)$ is derived Weil restiction of $E$ along $f.$ It agrees with (\ref{eqn: Derived Sections}).

Moreover, its construction applies to $E\in\PS_{/X}$ for which there exists an epimorphism of prestacks $\EuScript{G}\rightarrow X$ and a pull-back diagram in $\PS:$
\[
\begin{tikzcd}
    \EuScript{G}\times\EuScript{F}\arrow[d]\arrow[r] & E\arrow[d,"\pi"]
    \\
    \EuScript{G}\arrow[r] & X
\end{tikzcd}
\]
for some prestack $\EuScript{F}.$ The moduli of \emph{relative} sections $\Sect_{/X}(\EuScript{G}\times\EuScript{F}),$ arising in this case will be used in §§ \ref{Parameterized Solution and Derived DO} below, to describe $\EuScript{F}$-parameterized solutions. 

Proposition \ref{Weil Proposition} shows
 $\JetX(E)$ has a defining universal property via mapping prestacks. This observation further
 defines two useful prestacks endowed with corresponding universal families.

\begin{construction}
\label{cons: Solutions}
For morphisms of derived prestacks $\EQ_1\rightarrow \EQ_2\rightarrow S$ we consider $\mathsf{Res}_{\EQ_2/S}(\EQ_1)\in \PS_{/S}.$ 
We are interested in two cases, treated uniformly in this setting as:
\begin{equation*}
\textbf{(a)} \text{ Solutions, } \EQ\xrightarrow{\pi}X_{\DR}\xrightarrow{a_{X_{\DR}}}pt\hspace{2mm}
\textbf{(b)} \text{ Solutions with coefficients, } \EQ_1\xrightarrow{F}\EQ_2\xrightarrow{a_{\EQ_2}}pt.
\end{equation*}
Indeed, in case (a), we obtain
$\mathbb{R}\mathrm{Sol}(\EQ):=\Sect_{pt}(\EQ/X_{\DR}):=\mathsf{Res}_{X_{\DR}/pt}(\EQ)=\mathsf{Maps}_{/X_{\DR}}(X_{\DR},\EQ),$
characterized by
$$\mathrm{Maps}_{\PS}\big(T,\mathbb{R}\mathrm{Sol}(\EQ)\big)\simeq \mathrm{Maps}_{\PS_{/X_{\DR}}}(T\times X_{\DR},\EQ\big),$$
with universal family
\begin{equation}
    \label{eqn: RSol Universal Family 1}
    \begin{tikzcd}
      & \EQ\arrow[d,"\pi"]
      \\
      X_{\DR}\times^h\mathbb{R}\mathrm{Sol}(\EQ)\arrow[ur,"ev"] \arrow[d,"q_{\mathbb{R}\mathrm{Sol}}"] \arrow[r] & X_{\DR}\arrow[d,"a_{X_{\DR}}"]
      \\
      \mathbb{R}\mathrm{Sol}(\EQ)\arrow[r] & Spec(\mathbb{C})
    \end{tikzcd}
\end{equation}
It is the parameterizing space of $T$-families of sections $s_T:T\times X_{\DR}\rightarrow T\times \EQ,$ with homotopies $p_T\circ s_T\simeq id_{T\times_{Spec(\mathbb{C})}^h X_{\DR}}.$

In case (b), the derived moduli of sections we get are
$$\mathbb{R}\mathrm{Sol}(\EQ_1,\EQ_2):=\Sect_{pt}(\EQ_1/\EQ_2):=\mathsf{Res}_{\EQ_2/pt}(\EQ_1)\simeq \mathsf{Maps}_{/X_{\DR}}(\EQ_2,\EQ_1),$$
with corresponding universal family
\begin{equation}
    \label{eqn: RSol Universal Family 2}
\begin{tikzcd}
      & \EQ_1\arrow[d,"f"]
      \\
      \EQ_2\times_{X_{\DR}}^h\mathbb{R}\mathrm{Sol}(\EQ_1,\EQ_2)\arrow[ur,"ev"] \arrow[d,"q_{\mathbb{R}\mathrm{Sol}}"] \arrow[r] &\EQ_2\arrow[d]
      \\
      \mathbb{R}\mathrm{Sol}(\EQ_1,\EQ_2)\arrow[r] & X_{\DR}
    \end{tikzcd}
\end{equation}
suppressing the maps to $a_{\EQ_2}:\EQ_2\rightarrow Spec(\mathbb{C}),$ and $a_{X_{\DR}}:X_{\DR}\rightarrow Spec(\mathbb{C}).$
\end{construction}

In the notation of Construction \ref{cons: Solutions}, the space of solution sections in (a) is recovered from that in (b) by the formula
$\mathbb{R}\mathrm{Sol}(\EQ)=\mathbb{R}\mathrm{Sol}(\EQ,X_{\DR}).$
Denote for each morphism $f:\EQ_1\rightarrow \EQ_2\in \PS_{/X_{\DR}}$, the induced map
$$f:\mathbb{R}\mathrm{Sol}(\EQ_1)\rightarrow \mathbb{R}\mathrm{Sol}(\EQ_2),$$
given by precomposition with $f.$
\begin{construction}
\label{cons: De Rham Hom Stacks}
    Fix $(\pi:\EQ\rightarrow X_{\DR})\in \PS_{X_{\DR}}.$ Consider the functor
    $$\mathbb{R}\mathrm{Sol}_{\DR}(\EQ):=\mathsf{Res}_{\EQ/X_{\DR}}(\EQ\times^h -):\PS_{X_{\DR}}\rightarrow \PS_{X_{\DR}}.$$
    Namely, when evaluated on $(\EuScript{Y}\to X_{\DR})$ it first views $\EQ\times^h\EuScript{Y}$ as an object of $\big(\PS_{Z}\big)_{/X_{\DR}},$ and then evaluates the Weil-restriction along $\pi$ i.e.
    $$(\EuScript{Y}\to X_{\DR})\longmapsto \mathbb{R}\mathrm{Sol}_{\DR}(\EQ)(Y):=\mathsf{Res}_{\EQ/X_{\DR}}(\EQ\times^h\EuScript{Y}).$$
\end{construction}
Construction \ref{cons: De Rham Hom Stacks} gives 
 \emph{internal mapping $\D$-prestack} between $\EQ$ and $\EuScript{Y}.$
Indeed,
\begin{equation*}
\mathrm{Maps}_{\mathsf{PStk}_{X_{\DR}}}\big(T,\mathsf{Res}_{\EQ/X_{\DR}}(\EQ\times^h\EuScript{Y})\big)\simeq \mathrm{Maps}_{\PS_{/Z}}\big(T\times_{X_{\DR}}^h \EQ,\EQ\times^h\EuScript{Y})
\simeq\mathrm{Maps}_{\PS_{X_{\DR}}}\big(T,\Map_{/X_{\DR}}(\EQ,\EuScript{Y})\big),
\end{equation*}
by the usual adjunction between $-\times^h-$ and internal mapping stacks.
The $U$-points of this mapping space are given by
\begin{equation*}
\mathbb{R}\mathrm{Sol}(\EQ)(Y):\mathsf{dAff}_{/X_{\DR}}^{afp}\rightarrow \mathsf{Spc},\hspace{3mm}
U \mapsto  \mathsf{Maps}_{/X_{\DR}}(Y\times_{X_{\DR}}U,Z).
\end{equation*}
We have defined for any $\EQ\to X_{\DR}$ via derived Weil-restriction its moduli stack of solutions $\mathbb{R}\mathrm{Sol}(\EQ).$ 
\begin{remark}
In the sequel \cite{KSY2}, we prove the derived global sections of $\mathcal{O}_{\mathbb{R}\mathrm{Sol}(\EQ)}$ can be computed in terms of \emph{variational factorization homology} i.e. a global horizontal de Rham hypercohomology functor over configuration space of points of $X.$
\end{remark}

A related, but more down to earth discussion is given by the following.
\begin{proposition}
\label{prop: RSol is affine}
Suppose $\EQ_X\to X$ is affine over $X.$ Then $\mathbb{R}\mathrm{Sol}_X(q_{\DR*}\EQ_X)$ is affine.
\end{proposition}
\begin{proof}
By Proposition \ref{prop: AffXsoAffXDR} since $\EQ_X\to X$ is affine, so is $\EQ:=q_{\DR*}(Z_X).$ Then, for each test affine $X_{\DR}$-scheme $T$,
\begin{eqnarray*}
\mathbb{R}\mathrm{Sol}(\EQ)&\simeq& \Map_{/X_{\DR}}(T\times X_{\DR},\EQ)
\\
&\simeq&\mathrm{Maps}_{\mathsf{QCAlg}(T\times X_{\DR})}(A\boxtimes \mathcal{O}_T,\mathcal{O}_{X_{\DR}\times T})
\\
&\simeq& \mathrm{Maps}(p_*(A\boxtimes \mathcal{O}_T),R),
\end{eqnarray*}
where we use $p:T\times X_{\DR}\to T$ and so the argument of mapping space is $R\Gamma_{\DR}(X,A)\otimes_k R.$ We are using that $p_*$ is right-adjoint to pull-back so that the right-hand side is equivalently
$\mathbb{R}\mathrm{Sol}(\EQ)(T)\simeq \mathsf{Spec}\big(R\Gamma_{\DR}(X,A)\big)(T).$ Since
$R\Gamma_{\DR}(X,-):=a_{X_{\DR}*}:\mathsf{QCoh}(X_{\DR})\to \mathsf{QCoh}(pt)\simeq \mathsf{Vect}_k,$ is lax-symmetric monoidal $R\Gamma_{\DR}(X,A)$ has the structure of a commutative algebra in $\mathsf{Vect}_k,$ therefore, $\mathsf{Spec}_k\big(R\Gamma_{\DR}(X,A)\big)$ is an affine derived $k$-scheme. This is precisely $\mathbb{R}\mathrm{Sol}(\EQ).$
\end{proof}

\begin{proposition}
\label{Derived Sections are Derived Flat Sections of Jet Space}
Let $E\in \PS_{/X}^{laft}$ and suppose $X$ is eventually coconnective. Then we have an identification of prestacks  
$$\mathfrak{j}^{\infty}:\Sect_(X,E)\simeq \Sect^{\nabla}\big(X,\JetX(E)\big)=\underline{\mathsf{Maps}}_{X_{\DR}}\big(X_{\DR},\J(E)\big).$$
\end{proposition}
\begin{proof}
For a test affine scheme $T,$ the result follows from the chain of equivalences, 
\begin{eqnarray*}
\mathrm{Maps}_{k}(T,\mathsf{Res}_{X_{\DR}/pt}(\mathsf{Res}_{X/X_{\DR}}(E)\big)&=&\mathrm{Maps}_{k}(T,a_{X_{dR*}}q_{\DR,*}E)
\\
&=&
\Map_{/X_{\DR}}(a_{X_{\DR}}^*T,q_{\DR,*}E)
\\
&=&\mathrm{Maps}_{X}(q_{\DR}^*a_{X_{\DR}}^*T,E)
\\
&=&\mathrm{Maps}_{X}(T\times X,E).
\end{eqnarray*}
\end{proof}

Based on the classical isomorphism of sheaves $\mathcal{D}_X(\mathcal{E}_1,\mathcal{E}_2)\simeq\mathcal{H}om_{\mathcal{O}_X}(\mathcal{J}^{\infty}(\mathcal{E}_1),\mathcal{E}_2),$ the next result makes it clear what is a generalized operator acting from prestacks $E_1$ to $E_2.$
\begin{proposition}
\label{Derived Differential Operator Proposition}
Consider $E_1,E_2\in \PS_{/X}$ and a morphism
\begin{equation}
    \label{eqn: Derived Differential Operator}
    \mathsf{F}:\JetX(E_1)\rightarrow E_2,
\end{equation}
in $\PS_{/X},$ (the generalized operator from $E_1$ to $E_2$). The map \emph{(\ref{eqn: Derived Differential Operator})}, induces a morphism of prestacks 
$\mathsf{P}_{\mathsf{F}}:\Sect_(X,E_1)\rightarrow \Sect_(X,E_2).$
\end{proposition}
\begin{proof}
It suffices to notice for any $s:X\rightarrow E_1,$ there is a canonical morphism $X\rightarrow \JetX(E_1)$ via (\ref{eqn: Jet prolongation}) i.e. $j_{\infty}(s).$ One then sets $\mathsf{P}_{\mathsf{F}}(s):=\mathsf{F}\circ j_{\infty}(s).$
\end{proof}
A (dg-)differential operator $\mathsf{P}_{\mathsf{F}}$ given by Proposition \ref{Derived Differential Operator Proposition} is $n$-shifted if the underlying morphism (\ref{eqn: Derived Differential Operator}) is of the form $\mathsf{F}:\JetX(E_1)\rightarrow E_2[n].$
By adjunction (\ref{eqn: Derived Differential Operator}) is equivalent to a diagram 
\begin{equation}
\label{eqn: Derived Diff op diagram}
\begin{tikzcd}
    \J(E_1)\simeq q_{\DR,*}E_1\arrow[d]\arrow[r,"\mathsf{F}"] & q_{\DR,*}E_2\simeq \J(E_2)\arrow[d]
    \\
    X_{\DR}\arrow[r,"id"] & X_{\DR}
\end{tikzcd}.
\end{equation}

The following two results describe moduli of solutions relative to one another.

\begin{proposition}
Let $f:X\to Y$ be a morphism of smooth projective $\mathbb{C}$-schemes. Let $f_{\DR}:X_{\DR}\to Y_{\DR}$ be the induced map on de Rham stacks. Let $E,F$ be derived prestacks. Consider the commutative diagram 
    \[
    \begin{tikzcd}
        E\arrow[d]\arrow[r] & X\arrow[d,"f"] \arrow[r,"q_{\DR}^X"]& X_{\DR}\arrow[d,"f_{\DR}"]
        \\
        F\arrow[r] & Y\arrow[r,"q_{\DR}^Y"] & Y_{\DR}
    \end{tikzcd}
    \]
There a morphism $\mathsf{Res}_{X/X_{\DR}}(E)\rightarrow \mathsf{Res}_{Y/Y_{\DR}}(F)$ and an equivalence
$\underline{\J(F)}\big(q_{\DR,*}^X(E)\big)\simeq \mathsf{Maps}_{/Y}(\JetX(E),F).$
\end{proposition}
\begin{proof}
For a morphism $\eta:E\rightarrow F$ of derived $k$-schemes or prestacks, there is an induced map 
$q_{\DR,*}(\eta):q_{\DR,*}(E)\rightarrow q_{\DR,*}(F)$. Then,
\begin{eqnarray*}
\underline{\J(F)}\big(\J(E)\big)&\simeq& \underline{\mathsf{Res}}_{Y/Y_{\DR}}(F)\big(\mathsf{Res}_{X/X_{\DR}}(E)\big)
\\
&\simeq& \mathsf{Maps}_{/Y}\big(\mathsf{Res}_{X/X_{\DR}}(E)\times_{X_{\DR}}^hX,F\big)
\\
&\simeq& \mathsf{Maps}_{/Y}\big(\J E\times_{X_{\DR}}^h X,F\big)
\\
&\simeq& \mathsf{Maps}_{/Y}(\JetX(E),F).
\end{eqnarray*}

\end{proof}

\begin{proposition}
\label{Prop: Solutions}
Let $f:F\rightarrow E$ be a finitely presented morphism in $\PS_{/X}^{\mathrm{laft}}$.
\begin{itemize}
   \item[1.] The moduli of sections of the pull-back of $F$ over $\J(E)$ are compatible with the product of moduli i.e. there is an equivalence of $$\underline{\mathbb{R}\mathrm{Sect}}_{\J(E)}\big(X\times_{X_{\DR}}^h\J(E)\times_{E}^h F/X\times_{X_{\DR}}^h \J(E)\big)\simeq \J(F)\times_{X_{\DR}}^h\J(E).$$

    \item[2.] Let $a:=a_{X_{\DR}}:X_{\DR}\to \mathrm{Spec}(\mathbb{C})$ be the canonical map and consider the  commutative diagram
    \[
    \begin{tikzcd}
    \EQ_2\arrow[dr,"p_2"] \arrow[rr,"f"] & & \EQ_1\arrow[dl,"p_1"]
    \\
    & X_{\DR}& 
    \end{tikzcd}
    \]
There is an equivalence of derived spaces of sections over $\mathbb{R}\mathrm{Sol}(\EQ_1)$ of the pull-back of $\EQ_2$:
$$\Sect_{\mathbb{R}\mathrm{Sol}(\EQ_1)}\big(a^*a_*\EQ_1\times_{\EQ_1}^h\EQ_2/a^*a_*\EQ_1\big)\simeq \mathbb{R}\underline{\mathrm{Sol}}(\EQ_1)\times_{Spec(\mathbb{C})}^h\mathbb{R}\underline{\mathrm{Sol}}(\EQ_2).$$
    
\end{itemize}
\end{proposition}

\begin{proof}
Consider defining pull-back (base-change) diagram of $\J(E)$ i.e.
\[
\begin{tikzcd}
q_{\DR}^*q_{\DR,*}(E)\arrow[d,"\beta:=\eta^*q_{\DR}"] \arrow[r,"\alpha=q_{\DR}^*\eta"] & X\arrow[d,"q_{\DR}"]
\\
\J(E)=q_{\DR,*}(E)\arrow[r,"\eta"] & X_{\DR}
\end{tikzcd}
\]
Then, note that 
$$\underline{\mathbb{R}\mathrm{Sect}}_{\J(E)}\big(X\times_{X_{\DR}}^h\J(E)\times_{E}^h F/X\times_{X_{\DR}}^h \J(E)\big)=\beta_*\alpha^*(F),$$
by definition of the derived mapping spaces of sections. 
Note further that this diagram may be enlarged as
\[
\begin{tikzcd}
& F \arrow[dr,"f"]& & &
\\
f^*eval(\JetX(E))\arrow[dr,"eval^*f"] \arrow[ur] & & E\arrow[d,"\pi_E"] &
\\
&
  q_{\DR}^*q_{\DR,*}(E)\arrow[ur,"ev"] \arrow[d,"\beta:=\eta^*q_{\DR}"] \arrow[r,"\alpha=q_{\DR}^*\eta"] & X\arrow[d,"q_{\DR}"]
\\
& \J(E)=q_{\DR,*}(E)\arrow[r,"\eta"] & X_{\DR}  
\end{tikzcd}
\]
with the top diamond is also a pull-back.
By Beck-Chevalley-Lurie \cite[Lemma 6.1.6.3]{Lur17} and universal properties of right-adjoints to the base-change pull-back functors the mapping space is 
$\beta_* eval^*f(q_{\DR}^*q_{\DR,*}E),$
and since
$\alpha_!\simeq (q_{\DR,*}\eta)_!\circ (\eta^*q_{\DR})^*$ passing to right-adjoints gives, (when applied to $F$),
$$\eta^*q_{\DR,*}\simeq (\eta^*q_{\DR})_*\circ (q_{\DR}^*\eta)^*\simeq \eta^*q_{\DR,*}F.$$
Thus, 
$$\underline{\mathbb{R}\mathrm{Sect}}_{\J(E)}\big(X\times_{X_{\DR}}^h\J(E)\times_{E}^h F/X\times_{X_{\DR}}^h \J(E)\big)\simeq\Sect_{X_{\DR}}(F/X)\times_{X_{\DR}}^h\Sect_{X_{\DR}}(E/X),$$
from which the result follows by (\ref{eqn: dR-Jets as Mapping Space}).
The second claim is proven similarly using:
\[
\begin{tikzcd}
& \EQ_2 \arrow[dr,"f"]& & &
\\
f^*eval(a^*a_*\EQ_1)\arrow[dr,"eval^*f"] \arrow[ur] & & \EQ_1\arrow[d,"p_1"] &
\\
&
  a^*a_*\EQ_1\arrow[ur,"ev"] \arrow[d,"\beta"] \arrow[r,"\alpha"] & X_{\DR}\arrow[d,"a_{X_{\DR}}"]
\\
& \mathbb{R}\mathrm{Sol}(\EQ_1)=a_*(\EQ_1)\arrow[r] & Spec(\mathbb{C})  
\end{tikzcd}
\]
Namely, use the $(\beta^*,\beta_*)$-adjunction to compute restrictions along $\beta$ i.e. $a^*a_*\EQ_1\times_{\EQ_1}^h\EQ_2:$ 
$$\mathrm{Maps}_{\mathsf{dStk}_{/\mathbb{R}\mathrm{Sol}(\EQ_1)}}(T,\mathsf{Res}_{\beta}(a^*a_*\EQ_1\times_{\EQ_1}^h \EQ_2)\big)\simeq\mathrm{Maps}_{\mathsf{dStk}_{/a^*a_*\EQ_1}}(T\times_{\mathbb{R}\mathrm{Sol}(\EQ_1)}^ha^*a_*\EQ_1,a^*a_*\EQ_1\times_{\EQ_1}^h \EQ_2\big).$$
Then 
$\Sect_{\mathbb{R}\mathrm{Sol}(\EQ_1)}\big(a^*a_*\EQ_1\times_{\EQ_1}^h\EQ_2/a^*a_*\EQ_1\big)$ is agrees with $\Sect_{pt}(\EQ_1/X_{\DR})\times_{Spec(\mathbb{C})}^h\Sect_{pt}(\EQ_2/X_{\DR}).$
\end{proof}

\begin{proposition}\label{Laft Descent}
If $X$ is eventually coconnective and $\EQ\rightarrow X_{\DR}$ is a derived Artin stack, $\RS(\EQ)$ is locally almost of finite type. In particular, if $E$ is a laft prestack over $X$ satisfying descent, then $\Sect^{\nabla}\big(X,\JetX(E)\big)$ is also laft. 
\end{proposition} 
\begin{proof}
Follows from Proposition \ref{Derived Sections are Derived Flat Sections of Jet Space} by considering a filtered diagram $(V_i)_{i\in I}$ in $\hspace{.2mm}^{\leq n}\mathsf{dAff},$ with $V\simeq \underset{i}{\mathrm{colim}}\hspace{.3mm} V_i,$ so $X\times V_i$ gives a filtered diagram in $\hspace{.2mm}^{\leq m+n}\mathsf{dAff}$ and noticing since $E$ is laft, the natural maps
$$\underset{i}{\mathrm{colim}}\hspace{.3mm}\underline{\mathsf{Maps}}\big(X,E\big)(V_i)\simeq \underline{\mathsf{Maps}}\big(X,E\big)(V),$$
are equivalences. Moreover, notice that for any push-out diagram of eventially coconnective $X_{\DR}$-spaces
\[
\begin{tikzcd}
U\times_{X_{\DR}}X\arrow[d]\arrow[d] \arrow[r] & V\times_{X_{\DR}} X\arrow[d]
\\
W\times_{X_{\DR}} X\arrow[r] & S\times_{X_{\DR}}X
\end{tikzcd},
\]
there is a pull-back diagram as in Proposition \ref{Retract proposition}:
\[
\begin{tikzcd}
    \RS(\EQ)(S)\arrow[d]\arrow[r] & \RS(\EQ)(V)\arrow[d]
    \\
    \RS(\EQ)(W)\arrow[r] & \RS(\EQ)(U)
\end{tikzcd}.
\]
\end{proof}
\begin{definition}
\label{definition: Derived NLPDE System Definition}
 A \emph{derived non-linear partial differential system}, or simply a \emph{generalized PDE} on sections of $E\in \PS_{/X}$ is a closed sub-stack
$\EQ$ of $q_{\DR,*}(E)=\J(E).$
\end{definition}
A representable $X_{\DR}$-space $\mathsf{Spec}_k(A)\rightarrow X_{\DR}$ determines a derived affine $\D_X$-scheme $\Spec_{\D_X}(\mathcal{A})$ with $\mathcal{A}=(q_{\DR})_*^{\mathsf{QCoh}}A$, by Proposition \ref{D-PreStk equiv X dR PreStk}, via quasi-coherent push-forward.
Definition \ref{definition: Derived NLPDE System Definition} states the existence of a closed immersion 
$i:\EQ\rightarrow q_{\DR,*}(E)$ in $\PS_{/X_{\DR}}.$ That is, $i$ is representable and for each $Spec(A)\rightarrow q_{\DR,*}(E),$ the induced map
$$\EQ\times_{q_{\DR,*}(E)}^h\mathsf{Spec}_{\D_X}(\mathcal{A})\simeq \mathsf{Spec}_{\D_X}(\mathcal{B})\rightarrow \mathsf{Spec}_{\D_X}(\mathcal{A}),$$
induces an epimorphism of commutative $\D_X$-algebras 
$H_{\D}^0(\mathcal{A})\rightarrow H_{\D}^0(\mathcal{B})$, defining a $\D$-algebraic NLPDE $H_{\D}^0(\mathcal{A})\rightarrow H_{\D}^0(\mathcal{B}),$ as per Definition \ref{defn: RelAlgNLPDE}. 

More generally, consider 
$$\mathsf{F}_1:\JetX(E_1)\rightarrow \EuScript{F},\hspace{2mm} \mathsf{F}_2:\JetX(E_2)\rightarrow \EuScript{F},$$
with $\EuScript{F}\in \PS_{/X}.$ By Proposition \ref{Derived Differential Operator Proposition} and diagram (\ref{eqn: Derived Diff op diagram}), an (inhomogeneous) generalized \textsc{nlpde} is a derived fiber product:
\begin{equation}
\label{eqn: Inhomogeneous DNLPDE}
\EQ\simeq \J(E_1)\times_{\J(\EuScript{F})}\J(E_2
),
\end{equation}
taken in $\PS_{/X_{\DR}}.$

\begin{observation}
\label{obs: DNLPDE Cat}
The considerations above which culminated in Definition \ref{definition: Derived NLPDE System Definition}, lead us to define $\mathsf{NLPDE}_X^{co-free}\subseteq \PS_{/X_{\DR}}$ as the full sub-$(\infty,1)$-category generated\footnote{More precisely, one should consider a type of (free) sifted colimit completion.} by the essential image $\J,$ and for a \emph{fixed} $E\in \PS_{/X}$, 
$$\mathsf{NLPDE}_X(E):=(\PS_{/X_{\DR}})_{/^{cl}\J(E)},$$ where $/^{cl}$ means we consider only morphisms which are derived closed embeddings. 

More exactly, we consider the full-sub $\infty$-groupoid of this $\infty$-category, via
$$\mathsf{dStk}_{/X}\xrightarrow{q_{\DR,*}}\mathsf{dStk}_{/X_{\DR}}\rightarrow \big(\PS_{/^{cl-emb}}^{\mathrm{Tor}}\big)^{\simeq},$$
sending $(E\rightarrow X)$ to $(\mathrm{Jet}_{\DR}^{\infty}(E)\rightarrow X_{\DR})$ and assigning to this relative pre-stack the $\infty$-groupoid of closed embeddings $i:\EQ\hookrightarrow \mathrm{Jet}_{\DR}^{\infty}(E)$ which are locally of finite Tor-amplitude non-linear PDEs. In particular, quasi--smooth PDEs are those for which $\mathbb{L}_{\EQ/\J(E)}$ has amplitude $[-1,0]$ (Def. \ref{QuasiSmoothDerivedPDE}).
In full generality, the objects are too wild to admit any kind of infinitesimal geometry or deformation theory. This is the principal motivation for finiteness hypothesis ($\D$-finitary) in (\ref{Definition: AdmissiblyDFinitary}) below. See Theorem \ref{thm: RSol is laft-def}.
\end{observation}

We now introduce the stacks of solutions alluded to in Observation \ref{obs: DNLPDE Cat}.

\begin{definition}
\label{Definition: Derived Stack of Solutions Definition}
\normalfont Consider Definition \ref{definition: Derived NLPDE System Definition} and let $\EQ$ be a NLPDE on $E\in \PS_{/X}^{\mathrm{laft}}.$ The moduli of sections in Construction \ref{cons: Solutions} (a) is called the \emph{derived stack of solutions} $\RS(\EQ):=\mathbb{R}\underline{\mathsf{Maps}}_{X_{\DR}}\big(X_{\DR},\EQ\big).$
\end{definition}
As mentioned in \S\S~\ref{sssec: Relation to other works}, one may change the underlying $\infty$-site or $\infty$-geometry on which stacks are modelled. Here, we are considering $(\mathsf{dAff}_{\mathbb{K}}^{alg},\text{ét})$ i.e. $\PS:=\mathsf{PShv}\big(\mathsf{dAff}_k^{alg},\mathsf{Spc}\big).$ The derived stack of solutions is \emph{algebraic}. If considering $(\mathsf{DAn}_{\mathbb{C}}^{Stein},\text{ét})$ $\RS$ is \emph{analytic}. An overconvergent rigid-analytic solution space discussed below. 
Finally, sub-stacks
$$\underline{\mathcal{S}}\subset \mathsf{Maps}_{/X_{\DR}}(X_{\DR},-).$$
are called \emph{stable} if: for all closed immersion $N\hookrightarrow M$ of derived stacks, the natural morphism
$$\underline{\mathcal{S}}(N)\rightarrow \underline{\mathcal{S}}(M)\times_{\RS(M)}^h\RS(N),$$
is an equivalence.
Stability arises in the form of properties: finiteness, geometricity, representability etc. 

\begin{example}
    \normalfont
    A classical solution $\varphi:X\rightarrow \EQ$ of a non-linear \textsc{pde} on sections of $E\rightarrow X$ is then a section $j_{\infty}(s)'$ obtained from $s\in \Gamma(X,E)$ and its jet-prolongation (\ref{eqn: Jet prolongation}) such that 
    \[
    \begin{tikzcd}
    \EQ \arrow[r] & \JetX(E)\arrow[r] & \J(E)\arrow[d]
    \\
    X \arrow[u,"j_{\infty}(s)'"]\arrow[r] & \arrow[u,"j_{\infty}(s)"] X\arrow[r] & X_{\DR}
    \end{tikzcd}
\]
as morphisms pull-backed from $X_{\DR}.$
\end{example}
The derived prestack of solutions to $\EQ$ with values in $\EuScript{X}$ in terms of the assignment in Construction \ref{cons: Solutions} (b): 
\begin{equation}
\label{eqn:Derived solution stack}
\RS(\EQ)(-):\PS_{/{X}_{\DR}}\rightarrow \PS_{/{X}_{\DR}},\hspace{2mm} \RS(\EQ)(\EuScript{X}):=\Map_{X_{\DR}}\big(\EuScript{X},\EQ\big).
\end{equation}

This abstract context can be made more concrete, which is useful for applications. We give one construction/example leaving more systematic treatment to our future work.
\begin{construction}
\label{cons: Explicit Solutions}
Given two non-linear differential operators $\mathsf{P}_{\mathsf{F}_i}:E_i\rightarrow F$ from some space $E_i$ to another space $F$, they can be equivalently viewed as compositions:
$$\mathsf{P}_{\mathsf{F}_i}:=\mathsf{F}_i\circ j_k:\mathsf{Res}_{X/*}(E_i)\rightarrow \mathsf{Res}_{X/*}\big(\mathrm{Jets}_{X}^k(E_i)\big)\rightarrow \mathsf{Res}_{X/*}(F).$$
The space of solutions is of course morally the collection of pairs of section $s_i$ of $E_i$ for $i=1,2$ which determined the same value as a section of $F$ i.e.
$$\big\{s=(s_1,s_2)| \mathsf{P}_{\mathsf{F}_1}(s_1)=\mathsf{P}_{\mathsf{F}_2}(s_2)=f\big\}.$$
It must be appropriately identified with a homotopy fiber product in derived stacks,
\begin{equation}
    \label{eqn: RSol(P,P)}
\begin{tikzcd}
    \mathbb{R}\underline{\mathrm{Sol}}_{X/*}(P_1,P_2)\arrow[d]\arrow[r] & \mathsf{Res}_{X/*}(E_2)\arrow[d,"\mathsf{P}_{\mathsf{F}_2}"]
    \\
    \mathsf{Res}_{X/*}(E_1)\arrow[r,"\mathsf{P}_{\mathsf{F}_1}"] & \mathsf{Res}_{X/*}(F)
\end{tikzcd}
\end{equation}
More generally, replacing the point $*$ with a stack $S$, and considering $S$-families of submersions, diagram (\ref{eqn: RSol(P,P)}) gives
$$\mathbb{R}\underline{\mathrm{Sol}}_{X/S}(P_1,P_2)\simeq \Sect_{X/S}(X,E_2)\times_{P_2,\Sect_{X/S}(X,V),P_1}\Sect_{X/S}(X,E_1).$$

If one of the spaces is trivial e.g. a  point, and one of the operators is the zero map, $\mathbf{0}$, we obtain the more familiar 
\[
\begin{tikzcd}
    \mathbb{R}\underline{\mathrm{Sol}}_{X/*}(P=0)\arrow[d]\arrow[r] & \Sect_{X}(X,E)\arrow[d,"P"]
    \\
    *\arrow[r,"0"] & \Sect_X(X,F).
\end{tikzcd}
\]
Thus, for a non-linear differential operator $P:E\rightarrow F$ over $X$ relative to $S$, we have 
\begin{equation}
    \label{eqn: EllipticPF}
\underline{\mathrm{Sol}}_{X/S}(P):=\underline{\mathrm{Sect}}_{S}(E/X)\times_{\underline{\mathrm{Sect}}_S(F/X)}\{0\}.
\end{equation}
It's points are given by:
$$\{T\xrightarrow{t}  \underline{\mathrm{Sol}}_{X/S}(P)\}_{\mathrm{Stk}}\leftrightarrow \{(f_T,\varphi_T)|f_T:T\rightarrow S, \varphi_T:X\times_S T\rightarrow E| P\varphi_T=0_{F/X}\}.$$
\end{construction}
For submersive maps, they satisfy base-change properties.

\begin{proposition}
\label{prop: Pull-back of EllipticPF}
    For any submersion of smooth stacks $\pi:V\rightarrow S$ and fiber-wise elliptic operator $P:E\rightarrow F$ on $V$ with a smooth map $g:S'\rightarrow S,$ consider the diagram
    \[
    \begin{tikzcd}
        V':=V\times_S^h S'\arrow[d,"g^*\pi"] \arrow[r] & V\arrow[d]
        \\
        S'\arrow[r] & S
    \end{tikzcd}
    \]
Then there is an equivalence $g^*\mathbb{R}\underline{\mathrm{Sol}}_{V/S}(P)\simeq \mathbb{R}\underline{\mathrm{Sol}}_{V'/S'}(g^*P).$ 
\end{proposition}
\begin{proof}
Note that the operator $g^*P$ is given by $g^*P:g^*E=E\times_S^h S'\rightarrow g^*F=F\times_S S'$. Then, the one has that $g^*\mathbb{R}\underline{\mathrm{Sol}}_{V/S}(P)\simeq (g^*\pi)_*(g^*P),$ due to the properties of moduli of sections, as in Proposition \ref{Prop: Solutions}.
\end{proof}

Definition \ref{Definition: Derived Stack of Solutions Definition} captures both the standard $\D$-module (i.e. linear) solution space and the functor of points $\D$-geometric derived stacks. The following comparison result holds in the wider context of an $\infty$-geometric context admitting infinitesimals.

\begin{proposition}
\label{Main Theorem: Comparison}
Consider a derived algebraic non-linear \textsc{pde} $\EQ$ on sections of $E\in \PS_{/X}^{\mathrm{laft}}$ with $X$ smooth.
Suppose further that it is representable and consider its restriction to representable derived stacks over $X_{\DR}.$ Then there is an equivalence of functors
$\RS(\EQ)(-)\simeq \Spec_{\D_X}(\mathcal{A})$ from affine pre-stacks over $X_{\DR}$ to $\mathsf{Spc}.$ In particular, if $X$ is a scheme and $\mathcal{Z}$ is a quasi-coherent module over $X_{\DR}$, we recover the usual derived $\D$-module linear solution space \end{proposition}

\subsubsection{Quasi--smooth derived PDEs}
Recall Definition \ref{Quasi-smooth}, of a $\D$-quasismooth $\D$-space.

Definition \ref{definition: Derived NLPDE System Definition} holds analogously in the case of derived jet-spaces of finite order $k<\infty$, and gives rise to an explicit class of derived PDEs.

To state them, let $\pi:E\rightarrow X$ be an $X$-prestack, $\hat{X}^k$ the formal $k$-th order neighbourhood of the diagonal, and consider
\begin{equation}
    \label{eqn: inf-k}
\begin{tikzcd}
    \hat{X}_k\arrow[d,"q_k"] \arrow[r,"p_k"] & X\arrow[d]
    \\
    X\arrow[r] & S,
\end{tikzcd}
\end{equation}
where $p_k,q_k$ are the two projections from $\hat{X}^k.$ Via pull-back, $p_k^*E\rightarrow \hat{X}_k\xrightarrow{q_k} X,$ Weil-restriction defines the relative $k$-jet space of $E$,
$J_{X/S}^k(E):=\underline{\mathrm{Res}}_{\hat{X}_k/X}(p_k^*E).$ It satisfies,
$$\mathrm{Map}_{\mathsf{PStk}_X}(T,q_{k*}p_k^*E)=\mathrm{Map}_{\mathsf{PStk}_{\hat{X}_k}}(q_k^*T,p_k^*E),$$
thus agrees with $\mathrm{Map}_{\hat{X}_k}(T\times_X \hat{X}_k,p_k^*E).$
\begin{definition}
\label{def: S-derPDEk}
Let $X$ be an $S$-stack and consider an object $E\in \mathsf{PStk}_{X}.$
    An \emph{$S$-family of $k$-th order generalized PDEs on $E$} is a closed sub-stack
    $Z\subset J_{X/S}^kE.$
\end{definition}
This is explicit in the case of quasi--smooth PDEs (recall Def. \ref{QuasiSmoothClImm}).
\begin{definition}
\label{QuasiSmoothDerivedPDE} 
    Let $\mathsf{X}$ a quasi-smooth derived $\mathbb{C}$-scheme. Let $k\in\mathbb{N}$. A \emph{derived partial differential equation of order $k$} is a derived $\mathbb{C}$-subscheme 
$\iota:\EuScript{E}\hookrightarrow L\mathcal{J}^k(\mathsf{X}).$ 
It is \emph{quasi--smooth} if $\iota:\EuScript{E}\rightarrow L\mathcal{J}^k(\mathsf{X})$ is quasi-smooth.
\end{definition}
These objects admit an explicit descriptions using Koszul sequences \cite{KhanRydh2025}.
a polynomial ring over $\mathbb{C}$ in finitely-many variables. 
\begin{proposition}
    Let $\iota:\EuScript{E}\to L\mathcal{J}^k(\mathsf{X})$ be quasi--smooth, where $\mathsf{X}=\mathsf{Spec}(A_{\bullet})$.
 Then, Zariski-locally,
    $$\EuScript{E}\simeq \Spec_{\mathbb{C}}\big(L\mathcal{J}^k(A_{\bullet})\po \big<G\big>\big),\hspace{3mm}G:=\big(g_1,\ldots,g_c\big),$$
for a finite sequence of elements in $L\mathcal{J}^k(A_{\bullet}),$ such that 
$\mathcal{O}_{\EuScript{E}}^{\bullet}\xrightarrow{\sim} L\mathcal{J}^k(A_{\bullet})\otimes_{\mathbb{C}[T_1,\ldots,T_c]}^{\mathbb{L}}\mathbb{C},$
via the map $T_A\mapsto g_A,A=1,\ldots,c.$
\end{proposition}
\begin{proof}
Using Proposition \ref{prop: Geometry of Derived Jets1} (iii), we describe derived $k$-jets Zariski-locally as follows. Since $\mathsf{X}$ is quasi--smooth, by \cite{KhanRydh2025} we write it as a homotopy-pushout $A^{\bullet}/\!/\big<f\big>.$ for some functions. Then, for every $k\in\mathbb{N}\cup\big\{\infty\big\},$ there is a natural equivalence 
\begin{equation}
    \label{k-Comparison}
L\mathcal{J}^k\big(A_{\bullet}\po \big<f\big>\big)\xrightarrow{\simeq} L\mathcal{J}^k(A_{\bullet})\po \big<f,f^{(1)},\ldots,f^{(k)}\big>.
\end{equation}
Specifically, since $L\mathcal{J}^k(X)\simeq J^k(X)$ for a classical $\mathbb{C}$-scheme $X$, we have that
$$L\mathcal{J}^k\big(\mathbb{C}[T_1,\ldots,T_n]\big)\simeq J^k\big(\mathbb{C}[T_1,\ldots,T_n]\big)\simeq \mathbb{C}[T,T^{(1)},\ldots,T^{(k)}],$$
and thus
$$L\mathcal{J}^k\big(\eqref{HoPushOut1}\big)\xrightarrow{\simeq }\begin{tikzcd}
    \mathbb{C}[T,T^{(1)},\ldots,T^{(k)}]\arrow[d]\arrow[r,"L\mathcal{J}^k(\mathsf{f})"] & L\mathcal{J}^k(A_{\bullet})\arrow[d]
    \\
J^k(\mathbb{C})\arrow[r] & L\mathcal{J}^k(A_{\bullet}\po \big<f\big>).
\end{tikzcd}$$
This diagram is homotopy coCartesian since$L\mathcal{J}^k(-)$ is a left-adjoint, thus preserves colimits. Note that 
$$L\mathcal{J}^k(\mathsf{f}):\mathbb{C}[T,T_j^{(i)}]\rightarrow L\mathcal{J}^k(A_{\bullet}),$$
is given by $T_i\mapsto f_i$ and $T_i^{(j)}\mapsto f_{i}^{(j)},$ for $,i=1,\ldots,n,j=1,\ldots,k.$ 
Since $\EuScript{E}\to L\mathcal{J}^k(\mathsf{X})$ is quasi--smooth, Zariski-locally we may write 
\[
\begin{tikzcd} \mathcal{E} \arrow[d] \arrow[r,"\iota"] & L\mathcal{J}^k(\mathsf{X}) \arrow[d,"g"] \\ \{0\} \arrow[r] & \mathbb{A}_{\mathbb{C}}^c \end{tikzcd}
\]
as a homotopy Cartesian diagram (cf. \eqref{QSDiagram}) or dually as the homotopy push-out
\[
\begin{tikzcd} \mathbb{C}[T_1,\dots,T_c] \arrow[r,"\epsilon"] \arrow[d,"g^\#"] & \mathbb{C} \arrow[d] \\ \mathcal{O}_{L\mathcal{J}^k(\mathsf{X})} \arrow[r] & \mathcal{O}_{\mathcal{E}} \end{tikzcd}
\]
where $g^{\sharp}:T_A\mapsto g_A$ as a function $g_A\in \mathcal{O}_{L\mathcal{J}^k(A_{\bullet}}.$ It follows that 
$\pi_0(\mathcal{O}_{\EuScript{E}}^{\bullet})\simeq \pi_0\big(\mathcal{O}_{L\mathcal{J}^k(\mathsf{X})}\big)/\big(g_1,\ldots,g_c\big).$
In general $\pi_i(\mathcal{O}_{\EuScript{E}}^{\bullet})$ are higher cohomologies of the Koszul sequence for the functions $(g_1,\ldots,g_c).$ Namely,
$$\pi_i(\mathcal{O}_{\EuScript{E}}^{\bullet})=H^{-i}(K^{\bullet}_{L\mathcal{J}^k(A_{\bullet})}(G)\big).$$
\end{proof}
This result shows quite explicitly that derived PDEs $\EuScript{E}$, here in the quasi--smooth case, indeed carry higher cohomological information via the non-trivial $\pi_0(\mathcal{O}_{\EuScript{E}}^{\bullet})$-modules $\pi_i(\mathcal{O}_{\EuScript{E}}^{\bullet}).$

Singular supports in the sense of \cite{AriGai2015} are defined for ind-coherent sheaves on quasi-smooth derived schemes. In this setting, they measure the degree to which the derived PDE is singular.
\begin{remark}
    In this paper we always work with derived $\D$-spaces. This subsumes the existence of prolongations of finite-order derived PDEs. Classically, there are obstructions to prolongation and if these vanish, the PDE is said to be formally integrable. Thus, derived $\D$-geometry implicitly works with formally integrable objects in a derived sense. A detailed study of \emph{derived prolongation} of finite-order and in particular quasi--smooth derived PDEs (Def. \ref{QuasiSmoothDerivedPDE}) and the corresponding notion of \emph{derived integrability} is studied elsewhere (see e.g, \cite{KryQS2026}).
\end{remark}

\subsubsection{Parameterized Solutions} 
\label{Parameterized Solution and Derived DO} Proposition
\ref{FormalDisk prop} together with the jet-prolongation maps (\ref{eqn: Jet prolongation}) define $\EuScript{F}$-parameterized solutions $\RS(\EQ)$ to a derived non-linear \textsc{pde} $\EQ\rightarrow \J(E)$ imposed on sections of $E\rightarrow X.$ Let $\EuScript{F}\in \PS_{/X}$ be a laft prestack and consider the sub-prestack of sections 
$$\Sect_{\EuScript{F}}\big(X,\JetX(E)\big)\subset R\underline{\mathsf{Maps}}_{/X}\big(q_{\DR}^*q_{\DR,!}(\EuScript{F}),\JetX(E)\big),$$
consisting of those morphisms of prestacks $\psi^{\EuScript{F}}:q_{\DR}^*q_{\DR,!}\EuScript{F}\rightarrow \JetX(E)$ relative to $X$ such that the following diagram
\begin{equation}
    \label{eqn: Formal Solution Diagram}
    \begin{tikzcd}
        q_{\DR}^*q_{\DR,!}\EuScript{F}\arrow[r,"\psi^{\EuScript{F}}"] & \JetX(E)\arrow[d]
\\
\EuScript{F}\arrow[u,"g"] \arrow[r,"\tilde{\psi}^{\EuScript{F}}"] & \JetX(\JetX(E))
    \end{tikzcd}
\end{equation}
commutes. The morphism $\widetilde{\psi}^{\EuScript{F}}:\EuScript{F}:\rightarrow \JetX\JetX(E)$ appearing in (\ref{eqn: Formal Solution Diagram}) is determined by $\psi^{\EuScript{F}}$ via the $(p^*p_!,\JetX(E))$-adjunction and Proposition \ref{ComonadicJet}.

Note that $q_{\DR}^*q_{\DR,!}\EuScript{F}$ is determined by the pull-back diagram in $\PS_{/X_{\DR}}:$
\[
\begin{tikzcd}
    X\times_{X_{\DR}}\EuScript{F}\arrow[d]\arrow[r] & \EuScript{F}\arrow[d]
    \\
    X\times_{X_{\DR}}X\arrow[d]\arrow[r]& X\arrow[d,"q_{\DR}"]
    \\
    X\arrow[r,"q_{\DR}"] & X_{\DR}
\end{tikzcd}
\]
Consider $E\in \PS_{/X}^{laft}$, $\iota:\EQ\rightarrow \J(E)$ and let $\EuScript{F}\in \PS_{/X}^{laft}.$ The corresponding  pre-stack of $\EuScript{F}$-parameterized solutions is the pre-stack
$$\mathbb{R}\underline{\mathrm{Sol}}_{\EuScript{F}}(\EQ):=\underline{\mathsf{Maps}}(q_{\DR}^*q_{\DR,!}\EuScript{F},p^*\EQ),$$ 
such that for all $U$-points $\psi_U^{\EuScript{F}}\in\mathbb{R}\underline{\mathrm{Sol}}_{\EuScript{F}}(\EQ)(U),$ the map
$\iota \circ \psi_U^{\EuScript{F}}$ is an object of the space of sections 
$\Sect_{\EuScript{F}}(X,\JetX(E)\big)(U).$

\begin{proposition}
\label{proposition: Relation between solutions}
    If $i:\EQ_X\rightarrow \JetX(E)$ is a derived non-linear \textsc{pde} on sections of $E\in \PS_{/X}^{\mathrm{laft-def}},$ then $\EQ$ and $\EQ_X^{\infty}$ have the same solutions.
\end{proposition}
\begin{proof}
Let $s:X\rightarrow E$ be a section that is a solution of $\EQ$. Then there exists a uniquely defined (up to homotopy) morphism $s_{\infty}\in \RSol_{/id_X}(\EQ),$ and a morphism 
    $\alpha:\EQ_X^{\infty}\rightarrow \EQ_X,$ which induces an equivalence 
    $$\alpha^*:\RS(\EQ_{X}^{\infty})\rightarrow \RSol_{/id_X}(\EQ_X).$$
This might be seen, for instance, by noting $s_{\infty}:q_{\DR}^*q_{\DR,!}X\rightarrow \JetX(E),$ is a uniquely defined $X$-parameterized solution, given by comonadic extension,
    $X\xrightarrow{\simeq} \JetX\xrightarrow{\JetX(E)(j_{\infty}(s)}\JetX(\JetX(E)).$
    The diagram (\ref{eqn: Infinite prolongation}) together with 
    \begin{equation}
    \label{eqn: Universal alpha}
    \begin{tikzcd}
       \EQ_X^{\infty}\arrow[d,"\mathfrak{j}_{\infty}(s)"]\arrow[r,"i_{\infty}"] & \JetX\arrow[d,"\Delta_{E}"]
       \\
       \JetX\EQ_X\arrow[d]\arrow[r,"\JetX i"] & \JetX(\JetX(E))\arrow[d]
       \\
       \EQ_X\arrow[r,"i"] & \JetX(E)
    \end{tikzcd}
    \end{equation}
defines $\alpha,$ as the obvious composition.
\end{proof}

We give a brief example constructing a new jet-comonad via the `overconvergent de Rham space' in the context of rigid analytic geometry \cite{Cam}, as discussed in §§ \ref{sssec: Relation to other works}. It is included for completeness as it relates to the central theme of this paper, but the reader can safely skip ahead. 
\subsubsection*{Towards overconvergent solutions}
Let $X$ be a quasi-compact derived rigid analytic stack and consider the $n$-th overconvergent neighbourhood of the diagonal map; the analytic space $(\Delta_X^{n+1}X)^{\dagger}$ whose functions overconverge the locally closed subspace 
$|\Delta_X^{n+1}|\subset |X^{\times X(n+1)}|.$
Then,
$$X_{\DR}^{\dagger}\simeq \varinjlim_{[n]\in \Delta^{op}}\big(\Delta_X^{(n+1)}X\big)^{\dagger}.$$ 
On the $\infty$-category of bounded analytic rings there are natural notions of $\dagger$-reduced rings, as well as $\dagger$-nil radicals, whose $S$-points are maps $S\rightarrow A$ overconvergently close to zero. Similarly, there is a condensed nil-radical whose $S$-points are uniformly nilpotent \cite{Cam}. Following Construction \ref{cons: Solutions} (a) and Proposition \ref{D-PreStk equiv X dR PreStk}, one obtains the $\infty$-comonad of overconvergent derived jets and the following derived space of flat sections.
\begin{construction}
\label{cons: Analytic Solution Space}
\normalfont 
Consider the $\infty$-topos of analytic stacks \cite{Cam}, 
and a morphism $f:X\rightarrow Y$ together with a prestack $E\xrightarrow{q}X,$ such that $Y$ is of finite solid Tor-dimension and solid proper, there is an adjunction $(f^*,f_*)$ on slice categories of stacks and derived moduli of analytic sections of $q$,
$\mathbb{R}\mathsf{Sect}_{Y}(E/X):=f_*(E),$ is naturally defined \cite{Lurie2018}.
As $f^*$ commutes with colimits, $f_*$ exists as a right-adjoint \cite[Corollary 5.5.2.9]{Lur09}. Given a derived analytic stack $X$ with $X\rightarrow X_{\DR}^{\dagger},$ and $\EQ$ a closed substack of Weil restriction along this map, one obtains moduli space of \emph{overconvergent solutions} by Construction \ref{cons: Solutions} (a) as 
$\RSol_{X,\dagger}(\EQ):=\mathsf{Res}_{X_{\DR}^{\dagger}/pt}(\EQ).$
\end{construction}

\subsection{Cotangent Complexes and Finiteness Conditions}
\label{sec: Cotangent Complexes and Finitness Conditions}
In this subsection, we consider the tangent stack of $\RS.$ Heuristically,
$$T\mathbb{R}\underline{\mathrm{Sol}}_{X}(\EQ)\simeq \mathbb{R}\underline{\mathrm{Sol}}_{X}(\mathbb{T}_{Z}):=\{\text{solutions of (derived) linearization problem}\}.$$

To this end, consider a relative prestack $E\rightarrow Y,$ with $q_{\DR}^Y:Y\rightarrow Y_{\DR}$ the projection, as well as the $\mathcal{D}_Y$-relative jet space $\J(E)=q_{\DR,*}^YE,$ with its structure maps
that we denote 
\begin{equation}
\label{eqn: Jet Structure maps}
\mathsf{p}:\J(E)\rightarrow Y_{\DR},\hspace{2mm}\text{ and }\hspace{1mm} \mathsf{q}:\JetY(E):=q_{\DR}^{*}\J(E)\rightarrow Y
\end{equation}

The relative cotangent complex of the $\mathcal{D}_Y$-relative $\infty$-jet prestack will be an object 
\begin{equation}
\label{eqn: D-Relative Cotangent}
\scalemath{1.10}{\mathbb{L}}_{\J(E)/Y_{\DR}}\in \mathsf{Shv}_{Y_{\DR}}(\J(E)):=\mathcal{O}(\J(E))-\mathsf{Shv}(Y_{\DR}).
\end{equation}
It is understood here that since we are considering cotangents, employing the results of \cite{GR17b}, we are speaking about \emph{pro-coherent cotangent} complexes i.e. $\mathsf{Shv}$ is understood to be $\mathsf{ProCoh}.$
The sheaf  $\mathcal{O}_{\JetY(E)}$ is naturally a commutative $\mathcal{O}_Y$-algebra in $\mathsf{Shv}\big(\JetY(E)\big).$

\subsection{Finiteness and representability}
Suppose $E$ is a laft prestack admitting deformation theory relative to $Y.$ Then $\J(E)$ admits relative deformation theory to $Y_{\DR}$ and in this case there is an equivalence
$$\scalemath{1.10}{\mathbb{L}}_{\J(E)/Y_{\DR}}\simeq\scalemath{1.10}{\mathbb{L}}_{\J(E)},$$
as the the deformation theory of $Y_{\DR}$ is trivial. 

We would like mapping stacks of solutions (\ref{eqn:Derived solution stack}) to be homotopically finitely presented over $Y_{\DR}$, hereafter referred to as \emph{homotopically dR-finitely presented}. Roughly speaking, if $E$ is homotopically dR-finitely presented and representable it is of finite $\D$-presentation as in § \ref{ssec: Prestacks of Solutions and DNLPDES}.

A prestack $E$ is $dR$-finitely presented if $\Map_{/X_{\DR}}(\EQ,-)$ commutes with homotopy colimits in $\PS_{X_{\DR}}.$ In other words, for a filtered diagram (indexed by a small category $I$) of relative $X_{\DR}$-stacks $\{X_i\rightarrow X_{\DR}\},$ one has that
\begin{equation}
\label{eqn: Ho dR-fp}
\RS(\EQ)(\underset{i\in I}{\mathrm{colim}}\hspace{1mm} X_i)\simeq \underset{i\in I}{\mathrm{colim}}\hspace{1mm}\RS(\EQ)(X_i).
\end{equation}

It may be checked in practice on the sequences of cohomological truncations. Namely, it is enough to verify the claim when $\EQ$ is affine i.e. $\EQ=\mathsf{Spec}(\mathcal{A})$ for some $\mathcal{A}\in Comm\big(\mathsf{Shv}(X_{\DR})\big).$ In this case, one verifies each $\tau^{\leq -n}(\mathcal{A})$ is compact.

\begin{remark}[Terminology]
\normalfont
If an object $Z\in \PS$ has the property that its $\infty$-category of sheaves $\mathsf{Shv}(\EQ)$ is canonically anti self-dual in $\mathsf{Cat}_{st}^{\infty},$ we say that it is \emph{finitary}\footnote{With the additional assumptions that $\mathcal{O}_{Z}$ is a compact object and the diagonal map $\delta_{Z}$ is quasi-affine, one says that $Z$ is \emph{passable} \cite{GR17b}. } for $\mathsf{Shv}.$ 
\end{remark}

This definition includes, for example, the case when $Z$ is an ind-inf scheme finitary for $\IC$. Similarly if $Z$ is perfect (as an object of $\PS$), then it is finitary for $\Q.$ Furthermore, any quasi-compact scheme is finitary for $\Q.$

\begin{definition}
\label{Definition: AdmissiblyDFinitary}
A relative prestack $E\rightarrow Y$ is \emph{$\D$-finitary} if it admits a cotangent complex and the $(\infty,1)$-subcategory of $\mathsf{Shv}^!\big(\J(E)\big)$ spanned by the essential image of the functor $q_{\DR,*}^{\mathsf{Shv}}$ is dualizable. A derived non-linear \textsc{pde} $\EQ$ is $\D$-\emph{finitary} if it is obtained from a $\D$-finitary prestack $E\rightarrow Y$ and its derived stack of solutions is locally almost of finite type and admits a cotangent complex for every solution $\varphi.$
\end{definition}

Quasi-coherent pullback along smooth morphisms is symmetric monoidal, inducing a functor 
$g^*_{\mathrm{CAlg}}$. Moreover, since $\mathrm{CAlg}(-)$ is a functorial in symmetric monoidal $\infty$-functors we have a series of induced functors on $\mathsf{QCAlg}(X_{\DR})$.
Suppose $f:Y\to X$ is a morphism of smooth varieties over $\mathbb{C}$ and let $f_{\DR}:Y_{\DR}\to X_{\DR}$ be the induced map. Then,
$$f_{\DR}^{\mathsf{QCoh},*}:\mathsf{QCoh}(X_{\DR})\to \mathsf{QCoh}(Y_{\DR}),$$
induces a functor,
\begin{equation}
    \label{eqn: QCAlgpb}
f_{\DR}^{\mathsf{QCAlg},*}:\mathsf{QCAlg}(X_{\DR})\to \mathsf{QCAlg}(Y_{\DR}).
\end{equation}

We need some properties of these functors.

\begin{proposition}
\label{prop: QCAlgDRpb}
Let $f:Y\to X$ be a morphism of smooth varieties. Then, the following diagram commutes,
\[
\begin{tikzcd}
    \mathsf{QCAlg}(X_{\DR})\arrow[d,"U"]\arrow[r,"f_{\DR}^{\mathsf{QCAlg},*}"]& \mathsf{QCAlg}(Y_{\DR})\arrow[d,"U"]
    \\
    \mathsf{QCoh}(X_{\DR})\arrow[r,"f_{\DR}^{\mathsf{QCoh},*}"]& \mathsf{QCoh}(Y_{\DR}).
\end{tikzcd}
\]
Moreover, if $f$ is smooth (or more generally, finite Tor-dimension e.g. flat), then 
$\tau^{\leq -n}\circ f^*\simeq f^*\circ \tau^{\leq -n},$
for each $n.$
\end{proposition}
\begin{proof}
    Since truncation $\tau^{\leq -n}$ is a right adjoint to the inclusion $\mathsf{QCoh}^{\leq -n}\subset \mathsf{QCoh},$ it commutes with functors which preserve limits, but our pull-back is a left-adjoint. It is $t$-exact and preserves subcategories of (co)connective truncations, since $f$ is of finite Tor-amplitude because it is smooth. 
\end{proof}
This result is an application of functoralities discussed in \cite{GR14}.
\begin{proposition}
\label{prop: PreserveDafp}
    Let $A\in \mathsf{CAlg}(\D_X-\mathsf{Mod})$ and assume it is $\D_X$-afp. Let $f:Y\to X$ be smooth. Put $A_Y:=f_{\DR}^{\mathsf{QCAlg},*}(A)\in \mathsf{CAlg}(D_Y-\mathsf{Mod}).$ Then $A_Y$ is $\D_Y$-afp.
\end{proposition}
\begin{proof}
Since $f$ is smooth (\ref{eqn: QCAlgpb}) is a left-adjoint. Moreover it is each and commutes with filtered colimits. In particular, it preserves compact objects, and thus from Proposition \ref{prop: QCAlgDRpb}, since $\tau^{\leq -n}A$ is compact in 
$\mathrm{CAlg}(\mathsf{QCoh}(X_{\DR}))^{\leq -n},$ for every $n$, it follows that 
$$\tau^{\leq -n} f_{\DR}^{\mathsf{QCAlg},*}(A)=\tau^{\leq -n}A_Y\simeq f_{\DR}^{\mathsf{QCAlg},*}\tau^{\leq -n}A.$$
Since it preserves compacts, this is again compact for each $n$, therefore $A_Y$ is $\D_Y$-afp.
In more detail, for every \(n\ge0\), the truncation \(\tau^{\ge -n}A_Y\) is compact in \(\mathsf{QCAlg}(Y_{\DR})^{\ge -n}\). That is, the functor $\operatorname{Map}_{\mathsf{QCAlg}(Y_{\DR})}\big(\tau^{\ge -n}A_Y,\; -\big)$
preserves filtered colimits in $\mathsf{QCAlg}(Y_{\DR})^{\ge -n}$.

Because \(f_{\DR}^{\mathsf{QCAlg},*}\) is left adjoint to \(f_{DR,*}^{\mathsf{QCAlg}}\) by adjunction,
$$\operatorname{Map}_{\mathsf{QCAlg}(Y_{\DR})}\big(f_{\DR}^*B,\;C\big)
 \simeq
    \operatorname{Map}_{\mathsf{QCAlg}(X_{\DR})}\big(B,\; f_{DR,*}C\big),$$
for \(B\in\mathsf{QCAlg}(X_{\DR})\) and \(C\in\mathsf{QCAlg}(Y_{\DR})\).
Apply this with \(B=\tau^{\ge -n}A\) and \(C\) arbitrary. Then for any filtered diagram \(\{C_i\}\) in
\(\mathsf{QCAlg}(Y_{\DR})^{\ge -n}\) with colimit \(C=\varinjlim_i C_i\) we obtain
\[
\begin{aligned}
    \operatorname{Map}_{\mathsf{QCAlg}(Y_{\DR})}\big(\tau^{\ge -n}A_Y,\; C\big)
    &\simeq
    \operatorname{Map}_{\mathsf{QCAlg}(X_{\DR})}\big(\tau^{\ge -n}A,\; f_{DR,*}C\big) \\
    &\simeq
    \operatorname{Map}_{\mathsf{QCAlg}(X_{\DR})}\big(\tau^{\ge -n}A,\; \varinjlim_i f_{DR,*}C_i\big),
\end{aligned}
\]
where we used that \(f_{DR,*}\) commutes with filtered colimits because \(f\) is quasi-compact and quasi-separated. Concretely, the pushforward \(f_{DR,*}^{\Q}\) preserves filtered colimits and the same holds on algebra objects since underlying functors detect filtered colimits.

Because \(\tau^{\ge -n}A\) is compact in \(\mathsf{QCAlg}(X_{\DR})^{\ge -n}\) by hypothesis, for all filtered colimits,
$$\operatorname{Map}_{\mathsf{QCAlg}(X_{\DR})}\big(\tau^{\ge -n}A,\; \varinjlim_i f_{DR,*}C_i\big)
    \simeq
    \varinjlim_i \operatorname{Map}_{\mathsf{QCAlg}(X_{\DR})}\big(\tau^{\ge -n}A,\; f_{DR,*}C_i\big).$$
From Proposition \ref{prop: QCAlgDRpb} and by adjunction,
$$ \varinjlim_i \operatorname{Map}_{\mathsf{QCAlg}(X_{\DR})}\big(\tau^{\ge -n}A,\; f_{DR,*}C_i\big)
    \simeq
    \varinjlim_i \operatorname{Map}_{\mathsf{QCAlg}(Y_{\DR})}\big(\tau^{\ge -n}A_Y,\; C_i\big).$$
This shows \(\operatorname{Map}(\tau^{\ge -n}A_Y,-)\) preserves filtered colimits, i.e. \(\tau^{\ge -n}A_Y\) is compact in \(\mathsf{QCAlg}(Y_{\DR})^{\ge -n}\). Since \(n\) was arbitrary, \(A_Y\) is \(D_Y\)-afp.
\end{proof}

Given $f:Y\to X$ when is 
$$\mathrm{Maps}_{Y_{\DR}}(T,\mathsf{Spec}_{Y_{\DR}} A_Y) \to \Map_{/X_{\DR}}(T,\mathsf{Spec}_{X_{\DR}}A),$$
an equivalence? We will answer this in the following section. In order to do so, it is useful to have the following relation between $\D$-afp algebras and finiteness hypothesis on spaces of flat sections.

\begin{proposition}
    \label{prop: D-afp means RSolfinite}
    Suppose $A\in \mathsf{QCAlg}(X_{\DR})$ is $\D_X$-afp. Then $\mathbb{R}\mathrm{Sol}_{\D}(A)$ is a derived algebraic stack locally of finite presentation over $k$. Moreover, if $\A\simeq \mathrm{Sym}(\M)$ with $\M$ a compact generator of $\mathsf{QCoh}(D_X),$ then $\A$ is compact and $\mathbb{R}\mathrm{Sol}_{\D}(A)$ is representable by a derived affine $k$-scheme of finite presentation.
\end{proposition}
\begin{proof}
Writing $U:\mathsf{QCAlg}(X_{\DR})\to \mathsf{QCoh}(X_{\DR}),$ we have $$\mathrm{Maps}_{\mathsf{QCAlg}(X_{\DR})}(\mathrm{Sym}(\M),B)\simeq \mathrm{Maps}_{\Q(X_{\DR})}(\M,UB).$$
Let $\{B_i\}_{i\in I}$ be a filtered diagram in $\mathsf{QCAlg}(X_{\DR}).$ By Proposition \ref{Compact D algebras} $\mathrm{Maps}_{\Q(X_{\DR})}(\M,-)$ preserves filtered colimits. Therefore, 
$$\mathrm{Maps}_{\mathsf{QCAlg}(X_{\DR})}(\mathrm{Sym}(\M),\mathrm{colim}_iB_i)\simeq \mathrm{Maps}_{\Q(X_{\DR})}(\M,U\mathrm{colim}_iB_i),$$
which gives $$\mathrm{Maps}(\M,\mathrm{colim}_iU B_i)\simeq \mathrm{colim}_i\mathrm{Maps}(\M,UB_i)\simeq\mathrm{colim}_i\mathrm{Maps}_{\mathsf{QCAlg}(X_{\DR})}(\mathrm{Sym}(\M),B_i).$$ Thus $\mathrm{Maps}_{\mathsf{QCAlg}(X_{\DR})}(\mathrm{Sym}(\M),-)$ commutes with filtered colimits so $\mathrm{Sym}(\M)$ is compact. Now, suppose $\mathsf{Spec}_{X_{\DR}}A\to X_{\DR}$ is affine and $A$ is $\D$-afp. Then $L_{A}$ is perfect of finite Tor-amplitude. The derived mapping stack of sections,
$$R\underline{\mathrm{Sol}}_{\D}(A)=\underline{\mathrm{Sect}}_{X_{\DR}}\big(\mathsf{Spec}_{X_{\DR}}(A)/X_{\DR})\simeq \underline{\mathrm{Maps}}_{/X_{\DR}}(X_{\DR},\mathsf{Spec}_{X_{\DR}}(A)),$$
has as its $T$-points, 
$$R\underline{\mathrm{Sol}}_{\D}(A)(T)\simeq \mathrm{Maps}_{\mathsf{QCAlg}(X_{\DR})}(A,(X_{\DR}\times T)_*\mathcal{O}_{X_{\DR}\times T}).$$
Since \(\mathsf{Spec}_{X_{\DR}}(A)\to X_{\DR}\) is affine and \(A\) is \(D_X\)-afp, the relative cotangent complex \(\mathbb{L}_{A/X_{\DR}}\) is perfect and of finite Tor-amplitude.Concretely, the prestack
\[
    T\longmapsto \operatorname{Maps}_{/X_{\DR}}\!\big(X_{\DR}\times T,\; \mathsf{Spec}_{X_{\DR}}(A)\big)
\]
is locally on \(T\) controlled by the complex \(R\Gamma\big(X_{\DR},\; \mathbb{T}_{A/X_{\DR}}\otimes \mathcal{O}_T\big)\), and the finiteness of \(\mathbb{L}_{A/X_{\DR}}\) implies Artin-type representability conditions (cotangent amplitude and lfp) needed for the standard representability theorems. Hence \(R\underline{\mathrm{Sol}}_{\D}(A)\) is derived algebraic and locally of finite presentation over \(\mathbb C\).

\medskip\noindent\emph{(B) Free case \(A\simeq\operatorname{Sym}(\mathcal{M})\) with \(\mathcal{M}\) compact.} \\
Assume \(A=\operatorname{Sym}(\mathcal{M})\) for a compact \(\mathcal{M}\in\mathsf{QCoh}(X_{\DR})\). Then the mapping space of \(A\) into any test algebra is computed by the mapping space out of \(\mathcal{M}\) into underlying modules, because \(\operatorname{Sym}(-)\) is left adjoint to the forgetful functor:
\[
    \operatorname{Map}_{\mathsf{QCAlg}}\big(\operatorname{Sym}(\mathcal{M}),\, B\big)
    \;\simeq\;
    \operatorname{Map}_{\mathsf{QCoh}}\big(\mathcal{M},\, U(B)\big).
\]
Taking \(B=(X_{\DR}\times T)_*\mathcal{O}_{X_{\DR}\times T}\) (the test algebra for the mapping/section functor) and adjunctions, we obtain the equivalence of functors in \(T\):
\[
\begin{aligned}
    R\underline{\mathrm{Sol}}_{\D}(A)(T)
    &\simeq
    \operatorname{Map}_{\mathsf{QCAlg}(X_{\DR})}\big(\operatorname{Sym}(\mathcal{M}),\; (X_{\DR}\times T)_*\mathcal{O}_{X_{\DR}\times T}\big) \\
    &\simeq
    \operatorname{Map}_{\mathsf{QCoh}(X_{\DR})}\big(\mathcal{M},\; (X_{\DR}\times T)_*\mathcal{O}_{X_{\DR}\times T}\big)\\
    &\simeq
    \operatorname{Map}_{\mathsf{QCoh}(\mathrm{pt})}\big(R\Gamma_{\DR}(X,\mathcal{M}^\vee),\; \Gamma(T,\mathcal{O}_T)\big),
\end{aligned}
\]
where in the last line we used global duality: compactness of \(\mathcal{M}\) implies \(\mathcal{M}^\vee\) is bounded with coherent cohomology and
\[
    \operatorname{Map}_{\mathsf{QCoh}(X_{\DR})}\big(\mathcal{M},\; (X_{\DR}\times T)_*\mathcal{O}_{X_{\DR}\times T}\big)
    \simeq
    \operatorname{Map}\big(R\Gamma_{\DR}(X,\mathcal{M}^\vee),\; \Gamma(T,\mathcal{O}_T)\big).
\]
Therefore the functor of points of \(R\underline{\mathrm{Sol}}_{\D}(A)\) is represented by the affine derived scheme
\[
    \mathsf{Spec}\Big(\operatorname{Sym}^\bullet\big(R\Gamma_{\DR}(X,\mathcal{M}^\vee)\big)\Big),
\]
i.e. the free cdga on the finite complex \(R\Gamma_{\DR}(X,\mathcal{M}^\vee)\). Since \(X\) is proper and \(\mathcal{M}\) is coherent (compact), the complex \(R\Gamma_{\DR}(X,\mathcal{M}^\vee)\) is perfect of finite amplitude and finite-dimensional cohomology over \(\mathbb C\). Hence the resulting affine derived \(\mathbb C\)-scheme is finite type.

\end{proof}
As an example, taking further $A=\mathrm{Sym}(\M)$ with $\M=\mathrm{ind}(E)$ for a vector bundle $E,$ then
$\mathbb{R}\mathrm{Sol}(\mathrm{Sym}(\mathrm{ind}(E)))\simeq \mathsf{Spec}(\mathrm{Sym}^{\bullet}(R\Gamma(X_{\DR},ind E^{\vee})),$ which computes gives $\mathsf{Spec}(\mathrm{Sym}^{\bullet}(R\Gamma(X,E^{\vee})),$ since $R\Gamma(X_{\DR},\mathrm{ind}(E)^{\vee}))$ is $R\Gamma(X,E^{\vee}).$

We might say that $E$ is \emph{strongly $\D$-finitary} if $\J(E)$ is also locally almost of finite type, but typically this is too strong a condition. We relax it by asking that only the fibers of $\EuScript{Y}$ mapping into $\J(E)$ are finite.
\begin{proposition}
    \label{prop: Relatively D-laft}
    Consider an $X_{\DR}$-space $\EQ$ and a laft prestack $\EuScript{X}$ such that for $X$ an eventually co-connective affine scheme, there is a pull-back diagram in $\PS_{/X_{\DR}}:$
\begin{equation}
    \label{eqn: Relative Laft}
    \begin{tikzcd}
        \EuScript{X}\arrow[d]\arrow[r] & \EQ\arrow[d]\arrow[r] & \arrow[dl] \J(E)
        \\
        \EuScript{U}\arrow[r] & X_{\DR}& 
    \end{tikzcd}.
\end{equation}
Then $\EQ$ is $\D$-finitary.
\end{proposition}
In the context of Proposition \ref{prop: Relatively D-laft}, one may calculate the relative pro-cotangent complex of $\EQ$ from the laftness of $\EuScript{X}$ and the fact that if $X$ is coconnective then $\EuScript{U}$ is truncated. Moreover, this result holds if if we assume $E$ is also laft, via Proposition \ref{Laft Descent}.

\begin{corollary}
\label{cor: Laft to DAfp}
If $X$ is a proper $k$-scheme of finite type the
functor \emph{(\ref{eqn: PreStk dR Pushforward})} restricts to $\mathsf{Jets}_{\DR}:\PS_{/X}^{aft}\rightarrow \PS_{/X_{\DR}}^{\D-\text{fin}}.$ In particular if $E\rightarrow X$ is affine over $X$, then  $\J(E)$ is affine over $X_{\DR}.$

\end{corollary}
\begin{proof}[Sketch of proof]
We prove the second claim, following our discussion of compact objects in commutative monoids in  ind-coherent sheaves on $X_{\DR},$ as in \ref{ssec: Compactness}. Namely, that if $E=\mathsf{Spec}_{/X}(A_0^{\bullet})$ is a derived affine scheme with $A_0\in \mathsf{QCohCAlg}(X)^{\leq 0}.$ Then, we claim that
$$q_{\DR,*}(E)\simeq\mathsf{Spec}_{\D}\big(oblv_{Comm}^{\ell,\leq,\mathrm{L}}(A_0^{\bullet}))\big).$$

Here we write
$oblv^{\ell}:\mathsf{Mod}(\D_X)^{\leq 0}\rightarrow \mathsf{QCoh}(X)^{\leq 0},$
and 
$$oblv_{Comm}^{\ell}:\mathsf{CAlg}(\D_X)^{\leq 0}\rightarrow Comm\big(\mathsf{QCoh}(X)^{\leq 0}\big),$$
the corresponding functor on commutative monoids, where 
$$oblv^{\ell}=(q_{\DR}^X)_{\mathsf{QCoh}}^*:\Q(X_{\DR})^{\leq 0}\rightarrow \Q(X)^{\leq 0}.$$
Then, 
$oblv_{Comm}^{\ell,\mathrm{L}},$
is the left adjoint, commuting with colimts.
\end{proof}
In particular, if $A_0^{\bullet}=\mathrm{Sym}(\mathcal{E})$ for some vector bundle or coherent sheaf $\mathcal{E},$
$$\big(oblv_{Comm}^{\ell}\big)^{\mathrm{L}}\big(\mathrm{Sym}_{\mathcal{O}_X}(\mathcal{E})\big)\simeq \mathrm{Sym}^!\big((q_{\DR}^X)_*^{\Q}\mathcal{E}\big),$$
where we recall that $(q_{\DR}^X)_*^{\Q}$ is $\mathrm{ind}_{\D_X}^{\ell}.$

\begin{remark}

Definition \ref{Definition: AdmissiblyDFinitary} relies on the existence of a certain adjoint functor, but in our cases we do have push-forwards $q_{\DR,*}$ e.g. when $q_{\DR}$ is proper, thus a left-adjoint occurring for instance, if $X$ is derived Artin stack as per Observation \ref{obs: Artin}.
\end{remark}

When $E\in \PS_{/Y}$ is $\D$-finitary we denote its dualizable subcategory by $$\mathsf{Shv}_{fin}^!(E):=\big(q_{\DR}^Y\big)_*\mathsf{Shv}(E)\hookrightarrow \mathsf{Shv}^!\big(\J(E)\big).$$

Since $q_{\DR}^Y:Y\rightarrow Y_{\DR}$ induces a functor 
$$(q_{\DR}^Y)_{\mathsf{Shv}}^!:\mathsf{Shv}(Y_{\DR})\rightarrow \mathsf{Shv}(Y),$$
which is nothing but the $!$-pullback of sheaves on prestacks, it
extends to sheaves on relative prestacks denoted the same
\begin{equation}
    \label{eqn: Relative !-pullback functor}
    (q_{\DR})_{\mathsf{Shv}}^!:\mathsf{Shv}_{Y_{\DR}}\big(\J(E)\big)\rightarrow \mathsf{Shv}_Y^!\big(q_{\DR}^{Y,!}\J(E)\big)\simeq \mathsf{Shv}_Y\big(\JetY(E)\big).
\end{equation}

The $\infty$-category on the right-hand side of  (\ref{eqn: Relative !-pullback functor}) is understood as sheaves on the derived prestack $\JetY(E),$ which come with a canonical $\mathcal{O}_Y$-module structure.

Furthermore, structure maps (\ref{eqn: Jet Structure maps}) induce a $*$-pushforward functor
\begin{equation}
    \label{eqn: Relative *-pushforward functor}
    (\mathsf{q})_*^{\mathsf{Shv}}:\mathsf{Shv}_{Y}\big(\JetY(E)\big)\rightarrow \mathsf{Shv}(Y).
\end{equation}

\begin{proposition}
\label{Sheaf push-forward is right-lax monoidal}
 Functor \emph{(\ref{eqn: Relative *-pushforward functor})} is right-lax symmetric monoidal.
\end{proposition}
By \autoref{Sheaf push-forward is right-lax monoidal} we obtain that (\ref{eqn: Relative *-pushforward functor}) sends commutative monoids to commutative monoids, and we denote this enunciated situation by
\begin{equation}
\label{eqn: CAlgShv q-pf}
(\mathsf{q})_*^{\mathsf{CAlgShv}}:\mathsf{CAlgShv}_{Y}^!\big(\JetY(E)\big)\rightarrow \mathsf{CAlgShv}^!(Y).
\end{equation}

The functor (\ref{eqn: CAlgShv q-pf}) specialized to the case of quasi-coherent sheaves is written $(\mathsf{q})_*^{\mathsf{CAlgQCoh}}:\mathsf{CAlg}\mathsf{QCoh}(\JetY(E)\big)\rightarrow \mathsf{CAlgQCoh}(Y).$

Via the (co)-Cartesian Grothendieck construction, we have that $\mathsf{Shv}^!\big(\JetY(E)\big)$ is equivalent to 
$$\mathsf{Shv}^!\big(\JetY(E)\big)=\underset{A\in\mathsf{CAlg}^{\leq 0},f\in \JetY(E)(A)}{\mathrm{lim}}\mathsf{Shv}^!\big(Spec(A)\big),$$
with the limit taken in the $\infty$-category $\mathsf{Cat}_{\mathrm{pres},\mathrm{L}}^{\infty}.$ 
Since $\mathsf{CAlg}(-)$ is symmetric monoidal and commutes with limits, there
are equivalences of $\infty$-categories
\begin{eqnarray*}
\mathsf{CAlgShv}_Y\big(\JetY(E)\big)&\simeq& \underset{(\mathsf{dAff}_{/\JetY(E)})^{op}}{\mathrm{lim}}\mathsf{CAlg}\big(\mathsf{Shv}(\mathsf{Spec}(A)\big)
\\
&\simeq& \underset{(\mathsf{dAff}_{/\JetY(E)})^{op}}{\mathrm{lim}}\mathsf{CAlg}(\mathcal{O}_Y)_{A/},
\end{eqnarray*}
where the limit is taken over $U\in \mathsf{dAff}$ with a map to $\JetY(E),$ in the category $\PS.$

By Proposition \ref{Sheaf push-forward is right-lax monoidal} consider $\mathcal{O}_{\JetY(E)}$ as an object of $\mathsf{Shv}\big(\JetY(E)\big)$ and note it's image under the functor (\ref{eqn: CAlgShv q-pf}),
\begin{equation}
    \label{eqn: D-Algebra Constructed}
\mathcal{A}^{\bullet}:=\mathsf{q}_*^{\mathsf{CAlgShv}}\big(\mathcal{O}_{\JetY(E)}\big)\in \mathsf{Shv}(Y),
\end{equation}
naturally has the structure of a commutative algebra on $Y.$ Moreover, by push-forward of sheaves on prestacks relative to $Y$ along $q_{\DR}^Y,$ we get that 
$$\mathcal{A}_Y^{\bullet}:=(q_{\DR}^Y)_*^{\mathsf{Shv}}\mathcal{A}^{\bullet},$$
is a sheaf of commutative algebras on $Y$ with a flat connection. 
\begin{proposition}
\label{D-PreStack Construction}
The derived spectrum $Spec_{\mathcal{D}}\big(\mathcal{A}_Y^{\bullet}\big)$ is a $\mathcal{D}_Y$-prestack.
\end{proposition}
Moreover, we have the following.
\begin{proposition}
\label{Factoring of q pushforward}
The functor $\mathsf{q}_*$ acts as
\[
\begin{tikzcd}
\mathsf{Shv}\big(\JetY(E)\big)\arrow[dr,"\mathsf{q}_*^{\mathsf{Shv}}"] \arrow[r]& \mathcal{A}^{\bullet}-\mathsf{Mod}\big(\mathsf{Shv}(Y)\big)\arrow[d,"\mathbf{oblv}_{\mathcal{A}^{\bullet}}"]
\\
& \mathsf{Shv}(Y).
\end{tikzcd}
\]
\end{proposition}
 Recall (\ref{eqn: Jet Structure maps}), diagram (\ref{eqn: Jet pb}) and Proposition \ref{prop: Def but not laft proposition}. 

\begin{proposition}
\label{Sheaves on Jet prestack over Y Description}
There is an equivalence
$$\mathsf{Shv}^!\big(\mathrm{Jet}_Y^{\infty}(E)\big)\xrightarrow{\simeq} \mathbf{q}_*^{\mathsf{CAlgShv}}\big(\mathcal{O}_{\mathrm{Jet}_Y^{\infty}(E)}\big)\text{-}\mathsf{Mod}\big(\mathsf{Shv}(Y)\big).$$
Moreover, suppose that $E\rightarrow Y$ is $\D$-finitary. Then the canonical map
$$\mathsf{Shv}^!\big(\J(E)\big)\otimes_{\mathsf{Shv}(Y_{\DR})}\mathsf{Shv}(Y)\rightarrow \mathsf{Shv}^!\big(\J(E)\times_{Y_{\DR}}Y\big),$$
is an equivalence.
\end{proposition}

\begin{proof}
Notice that $\mathcal{O}_Y$ is compact and $\mathsf{Shv}(Y_{\DR})$ is dualizable. The result then follows from \cite[Proposition. 3.5.3]{GR17a}. 
In particular, we observe that equivalences,
$\mathsf{Shv}\big(\J(E)\big)\otimes_{\mathsf{Shv}(Y_{\DR})}\mathsf{Shv}(Y)$ is equivalent to 
$$\mathsf{p}_*^{\mathsf{Shv}}\big(\mathcal{O}_{\J(E)}\big)\text{-}\mathsf{Mod}\big(\mathsf{Shv}(Y_{\DR})\otimes_{\mathsf{Shv}(Y_{\DR})}\mathsf{Shv}(Y),$$
and therefore  to $\big(q_{\DR}^Y\big)_{\mathsf{Shv}}^!\big(\mathsf{p}_*^{\mathsf{Shv}}(\mathcal{O}_{\J(E)})\big)\text{-}\mathsf{Mod}\big(\mathsf{Shv}(Y)\big),
$
using the functors associated to the maps (\ref{eqn: Jet Structure maps}),
together with the equivalence 
$$\mathsf{q}_*^{\mathsf{CAlgShv}}\big(\mathcal{O}_{\JetY(E)}\big)\simeq \big(q_{\DR}^Y\big)_{\mathsf{Shv}}^!\big(\mathsf{p}_{*}^{\mathsf{Shv}}\mathcal{O}_{\J(E)}\big),$$ in $\mathsf{Shv}(Y)$ which is moreover an equivalence of commutative monoids.
\end{proof}
By push-forward along $q_{\DR}^Y,$ obtain $\mathcal{A}^{\bullet}$-modules in sheaves on $Y_{\DR}.$

\subsection{Cotangent and tangent $\D$-complex of a derived NLPDE}
\label{sssec: Cotangent and Tangent D-Complex of a DNLPDE}
The main use of the finiteness conditions introduced in the previous subsection are given by the following result.

\begin{theorem}
\label{thm: RSol is laft-def}
Consider a $\D$-finitary non-linear PDE $\EQ\rightarrow X_{\DR}.$ Its derived space of solutions $\RS(\EQ)$ restricts to a functor
$$\RS(\EQ):\PS_{/X_{\DR}}\rightarrow \PS_{\mathbb{C}}^{\mathrm{\mathrm{laft-def}}}.$$
Furthermore, suppose that for a non-negative integer $k$ we have that $Z$ is $\D$-finitary with the property that $\mathbb{L}_{Z}$ exists as an object of $\mathrm{Shv}^{\leq k}.$ Then, $\RS(\EQ)$ admits a cotangent complex as an object of $\mathrm{Shv}^{\leq k}$ as well.
\end{theorem}
\begin{proof}
Since $\EQ\to X_{\DR}$ is $\D$-finitary, we have that $\RS(\EQ)$ is laft. Let $\widetilde{u}_T\in \mathrm{Maps}(T,\RS(\EQ))$ be a $T$-family of solutions, corresponding to $u_T\in \mathrm{Maps}(T\times X_{\DR},\EQ).$ The pro-cotangent space to $\RS(\EQ)$ at $\widetilde{u}_T$ is the $\infty$-functor
\begin{equation}
    \label{TangentRSol}
\mathsf{T}_{\widetilde{u}_T}^*\big(\RS(\EQ)\big):\mathsf{QCoh}(T)\to \mathsf{Vect}_{\mathbb{C}}^{\leq 0},\end{equation}
sending a quasi-coherent sheaf $\EuScript{E}$ to the connective complex,
$$\mathsf{T}_{\widetilde{u}_T}^*\big(\RS(\EQ)\big)(\EuScript{E})\simeq \tau^{\leq 0}\big(Maps_{\mathsf{QCoh}(T\times X_{\DR})}(\mathsf{T}_{u_T}^*(\EQ/X_{\DR}),\mathscr{E}\boxtimes \mathcal{O}_{X_{\DR}})\big).$$
Since $Z$ is corepresentable and admits $\mathbb{L}_{\EQ/X_{\DR}}$ which is $-k$-connective, then we claim that \eqref{TangentRSol} is also $(-k)$-connective. That is, we must prove that this functor is corepresentable by an object of $\mathsf{QCoh}^{\leq k}.$ For this, it is enough to show that there exists a left-adjoint to $\mathscr{E}\mapsto \mathscr{E}\boxtimes \mathcal{O}_{X_{\DR}}$ acting from $\mathsf{QCoh}(T)\to \mathsf{QCoh}(T\times X_{\DR}).$ Since $X$ is smooth (hence $X\to X_{\DR}$ is flat effective epimorphism), we have 
$$\Q(T\times X_{\DR})\simeq \Q(T)\otimes \Q(X_{\DR}),$$
via \cite[Prop 1.4.4]{QCoh}. It suffices to consider $T=\mathrm{pt}$, and notice that $\mathsf{QCoh}(X_{\DR})$ is comapctly generated, so we can look at $\mathsf{QCoh}(X_{\DR})^{\omega}.$ It is enough to construct the left-adjoint on compact objects. Let $\mathcal{M}\in \Q(X_{\DR})^{\omega},$ be such a compact object in complexes of left $\D_X$-modules on $X.$ Then, the complex $Maps_{\Q(X_{\DR}}(\mathcal{M},\mathcal{O}_{X_{\DR}})\in\mathsf{Vect}_{\mathbb{C}}$ has finitely-many nonzero cohomology sheaves which are finite-dimensional since $\mathcal{M}$ is compact. In other words $Maps_{\Q(X_{\DR}}(\mathcal{M},\mathcal{O}_{X_{\DR}})\in\mathsf{Vect}_{\mathbb{C}}^{\omega}$ i.e. is itself compact. Then, the desired left-adjoint simply sends the monoidal unit $\mathbb{C}$ to the monoidal dual complex of $Maps_{\Q(X_{\DR}}(\mathcal{M},\mathcal{O}_{X_{\DR}})\in\mathsf{Vect}_{\mathbb{C}}.$

\end{proof}
Suppose $E\rightarrow Y$ is $\D$-finitary and assume that $Y$ is a complex analytic manifold (the case of a derived base $Y$ is treated in § \ref{sec: Derived Linearization and the Equivariant Loop Stack}). Consider the pullback of $\mathbb{L}_{\J(E)/Y_{\DR}}$ as in equation (\ref{eqn: D-Relative Cotangent}) under the functor (\ref{eqn: Relative !-pullback functor}), denoted by
\begin{equation}
    \label{eqn: Relative Cotangent D-complex}
    \scalemath{1.10}{\EuScript{L}}(\J(E)):=(q_{\DR})_{\mathsf{Shv}}^!\big(\mathbb{L}_{\J(E)/Y_{\DR}}\big).
    \end{equation}
We suppose $E$ admits deformation theory relative to $Y_{\DR}$ i.e. $\mathbb{L}_{\J(E)/Y_{\DR}}\simeq\mathbb{L}_{\J(E)}.$ 
The object \emph{(\ref{eqn: Relative Cotangent D-complex})} is an object of the $\infty$-category of $\mathcal{A}^{\bullet}\otimes_{\mathcal{O}_Y}\mathcal{D}_Y$-modules over $Y,$ with $\mathcal{A}^{\bullet}$ as in \emph{(\ref{eqn: D-Algebra Constructed})}.  
More exactly, it is an object of $\mathsf{Shv}\big(\JetY(E)\big)$, and this interpretation follows from the description in \autoref{Sheaves on Jet prestack over Y Description}, equation (\ref{eqn: D-Algebra Constructed}) and the definition of its push-forward $\mathcal{A}_Y^{\bullet}$.

Indeed, notice that the relative cotangent complex is an object of
$$\mathsf{Shv}_{Y_{\DR}}\big(\J(E)\big)\simeq \mathsf{Shv}\big(\J(E)\times Y_{\DR}\big)\simeq \mathsf{Shv}\big(\J(E)\big)\otimes_{\mathsf{Shv}(Y)}\mathsf{Shv}\big(Y_{\DR}\big),$$
and since $Y$ is smooth, so that $\IC(Y)=\Q(Y),$ this is equivalent to the $\infty$-category
$$\PC\big(\J(E)\big)\otimes_{\mathsf{QCoh}(Y)} \mathsf{Mod}(\mathcal{D}_Y).$$
The $!$-pullback of sheaves then reads as
$$\PC\big(\J(E)\big)\otimes_{\Q(Y)} \mathsf{Mod}(\mathcal{D}_Y)\rightarrow \PC\big(\JetY(E)\big)\otimes_{\Q(Y)}\mathsf{Mod}(\mathcal{D}_Y).$$
Composition with the functor (\ref{eqn: Relative *-pushforward functor}), gives that the $!$-pullback of the relative cotangent complex is thus an object of $\mathsf{Mod}\big(\mathcal{A}^{\bullet}\big)\otimes_{\Q(Y)}\mathsf{Mod}(\mathcal{D}_Y),$ thus
appropriately interpreted in 
$\mathsf{Mod}\big(\mathcal{A}^{\bullet}\otimes_{\mathcal{O}_Y}\mathcal{D}_Y\big)$.

More precisely, $ev_{u_T}^{\sharp}\mathbb{L}_{\EQ/X_{\DR}},$ exists and is a perfect complex with dual $\mathbb{T}_{\EQ/X_{\DR}}\in \mathsf{Mod}_Z(\mathsf{IndCoh}(X_{\DR})),$ such that $ev_{u_T}^{!,\mathsf{IndCoh}}\mathbb{T}_{\EQ/X_{\DR}}\in \IC(T\times X_{\DR})^{\omega}.$ Let $\mathbb{D}_Z$ denote the induced duality functor on $\mathsf{Mod}_Z(\IC(X_{\DR})).$
There is a natural self-duality,
\begin{equation}
    \label{eqn: DSerreVerdier}
\mathbb{D}_{T\times X_{\DR}}:\mathsf{IndCoh}(T\times X_{\DR})^{\omega}\simeq \big(\mathsf{IndCoh}(T\times X_{\DR})^{\omega}\big)^{op}.
\end{equation}
Putting $\mathbb{L}_{\EQ/X_{\DR},u_T}:=\mathbb{D}_{T\times X_{\DR}}(ev_{u_T}^{!,\mathsf{IndCoh}}\mathbb{T}_{\EQ/X_{\DR}}),$ due to compatibility between dualities and $!$-pullbacks, there is an equivalence
$$\mathbb{D}_{T\times X_{\DR}}(ev_{u_T}^{!,\mathsf{IndCoh}}\mathbb{T}_{\EQ/X_{\DR}})\simeq ev_{u_T}^{\sharp}\big(\mathbb{D}_Z\mathbb{T}_{\EQ/X_{\DR}}\big).$$

The object (\ref{eqn: Relative Cotangent D-complex}) is called the \emph{$\D$-cotangent complex} of the derived $\D$-jet prestack $\J(E).$

As $Y$ is smooth $\J(E)$ is also smooth so that $\IC(\J(E)\times Y_{\DR})$ is given by $\mathcal{O}_{\J(E)}\otimes\mathcal{D}_Y$-modules and as this ring has finite cohomological dimension, every finitely generated module in its heart is compact object in $\mathcal{A}\otimes\mathcal{D}_Y$-modules. 

\begin{definition}
\label{Definition: Tangent Local Verdier Dual}
  \normalfont  Let $E\rightarrow Y$ satisfy the assumptions from the beginning of § \ref{sssec: Cotangent and Tangent D-Complex of a DNLPDE}. The dual object to the $\D$-cotangent complex is the \emph{$\D$-geometric tangent complex} given by the (leftified) \emph{local Verdier $\D$-dual}
$$\mathsf{T}_{\J(E)}^{\ell}:=\scalemath{1.10}{\EuScript{L}}_{\J(E)}^{\circ}\otimes\omega_Y^{-1}.$$
\end{definition}
We have an extension of the sequence (c.f. §§\ref{sssec: D-Geom Tangent and Cotangent Complexes}),
$$\mathbb{T}_{\EQ/\J(E)}\rightarrow \mathbb{T}_{\EQ}\rightarrow i^*\mathbb{T}_{\J(E)}\rightarrow \mathbb{T}_{\EQ/\J(E)}[1],$$
to the leftified local Verdier duals for $i:\EQ\rightarrow \J(E).$

This follows from the fact that given $f:\EQ \rightarrow \J(E),$ in pre-stacks over $X_{\DR},$ there is a cofiber sequence of pro-cotangent complexes
$$f^*\mathbb{L}(\J(E))\rightarrow \mathbb{L}(\EQ)\rightarrow \mathbb{L}(\EQ/\J(E)),$$
and by assuming our objects are $\D$-finitary, the duality operation sends pro-coherent objects to the opposite category of ind-coherent sheaves via the commutative diagram
\[
\begin{tikzcd}
\IC(\J(E))^{op}\arrow[d]\arrow[r]& \IC(\EQ)^{op}\arrow[d]
\\
\PC(\J(E))\arrow[r] & \PC(\EQ).
\end{tikzcd}
\]

The objects and their dualities described here are used to produce linearization sheaves associated with $\D$-finitary non-linear \textsc{pde}. In the following we describe pull-backs of tangent and cotangent complexes for trivially parameterized solutions i.e. $\EuScript{F}\rightarrow X$ to be $id_X:X\rightarrow X\in \PS_{/X}$ in the sense of §§ \ref{Parameterized Solution and Derived DO}.
\begin{remark}
\label{RSol Tangent Remark}
    These facts, together with Theorem \ref{thm: RSol is laft-def}, provide enough information to determine $\mathbb{T}[\RS(\EQ)]$.
\end{remark}

\subsection{The pull-back along solutions}
\label{ssec: The pull-back along solutions}
Let $E\in \PS_{/X}$ and consider a derived non-linear \textsc{pde} $\EQ$ on sections of $E.$ By \autoref{Derived Sections are Derived Flat Sections of Jet Space}, 
$$\Sect_(X,E)\simeq \mathbb{R}\mathrm{Sect}^{\nabla-\mathrm{cofree}}\big(X,E\big)\simeq \RS\big(\J(E)\big).$$
A classical global solution $s$ is a morphism $s:X\rightarrow p^*\J(E)$ which is also a map of de Rham spaces. Parameterized solutions are given by $U$-points $s_U:U\rightarrow \RS\big(\J(E)\big),$ with corresponding point in the solution space given by its infinite jet prolongation (\ref{eqn: Jet prolongation})
$\mathfrak{j}_{\infty}(s_U):U\times X_{\DR}\rightarrow \J(E)$ and in the space of homotopy co-free sections, $\mathfrak{j}_{\infty}(s_U)'.$
Consider the diagrams in $\PS_{/X_{\DR}}$,
\[
\adjustbox{scale=.95}{
\begin{tikzcd}
U\times X\arrow[d,"\mathbf{1}_U\times q_{\DR}^X"]\arrow[rr, bend left, "\mathfrak{j}_{\infty}(s_U)'"] \arrow[r,"s\times \mathbf{1}_X"] & \J(E)\times_{X_{\DR}}X\arrow[d]\arrow[r] & E
\\
U\times X_{\DR}\arrow[r,"\mathfrak{j}_{\infty}(s_U)"] & \J(E)& 
\end{tikzcd}
\hspace{2mm},
\begin{tikzcd}
U\times X\arrow[d,"\mathbf{1}_U\times q_{\DR}^X"] \arrow[r,"s\times \mathbf{1}_X"] & \EQ\times_{X_{\DR}}X\arrow[d]
\\
U\times X_{\DR}\arrow[r,"\mathfrak{j}_{\infty}(s_U)"] & \EQ& 
\end{tikzcd}
}
\]
Solutions to $\EQ$ appearing in the right-hand side are induced from those in the diagram on the left.

We have $\mathbb{T}_{\J(E)}\in\IC\big(\J(E)\big),$ and $\mathbb{L}_{E}\in\PC\big(E\big).$ Moreover, there are induced functors by pull-back of ind-coherent sheaves on prestacks
$$\big(\mathfrak{j}_{\infty}(s_U)\big)_{\mathsf{IndCoh}}^!:\mathsf{IndCoh}\big(\J(E)\big)\rightarrow \mathsf{IndCoh}\big(U\times X_{\DR}\big),$$ as well
$$\big(\mathfrak{j}_{\infty}(s_U)'\big)^!:\PC\big(E\big)\rightarrow \PC\big(U\times X\big),$$
under which we can consider the image $\big(\mathfrak{j}_{\infty}(s_U)'\big)_{\PC}^!\mathbb{L}_{E},$ and its subsequent push-forward,
$$\big(\mathbf{1}_U\times q_{\DR}^X\big)_{*}\circ \big(\mathfrak{j}_{\infty}(s_U)'\big)^!(\mathbb{L}_{E})\in\PC\big(U\times X_{\DR}\big).$$
Since the diagram is Cartesian, it follows that
$$\big(\mathfrak{j}_{\infty}(s_U)\big)_{\mathsf{IndCoh}}^!\big(\mathbb{L}_{\J(E)}\big)\simeq \big(\mathbf{1}_U\times q_{\DR}^X\big)_{*}\circ \big(\mathfrak{j}_{\infty}(s_U)'\big)^!(\mathbb{L}_{E})\in\PC\big(U\times X_{\DR}\big).$$

\begin{proposition}
\label{T and D-T equivalence}
    Suppose that $\EuScript{Y}$ is a $\D$-finitary derived non-linear \textsc{pde} imposed on sections of a laft pre-stack over $X$, a smooth variety of dimension $n.$ Then $s_U^*\EuScript{T}_{\EQ}^{\ell}$ is equivalent to 
    $\mathfrak{j}_{\infty}(s_U)^!\mathbb{T}_{\EQ}\otimes \omega_X[-2n],$ functorially in $U$.
\end{proposition}

We now study describe the linearization as a compact object in ind-coherent sheaves over $X$ which is not necessarily smooth.

\begin{proposition}
\label{Linearization D-Module Proposition}Suppose that $\EQ$ is affine $\D$-finitary over an eventually coconnective affine scheme $X$ and consider the object $\scalemath{1.10}{\EuScript{L}}(\J(E))$ as in (\ref{eqn: Relative Cotangent D-complex}).
Then a choice of any classical solution $\overline{\varphi}^{cl}$ as above yields a well-defined, perfect object $\mathsf{T}(\J(E))_{\varphi^{\mathrm{cl}}},$ given by
$$(\overline{\varphi}^{\mathrm{cl}}\times\mathrm{id})_{\IC}^!\big(
(\Upsilon_{X_{\DR}}^{-1}\mathbb{D}^{loc-verd}(q_{\DR})_{\mathsf{Shv}}^!\big(\mathbb{L}_{\J(E)/X_{\DR}})\big)\in \mathsf{Perf}(X_{\DR}).$$
\end{proposition}
\begin{proof}
By \autoref{D-PreStack Construction} consider the
  pull-back along the functor induced by the solution $\overline{\varphi}^{\mathrm{cl}}$, given by
$$(\overline{\varphi}^{\mathrm{cl}}\times\mathrm{id})_{\IC}^!:\IC\big(Spec_{\mathcal{D}}(\mathcal{A}^{\bullet})\big)\underset{\Q(X)}{\otimes}\IC(X_{\DR})\rightarrow \IC(X_{\DR}),$$
of the leftified local Verdier dual and applying the (inverse) duality functor $\Upsilon_{X_{\DR}}.$ Just remark that the equivalence $\Upsilon_{X_{\DR}}$ is not $t$-exact, but it is $t$-bounded \cite[Prop. 4.4.4.]{GR14}. In the case $X$ is smooth of dimension $n$, this functor is $t$-exact up to a shift by $[n].$
In any case, the functor does induce an  equivalence
$\Upsilon_{X_{\DR}}:\mathsf{Coh}(\D_X)\simeq \mathsf{Coh}(\mathcal{D}_X^{op}).$ 

For the second part, consider the infinite jet of a classical $U$-parameterized solution, $j_{\infty}(\varphi_U):U\rightarrow \EQ,$ which is the jet-prolongation (\ref{eqn: Jet prolongation}) of a $U$-point $\varphi_U:U\rightarrow \RS(\EQ).$ As objects of $\PS_{/X_{\DR}}$, we have the universal family
\[
\begin{tikzcd}
    U\times X_{\DR}\arrow[d,"q_U"]\arrow[r,"\varphi_U"] & \RS(\EQ)\times X_{\DR}\arrow[d,"\rho"] \arrow[r,"ev_{\varphi_U}"] & \EQ
    \\
    U\arrow[r,"\varphi_U"] & \RS(\EQ) & 
\end{tikzcd}
\]
From this diagram since $\EQ$ is $\D$-finitary, both $\mathbb{L}_{\EQ}$ and $\mathbb{L}\big(\RS(\EQ)\big)$ exists as objects of $\mathsf{ProCoh}(\EQ),$ and for every solution we have
$$\mathbb{L}_{\RS(\EQ),\varphi_U}:=\varphi_U^!\mathbb{L}_{\RS(\EQ)}\simeq \varphi_U^! \rho_! ev_{\varphi_U}^!\mathbb{L}_{\EQ}\in \mathsf{ProCoh}(U).$$
Identifying the pro-coherent complexes with their ind-coherent duals e.g. $\mathbb{T}_{\EQ,\varphi_U}$ is given by 
$\mathbb{D}^{loc-ver}(ev_{\varphi_U}^!\mathbb{L}_{\EQ})$ with $\mathbb{D}^{loc-ver}$ the local Verdier duality,
$$\mathsf{ProCoh}(\RS(\EQ)\times X_{\DR})\simeq \IC(\RS(\EQ)\times X_{\DR})^{op}.$$

\end{proof}

We now have an analog of \autoref{Relative D-cotangent Complex}. 
\begin{proposition}
\label{Derived Analytic Relative Linearization Complex}
Let $X$ be a complex analytic manifold and suppose that $\EQ\xrightarrow{i}\J(E)$ is a derived non-linear \textsc{pde}. Then for every classical solution $\varphi$ there is a map in $\IC(X_{\DR}),$
$$\alpha:\mathsf{T}^{\ell}(\EQ)_{\varphi^{\mathrm{cl}}}\rightarrow i^*\mathsf{T}^{\ell}(\J(E))_{\varphi^{\mathrm{cl}}}.$$
\end{proposition}

\begin{remark}
    The object as defined by Proposition \ref{Linearization D-Module Proposition} is the \emph{linearization sheaf (complex)} of the non-linear system $\mathsf{Sol}(\mathcal{B})$ along the chosen solution $\varphi.$ Considering Proposition \ref{Derived Analytic Relative Linearization Complex}, the homotopy fiber $hofib(\alpha)$ defines the relative $\D$-geometric linearization complex.
\end{remark}

\begin{proposition}
\label{Higher Symmetry Result}
Let $i:\EQ\hookrightarrow \mathsf{Jets}_X^{\infty}(E)$ be a derived non-linear \textsc{pde} with canonical sequence of tangent complexes
defining the $\mathcal{O}_{\EQ}[\mathcal{D}_X]$-module $\mathbb{T}_{\EQ}$ as a sub-module of $i^*\mathbb{T}_{\mathsf{Jets}_X^{\infty}(E)}.$ Suppose $\EQ$ is a classical $\D$-subscheme, whose defining ideal $\mathcal{I}$ is differentially generated. Then there is an isomorphism 
$H_{\D}^0(\mathbb{T}_{\EQ})\simeq \Theta_{\EQ}$ whose tangent sheaf is generated by evolutionary vector fields $E_{\xi}|_{\EQ}$ for $\xi\in p_{\infty}^*\Theta_{E/X}$ satisfying
$i^*\big(\ell_{\mathsf{F}_i}(\xi)\big)=0.$
\end{proposition}
\autoref{Higher Symmetry Result} states that in the smooth situation, vector fields are given by homogeneous solutions of the universal linearization of the operators $\mathsf{F}_i$ which generate the ideal $\mathcal{I}.$ 
\begin{example}
\label{Tangent Complex of Formal Lin over Loops}
Consider \autoref{Formal Linearization Example}. The derived linearization sheaf $\mathbb{T}_{\mathbb{R}\mathbf{Sol}_{\D_{\mathbb{A}^1}}(\mathsf{P})}$ viewed as a complex over $\mathcal{L}\mathbb{A}^1,$ can be identified with the complex of sheaves 
$$\mathbb{T}_{\mathbb{R}\mathbf{Sol}_{\D_{\mathbb{A}^1}}(\mathsf{P})}\simeq i^*\big(\mathcal{O}(\mathcal{L}\mathbb{A}^1)(\!(\EQ)\!)\rightarrow \mathcal{O}(\mathcal{L}\mathbb{A}^1)\big)(\!(\EQ)\!)\big),$$
in degrees $[0,1]$ given by the linear over $\mathcal{L}\mathbb{A}^1$-operator $lin(\mathsf{P})$ acting on functions $g(\EQ)$ by 
$$lin(\mathsf{P})g(\EQ)=\mathsf{F}\big(z,f(\EQ)+\epsilon g(\EQ),\partial f(x)+\epsilon\partial g(\EQ),\ldots\big)=0\hspace{2mm} mod\hspace{1mm} \epsilon^2,$$
where $i:\mathbb{R}\mathbf{Sol}(\mathsf{P})\rightarrow \mathcal{L}\mathbb{A}^1$ is the projection. The cotangent $\D$-complex will thus be in degrees $[-1,0]$ and the corresponding operator is the local Verdier duality of $lin(\mathsf{P}).$ In PDE language, it is nothing but the formal adjoint operator.
\end{example}

\section{Non-linear algebraic microlocal analysis}
\label{sec: Derived Non-linear Microlocal Analysis}
In this section we describe the derived version of the constructions in \S.~\ref{ssec: Non-linear Microlocal Analysis} and \S\S.~\ref{ssec: D-Algebraic Micro}.
\subsection{Characteristic varieties and singular supports}
\label{ssec: Characteristic Varieties}
In the complex domain conditions should be imposed on hypersurfaces for solutions of a
given analytic non-linear \textsc{pde} to have singularities along this hypersurface (see §§ \ref{sssec: On the non-linear generalizations}). For linear equations such a hypersurface must be a characteristic hypersurface for the operator. For quasilinear equations such a hypersurface must be also characteristic for the
linear part of the operator if the solution is not too singular, which practically means certain boundedness is required (see \autoref{ex: Boundedness}).

In the real domain, the propagation of
singularities for quasi-linear equations is performed along bicharacteristic curves of the linearized characteristic equation microlocally. In the language of $\D$-modules, this phenomena is captured by the notion of second microlocalization \cite{Laurent1985}.

\subsubsection{$V$-filtrations}
Following \cite{KashiwaraOshima1977}, denote
by $\mathcal{J}_V$ the subsheaf of $\mathcal{E}_X$ consisting of microdifferential operators of order $\leq 1$ whose symbol of order $1$ vanishes on $V$ i.e. $P\in\mathcal{E}_X(1)|_V$ with $\sigma_1(P)|_V=0.$
It generates a subsheaf of $\mathcal{E}_X$ denoted $\mathcal{E}_V$,
$$\mathcal{E}_V=\bigcup_{j\geq 0}\mathcal{J}_V^j.$$

\begin{remark}
Outside of the zero section of $\tau:T_VT^*X\rightarrow V$ the $\mathsf{V}$-filtration is given by $\mathsf{V}^j\mathcal{E}_X=\bigcup_{j}\mathsf{V}^j\mathcal{E}_X,$ with $\mathsf{V}_j\mathcal{E}_X:=(\mathsf{F}^j\mathcal{E}_X)\cdot \mathcal{E}_V.$
\end{remark}
An \emph{$\mathcal{E}_V$-lattice} in a coherent $\mathcal{E}_X$-module $\mathcal{M}$ is an $\mathcal{E}_V$-submodule $\mathcal{N}_0$ of $\mathcal{M}$ such that it is $\mathcal{E}_X(0)$-coherent and generates $\mathcal{M}$ over $\mathcal{E}_X.$ For a coherent $\mathcal{E}_X(0)$-submodule $\mathcal{N}_0$ of $\mathcal{M}$, define
$\mathcal{N}_0(m):=\mathcal{E}_X(m)\cdot \mathcal{N}_0$ for $m\in\mathbb{Z}.$ We say that $\mathcal{M}$ is \emph{regular along $V$} (or that it has regular singularities along $V$) if it has an $\mathcal{E}_V$-lattice locally on $\mathring{T}^*X.$

A coherent $\D_X$-module is said to be regular along $V$ if $\Char(\mathcal{M})$ is contained in $V$ and if $\mu(\mathcal{M})$ is regular along $\mathring{V}=V\backslash T_X^*X.$
There is a thick abelian subcategory of coherent $\mathcal{E}_X$-modules that are regular along $V,$ so a corresponding bounded derived $\infty$-subcategory of $\mathsf{Coh}(\mathcal{E}_X)$ whose cohomologies are regular along $V$; $\mathcal{M}^{\bullet}\in\mathsf{Coh}(\mathcal{E}_X)$ is regular along $V$ if $\mathcal{H}^k(\mathcal{M}^{\bullet})$ is regular along $V$ for each $k\in \mathbb{Z}.$ Then $\mathcal{M}^{\bullet}\in\mathsf{Coh}(\D_X)$ is regular along $V$ if $\mu(\mathcal{M}^{\bullet})\in \mathsf{Coh}(\mathcal{E}_X).$
A system with regular singularities along $V$ is supported by $V.$ 

\begin{example}
Let $Y\subset X$ be a complex analytic sub-manifold with ideal sheaf $\mathcal{I}_Y$ in $\mathcal{O}_X.$ For $V=T_Y^*X,$ one has
$\mathsf{V}_j\mathcal{E}_X=\sum_{k+\ell\leq j}(\mathsf{F}_k\mathcal{E}_X)\cdot \pi^{-1}(\mathsf{V}_{\ell}\mathcal{D}_X),$
where $\mathsf{V}_{\ell}\mathcal{D}_X:=\{P\in \D_{X|Y}|\forall k\in\mathbb{Z}, P\bullet \mathcal{I}_Y^k\subset \mathcal{I}_Y^{k-\ell}\}.$
\end{example}
Denote the $1$-micro-characteristic variety of a coherent $\mathcal{E}_X$-module by $\mathrm{Char}_V^1(\mathcal{M}),$ a subvariety of the normal bundle $\tau:T_VT^*X\rightarrow T^*X.$
More generally, there is a sheaf $\mathcal{E}_V^{2(s)}$ given as a Gevrey localization of $F_{\mathsf{V}}^{(s)}\mathcal{E}_X,$ with $s$-characteristic varieties
\begin{equation}
\label{MicrolocalizationClassical}
\mathrm{Char}_V^{(s)}(\mathcal{M}):=supp\big(\mathcal{E}_V^{2(s)}\otimes_{\tau^{-1}\mathcal{E}_{X|V}}\mathcal{M}|_V)\subset T_VT^*X.
\end{equation}
\subsection{$\D$-geometric micro-characteristics}
\label{sssec: D-Geometric Microcharacteristics}
Let $X$ be a complex analytic manifold and consider a derived $\D_X$-algebra $\mathcal{A}^{\bullet}$ that we assume admits a globally defined perfect tangent complex $\mathbb{T}_{\mathcal{A}}^{\ell}.$ Consider the 
composition
$$\eta:T_VT^*X\xrightarrow{\tau}T^*X\xrightarrow{\pi}X.$$
Pulling back $\mathbb{T}_{\mathcal{A}}^{\ell}$ as an $\mathcal{A}[\mathcal{D}_X]$-module and tensoring with $\mathcal{O}_{T_VT^*X}$ gives a \emph{$V$-microlocalization}, denoted by
\begin{equation}
    \label{eqn: Twisted V-Microlocalization}
\mu_{\mathcal{D},V}\big(\mathbb{T}_{\mathcal{A}}^{\ell}\big):=\mathcal{O}_{T_VT^*X}\otimes_{\eta^*\mathcal{A}[\tau^{-1}\mathrm{Gr}\mathcal{E}_V]}\eta^*\mathbb{T}_{\mathcal{A}}^{\ell},
\end{equation}
as an object of $\mathsf{Mod}\big(\mathsf{Spec}_{T_VT^*X}(\eta^*\mathcal{A}^{\bullet})\big).$
Note this relative spectrum is equivalent to
\begin{eqnarray*}
    \mathsf{Spec}_{T_VT^*X}(\eta^*\mathcal{A}^{\bullet})&\simeq& \mathsf{Spec}_{T_VT
^*X}\big(\tau^*\pi^*\mathcal{A}^{\bullet}\big)
\\
&\simeq& \mathsf{Spec}_{T^*X}(\pi^*\mathcal{A}^{\bullet})\times_{T^*X}T_VT^*X.
\end{eqnarray*}

Remark that $\mathcal{O}_{T_VT^*X}$ is naturally a $\tau^{-1}\mathrm{Gr}(\mathcal{E}_V)$-module, which itself is a Noetherian graded algebra.
Denoting $\mathcal{I}_V(k):=\mathcal{I}_V\cap\mathcal{O}_{T^
*X}(k)$ the sheaf ideal of sections of $\mathcal{O}_{T^*X}(k)$ vanishing on $V$ note that
$\mathrm{Gr}(\mathcal{E}_V)\simeq \bigoplus_{k\in\mathbb{Z}}\mathcal{I}_V^k(1)|_V,$
with the convention $\mathcal{I}_V^k(1)=\mathcal{O}_{T^*X}(k)$ for $k\leq 0.$
Consider the associated coherent graded ring
$\mathcal{K}_V(1):=\bigoplus_{k\geq 0}\mathcal{I}_V^k(1)/\mathcal{I}_V(0)\cdot \mathcal{I}_V^k(1).$

\begin{definition}
\label{definition: 1-Microchar}
Let $U$ be an open subset of $T^*X$ and suppose $V\subset U$ is a closed conic analytic subset. Let $\tau:T_VT^*X\rightarrow T^*X,$ denote the normal bundle and let $\mathcal{A}^{\bullet}$ be a derived $\D$-algebra on $X$.
    The \emph{$\D$-geometric $1$-micro-characteristic variety of $\mathcal{A}^{\bullet}$ along} $V$ is
    $$\mathsf{Ch}_{\D,V}^1(\mathcal{A}):=\bigcup_{i\in\mathbb{Z}}supp\big(\mathcal{H}^i(\mu_{\mathcal{D},V}\mathbb{T}_{\mathcal{A}}^{\ell})\big),$$
    in $\mathsf{Spec}_{T_VU\subset T_VT^*X}(\tau^*\mathcal{A}^{\bullet})\simeq \mathsf{Spec}_{\pi(U)\subset X}(\mathcal{A}^{\bullet})\times_X T_VT^*X.$
\end{definition}
Note that $\mathcal{H}^i\big(\mu_{\mathcal{D},V}\mathbb{T}_{\mathcal{A}}^{\ell}\big)=\mu_{\mathcal{D},V}\mathcal{H}^i(\mathbb{T}_{\mathcal{A}}^{\ell}),$ for each $i.$
Pullback along a classical solution $\varphi$ gives a complex of $\D_X$-modules $\mathbb{T}_{\mathcal{A},\varphi}^{\ell}$, and one can consider its classical $V$-microlocalization \eqref{MicrolocalizationClassical}. Namely, by taking supports we obtain
$$\mathrm{Char}_V^1(\mathbb{T}_{\mathcal{A},\varphi}^{\ell}):=supp\big( \varphi^* \mu_{\mathcal{D},V}\mathbb{T}_{\mathcal{A}}^{\ell}\big)\simeq supp\big(\mu_V\mathbb{T}_{\mathcal{A},\varphi}^{\ell}\big).$$
Similarly, one may consider 
$$\mathsf{Ch}_{\D,V}^{(s)}(\mathcal{A}):= supp\big(\eta^*\mathcal{A}[\mathcal{E}_V^{2(s)}]\otimes_{\eta^*\mathcal{A}\tau^{-1}\mathcal{E}_{X|V}}\eta^*\mathbb{T}_{\mathcal{A},\varphi}|_V\big),\hspace{2mm} s\geq 1.$$
Note $\mathrm{supp}(\varphi^*\mu_{\mathcal{D},V}\mathbb{T}_{\mathcal{A}}^{\ell})$ is then contained in $\varphi^{-1} \mathrm{supp}(\mu_{\mathcal{D},V}\mathbb{T}_{\mathcal{A}}^{\ell}).$
\begin{proposition}
\label{Prop: CharPBSolution}
Let $\mathcal{A}$ be a such that $\mathbb{L}_{\mathcal{A}}$ is perfect with dual $\mathbb{T}_{\mathcal{A}}^{\ell}.$ Then for every solution $\varphi$ we have
$$\mathrm{Char}_V^1(\mathbb{T}_{\A,\varphi}^{\ell})\subseteq \varphi^{-1}\mathsf{Ch}_{\D,V}^1(\A).$$
Suppose that $L^j\varphi^*\mathcal{H}_{\D}^j(\mu\mathbb{T}_{\A}^{\ell})=0,$ for all $j>0.$ Then, $\mathrm{Char}_V^1(\mathbb{T}_{\mathcal{A},\varphi}^{\ell})$ and $\varphi^{-1}\mathsf{Ch}_{\D,V}^1(\mathcal{A})$ coincide.
\end{proposition}
\begin{proof}
The first claim is a consequence of behaviour of supports while the second follows via a spectral sequence calculation from which it follows for each $i$,
$\mathrm{supp}(\varphi^*H_{D}^i(\M))=\varphi^{-1}\mathrm{supp}(H_{D}^i(\M)),$ implying $\mathrm{supp}(\varphi^*\M)=\varphi^{-1}\mathrm{supp}(\M),$ with $\M=\mu_{\D,V}(\mathbb{T}_{\A}^{\ell})$ in our case.
\end{proof}

For such $\mathcal{A}^{\bullet}\in \mathrm{cdga}_{\D_X}^{\leq 0},$ with $\mathcal{A}^{\bullet}[\mathcal{D}_X]:=\mathcal{A}^{\bullet}\otimes_{\mathcal{O}_X}\mathcal{D}_X,$ it follows from the inclusion $\pi^{-1}\mathcal{O}_X\hookrightarrow \pi^{-1}\mathcal{D}_X,$ we have a canonical morphism
$$\mathcal{A}^{\bullet}\otimes_{\mathcal{O}_X}\mathcal{D}_X\rightarrow \pi_{\mathcal{D}}^*\mathcal{A}^{\bullet}\otimes_{\pi^{-1}\mathcal{O}_X}\pi^{-1}\mathcal{D}_X.$$
We also have an inclusion 
$$\pi_{\mathcal{D}}^*\mathcal{A}^{\bullet}\otimes_{\pi^{-1}\mathcal{O}_X}\pi^{-1}\mathcal{D}_X\hookrightarrow \pi_{\mathcal{D}}^*\mathcal{A}^{\bullet}\otimes_{\pi^{-1}\mathcal{D}_X}\mathcal{E}_X.$$

\begin{proposition}
\label{Induced functor microlocal}
The induced functor
$\widetilde{\pi}_{\mathcal{D}}^!:\mathrm{Mod}\big(\mathcal{A}^{\bullet}\otimes_{\mathcal{O}_X}\mathcal{D}_X\big)\rightarrow\mathrm{Mod}\big(\pi_{\mathcal{D}}^*\mathcal{A}^{\bullet}\otimes_{\pi^{-1}\mathcal{D}_X}\mathcal{E}_X\big),$
extends to
$$\widetilde{\pi}_{\mathcal{D}}^!:\mathsf{Mod}\big(\mathcal{A}^{\bullet}\otimes_{\mathcal{O}_X}\mathcal{D}_X\big)\rightarrow \mathsf{Mod}\big(\pi_{\mathcal{D}}^*\mathcal{A}^{\bullet}\otimes_{\pi^{-1}\mathcal{D}_X}\mathcal{E}_X\big).$$
\end{proposition}
Assignment $_{\mathcal{A}}\mu$ is the \emph{microlocalization with coefficients} (in $\mathcal{A}^{\bullet}),$ given for an $\mathcal{A}^{\bullet}\otimes_X^!\mathcal{D}_X$-module $\mathcal{M}^{\bullet}$ by 
\begin{equation}
\label{eqn: Microlocalization with Coefficients}
_{\mathcal{A}}\mu(\mathcal{M}^{\bullet}):=\pi^{-1}\mathcal{M}^{\bullet}\otimes_{\pi_{-1}\mathcal{A}^{\bullet}[\mathcal{D}_X]}\pi_{\mathcal{D}}^*\mathcal{A}^{\bullet}[\mathcal{E}_X],
\end{equation}
where $\pi_{\mathcal{D}}^*\mathcal{A}^{\bullet}[\mathcal{E}_X]:=\pi_{\mathcal{D}}^*\mathcal{A}^{\bullet}\otimes_{\pi^{-1}\mathcal{A}^{\bullet}[\mathcal{D}_X]}\mathcal{E}_X.$ 
Proposition \ref{Induced functor microlocal} extends to give a twisted microlocalization 
$$_{\mathcal{A}}\mu:\mathsf{Mod}\big(\mathcal{A}^{\bullet}\otimes_{\mathcal{O}_X}\mathcal{D}_X\big)\rightarrow \mathsf{Mod}\big(\pi_{\mathcal{D}}^*\mathcal{A}^{\bullet}\otimes_{\pi^{-1}\mathcal{A}^{\bullet}[\mathcal{D}_X]}\mathcal{E}_X\big),$$
where $\pi^{-1}\mathcal{A}^{\bullet}[\mathcal{D}_X]=\pi^*\mathcal{A}^{\bullet}\otimes_{\pi^{-1}\mathcal{O}_X}\pi^{-1}\mathcal{D}_X.$ 

The $\D$-geometric characteristic variety (c.f. \cite{Paugam2022}) is
$$\mathsf{Ch}_{\D_X}(\mathcal{A}^{\bullet}):=supp\big(_{\mathcal{A}}\mu(\mathbb{T}_{\mathcal{A}}^{\ell})\big)=\bigcup_{j\geq 0}supp\big(\mathcal{H}^j(_{\mathcal{A}}\mu\mathbb{T}_{\mathcal{A}}^{\ell})\big)\subset \mathsf{Spec}_{\D_X}(\mathcal{A}^{\bullet})\times_X T^*X.$$
Pull-back of the tangent complex along a solution $\varphi^*\mathbb{T}_{\mathcal{A}}^{\ell}$ gives a $\D$-module on $X$ and so pulling back $\mu\mathbb{T}_{\mathcal{A}}^{\ell}$ gives us an $\mathcal{E}_X$-module on $T^*X.$
Its support is
$supp\big(\varphi^*\mu\mathbb{T}_{\mathcal{A}^{\ell}}\big),$
which is the $\D$-module characteristic variety of the linearization along $\varphi$, denoted by
$\mathrm{Char}(\mathbb{T}_{\mathcal{A},\varphi}^{\ell}):=supp\big(\mu \mathbb{T}_{\mathcal{A},\varphi}^{\ell}\big).$

Consider the induction functor
$\mu\mathrm{ind}_{\mathcal{A}^{\bullet}[\mathcal{D}_X]}^r:\mathsf{Mod}(\mathcal{O}_X)\rightarrow \mathsf{Mod}\big(\pi_{\mathcal{D}}^*\mathcal{A}^{\bullet}[\mathcal{E}_X]\big),$ given by
$$\mu\mathrm{ind}_{\mathcal{A}^{\bullet}[\mathcal{D}_X]}^r(\mathcal{F}^{\bullet}):=\pi^{-1}(\mathrm{ind}_{\mathcal{A}[\mathcal{D}_X]}^r\mathcal{F}^{\bullet})\otimes_{\pi^{-1}\mathcal{A}^{\bullet}[\mathcal{D}_X]}\mathcal{E}_X.$$

Note that $T_VT^*X$ has induced $\mathbb{C}^{\times}$-actions by the projections $\pi$ and $\tau$ so $\tau_*\mathcal{O}_{T_VT^*X}$ is bi-filtered. 
\begin{proposition}
\label{1-microchar equivalence}
$\mathsf{Ch}_{\D,V}^1(\mathcal{A})$ is 
$supp\big(\eta^*\mathcal{A}[\mathcal{O}_{T_VT^*X}]\otimes_{\eta^*\mathcal{A}[\tau^{-1}\mathrm{Gr}\mathcal{E}_V]}\tau^*\big(_{\mathcal{A}}\mu\mathbb{T}_{\mathcal{A}}^{\ell}\big).$
\end{proposition}
\begin{proof}
    Since
    $\mu_{\mathcal{D},V}(\mathbb{T}_{\mathcal{A}}^{\ell})\simeq \mathcal{O}_{T_VT^*X}\otimes_{\eta^*\mathcal{A}[\tau^{-1}\mathcal{K}_V(1)]}\tau^{*}(_{\mathcal{A}}\mu\mathbb{T}_{\mathcal{A}}^{\ell}),$
we observe
$$
supp\big(\eta^*\mathcal{A}[\mathcal{O}_{T_VT*X}]\otimes_{\eta^*\mathcal{A}[\tau^{-1}\mathrm{Gr}\mathcal{E}_V]}\tau^*\big(\pi^*\mathbb{T}_{\mathcal{A}}^{\ell}\otimes_{\pi^*\mathcal{A}[\pi^{-1}\mathrm{Gr}\D_X]}\pi^*\mathcal{A}[\mathcal{O}_{T^*X}]\big),$$
is equivalently written as 
$$supp\big(\eta^*\mathcal{A}[\mathcal{O}_{T_VT*X}]\otimes_{\eta^*\mathcal{A}[\tau^{-1}\mathrm{Gr}\mathcal{E}_V]}\tau^*\big(\pi^*\mathbb{T}_{\mathcal{A}}^{\ell}\big)\otimes_{\tau^*\pi^*\mathcal{A}[\tau^{-1}\pi^{-1}\mathrm{Gr}\D_X]}\tau^*\pi^*\mathcal{A}[\tau^{-1}\mathcal{O}_{T^*X}]\big),$$
and thus
$supp\big(\eta^*\mathcal{A}[\mathcal{O}_{T_VT*X}]\otimes_{\eta^*\mathcal{A}[\tau^{-1}\mathrm{Gr}\mathcal{E}_V]}\tau^*\big(\pi^*\mathbb{T}_{\mathcal{A}}^{\ell}\big).$
\end{proof}

We now study microsupports of solution stacks.

\subsubsection{The solutionwise microsupport}
To our knowledge, the notion of sheaf of derived stacks does not yet exist. Thus, there is not available notion of microsupport for such objects, nor of microsupport for sheaves of simplicial sets. 

In this subsection we propose a notion which is suitable for our purposes in this work. It captures the idea of the microlocal support of the (derived) solution prestack being the locus of covectors in $T^*X$ where some classical solution has nonzero microlocalized tangent directions.

Given a $D$-scheme $\EQ\to X_{\DR}$, set $\RS(\EQ):=\Map_{/X_{\DR}}(X_{\DR},\EQ)$. If $Z$ is relatively affine, then $\RS(\EQ)$ is equivalent to $\mathrm{Sol}_{\D_X}(\A).$ Fix a classical solution $\varphi\in \pi_0(\mathrm{Sol}_{\D_X}(\A).$ Then, 
$$\mathbb{T}_{\RS(\A),\varphi}=R\Gamma(X_{\DR},\mathbb{T}_{\A,\varphi}^{\ell}).$$
Fix $V\subset T^*X$ as before.
\begin{definition}
\label{PointwiseSS}
The \emph{point-wise $1$-microcharacteristic set along $V$, at $\varphi$} is the subset of $T_VT^*X$, given by
$\mathrm{Char}_V^1(\mathbb{T}_{\RS(\A)};\varphi):=\mathrm{supp}\big(\mu_{\D,V}(\varphi^*\mathbb{T}_{\A}^{\ell}))\big).$
The \emph{microlocal support} is the set
$$\mathrm{SS}(\mathbb{R}\mathrm{Sol}_{\D}(\A)):=\bigcup_{\varphi\in \pi_0\mathbb{R}\mathrm{Sol}(A)}\mathrm{Char}_V^1(\mathbb{T}_{\RS(\A)};\varphi).$$
\end{definition}

\subsection{Property of non-micro-characteristic }
\label{ssec: Property of non-micro-characteristic}
Consider a morphism of complex analytic manifolds $f:X\rightarrow Y$, an involutive closed conic subset $V$ of an open subset $V\subseteq \mathring{T}^*Y,$ a derived $\D$-algebra $\mathcal{A}$ on $Y.$ In this subsection, we will write the derived $\D$-space as $\mathbb{R}\mathsf{Sol}_{Y_{\DR}}.$ To this datum, we have a Cartesian diagram
\begin{equation}
    \label{eqn: D-Algebra CKK Diagram}
    \begin{tikzcd}
    \mathbb{R}\mathsf{Sol}_{Y_{\DR}}(\mathcal{A}^{\bullet})\times_Y T^*X\arrow[d] & \arrow[l,"f_d"] \mathbb{R}\mathsf{Sol}_{Y_{\DR}}(\mathcal{A}^{\bullet})\times_Y X\times_Y T^*Y\arrow[d]\arrow[r,"f_{\pi}"] & \mathbb{R}\mathsf{Sol}_{Y_{\DR}}(\mathcal{A}^{\bullet})\times_Y T^*Y\arrow[d]
    \\
    X\arrow[r,"\mathbf{1}_X"] & X\arrow[r,"f"] & Y
    \end{tikzcd}
\end{equation}
Note that we have a canonical projection map
$$\rho_V:V\times_{T_Y^*}T^*T^*Y\xrightarrow{\simeq} V\times_{T^*Y}TT^*Y\rightarrow T_VT^*X,$$
which is a vector bundle morphism. By $\rho_V$ and using the base-change of the normal bundle $\tau_Y:T_VT^*Y\rightarrow T^*Y,$ the Cartesian diagram (\ref{eqn: D-Algebra CKK Diagram}) is extended to a diagram in prestacks,
\begin{equation}
    \label{eqn: D-Algebra 1-CKK Diagram}
    \begin{tikzcd}
     \mathbb{R}\mathsf{Sol}_{Y_{\DR}}(\mathcal{A}^{\bullet})\times X\times V\times T^*T^*Y\arrow[d]\arrow[r] & \mathbb{R}\mathsf{Sol}_{Y_{\DR}}(\mathcal{A}^{\bullet})\times V\times T^*T^*Y\arrow[d]
    \\
    \mathbb{R}\mathsf{Sol}_{Y_{\DR}}(\mathcal{A}^{\bullet})\times_Y X\times_Y T_VT^*Y\arrow[d]\arrow[r] & \mathbb{R}\mathsf{Sol}_{Y_{\DR}}(\mathcal{A}^{\bullet})\times T_VT^*Y\arrow[d]
    \\
     \mathbb{R}\mathsf{Sol}_{Y_{\DR}}(\mathcal{A}^{\bullet})\times_Y X\times_Y T^*Y\arrow[d, "f_d"] \arrow[r,"f_{\pi}"] & \mathbb{R}\mathsf{Sol}_{Y_{\DR}}(\mathcal{A}^{\bullet})\times_Y T^*Y
     \\
     \mathbb{R}\mathsf{Sol}_{Y_{\DR}}(\mathcal{A}^{\bullet})\times_Y T^*X & & 
    \end{tikzcd}
\end{equation}

\begin{definition}
\label{Definition: 1MicroNC}
\normalfont    A morphism $f:X\rightarrow Y$ is {\color{white!15!black}{\textbf{\emph{$1$-micro-noncharacteristic}}}} for $\mathcal{A}^{\bullet}\in \mathsf{CAlg}_Y
(\mathcal{D}_Y)^{c},$ \emph{along $V$ at $p\in V$} if for every function $\psi$ defined on a neighbourhood $W$ of $p$ such that $\psi\circ f_{\pi}=0$ and $d\psi(p)\neq 0,$ then the horizontal lift $\widehat{H}_{\psi}(p)$ of the Hamiltonian vector field $H_{\psi}(p)$ is not an element of $\mathsf{Ch}_{\D,V}^1(\mathcal{A})$.
\end{definition}
Therefore $f:X\rightarrow Y$ is micro-non-characteristic for $\varphi^*\mathbb{T}_{\mathsf{Sol}_Y(\mathcal{A})}^{\ell}$ for every classical solution $\varphi,$ in the sense that $H_{\psi}(p)$ is not in $\mathrm{Char}_V^1(\mathbb{T}_{\mathcal{A},\varphi}^{\ell}).$ 

In other words, considering  
$T_X^*Y\hookrightarrow T_X^*Y\times_Y V\rightarrow T_V(T^*Y),$ and the normal cone $C_V\big(\mathrm{Char}(\mathbb{T}_{\mathcal{A},\varphi}^{\ell})\big),$ then $X$ is non-micro characteristic at $p\in V$ if 
$$\mathring{\pi}_X^{-1}(p)\cap C_V\big(\mathrm{Char}(\mathbb{T}_{\mathcal{A},\varphi}^{\ell})\big)=\emptyset,$$
where $\mathring{\pi}_X:\mathring{T}_X^*Y\times_Y V\rightarrow V.$

This implies that $f$ is non-characteristic for $\varphi^*\mathbb{T}_{\mathsf{Sol}_Y(\mathcal{A})}^{\ell}$ in a neighbourhood $W$ of $p$ i.e. the morphism $f_d$ is a proper morphism on $f_{\pi}^{-1}\big(\mathsf{Ch}_{\D,\varphi}(\mathcal{A})\cap W\big).$

In particular, we say that $f:X\rightarrow Y$ is non-characteristic for $\mathcal{A}$ if it is $1$-micro-noncharacteristic for $\mathcal{A}$ along the zero section Lagrangian sub-manifold $V=T_X^*X$. In this case, one has that
\begin{equation}
    \label{eqn: NC Condition for D-Algebras}
f_{\pi}^{-1}\big(\mathsf{Ch}_{\D_Y}(\mathcal{A}^{\bullet})\big)\times^h\mathbb{R}\mathsf{Sol}_{Y_{\DR}}(\mathcal{A}^{\bullet})\times_Y T_X^*Y\subset \mathbb{R}\mathsf{Sol}_{Y_{\DR}}(\mathcal{A}^{\bullet})\times_Y X\times_Y T_Y^*Y.
\end{equation}

We can also incorporate sub-spaces $W\subset \mathsf{Sol}_Y(\mathcal{A})\times_Y T^*X$. Namely, $f$ is NC for $\mathcal{A}$ (over $U$) \textit{on $W$} if 
the morphism 
$f_{\pi}$ restricted to 
$$f_{d}^{-1}(W)\cap f_{\pi}^{-1}(\mathrm{Char}_{\mathcal{D}}(\mathcal{A})\big)\times\mathsf{Sol}_Y(\mathcal{A}),$$
is proper.
One recovers the usual NC condition by setting $W=T^*X.$

For example, supposing that
$\mathcal{A}^{\bullet}\simeq \mathcal{A}$ is concentrated in degree zero with regular singularities along $V$ i.e. $\varphi^*\mathbb{L}_{\mathcal{A}}$ is coherent with singularities along $V,$ then 
\begin{eqnarray*}
    \mathrm{SS}\big(\mathbb{R}\mathsf{Sol}_{X,lin}(\varphi^*\mathbb{L}_{\mathcal{A}})\big)&=&\mathrm{Char}(\varphi^*\mathbb{L}_{\mathcal{A}})\subset V\cup \mathrm{supp}(\varphi^*\mathbb{L}_{\mathcal{A}})
    \\
       &\simeq& \big\{(x;\xi)\in \mathring{T}^*X|(x;\xi)\in V\big\}\cup \big\{(x;0)\in T_X^*X|x\in \mathrm{supp}(\varphi^*\mathbb{L}_{\mathcal{A}})\big\}.
       \end{eqnarray*}

Equivalently, we could have restated this for $\mathsf{Ch}_{\D,V}^1(\mathcal{A})$ i.e. using $\mathbb{T}_{\mathcal{A}}$ instead, due to the fact that $\mathrm{Char}(\mathcal{M})=\mathrm{Char}(\mathbb{D}\mathcal{M})$ for a coherent $\D$-module \cite[Prop. 3.1.4]{KashiwaraMicro}.

\begin{remark}
Concretely, given a (holomorphic) non-linear operator of order $\leq m,$
$$F(x,u(x),\ldots,\partial^{\sigma}u,\ldots)=0,\hspace{1mm}|\sigma|\leq m,$$
a point $(x_0,\xi_0)\in \mathbb{R}^n\times (\mathbb{R}^n\backslash \{0\})$ is non-characteristic if 
\begin{equation}
    \label{eqn: Symbol of Linearization of Non-linear}
    \sigma_m(F)(x,;\xi)=\sum_{|\sigma|=m}\frac{\partial F}{\partial u_{\sigma}}(x,u(x),\ldots)(i\xi)^{\sigma}\neq 0.
\end{equation} 
\end{remark}

\subsubsection{Remarks on ellipticity}
\label{sssec: A Remark on Ellipticity}
The tools in §§ \ref{sssec: D-Geometric Microcharacteristics}, as well as the characterization of linear ellipticity (\ref{ssec: Propagation of Solutions}) can be used to make sense of geometric non-linear ellipticity.

Let $M$ be a smooth real analytic manifold with complexification $X,$ and consider a $\D$-algebraic non-linear \textsc{pde} $\mathcal{A}\rightarrow \mathcal{B}.$ It is said to be \emph{elliptic} if 
$$supp\big(\mu\mathbb{T}_{\mathcal{B},\varphi}^{\ell})\cap T_M^*X\subset M\times_X T_X^*X,$$
holds for every solution $\varphi$ i.e. $\mathbb{T}_{\mathcal{B},\varphi}^{\ell}$ is an elliptic $\D$-module.

In a more satisfactory way, we consider (\ref{eqn: Derived non-linear solution functor}) \eqref{eqn: Quillen adjunction for slice D stacks} and the general discussion in § \ref{sssec: Derived Solutions for RelAlgNLPDEs} and may say that our corresponding derived $\D_X$-prestack $R\mathrm{Sol}_{\D_X}(\mathcal{B})$ is elliptic if 
$$\big(\mathrm{Char}_{\mathcal{D}}(\mathcal{B})\times \mathbb{R}\mathsf{Sol}_X(\mathcal{B})\big)\cap T_M^*X\subset \mathbb{R}\mathsf{Sol}_X(\mathcal{B})\underset{X}{\times} M\times T_X^*X.$$

This may also be phrased directly in terms of the twisted local solution functor (§ \ref{sssec: Solutions with Coefficients}) by considering hyperfunctions $\mathscr{B}_M$ and the induced $\mathcal{A}[\mathcal{D}_X]$-module by base change $\mathrm{ind}_{\mathcal{A}[\mathcal{D}_X]}(\mathscr{B}_M).$ Working in $\mathcal{A}[\mathcal{D}]$-modules,
and replacing $\mathscr{B}_M$ by its extension of scalars from $\mathcal{D}_X=\mathcal{O}_X[\mathcal{D}_X]$ to $\mathcal{A}^{\ell}\otimes_{\mathcal{O}_X}\mathcal{D}_X,$ we have (\ref{sssec: Solutions with Coefficients}),
$$\mathbb{R}\mathsf{Sol}_{\mathcal{A}}(\mathbb{L}_{\mathcal{A}},\mathcal{A}[\mathcal{D}_X]\otimes\mathscr{B}_M)=\mathbb{R}\mathcal{H}om_{\mathcal{A}^{\ell}[\mathcal{D}_X]}(\mathbb{L}_{\mathcal{A}},\mathcal{A}[\mathcal{D}_X]\otimes\mathscr{B}_M).$$
By universal properties of $\mathbb{L}_{\mathcal{A}}$ and the hypothesis that it is perfect, for a given (infinite jet) of a solution $j_{\infty}(x):X\rightarrow \mathsf{Jets}^{\infty}(E/X)$, we consider local solutions of the leftified $\mathcal{A}[\mathcal{D}_X]$-dual $\mathbb{T}_{\mathcal{A}}^{\ell}.$ These considerations lead to an isomorphism of sheaves of $\mathsf{Jet}(E/X)$-vector spaces which we may appropriately pull-back to get a version of elliptic regularity for the linearization $j_{\infty}(x)^*\mathbb{T}_{\mathcal{A}}^{\ell}$ as an $\mathcal{A}[\mathcal{D}_X]$-module.

By base-change we pass to $\mathcal{E}_X$-modules. We can then analyze micro-function solutions and the ellipticity conditions obtained from the solution sheaves $\mathbb{R}\mathsf{Sol}_{\mathcal{A}}(\mathbb{L}_{\mathcal{A}},\mathcal{A}[\mathcal{E}_X]\otimes\mathscr{C}_M).$

The situation is more simple in the case of free $\D$-algebras.
\begin{example}
\label{Elliptic1}
    \normalfont 
Consider a $\D$-algebra $\mathcal{B}$ as in \autoref{Free Tangent Cotangent Example} with underlying coherent $\D$-module $\mathcal{M}$. Suppose that $\mathcal{M}$ is elliptic.
Then $\mathbb{R}\mathsf{Sol}_X(\mathcal{B})$ is elliptic, in the above sense.
\end{example}

\begin{example}
    \normalfont 
Consider \autoref{Elliptic1} and let $S$ denote the zero section of $T_M^*X$. If $\mathcal{B}$ is of the form $\mathrm{Sym}(\mathcal{M}_P)$ for $\mathcal{M}_P\simeq \mathcal{D}_X/\mathcal{D}_X\cdot  P$ and $P$ is an elliptic operator, we have that microfunction solutions of $\varphi^*\mathbb{L}_{\mathcal{B}}$ are determined by micro-function solutions of $\mathcal{M}_P$, and pull-back to $T^*X$ gives 
$\mathbb{R}\mathcal{H}om(\mathcal{M},\mathscr{C}_M)|_{T_M^*\backslash S}\simeq 0.$ Indeed, (\ref{eqn: Free-forget AD-Mod/Alg Adjunction}) derived solutions of $\mathcal{B}$ are derived solutions of $\mathcal{M}_P$ which by ellipticity identifies
$\mathbb{R}\mathsf{Sol}_{X,lin}(\mathcal{M}_P,\mathscr{A}_M)$ with $\mathbb{R}\mathsf{Sol}(\mathcal{M}_P,\mathscr{B}_M).$
This implies that the surjective linear differential operator $P$ defines an isomorphism acting on $\Gamma(\mathscr{B}_M/\mathscr{A}_M).$ By the short-exact sequence (\ref{eqn: SESHyperfunction}) this gives the vanishing result by restricting to microfunction solutions outside $S.$
\end{example}

\subsubsection*{Extension to non-affine $\D$-spaces}
Suppose we have a not necessarily representable $\D_X$-prestack, admitting a $\D_X$-étale atlas (c.f. §\ref{sssec: D-Geom Tangent and Cotangent Complexes}). Specifically, consider an object $\mathcal{X}\in \mathsf{dStk}_X(\D_X)$ and make the assumption that $\mathbb{L}_{\mathcal{X}}\in \mathsf{Shv}_{\mathcal{D}}^!(\mathcal{X})$ exists and is a perfect object. Recall $\mathsf{Shv}_{\mathcal{D}}^!(\mathcal{X})$ is the $\infty$-category of $\mathcal{O}_{\mathcal{X}}$-modules in $\D_X$-modules. In this way, 
$$\mathsf{Shv}_{\mathcal{D}}^!\big(\mathsf{Spec}_{\D_X}(\mathcal{A}^{\bullet})\big)\simeq \mathsf{Mod}_{\D_X}\big(\mathcal{A}^{\bullet}\big)\simeq \mathsf{Mod}\big(\mathcal{A}^{\bullet}\otimes_{\mathcal{O}_X}\mathcal{D}_X\big).$$

The tautological embedding  $\iota:\mathsf{Sch}\hookrightarrow \mathsf{Sch}_{\mathcal{D}}$ in permits viewing any scheme $(X,\mathcal{O}_X)$ as a $\D_X$-scheme, since $\mathcal{O}_X$ is a $\D_X$-algebra. Thus any scheme can be viewed as a derived $\D_X$-stack. 
Therefore, we may consider the homotopy fiber product (we will omit $\iota$ going forward) $$\mathcal{X}\times_{\iota(X)}\iota(T^*X)\hspace{2mm}\text{taken in }\mathsf{dStk}_X(\D_X).$$

 We want to transport our definitions to arbitrary algebraic $\D$-spaces and we consider $\mu_{\mathcal{D},V}\big(\mathbb{T}_{\mathcal{X}}^{\ell}\big),$ as a module over $\mathcal{X}\times_{T^*X} T_VT^*X$.
Moreover, notice that we have an equivalence
$$\mathsf{Shv}_{\mathcal{D}}(\mathcal{X}\times_XT^*X)\simeq \mathsf{Shv}_{\mathcal{D}}(\mathcal{X})\otimes_{\mathsf{Shv}_{\mathcal{D}}(X)}\mathsf{Shv}_{\mathcal{D}}(T^*X).$$

Now, given an affine derived $\D$-prestack $\mathsf{Spec}_{\D}(\mathcal{B}^{\bullet})$ with a $\D$-étale map $g$ to $\mathcal{X}$, set
\begin{equation}
\label{eqn: Non-affine D Char}
\mathrm{Char}_{\mathcal{D}}(\mathcal{X})\times_{\mathcal{X}}\mathsf{Spec}_{\D}(\mathcal{B})\rightarrow \mathrm{Char}_{\mathcal{D}}(\mathcal{B}^{\bullet}).
\end{equation}
The assignment (\ref{eqn: Non-affine D Char}) defines $\mathrm{Char}_{\mathcal{D}}(\mathcal{X})$ for non-affine $\D$-spaces.

Let us say that $\mathcal{X}$ is elliptic as a derived $\D_X$-stack over $X$ if it admits an atlas by derived affine $\D$-spaces which are elliptic (an `elliptic atlas') as in (\ref{sssec: A Remark on Ellipticity}).
\begin{proposition}
The assignment (\ref{eqn: Non-affine D Char}) on the category $\mathsf{dStk}_X(\D_X)_{/\mathcal{X}},$ with $\D$-étale maps is compatible with descent.
\end{proposition}
\begin{proof}
This follows from the general observation that if $\{f_{\alpha}\}:\bigcup_{\alpha}\mathsf{Spec}_{\D}(\mathcal{B}_{\alpha})\rightarrow \mathcal{X}$ is a $\D$-étale atlas, we have induced maps
$$lim_{\alpha} \big(\pi^*(f_{\alpha}^*\mathbb{T}_{\mathcal{X}}^{\ell})\otimes_{\pi^*(f_{\alpha}^*\mathcal{O}_{\mathcal{X}})[\pi^{-1}\mathcal{D}_X]}\pi^*(f_{\alpha}^*\mathcal{O}_{\mathcal{X}})[\mathcal{E}_X]\big)\rightarrow lim_{\alpha}\big(\pi^*(\mathbb{T}_{\mathcal{B}_{\alpha}}^{\ell}\otimes_{\pi^*\mathcal{B}_{\alpha}[\pi^{-1}\mathcal{D}]}\pi^*\mathcal{B}_{\alpha}[\mathcal{E}_X]\big),$$
compatible with all morphisms $g_{\alpha}$
arising from the sequence
$$lim_{\alpha}f_{\alpha}^*\mathbb{T}_{\mathcal{X}}^{\ell}\leftarrow lim_{\alpha}\mathbb{T}_{\mathcal{B}_{\alpha}}^{\ell}\leftarrow lim_{\alpha} \mathbb{T}_{\mathcal{B}_{\alpha}/\mathcal{X}},$$
by applying the twisted microlocalization functor, and using the fact that $f_{\alpha}$ is $\D$-étale as in § \ref{sssec: Model Categorical Interlude}, for each $\alpha.$
\end{proof}
Consider a complex analytic manifold $Y$ with $\pi:T^*Y\rightarrow Y$, and Cartesian diagram taken in $\PS,$
\[
\begin{tikzcd}
\JetY(E)\times_Y T^*Y\arrow[d]\arrow[r] & T^*Y\arrow[d,"\pi"]
\\
\JetY(E)\arrow[r,"\mathsf{q}"]& Y
\end{tikzcd}
\]
By pullback of sheaves on relative prestacks, we have that $\pi$ defines a functor
\begin{equation}
\label{eqn: Jet-Cotangent PB}
\pi_{\mathsf{Shv}}^!:\mathsf{Shv}\big(\JetY(E)\big)\rightarrow \mathsf{Shv}\big(\JetY(E)\times_Y T^*Y\big).
\end{equation}

Since $Y$ is a classical complex analytic manifold, there is a well-defined sheaf of microdifferential operators $\mathcal{E}_Y$ on $T^*Y,$ as in § \ref{ssec: Characteristic Varieties}.

Considering the pull-back under of the tangent $\D$-complex under (\ref{eqn: Jet-Cotangent PB}), viewed as an object
$\pi_{\mathsf{Shv}}^!\mathsf{T}^{\ell}[\J(E)]$ in
$$\pi_{\mathsf{Shv}}^!\mathcal{A}^{\bullet}-\mathsf{Mod}\big(\mathsf{Shv}(\pi^{-1}Y)\big)= \pi_{\mathsf{Shv}}^!\mathsf{q}_*^{\mathsf{CAlgShv}}\big(\mathcal{O}_{\JetY(E)}\big)-\mathsf{Mod}\big(\mathsf{Shv}(\pi^{-1}Y)\big).$$

Base-change and form the induced right-$\mathcal{E}_Y$-module $$\mathrm{ind}_{\mathcal{E}_Y}\big(\pi_{\mathsf{Shv}}^!\mathsf{T}^{\ell}[\J(E)]\big),$$
gives the following.

\begin{proposition}
\label{Derived Characteristic Variety Lemma}
Pull-back of ind-coherent sheaves along the smooth morphisms $\pi:T^*Y\rightarrow Y$ and tensoring with microlocal differential operators $\mathcal{E}_Y$ yields an object
$$\mu\big(\mathsf{T}^{\ell}[\J(E)]\big):=\mathsf{T}^{\ell}[\J(E)]\otimes_{\pi^!\mathcal{A}^{\bullet}\otimes_Y^!\pi^{-1}\mathcal{D}_Y}\pi^!\mathcal{A}^{\bullet}[\mathcal{E}_Y].$$
\end{proposition}
The object obtained as in \autoref{Derived Characteristic Variety Lemma} is thus an object of 
$$\big(\mathcal{A}^{\bullet}\otimes\mathcal{E}_Y\big)-\mathsf{Mod}\big(\mathsf{Shv}(\JetY(E)\times_YT^*Y)\big).$$

In particular, it is an ind-coherent sheaf and we may consider its support, that we denoted by
\begin{equation}
    \label{Singular Support of Derived NLPDE}
\mathsf{Ch}_{Y_{\DR}}\big(\J(E)\big):= supp\big(\mathsf{T}^{\ell}[\J(E)]\otimes_{\pi^!\mathcal{A}^{\bullet}\otimes_Y^!\pi^{-1}\mathcal{D}_Y}\pi^!\mathcal{A}[\mathcal{E}_Y]\big),
   \end{equation}
   in complexes of sheaves over
   $\JetY(E)\times_Y T^*Y.$

Consider a morphism $f:X\rightarrow Y$ of complex analytic manifolds together with the associated microlocal correspondence (\ref{eqn: Microlocal correspondence}). We have
\begin{equation}
    \label{eqn: Derived dR CKK Diagram}
    \begin{tikzcd}
    & \arrow[dl,"f_d", labels=above left] \JetY(E)\times_Y\times X\times_Y\times T^*Y\arrow[dd] \arrow[dr,"f_{\pi}"] & 
    \\
    \JetY(E)\times_Y T^*X \arrow[d] & & \JetY(E)\times_Y T^*Y\arrow[d]
    \\
    X\arrow[r,"\mathbf{1}_X"] & X \arrow[r,"f"] & Y
    \end{tikzcd}.
\end{equation}

Suppose further that $V\subset T^*Y$ is a conic subset (c.f. § \ref{sssec: D-Geometric Microcharacteristics}), and consider the diagram induced from (\ref{eqn: Derived dR CKK Diagram}),

\begin{equation}
    \label{eqn: Derived dR Sing CKK Diagram}
    \adjustbox{scale=.85}{
    \begin{tikzcd}
    & \arrow[dl,"f_{d,V}", labels=above left] \mathsf{Jets}^{\infty}(E) \times_Y\times X\times_Y V\times T^*T^*Y\arrow[dd] \arrow[dr,"f_{\pi,V}"] & 
    \\
    \mathsf{Jets}^{\infty}(E)\times X\times T_VT^*Y\arrow[d] & & \mathsf{Jets}^{\infty}(E)\times T_VT^*Y\arrow[d]
    \\
    \mathsf{Jets}^{\infty}(E)\times_YT^*X\arrow[d] & \arrow[l,"f_d"] \mathsf{Jets}^{\infty}(E)\times_Y\times X\times_Y T^*Y\arrow[d] \arrow[r,"f_{\pi}"] & \mathsf{Jets}^{\infty}(E)\times_Y T^*Y\arrow[d]
    \\
    X\arrow[r,"id_X"] & X\arrow[r,"f"] & Y
    \end{tikzcd}}.
\end{equation}

\begin{definition}
\label{NC for de Rhams}
A morphism $f:X\rightarrow Y$ of complex analytic manifolds is said to be \emph{non-characteristic} for $\J(E)\in \PS_{/Y_{\DR}},$ if is non-micro-characteristic for the Lagrangian zero section i.e.
\begin{equation}
    \label{eqn: Derived dR NC Condition}
f_{\pi}^{-1}\mathsf{Ch}_{Y_{\DR}}\big(\J(E)\big)\cap^{L}q_{\DR}^*\J(E)\times_Y T_X^*Y\subset q_{\DR}^*\J(E)\times_Y X\times_Y T_Y^*Y.
\end{equation}
\end{definition}

In the following section we will formulate a Cauchy problem for prestacks admitting an atlas by affines for which a morphism is non-characteristic, thus set up some notation now. Consider $\EQ\in \PS_{Y_{\DR}},$ which is expressible as a colimit,
$$\EuScript{Y}\simeq \underset{i\in I}{\mathrm{colim}}\hspace{1mm} Y_i,$$
with $Y_i$ of the form $\mathsf{RSol}_Y(\mathcal{B}_i)$ for some $\mathcal{B}_i\in \mathsf{Comm}(\mathsf{IndCoh}(Y_{\DR})).$ Suppose that $f$ is non-characteristic for each $Y_i$ as in § \ref{ssec: Characteristic Varieties}, and the induced maps $p(i):Y_i\rightarrow Y_{\DR}$ we have
$$p(i)^*Y_i\times_YX\times_YT^*Y\xrightarrow{f(i)_{\pi}}p^*(i)Y_i\times_Y T^*Y\rightarrow Y,$$
such that 
$$f(i)_{\pi}^{-1}\mathsf{Ch}_{\DR}(Y_i)\cap p(i)^*Y_i\times_Y T_X^*Y\subset p(i)^*Y_i\times_YX\times_Y T_Y^*Y,$$
by \ref{NC for de Rhams},
where $\mathsf{Ch}_{\DR}$ is the corresponding characteristic variety (equation  \ref{Singular Support of Derived NLPDE}). We say that $\EQ$ is of non-characteristic presentation for $f.$

\subsubsection{Functoriality}
We establish functorial behaviour of $D$-geometric characteristic varieties.

\begin{proposition}
\label{CharSmoothPB}
    Suppose $f:Y\to X$ is smooth, with $\EQ\to X_{\DR}$ a bounded $\D_X$-afp generalized PDE. Denote by $Z_Y$ the pullback generalized PDE. Then,
    $$\mathrm{Char}(Z_Y/Y_{\DR};u_T)=(id_T\times f_d)(id_T\times f_{\pi}^{-1})(\mathrm{Char}(\EQ/X_{\DR};u_T)).$$
\end{proposition}
\begin{proof}
    Recall $X$ is a smooth projective $\mathbb{C}$-scheme, and let $\EQ\to X_{\DR}$ be $\D_X$-afp. Let $T$ be any classical test scheme with $u_T:T\to Z$ a solution. Let $f:Y\to X$ be smooth with corresponding de Rham map $f_{\DR}:Y_{\DR}\to X_{\DR}.$
Put, $ev_{u_T}:T\times X_{\DR}\to Z,$ and $ev_{u_T}^Y:T\times Y_{\DR}\to Z_Y,$ where $Z_Y:=f_{\DR}^*Z.$ Note the Cartesian diagrams,
\[
\begin{tikzcd} Z_Y \ar[r]\ar[d] & Z \ar[d] \\ Y_{\DR} \ar[r,"f_{\DR}"] & X_{\DR} \end{tikzcd}\hspace{2mm} \text{and }\hspace{1mm}
\begin{tikzcd} T\times Y_{\DR} \ar[r,"\,id_T\times f_{\DR}\,"] \ar[d,"\mathrm{ev}^Y_{u_T}"'] & T\times X_{\DR} \ar[d,"\mathrm{ev}_{u_T}"] \\ Z_Y \ar[r] & Z \end{tikzcd}
\]
$$ev_{u_T}^{Y,!,\IC}(\mathbb{T}_{Z_Y/Y_{\DR}})\simeq (id_T\times f_{\DR})_{\IC}^!\big(ev_{u_T}^{!,\IC}\mathbb{T}_{\EQ/X_{\DR}}\big)\in \IC(T\times Y_{\DR})^{\omega}.$$
Moreover,
$$\mathbb{L}_{Z_Y/Y_{\DR},u_T}\simeq \mathbb{D}_{T\times Y_{\DR}}(ev_{u_T}^{Y,!}\mathbb{T}_{Z_Y/Y_{\DR}})\simeq \mathbb{D}_{T\times Y_{\DR}}\big((id_T\times f_{\DR})_{\IC}^!ev_{u_T}^!\mathbb{T}_{\EQ/X_{\DR}}\big),$$
and as $f$ is smooth, $(id_T\times f_{\DR})^!$ coincides with $(id_T\times f_{\DR})^*[d],$ where $d$ is the relative dimension of $f:Y\to X.$

Since $X$ is smooth and projective, $\IC(X_{\DR})^{\omega}\simeq \mathsf{Perf}(X_{\DR})\simeq \mathsf{Coh}(X_{\DR})$ and $\IC(T\times X_{\DR})^{\omega}\simeq \mathsf{Coh}(T)\otimes \IC(X_{\DR})^{\omega}$. Thus any object decomposes as an exterior tensor product of a coherent complex on $T$ and a perfect complex on $X_{\DR}$.
We have
$$ev_{u_T}^{Y,!}\mathbb{T}_{Z_Y/Y_{\DR}}\simeq(id_T\times f_{\DR})^!(E_T\boxtimes M_X)\simeq E_T\boxtimes f_{\DR}^!(M_X).$$
Therefore,
$$\mathbb{L}_{Z_Y/Y_{\DR},u_T}\simeq \mathbb{D}_T(E_T)\boxtimes \mathbb{D}_Y(f_{\DR}^!M_X).$$
Consider the diagram.
\[
\begin{tikzcd} Y\times_X T^*X \ar[r,"f_\pi"] \ar[d,"{^t f'}"'] & T^*X \ar[d,"\pi_X"] \\ T^*Y \ar[r,"\pi_Y"] & Y \end{tikzcd}\]
Since $\mathrm{SS}(E_T\boxtimes f_{\DR}^!M_X)\simeq \mathrm{supp}(E_T)\times \mathrm{Char}(f_{\DR}^!M_X)\subset T\times_Y T^*Y,$ since $f$ is smooth, we have
$\mathrm{Char}(f_{\DR}^!M_x)\simeq df\big(f_{\pi}^{-1}\mathrm{Char}(M_X)\big)\subset T^*Y,$ therefore,
$$\mathrm{Char}(Z_Y/Y_{\DR};u_T)=\mathrm{SS}\big(ev_{u_T}^{Y,!}\mathbb{T}_{Z_Y/Y_{\DR}}\big)\simeq \mathrm{supp}(E_T)\times df(f_{\pi}^{-1}\mathrm{Char}(M_X)\big),$$
thus is given by
$$(id_T\times df)\big(\mathrm{supp}(E_T)\times f_{\pi}^{-1}\mathrm{Char}(M_X)\big)\simeq (id_T\times df)(id_T\times f_{\pi}^{-1})\mathrm{Char}(\EQ/X_{\DR};u_T).$$

Moreover, let $f:Y\to X$ be a smooth morphism of smooth projective schemes. Consider the base change
\[
\begin{tikzcd}
T \times Y_{\DR} \arrow[r,"id_T\times f_{\DR}"] \arrow[dr,swap,"ev_{u_T}"] & T \times X_{\DR} \arrow[d,"ev_{u_T}"]\\
& Z
\end{tikzcd}
\]
Then the pullback
\[
(id_T\times f_{\DR})^! \, ev_{u_T}^{!,\mathsf{IndCoh}} \mathbb{T}_{\EQ/X_{\DR}} \in \IC(T\times Y_{\DR})^{\omega}
\]
satisfies the decomposition
\[
ev_{u_T}^{!,\mathsf{IndCoh}} \mathbb{T}_{f_{\DR}^*Z/Y_{\DR}}\simeq E_T\boxtimes f_{\DR}^! M_X,
\]
so that the characteristic variety is preserved under smooth pullback:
\[
\mathrm{Char}(f_{\DR}^*Z/Y_{\DR};u_T)\simeq \mathrm{supp}(E_T)\times \mathrm{Char}(f_{\DR}^! M_X)\simeq \mathrm{supp}(E_T)\times \mathrm{Char}(M_X).
\]
\end{proof}
The following is repeatedly used in the proof of the non-linear CKK theorem.

\begin{proposition}
\label{Codim1Prop}
    Let $i:X\to Y$ be a closed-immersion of $\mathbb{C}$-schemes of codimension $1.$ Assume $Z\to Y_{\DR}$ is $\D$-afp and bounded and that $i$ is non-characteristic for $Z.$ The, 
    $i_{\DR}^*\mathbb{T}_{Z/Y_{\DR}}\simeq \mathbb{T}_{i_{\DR}^*\EQ/X_{\DR}}$ and 
    $$\mathrm{Char}(i_{\DR}^!\mathbb{T}_{Z/Y_{\DR}})=i_d\big(i_{\pi}^{-1}\mathrm{Char}(\mathbb{T}_{Z/Y_{\DR}})\big).$$

\end{proposition}
\begin{proof}
    To $i:X\to Y$, let us denote the pull-back functor on quasi-coherent $\D$-algebras by 
    $$i_{\DR}^* : \mathsf{CAlg}(Y_{\DR}) \to \mathsf{CAlg}(X_{\DR}).$$
Let 
\begin{equation}
    \label{ClPBZ}
\begin{tikzcd} i_{\DR}^*Z \ar[r] \ar[d] & Z \ar[d] \\ X_{\DR} \ar[r,"i_{\DR}"] & Y_{\DR} \end{tikzcd}
\end{equation}
be the Cartesian diagram defining $i_{\DR}^*Z.$
The cotangent complex satisfies the base‑change formula$
L_{i_{\DR}^*Z / X_{\DR}} \simeq i_{\DR}^* L_{Z/Y_{\DR}},$
and by non-characteristic hypothesis, we have properness of supports from which the equivalence follows since the natural map $i_{\DR}^* \mathcal{O}_Z \to \mathcal{O}_{i_{\DR}^*Z},$ is an equivalence.
Working locally on $Y$, assume $X=\{t:=y_1=0\}.$ Then, by assumption $dt\notin \mathrm{Char}(Z/Y_{\DR}).$
Consider $\pi_{\DR,i}:X\to X_{\DR}\times_{Y_{\DR}}Y.$ Since $i:X\hookrightarrow Y$ is a closed-embedding, $X_{\DR}\times_{Y_{\DR}}Y$ is given by the formal completion $Y_X^{\wedge}$ and the induced functor
\begin{equation}
    \label{PBIndCohClosed}
(\pi_{\DR,i})_{\IC}^!:\IC(X_{\DR}\times_{Y_{\DR}}Y)\simeq \IC(Y_X^{\wedge})\to \IC(X),\end{equation}
is conservative and possesses a left-adjoint, denoted $(\pi_{\DR,i})_*^{\IC}.$ 
Let $u_T$ be a $T$-parameterized section of $Z\to Y_{\DR}.$ Note that by our assumptions on $X,Y$ we always have $\IC(T\times Y_{\DR})\simeq \IC(T)\otimes \IC(Y_{\DR})$ and similarly for $\Q$, thus we can consider the case $T=\mathrm{pt}.$
Moreover, $\Q(Y_{\DR})$ and $\IC(Y_{\DR})$ are compactly generated.
For each such solution consider $\mathrm{Ch}(Z/Y_{\DR};u)\subset T^*Y$. Let us omit reference to $u$ for simplicity. Similarly, consider the pulled-back solution to $i_{\DR}^*Z$ via \eqref{ClPBZ} write $\mathrm{Ch}(Z_X/X_{\DR})$ for simplicity. Let 
$$\IC_{\mathrm{Ch}(Z_X/X_{\DR})}(X_{\DR}\times_{Y_{\DR}}Y)\hookrightarrow \IC(X_{\DR}\times_{Y_{\DR}}Y),$$
denote the preimage of $\IC_{\mathrm{Ch}(i_{\DR}^*\EQ/X_{\DR})}(X_{\DR})\subset \IC(X_{\DR}),$ by the pull-back \eqref{PBIndCohClosed}. By Kashiwara's lemma \cite{GR14}, we identify $\IC(X_{\DR}\times_{Y_{\DR}}Y)$ with the full-subcategory $\IC(Y)_X\subset \IC(Y)$ of objects with set-theoretic support on $X\subset Y.$
We may \eqref{PBIndCohClosed} with $i^!$ restricted to this category and whose left-adjoint is $*$-pushforward. Consider the induced map $i_d$ put 
$$\Lambda:=(i_d)^{-1}\big(\mathrm{Ch}(i_{\DR}^*\EQ/X_{\DR})\big)\hookrightarrow X\times_Y T^*Y\subset T^*Y.$$
To prove the proposition, we show that there is an equivalence
$$\IC_{(i_d)^{-1}(\mathrm{Ch}(i_{\DR}^*\EQ/X_{\DR}))}(X)\xrightarrow{\simeq} \IC_{\mathrm{Ch}(i_{\DR}^*\EQ/X_{\DR})}(Y_X^{\wedge}),$$
as subcategories of $\IC(X_{\DR}\times_{Y_{\DR}}Y).$
To this end, we must show that the category on the left-hand side is equival to the pre-image of $\IC_{\mathrm{Ch}}(X_{\DR})$ under the $!$-pullback functor. For this, it is possible to apply \cite[Prop.7.1.3]{AriGai2015}: there is an inclusion 
$$\IC_{(i_d)^{-1}(\mathrm{Ch}(i_{\DR}^*\EQ/X_{\DR}))}(X_{\DR})\hookrightarrow (i_{\DR}^!)^{-1}\big(\IC_{\mathrm{Ch}(i_{\DR}^*\EQ/X_{\DR})}(X_{\DR})\big).$$
In the opposite direction, it suffices to note that the essential image of $\IC_{\mathrm{Ch}(i_{\DR}^*\EQ/X_{\DR})}(X_{\DR})$ by $i_{\DR,*}^{\IC}$, which takes values in $\IC_X(Y_{\DR}),$ is contained in $\IC_{(i_d)^{-1}(\mathrm{Ch}(i_{\DR}^*\EQ/X_{\DR}))}(Y_{\DR}).$ 

\end{proof}

\section{Cauchy problems and propagation of solution singularities}
\label{sec: Cauchy Problems and Propagation of Solution Singularities}
In this section we apply the tools established in \S.~\ref{sec: Derived D-NLPDES} to prove a propagation theorem. We then study the related non-linear Cauchy problem.
\begin{theorem}
\label{MainTheoremPropagationLinearizedBody}
    Let $\mathcal{I}\rightarrow \mathcal{A}\rightarrow \mathcal{B}$ be a $\D$-algebraic non-linear \textsc{pde} on a complex analytic manifold $X$, denoting by $Q\mathcal{B}$ the cofibrant replacement differential graded $\D_X$-algebra under $\mathcal{A}.$ Suppose that for every classical solution $\mathbb{T}_{\mathcal{B},\varphi}^{\ell}$ is perfect, and let $V$ be a conic closed involutive subset of an open subset $U\subseteq \mathring{T}^*X$ and $\Sigma$ is a bicharacteristic leaf of $V$, with $p_0\in \Sigma$. Let $L$ denote a coherent $\D$-module with simple characteristics and induced microlocalization $\mathcal{L}.$
    Then for every real $C^1$-function $\psi$ on $\Sigma,$ if $H_{\psi}(p_0)\notin \mathrm{Char}_{V,\varphi}^1\big(\mathcal{B}\big)$ for every $\varphi$, then 
    \begin{equation}\label{SSEstimateL}\mathrm{SS}\big(R\mathcal{S}ol_{\mathcal{E}_X}\big(\mu\mathbb{T}_{\mathcal{B},\varphi}^{\ell},\mathcal{L}\big)|_{\Sigma}\subset \mathrm{Char}_{V,\varphi}^1(\mathcal{B})\times_V\Sigma.\end{equation}
    \end{theorem}
    
In particular, Theorem \ref{MainTheoremPropagationLinearizedBody} does not depend on the choice of solution $\varphi$ and one can formulate this estimate by varying over all $\varphi:X\rightarrow\mathsf{Spec}_X(\mathcal{B}),$ as in \eqref{PointwiseSS}.
This has the effect of producing an estimate directly in terms of $\mathsf{Ch}_{\mathcal{D},V}^1(\mathcal{B})$. Namely if $p$ is the projection from $\mathsf{Spec}_X(\mathcal{B})\times T_VT^*X$ to bicharacteristic leaves of $V$ the `solution-independent' estimate is contained in
$\mathsf{Ch}_{\mathcal{D},V}^1(\mathcal{B})\times p^{-1}(\Sigma)\times \mathsf{Spec}_X(\mathcal{B}).$

In the following proof, we will specify to the case when $\mathcal{L}$ is base-change of microfunctions $\mathscr{C}_M$, so that \eqref{SSEstimateL} becomes
$$\mathrm{supp}\big(R\mathcal{S}ol_{\mathcal{B}_X[\mathcal{E}_X]}\big(_{\mathcal{B}}\mu(\mathbb{T}_{\mathcal{B}}^{\ell}),\mathrm{ind}_{\mathcal{B}_X[\mathcal{E}_X]}(\mathscr{C}_M)\big)\subset \mathrm{Char}\big(_{\mathcal{B}}\mu \mathbb{T}_{\mathcal{B},\varphi}^{\ell}\big)\cap T_M^*X.$$

\begin{proof}[Proof of Theorem \ref{MainTheoremPropagationLinearizedBody}]
We will use the the hypothesis on $\mathcal{B}$ to reduce the statement to each $n$-coconnective truncation $\mathcal{B}_{\leq n}:=\tau_{\D}^{\geq -n}\mathcal{B},$ which is a compact object of $\mathsf{QCAlg}(X_{\DR})^{\leq 0,\geq -n},$ admitting a map from $q^*q_*(E),$ via the composition 
$$\mathcal{O}_{q^*q_*(E)}\xrightarrow{p}\mathcal{B}\xrightarrow{f_n}\mathcal{B}_{\leq n},\forall n\geq 0.$$
In this compact case, we can express $\mathcal{B}_{\leq n}$ as retract of a finite homotopy colimit of free algebra on compact $\D$-modules by Proposition \ref{Compact D algebras}, from which it suffices to prove the theorem in this case. Since we consider the base-change to $\mathcal{E}_X$-modules, the latter follows from the extension of Theorem \ref{Thm: LinProp} to objects $\mathcal{M}^{\bullet}$ of the derived category of coherent $\mathcal{E}_X$-modules. We then use the fact that supports of limits are contained in the (closure of the) union of supports.
We now give the details.

Let $f_n:\mathcal{B}\to \mathcal{B}_{\leq n},$ be the map to the $n$-coconnective truncation. By Proposition \ref{Relative D-cotangent Complex}, we have a cofiber sequence of $\mathcal{B}_{\leq n}\otimes \D_X$-modules,
$$\mathbb{L}_{\mathcal{B}_{\leq n}/\mathcal{B}}[-1]\to f_n^*\mathbb{L}_{\mathcal{B}}\to \mathbb{L}_{\mathcal{B}_{\leq n}}.$$ 
Pullback along an infinite jet of a classical solution and take microlocalizations, by Proposition \ref{Prop: CharPBSolution} we have that 
$$\mu\big(\mathbb{L}_{\mathcal{B}_{\leq n}/\mathcal{B}}[-1]\big)\to \mu (f_n^*\mathbb{L}_{\mathcal{B}})\to \mu(\mathbb{L}_{\mathcal{B}_{\leq n}}),$$
and we then take supports. We will be interested in directly analyzing the support of each term. To this end,
note that
$\mathrm{ind}_{\mathcal{B}[\mathcal{E}_X]}(\mathscr{C}_M)\otimes_{\mathcal{B}[\mathcal{E}_X]}^L\mathcal{B}_{\leq n}[\mathcal{E}_X]\simeq \mathrm{ind}_{\mathcal{B}_{\leq n}[\mathcal{E}_X]}(\mathscr{C}_M).$
Moreover,
$$\mathsf{Free}(M)[\mathcal{E}_X]\simeq \pi^*\mathrm{Sym}_{\mathcal{O}_X}(M)\otimes_{\pi^{-1}\mathcal{D}_X}\mathcal{E}_X\simeq \mathrm{Sym}_{\mathcal{O}_{T^*X}}(\pi^*M)\otimes_{\mathcal{O}_{T^*X}}\mathcal{E}_X.$$
Combining these relations,
$$\mathrm{ind}_{\mathsf{Free}_{\D}[\mathcal{E}_X]}(\mathscr{C}_M)\simeq \mathsf{Free}_{\D}(M)[\mathcal{E}_X]\otimes\mathscr{C}_M\simeq \mathrm{Sym}_{\mathcal{O}_{T^*X}}(\pi^*M)\otimes_{\mathcal{O}_{T^*X}}^L\mathscr{C}_M,$$ 
as $\mathrm{Sym}_{\mathcal{O}_{T^*X}}(\pi^*M)$-modules.
In this case,
$$R\mathcal{S}ol_{\mathsf{Free}_{\D}(M)[\mathcal{E}_X]}\big(\mu (\mathbb{T}_{\mathsf{Free}_{\D}(M)}),\mathrm{ind}_{\mathsf{Free}_{\D}(M)[\mathcal{E}_X]}(\mathscr{C}_M)\big)\simeq R\mathcal{H}om_{\mathrm{Sym}(\pi^*M)[\mathcal{E}_X]}\big(\mu T_{\mathsf{Free}_{\D}(M)},Sym\pi^*M\otimes\mathscr{C}_M\big).$$
Since $\mathbb{L}_{\mathsf{Free}_{\D}(M)}\simeq \mathsf{Free}_{\D}(M)\otimes M$ with Verdier dual
$\mathbb{T}_{\mathsf{Free}_{\D}(M),\phi}^{\ell}\simeq M^{\vee}\otimes \omega_X^{-1}[\dim X],$ we have that 
$$\mu(\mathbb{T}_{\mathsf{Free}_{\D}(M),\phi}^{\ell})\simeq \mathcal{E}_X\otimes_{\D_X}M^{\vee}\otimes\omega_X^{-1}[\dim X]\simeq \mu(M^{\vee})\otimes\omega_X^{-1}[\dim X].$$ Since $M$ is perfect, $N:=M^{\vee}$ is perfect as a $\D$-module. We claim that we may then apply Theorem \ref{Thm: LinProp} to each free $\D$-algebra on compact $\D$-modules:
\begin{eqnarray*}
R\mathcal{S}ol_{\mathsf{Free}_{\D}(M)[\mathcal{E}_X]}\big(\mu(\mathbb{T}_{\mathsf{Free}_{\D}(M),\phi}^{\ell},\mathrm{ind}_{\mathsf{Free}_{\D}(M)[\mathcal{E}_X]}(\mathscr{C}_M)\big)&\simeq& R\mathcal{H}om_{\mathsf{Free}_{\D}(M)[\mathcal{E}_X]}\big(\mu(\mathbb{T}_{\mathsf{Free}_{\D}(M),\phi}^{\ell},\mathrm{ind}_{\mathsf{Free}_{\D}(M)[\mathcal{E}_X]}(\mathscr{C}_M)\big)
\\
&\simeq& R\mathcal{H}om_{\mathcal{E}_X}\big(N,\mathscr{C}_M)^{\ell},
\end{eqnarray*}
so that 
$$\mathrm{SS}R\mathcal{H}om_{\mathcal{E}_X}\big(N,\mathscr{C}_M)^{\ell}\subset \mathrm{Char}(N)\cap T_M^*X,$$
as $SS(\mathscr{C}_M)\simeq T_M^*X.$
Note we may identify 
$$\mathrm{Char}(N)=\mathsf{Ch}_{\D}\big(\phi^*\mu(\mathbb{T}_{\mathsf{Free}_{\D}(M)})\big).$$
We now extend this propagation to pushouts:
$\mathcal{B}_{i+1} = \mathcal{B}_i \sqcup_{\Free(\mathcal{M}_i)} \Free(\mathcal{M}_{i+1})$ in $\mathsf{QCAlg}(X_{\DR})$ whose underlying $\D$-modules are compact.
Since homotopy-pushout are sent to pullbacks via Proposition \ref{Retract proposition}, the pushout
\[\begin{tikzcd} \Free(\mathcal{M}_i) \ar[r,"c_i"] \ar[d,"\alpha_i"] & \Free(\mathcal{M}_{i+1}) \ar[d,"\alpha_{i+1}"] \\ \mathcal{B}_i \ar[r,"\beta_i"] & \mathcal{B}_{i+1} 
\end{tikzcd}\]
gives that $R\mathcal{S}ol_{\mathcal{B}_{i+1}[\mathcal{E}_X]}\bigl( {}_{\mathcal{B}_{i+1}}\mu(\mathbb{T}_{\mathcal{B}_{i+1}}^\ell),\; \mathrm{ind}(\mathcal{L}) \bigr)$ is naturally equivalent to 
$$R\mathcal{S}ol_{\mathcal{B}_i[\mathcal{E}_X]}\bigl( {}_{\mathcal{B}_i}\mu(\mathbb{T}_{\mathcal{B}_i}^\ell),\; \mathrm{ind}(\mathcal{L}) \bigr) \underset{R\mathcal{S}ol_{\Free(\mathcal{M}_i)[\mathcal{E}_X]}\bigl(\mu(\mathbb{T}_{\Free(\mathcal{M}_i)}^\ell),\; \mathrm{ind}(\mathcal{L}) \bigr)}{\times} R\mathcal{S}ol_{\Free(\mathcal{M}_{i+1})[\mathcal{E}_X]}\bigl( {}_{\Free(\mathcal{M}_{i+1})}\mu(\mathbb{T}_{\Free(\mathcal{M}_{i+1})}^\ell),\; \mathrm{ind}(\mathcal{L}) \bigr).$$
Note
$$R\mathcal{S}ol_{\mathcal{B}_{\leq n}[\mathcal{E}_X]}\big(\mu f_n^*\mathbb{T}_{\mathcal{B},\phi}^{\ell},\mathrm{ind}_{\mathcal{B}_{\leq n}[\mathcal{E}_X]}(\mathscr{C}_M)\big)\to R\mathcal{S}ol_{\mathcal{B}_{\leq n}[\mathcal{E}_X]}\big(\mu\mathbb{T}_{\mathcal{B}_{\leq n},\phi}^{\ell},\mathrm{ind}_{\mathcal{B}_{\leq n}[\mathcal{E}_X]}(\mathscr{C}_M)\big).$$
Since $\mathcal{B}_{\leq n}\in \mathsf{QCAlg}(X_{\DR})^{\leq 0,\geq -n}$ compact we have $\mathcal{B}_{\leq n}\simeq |\mathcal{B}_{n,\bullet}|$ is a homotopy retract i.e. we can reduce to the case of $\mathcal{B}_{n,i}\simeq \mathrm{Sym}(M_{n,i})$ for $\{M_{n,i}\}_{i\in I}\in \mathsf{Perf}(\D_X)$ for some finite index set $I,$ for every $n\geq 0.$ Then, each 
$$R\mathcal{S}ol_{\mathcal{B}_{n,i}[\mathcal{E}_X]}\simeq R\mathcal{H}om_{\mathrm{Sym}(M_{n,i})[\mathcal{E}_X]}\big(\mu \phi^*(\mathrm{Sym}(M_{n,i})\otimes M_{n,i})^{\ell},\mathrm{ind}_{\mathrm{Sym}(M_{n,i})[\mathcal{E}_X]}(\mathscr{C}_M)\big).$$
This is nothing but $R\mathcal{H}om_{\mathcal{E}_X}\big(\mu(M_{n,i}^{\ell}),\mathcal{E}_X\otimes\mathscr{C}_M\big)=R\mathcal{S}ol_{\mathcal{E}_X}\big(M_{n,i}^{\ell},\mathscr{C}_M\big),$ so that
\begin{equation}
    \label{n,iPropEstimate}
\mathrm{SS}\big(R\mathcal{H}om_{\mathcal{E}_X}\big(\mu(M_{n,i}^{\ell}),\mathcal{E}_X\otimes\mathscr{C}_M\big)\big)\subset \mathrm{Char}(M_{n,i})\cap T_M^*X.\end{equation}
We identified as above $\mathrm{Char}(M_{n,i})=\mathsf{Ch}\big(\phi^*\mu_{\mathsf{Free}(M_{n,i})}(\mathbb{T}_{S(M_{n,i})})\big).$ 

Thus, 
$$\mathrm{SS}\big(R\mathcal{S}ol_{\mathcal{B}[\mathcal{E}_X]}\big)\hookrightarrow \mathrm{SS}\big(\mathrm{holim}_{n\geq 0}R\mathcal{S}ol_{\mathcal{B}_{\leq n}[\mathcal{E}_X]}\big)\subset \bigcup_{n\geq 0}\mathrm{SS}\big(R\mathcal{S}ol_{\mathcal{B}_{\leq n}[\mathcal{E}_X]}\big).$$
We then apply the comapctness property reducing each singular support in the union over $n\geq 0,$ to a union over $i\in I$ of 
$$\big[\bigcup_{i\in I}\mathrm{Char}(M_{n,i})\big]\cap T_M^*X=\bigcup_{i\in I}\mathrm{Char}(M_{n,i})\cap T_M^*X.$$
Putting it together 
$R\mathcal{S}ol_{\mathcal{B}[\mathcal{E}_X]}\hookrightarrow \bigcup_{n\geq 0}\big[\bigcup_{i\in I}\mathrm{Char}(M_{n,i})\cap T_M^*X\big],$
thus in light of the estimates  \eqref{n,iPropEstimate}, we obtain the desired inclusion. Explicitly, since $V$ is involutive by the Frobenius theorem it defines a foliation whose leaves are bicharacteristics, so $\Sigma$ does indeed exist for such $V$, and $\Sigma$ are generically non-empty.
The proof amounts to demonstrating that
\begin{equation}
\label{RiGammaSupp}
R^i\Gamma_{\{\psi\geq 0\}}\big(\mathcal{E}xt_{\pi^*\mathcal{B}[\mathcal{E}_X]}^j\big(\mu\mathbb{T}_{\mathcal{B},\varphi}^{\ell},\mathcal{L}|_{\Sigma}\big)_{p_0}\simeq 0,\hspace{2mm} i\geq 0,j\geq 0,
\end{equation}
for $\mathcal{L}$ a given suitable function space for micro-differential equations with simple characteristics base-changed to obtain coefficients in $\mathcal{B}$. In this case, we consider
$$
R^j\Gamma_{\{\psi\geq 0\}}\big(\mathcal{H}om_{\mathcal{B}[\mathcal{E}_X]}\big(\mu\mathbb{T}_{\mathcal{B},\varphi}^{\ell},\mathcal{L}\big)|_{\Sigma},$$
is the $j$-th cohomology group of the complex
$\mathcal{S}^{\bullet}:=\bigoplus_{i+j=p}\mathcal{S}^{p},$
associated to the double complex
$$0\rightarrow \Gamma_{\{\psi\geq 0\}}\big(\mathcal{R}_{\mathcal{L}}^{\bullet}\big)^{N_0}\rightarrow \cdots\rightarrow  \Gamma_{\{\psi\geq 0\}}\big(\mathcal{R}_{\mathcal{L}}^{\bullet}\big)^{N_r}\rightarrow 0,$$
whose terms $\mathcal{S}^{i,j}:=\Gamma_{\{\psi\geq 0\}}(\mathcal{R}_{\mathcal{L}}^i)^{N_j},$ are obtain by taking a locally free resolution of the microlocalization linearized $\D_X$-module, by a bounded complex of free and finite rank $\mathcal{B}[\mathcal{E}_X]$-modules. These exist locally on $X$ and by pulling back along a solution, gives a locally on $X$ free resolution by $\mathcal{E}_X$-modules. Fix a resolution of the $\mathcal{E}_X$-module $\mathcal{L}$ on $\Sigma.$ One may then observe the $i$-th cohomology of 
$$\Gamma_{\{\psi\geq 0\}}\big(\mathcal{R}_{\mathcal{L}}^{\bullet}\big)=\bigg(\Gamma_{\{\psi\geq 0\}}\big(\mathcal{R}_{\mathcal{L}}^{-k}\big)\rightarrow \cdots \rightarrow \Gamma_{\{\psi\geq 0\}}\big(\mathcal{R}_{\mathcal{L}}^{0t}\big)\bigg),$$
which is bounded, is given by $R^i\Gamma_{\{\psi\geq 0\}}\big(\mathcal{L}\big).$
It follows that the complex
$$0\rightarrow H^j\big(\Gamma_{\{\psi\geq 0\}}\big(\mathcal{R}_{\mathcal{L}}^{\bullet}\big)^{N_0}\rightarrow \cdots \rightarrow H^j\big(\Gamma_{\{\psi\geq 0\}}\big(\mathcal{R}_{\mathcal{L}}^{\bullet}\big)\big)^{N_r}\rightarrow 0,$$
is exact for each $j$. By considering the $\mathcal{B}[\mathcal{D}_X]$-tangent complex $\D_X$-module $\mathbb{T}_{\mathcal{B}}^{\ell}$, rather than its pullback by base change from $\mathcal{O}_X[\mathcal{E}_X]$-modules to $\mathcal{B}[\mathcal{E}_X]$-modules on $\mathsf{Spec}(\mathcal{B})\times_{T^*X}T_VT^*X,$ this is equivalent to
$$\mathrm{SS}\big(R\mathcal{S}ol_{\mathcal{E}_X}\big(\mu\mathbb{T}_{\mathcal{B},\varphi}^{\ell},\mathcal{L}\big)|_{\Sigma}\subset \mathrm{Char}_{V,\varphi}^1(\mathcal{B})\times_V\Sigma,$$ 
as required.
\end{proof}
This result states
the microlocalized $\D$-geometric linearization $\mu\mathbb{T}$ has no microlocal solutions with values in $\mathcal{L}$ supported by $\{x|\psi(x)\geq 0\}$ in a neighbourhood of the point $p_0$ on $\Sigma.$ In other words, we may find a sufficiently close point $p_1=(x_1;\xi_1)$ and $\varphi$ with $\psi(x_1)=0$ and $d\psi(x_1)=\xi_1$ for which no sections are supported in this closed subset in the neighbourhood of $x_1.$

With additional assumptions on the linearization $\D$-module, we have estimates for propagation in terms of the $\D$-geometric microsupport.
\begin{corollary}
\label{NonlinPropCorollaryBody}
    Let $\mathcal{I}\rightarrow \mathcal{A}\rightarrow \mathcal{B}$ be a finitely presented $\D$-algebraic PDE such that for each solution $\varphi$ the linearization $\D$-module $\varphi^*\mathbb{T}_{\mathcal{B}}^{\ell}$ is a bounded complex with coherent cohomology. Then putting $V=T_X^*X,$ and $\mathcal{L}=\mathcal{O}_X$ one has that 
        $\mathrm{SS}\big(R\mathcal{S}ol_Y(\varphi_{\mathsf{IndCoh}}^*\mathbb{T}_{\mathsf{Sol}_X(\mathcal{B})}\big)\big)=\mathrm{Char}(\varphi_{\mathsf{IndCoh}}^*\mathbb{T}_{\mathsf{Sol}_X(\mathcal{B})}).$
    
\end{corollary}
This is the propagation result for the linearization $\D$-module:
assume $\psi$ a $C^1$ function on $X$ is such that $d\psi(x_0)\notin \mathrm{Char}(\varphi^*\mathbb{T}_{\mathsf{Sol}(\mathcal{B})}).$ Then, the corresponding $j$--th derived functor \eqref{RiGammaSupp} gives 
$$\mathcal{E}xt_{\D_X}^j\big(\varphi^*\mathbb{T}_{\mathsf{Sol}(\mathcal{B})},R\Gamma_{\{\varphi\geq 0\}}(\mathcal{O}_X)\big)_{x_0}\simeq 0,$$
which is the propagation estimate \eqref{eqn: ShvProp@x} at $x_0\in X.$

\begin{example}
\label{Free Propagation Example}
Let $M$ be a real analytic manifold with complexification $X$ and consider \autoref{Free Tangent Cotangent Example} i.e. the free derived $\D_X$-algebra 
$\mathcal{A}:=\mathrm{Sym}^*\big(Cone(\mathcal{D}_X\xrightarrow{P}\mathcal{D}_X)\big),$
with underlying $\D_X$-module $\mathcal{M}_P,$ i.e. $\mathcal{D}_X\xrightarrow{P}\mathcal{D}_X\rightarrow \mathcal{M}_P.$
Suppose that 
\begin{equation}
\label{Assumptions}
\begin{cases}
\mathrm{Char}(\mathcal{M}_P)\cap T_M^*X\subset T_X^*X \textit{ and $\gamma$ is a closed conic subset of the manifold $X$, }
\\

\textit{ with the the exterior conormals to $\gamma$ at a point $z_0\in \partial \gamma$ satisfying }
\\

\textit{ geometric hyperbolicity condition for $P$.}
\end{cases}
\end{equation}
With assumptions (\ref{Assumptions}) we get
$\mathsf{Sol}_X(\mathcal{A},\mathscr{A}_M)\simeq \mathbb{R}\mathsf{Sol}_{X,lin}(\mathcal{M}_P,\mathscr{A}_M),$
by adjunction (\ref{eqn: Free-forget AD-Mod/Alg Adjunction}) and by elliptic regularity
$\mathsf{Sol}_X\big(\mathcal{A},\mathscr{A}_M)\simeq \mathbb{R}\mathsf{Sol}_{X,lin}(\mathcal{M}_P,\mathscr{B}_M).$
Via geometric hyperbolicity, \ref{Propagation for hyperbolic directions}, gives that 
$\mathsf{Ext}_{\mathcal{R}}^j\big(\mathcal{M},\Gamma_{\gamma}(\mathscr{B}_M)\big)_{z_0}\simeq 0.$
That is, we have propagation in the hyperbolic directions e.g. the microsupport is contained in $\mathrm{Char}_M^{hyp}(\mathcal{M}_P).$ In other words, we have an isomorphism
$$R\mathcal{S}ol_{X}(\mathcal{M}_P,\mathscr{B}_M)\xrightarrow{\simeq} R\Gamma_{\gamma}(R\mathcal{S}ol_X(\mathcal{M}_P,\mathscr{B}_M)\big).$$
\end{example}

\subsection{Singular non-linear Cauchy problems}
Consider
the singular non-linear initial value problem in the complex analytic domain imposed on functions $u^{\alpha}(t,x),\alpha=1,\ldots,m$ of a trivial rank $m$ vector bundle (i.e. sections $u:X\rightarrow E$) given by
\begin{equation}
    \label{eqn: NLCauchyProblem}
\begin{cases}
    \frac{\partial^L}{\partial t^L}u^{\alpha}=\mathsf{F}\big(t,x,u^{\alpha},\partial_x^{\sigma}\partial_t^ju^{\alpha}\big),\hspace{2mm} |\sigma|+j\leq L,j<L
    \\
    \\
\frac{\partial^k}{\partial t^k}u^{\alpha}\big|_{t=0}=\Phi_k(x),\hspace{2mm} k=0,1,\ldots,m-1,
\end{cases}
\end{equation}
where  
for $x\in U\subset \mathbb{C}^n$ where $U$ is an open subset, and $t\in \mathbb{R}$, where $\mathsf{F}$ depends on the derivatives of $u$ with respect to $x,t$ of order $\leq L,$ and we take $\Phi_k(x)$ to be generically of the form
\begin{equation}
\label{eqn: NLCauchyFunctions}
\Phi_k(x)=\prod_{\ell=1,\beta_{\ell}\in\mathbb{Z}}^{n-1}C_{\beta_{\ell}}\cdot x_{\sigma_{\ell}}^{\beta_{\ell}},
\end{equation}
for complex coefficients $C_{\beta_{\ell}}\in\mathbb{C}$ and mult-indices $\sigma_{\ell}=(\sigma_1,\ldots,\sigma_{\ell})$.

\begin{remark}
If $\mathsf{F}$ is analytic in all its
arguments and if the $\Phi_k$
 are analytic, then the CK theorem
asserts the existence of a unique analytic solution in a neighborhood of any
initial point.
In the case of linear systems it is not necessary to assume analyticity in $t$\footnote{For instance, if $\mathsf{F}$ is continuous in $t$, analytic in other variables, there exists a unique solution $u(t,x)$ continuously differentiable in $t$ with values in
analytic functions of $x$ around our initial point.}.
\end{remark}

System (\ref{eqn: NLCauchyProblem}) generalizes what is already known in the setting of singular linear Cauchy problems
\begin{equation*}
    \begin{cases}
        P(u)=0,
        \\
        \frac{\partial^k u}{\partial t^k}\big|_{t=0}=\psi_k(t,x),\hspace{1mm}k=0,\ldots,m-1,
    \end{cases}
\end{equation*}
for a linear differential operator $P$ of order $\leq m.$ If $X=\mathbb{C}^n,$ and $\Sigma=\{t=0\}$ is the initial surface, let us set $Z=\{(t,x_1,\ldots,x_n)\in X|t=x_1=0\},$ and thus one considers those systems with initial data either of the form
$$\textbf{(a)} \begin{cases}
    u(t,x)|_{\Sigma}\equiv u(0,x)=\psi_k(x),
    \\
    \partial_t^k u(t,x)|_{t=0}=0,
\end{cases}
\hspace{5mm}\textbf{(b)}\begin{cases}
    u(t,x)|_{\Sigma}\equiv u(0,x)=0
    \\
    \partial_t^k u(t,x)|_{t=0}=\psi_k(x),\end{cases},$$
where $k=0,1,\ldots,m-1.$
In \textbf{(a)}, one may have singular data on the initial surface $\Sigma,$ say of the form $\psi_k(x)=\frac{1}{x_1}$. It may be geometrically modelled by a hyper-surface $Z=\{t=x_1=0\}$ in $\Sigma.$ In this situation one looks for solutions $u=u(t,x)$ whose singularities are projections onto $X$ of bicharacteristic curves of $\mathrm{Char}(P)$ issuing from the data $(u|_{\Sigma},\partial_tu|_{\Sigma})$ i.e. from
$$f_{\pi}^{-1}(T_Z^*\Sigma)\cap \mathrm{Char}(P).$$

In \textbf{(b)}, an example of singular data (\ref{eqn: NLCauchyFunctions}) are functions  
$$\psi_k(x)=\rho_k(x)\cdot \gamma_{a}(x),$$
where $\rho(x)$ is a holomorphic function on $\Sigma\cap \Omega$, where $\Omega$ is an open subset of $\mathbb{C}^n$ containing the origin and where $\gamma_a(x)$ are singular functions of the variables $x$, depending on complex parameters $a\in\mathbb{C}$ e.g. $\gamma_a(x)=x^{a}/\Gamma(a+1).$
For instance, $Z=\{x_2=x_3=0\}$ then taking $\psi_1(x)=1/x_2^3\cdot x_3^2$, the singular linear problem can be solved for several different operators e.g. $P=\partial_1^2+(x_2\partial_2)^2+(x_3\partial_3)^2$ or $Q=\partial_1^2+(x_2\partial_3)^2+x_3^2\partial_2.$

\begin{example}
\normalfont
Consider $\mathbb{C}^2\simeq \mathbb{C}_t\times\mathbb{C}_x,$ and the singular Cauchy problem,
$$\partial_t^2u-\partial_xu=0,\hspace{2mm} u(t,x)=x^{-1},\hspace{2mm} \partial_tu|_{t=0}=0,$$
Set $\Sigma=\{t=0\}$ and $Z=\{t=x=0\}$ so $\mathrm{Char}(P)=\{(x,\xi)|\xi_0^2=0\}.$ Putting $t=x_0$ and $x_1=x$ then $(t,x;\xi_0,x_1)\in T^*X$ and 
$f_{\pi}^{-1}(T_Z^*\Sigma)\cap \mathrm{Char}(P)=\{(x;\xi)|x_0=x_1=0, \xi_0^2=0,\xi_1=0\}.$
The initial data has an order $1$ pole, while solutions which take the form
$$u(t,x)=x^{-1}+x^{-1}\sum_{n=1}^{\infty}(-1)^n\frac{n!}{(2n)!}\big(\frac{t^2}{x}\big)^n,$$
have essential singularities along $Z=\{x=0\}.$
\end{example}

We reformulate this in the non-linear setting of $\D$-algebras.

\subsubsection{} 
The $\D$-algebraic non-linear \textsc{pde} on sections of the trivial rank $m$ vector bundle $E\rightarrow X$ which corresponds to the system (\ref{eqn: NLCauchyProblem}) 
is 
$0\to \mathcal{I}_{\mathsf{F}}\rightarrow \mathcal{A}=\mathcal{O}\big(\mathsf{Jet}^{\infty}(\mathcal{O}_E^{alg})\big)\rightarrow \mathcal{B}\to 0,$
where the ideal $\mathcal{I}$ is the $\D_X$-ideal generated freely by the relation $u_t^L-\mathsf{F}.$
In other words, 
$\mathcal{B}_{\mathsf{F}}:=\mathcal{A}/\mathcal{I}_{\mathsf{F}}=\mathrm{Jet}^{\infty}\big(\mathcal{O}_E^{\mathrm{alg}}\big)/<u_t^L-\mathsf{F}>.$
Proposition \ref{Jet Representability} gives a representable space of solutions 
$\mathsf{Sol}(\mathcal{B}_{\mathsf{F}})\simeq \mathsf{Spec}_{\D_X}(\mathcal{A}/\mathcal{I}_{\mathsf{F}}).$
We are interested in the $\mathcal{O}_X$-valued classical solutions 
\begin{equation*}
\mathbb{R}\underline{\mathrm{Sol}}_X(\mathcal{B}_{\mathsf{F}})(\mathcal{O}_X)\simeq \Map\big(\iota\mathsf{Spec}_{\D_X}(\mathcal{O}_X),\Spec_{\D_X}(\mathsf{B})\big).
\end{equation*}
The singular initial value problem (\ref{eqn: NLCauchyProblem}) then corresponds to studying pull-back solutions along the hypersurface embedding $\iota:\Sigma_0:=\{t=0\}\hookrightarrow X$ for which the $\D$-pullback operation is given by
$$
\iota_{\mathcal{D}}^*\big(\frac{\partial^{\alpha} u}{\partial x^{\alpha}}\big)=\frac{\partial^{\alpha}}{\partial x^{\alpha}}\big(f^*u\big),\hspace{5mm}\text{ and }\hspace{1mm}
\iota_{\mathcal{D}}^*\big(\frac{\partial^{k} u}{\partial t^{k}}\big)=\frac{\partial^k}{\partial t^k}\big(u|_{\Sigma_0}\big),$$
where moreover we have singularities along a sub-manifold 
$$Z=\{(t,x_1,\ldots,x_{n-1})\in X|x'=0\}\subset \Sigma_0,$$
where $x'$ is some finite subset $(x_1,\ldots,x_{p})$ of coordinates. 

\begin{proposition}
\label{eqn: Non-lin CKK prop}
The map
$\iota^{-1}\mathcal{H}\mathrm{om}_{\mathrm{CAlg}_{\mathcal{D}}}\big(\mathcal{A},\mathcal{O}_Y\big)\rightarrow \mathcal{H}\mathrm{om}_{\mathrm{CAlg}_{\mathcal{D}}}\big(\iota_{\mathcal{D}}^*\mathcal{A},\mathcal{O}_Y\big),
$ is an isomorphism.
\end{proposition}
\begin{proof}
Note the pull-back is compatible with both
 the algebra structure and the $\D_X$-action, since
$\mathcal{A}$ is generated as an $\mathcal{O}_E^{\mathrm{alg}}$-algebra by coordinates $\{u_{t,x}^{k,\alpha}\}$ with
$$\mathcal{D}_X\bullet \mathcal{A}\rightarrow\mathcal{A},\hspace{5mm} \sum_{k,\alpha}\Box_{k,\alpha}\partial_t^k\partial_x^{\alpha},u_{t,x}^{\tilde{k},\tilde{\alpha}}\mapsto \Box_{k,\alpha}u_{t,x}^{k+\tilde{k},\alpha+\tilde{\alpha}},$$
and so $f_{\mathcal{D}}^*\mathcal{A}=\mathrm{span}_{f^*\mathcal{O}_E^{\mathrm{alg}}}\big(\{u_{t,x}^{k,\alpha}\big)$ with $\D$-action 
$\mathcal{D}\otimes f_{\mathcal{D}}^*\mathcal{A}\rightarrow f_{\mathcal{D}}^*\mathcal{A}$ given explicitly by $\Box_I\partial_x^I\otimes G(x,u_{t,x}^{k,\alpha})\mapsto G(x,u_{t,x}^{k,I+\alpha}).$
Consider $\mathrm{Sol}_{\mathcal{D}}(\mathcal{A}):=\mathcal{H}\mathrm{om}_{\mathrm{CAlg}_X(\D_X)}\big(\mathcal{A},\mathcal{O}_X\big)$, as in (\ref{eqn: Classical D-Solutions}). The quotient map $\mathrm{Jet}^{\infty}(\mathcal{O}_E)\rightarrow \mathcal{A},$ induces the inclusion 
$$\mathrm{Sol}_{\mathcal{D}}(\mathcal{A})\hookrightarrow \mathcal{H}om_{\mathcal{D}_X-CAlg}(Jet^{\infty}(\mathcal{O}_E),\mathcal{O}_X)\simeq \mathcal{H}\mathrm{om}_{\mathcal{O}_X-\mathrm{CAlg}}\big(\mathcal{O}_E,\mathcal{O}_X\big)\xrightarrow{\eta}\mathcal{O}_X^{\oplus m},$$ via Proposition \ref{Jet-adjunction}. The morphism $\eta$ is given by taking the first $(m-1)$-traces of a solution 
$\eta(u):=\big(u|_{\Sigma},\partial_1u|_{\sigma},\ldots,\partial_{1}^{m-1}u|_{\Sigma}\big).$ The result follows by applying Theorem \ref{Thm: LinCKK}.
\end{proof}
Given a morphism $f:X\rightarrow Y$, set $\mathrm{CAlg}_{f-\mathsf{NC}}(\mathcal{D}_Y)$ to be the sub-category of left $\mathcal{D}_Y$-algebras for which $f$ is non-characteristic. Recall the notion of being $\D$-CK \ref{eqn: D-CK Ideals}). 
\begin{proposition}
\label{D-CK-NC}
A $\mathcal{D}_Y$-algebraic non-linear \textsc{pde} $\IAB$ which is $\mathcal{D}_Y$-Cauchy-Kovalevskaya is an object of $\mathrm{CAlg}_{f-\mathsf{NC}}(\mathcal{D}_Y).$
\end{proposition}

    \subsection{Derived $\D$-Geometric Non-linear CKK}
In this subsection we prove the derived $\D$-geometric non-linear CKK theorem (see \cite{Paugam2022} and \cite{Nirenberg}). See also \cite{Ishii},\cite{Kobayashi} for classical treatments in the singular setting. Recall the notion of pair-wise transversality (\ref{sssec: Cellular Objects}).

\begin{theorem}
\label{theorem: DNLCKK}
Let $f:X\rightarrow Y$ be a morphism of complex analytic manifolds and $V$ a closed conic subset of $T^*Y$. Suppose that 
$\mathcal{I}_Y\rightarrow \mathcal{A}_Y\rightarrow \mathcal{B}_Y,$
is a $\D$-algebraic PDE over $Y$ with pair-wise transversal homotopically finitely presented solution space $\mathsf{Sol}_Y(\mathcal{B}_Y)$ for which $f$ is $1$-micro-non-characteristic along $V$ with $p\in V.$
Then the functors of pull-back commute with formation of derived $\D$ solution stacks. In particular, if $V=T_Y^*Y$ is the zero-section Lagrangian sub-manifold, then we have equivalences 
$f^{-1}\mathsf{Sol}_Y(\mathcal{B}_Y)\xrightarrow{\simeq}\mathsf{Sol}_X(f^*\mathcal{B}_Y).$ 
\end{theorem}

Since the Cauchy problem is a local question, one may decompose the morphism $f$ as a composition of a closed (graph) immersion and a smooth projection,
$X\overset{j}{\hookrightarrow} X\times Y\overset{p}\rightarrow Y,x\mapsto \big(x,f(x)\big)$ with $p(x,y)=y$ and discuss the proof of each case separately.

The proof of the result will rely on some additional facts that we establish first. 

\subsubsection{}Let $f_0:\mathcal{A}^{\bullet}\rightarrow \mathcal{B}^{\bullet}$ in $\mathsf{CAlg}_X(\D_X)^{\leq 0}$ and consider,
$$f_!(\mathbb{L}_{\mathcal{A}}):=\mathbb{L}_{\mathcal{A}}\otimes_{\mathcal{A}}\mathcal{B}\rightarrow \mathbb{L}_{\mathcal{B}}\rightarrow \mathbb{L}_{\mathcal{B}/\mathcal{A}}.$$
We have a relative cofiber sequence
\[
\begin{tikzcd}
    \mathbb{L}_{\mathcal{A}}\arrow[d]\arrow[r] & \mathbb{L}_{\mathcal{B}}\arrow[d]
    \\
    0\arrow[r] & \mathbb{L}_{\mathcal{B}/\mathcal{A}}
\end{tikzcd}
\]
if and only if there is a distinguished triangle 
$f_!\mathbb{L}_{\mathcal{A}}\rightarrow \mathbb{L}_{\mathcal{B}}\rightarrow \mathbb{L}_{\mathcal{B}/\mathcal{A}}\rightarrow f_!\mathbb{L}_{\mathcal{A}}[1].$

We are interested in $\mathbb{L}_{\mathcal{B}},$ and its $\mathcal{B}[\mathcal{D}_X]$-dual, under the finiteness assumptions that
 $\mathcal{B}$ is homotopically finitely $\D$-presented. 
In what follows we use the cotangent complex.
In this way, any push-out (\ref{eqn: Hopushout squares}) gives us that
$$(\alpha_{i+1})_!\mathbb{L}_{c_i}\simeq \mathbb{L}_{\mathcal{A}_{i+1}/\mathcal{A}_i},$$
which follows from
\begin{equation}
    \label{eqn: Pushout cotangent diagram}
    \begin{tikzcd}
      (\alpha_{i+1})_!c_{i!}\mathbb{L}_{\Free(\mathcal{M}_i)}\arrow[d]\arrow[r] & (\alpha_{i+1})_!\mathbb{L}_{\Free(\mathcal{M}_{i+1})}\arrow[d]
      \\
      \beta_{i!}\mathbb{L}_{\mathcal{A}_i}\arrow[d] \arrow[r] & \mathbb{L}_{\mathcal{A}_{i+1}}
      \arrow[d]
      \\
      0\arrow[r] & \mathbb{L}_{\mathcal{A}_{i+1}/\mathcal{A}_i}.
    \end{tikzcd}
\end{equation}

In fact, (\ref{eqn: Pushout cotangent diagram}) simplifies further to 
\[
\begin{tikzcd}
    \mathcal{M}_i\otimes^!\mathcal{A}_{i+1}\arrow[d]\arrow[r] & \mathcal{M}_{i+1}\otimes^!\mathcal{A}_{i+1}\arrow[d]
    \\
    \mathbb{L}_{\mathcal{A}_i}\otimes_{\mathcal{A}_i}\mathcal{A}_{i+1}\arrow[d] \arrow[r] & \mathbb{L}_{\mathcal{A}_{i+1}}\arrow[d]
    \\
    0\arrow[r] & \mathbb{L}_{\mathcal{A}_{i+1}/\mathcal{A}_i}.
\end{tikzcd}
\]
Suppose $\mathcal{B}_0$ is smooth as a $\D_X$-algebra (\ref{D smooth definition}). In particular $\Omega_{\mathcal{B}_0}^1$ is free and finitely presented as a $\mathcal{B}_0[\mathcal{D}_X]$-module and there is a finite type $\mathcal{O}_X$-module $\mathcal{E}_0$ together with an ideal $\mathcal{I}_0$ of $\mathrm{Sym}_{\mathcal{O}_X}(\mathcal{E}_0)$ such that there is a short-exact sequence
$$\mathcal{K}_0\rightarrow Jet^{\infty}\big(\mathrm{Sym}(\mathcal{E}_0)/\mathcal{I}_0\big)\otimes
_{\mathcal{O}_X}\mathcal{N}\rightarrow \mathcal{B}_0\otimes \mathcal{N},$$
where the kernel $\mathcal{K}_0$ is $\D_X$-finitely generated and $\mathcal{N}$ can be taken to be $\D_X$.

Suppose we may choose a free $\mathcal{B}_0[\mathcal{D}]$-module $\mathcal{M}_{-1}$ of finite rank i.e. it is a vector $\D_X$-bundle over $\mathsf{Spec}_{\D_X}(\mathcal{B}_0),$ together with a morphism 
$\beta_{-1}:\mathcal{M}_{-1}\rightarrow \mathcal{B}_0,$ such that $f:\Sigma\rightarrow X$ is NC for $\mathcal{M}_{-1}^{\bullet}.$ 
Then in particular \autoref{Thm: LinCKK} holds for $\mathcal{M}_{-1}$ i.e
$f^{-1}\mathbb{R}\mathsf{Sol}_{X,lin}(\mathcal{M}_{-1}^{\bullet})\simeq \mathbb{R}\mathsf{Sol}_{\Sigma,lin}(f^*\mathcal{M}_{-1}^{\bullet}).$

Note the adjunction (\ref{eqn: Free-forget AD-Mod/Alg Adjunction}) in this case given by
\[
\begin{tikzcd}
\mathrm{Sym}_{\mathcal{B}_0}^{*}(-):\mathsf{Mod}(\mathcal{B}_0[\mathcal{D}_X])\arrow[r,shift left=.5ex] & \mathsf{CAlg}_{\mathcal{B}_0/}(\D_X)\arrow[l,shift left=.5ex]:\mathsf{For}_{\mathcal{B}_0},
\end{tikzcd}
\]
i.e. $\mathbb{R}\mathcal{H}\mathrm{om}_{\mathcal{B}_0[\mathcal{D}_X]-Alg}\big(\mathrm{Sym}_{\mathcal{B}_0}^*(\mathcal{M}_{-1}^{\bullet}),\mathcal{O}_X\big)\simeq \mathbb{R}\mathcal{H}\mathrm{om}_{\mathcal{B}_0[\mathcal{D}_X]-Mod}\big(\mathcal{M}_{-1}^{\bullet},\mathsf{For}_{\mathcal{B}_0}\mathcal{O}_X\big).$
Applying the functor $f^{-1}$  we obtain that
$$
f^{-1}\mathbb{R}\mathcal{H}\mathrm{om}_{\mathcal{B}_0[\mathcal{D}_X]-Alg}\big(\mathrm{Sym}_{\mathcal{B}_0}^*(\mathcal{M}_{-1}^{\bullet}),\mathcal{O}_X\big),$$
is equivalently 
$$f^{-1}\mathbb{R}\mathcal{H}\mathrm{om}_{\mathcal{B}_0[\mathcal{D}_X]}\big(\mathcal{M}_{-1}^{\bullet},\mathsf{For}_{\mathcal{B}_0}\mathcal{O}_X\big)
\simeq \mathbb{R}\mathcal{H}\mathrm{om}_{Lf^*\mathcal{B}_0[\D_{\Sigma}]}\big(f^*\mathcal{M}_{-1}^{\bullet},\mathsf{For}_{\mathcal{B}_0}\mathcal{O}_\Sigma\big).$$
These considerations provide an equivalence via \autoref{Thm: LinCKK}
$$\gamma_{-1}:f^{-1}\mathsf{Sol}_{X}\big(\mathrm{Sym}^*(\mathcal{M}_{-1}^{\bullet})\big)\simeq \mathbb{R}\mathsf{Sol}_{\Sigma,lin}(f^*\mathcal{M}_{-1}^{\bullet}).$$

Now we use this choice of coherent module to get the first cell in the retract. To this end, define $\mathcal{B}_1^{\bullet}\in\mathsf{CAlg}_{X}(\D_X)$ which is free as a graded $\mathcal{B}_0$-algebra generated in degree $-1$ by $\mathcal{M}_{-1}.$ The differential $d_{\mathcal{B}_1}$ is defined by the projection $\beta_{-1}$. Notice that this implies that $H_{\D}^0(\mathcal{B}_1^{\bullet})\simeq \mathcal{B}_0/\mathcal{I}_{-1},$ where $\mathcal{I}_{-1}$ is the $\D_X$-ideal of $\mathcal{B}_0$ generated by the image of $\beta_{-1}.$
In other words, $\mathcal{B}_1^{\bullet}$ is the homtopy-pushout in $\mathsf{CAlg}_X(\D_X),$
\[
\begin{tikzcd}
  \mathrm{Sym}_{\mathcal{B}_0}^*(\mathcal{M}_{-1}^{\bullet}[-1])\arrow[d]\arrow[r] & \mathcal{B}_0\arrow[d]
  \\
  \mathcal{B}_0\arrow[r,"g_0"] & \mathcal{B}_1^{\bullet}
\end{tikzcd}
\]
to which, one applies the functors $R\mathrm{Sol}_{\D_X}(-)$ and $R\mathrm{Sol}_{\D_{\Sigma}}(f^*),$ giving the respective commutative diagrams
\begin{equation}
    \label{eqn: CKK Diagrams 1}
\begin{tikzcd}
R\mathrm{Sol}_{\D_X}\big(\mathrm{Sym}_{\mathcal{B}_0}^*(\mathcal{M}_{-1}^{\bullet}[-1])\big)& R\mathrm{Sol}_{\D_X}(\mathcal{B}_0)\arrow[l]
  \\
  R\mathrm{Sol}_{\D_X}(\mathcal{B}_0)\arrow[u] & R\mathrm{Sol}_{\D_X}(\mathcal{B}_1^{\bullet}),\arrow[l]\arrow[u]
\end{tikzcd}
\hspace{2mm}
\begin{tikzcd}
R\mathrm{Sol}_{\D_X}\big(f^*\mathrm{Sym}_{\mathcal{B}_0}^*(\mathcal{M}_{-1}^{\bullet}[-1])\big)& R\mathrm{Sol}_{\D_X}(f^*\mathcal{B}_0)\arrow[l]
  \\
  R\mathrm{Sol}_{\D_X}(f^*\mathcal{B}_0)\arrow[u] & R\mathrm{Sol}_{\D_X}(f^*\mathcal{B}_1^{\bullet}).\arrow[l]\arrow[u]
\end{tikzcd}
\end{equation}
By applying $f^{-1}(-)$ to the left-most diagram in (\ref{eqn: CKK Diagrams 1}), there is an induced isomorphism between the two.
In particular, using $\gamma_{-1}$ we therefore reduce the statement to asking for an isomorphism 
$f^{-1}R\mathrm{Sol}_{\D_X}(\mathcal{B}_1^{\bullet})\simeq R\mathrm{Sol}_{\D_{\Sigma}}(f^*\mathcal{B}_1^{\bullet}),$ but this is clear.
Note that we have used the fact that
$$R\mathrm{Sol}_{\D_{\Sigma}}\big(f^*\mathrm{Sym}_{\mathcal{B}_0}^*(\mathcal{M}_{-1}^{\bullet}[-1])\big)\simeq R\mathrm{Sol}_{\D_{\Sigma}}(\mathrm{Sym}(f^*\mathcal{M}_{-1}^{\bullet})\big).$$

We now choose a free $\mathcal{B}_1^{\bullet}[\mathcal{D}_X]$-module $\mathcal{M}_{-2}^{\bullet}$ which is moreover $\D_X$-coherent a map as above $\beta_{-2}:\mathcal{M}_{-2}^{\bullet}\rightarrow \mathcal{B}_1^{\bullet}$ in $\mathsf{Mod}\big(\mathcal{B}_1^{\bullet}[\mathcal{D}_X]\big).$ One defines an object $\mathcal{B}_2^{\bullet}$ by the homotopy pushout square
\[
\begin{tikzcd}
    \mathrm{Sym}_{\mathcal{B}_1}^*(\mathcal{M}_{-2}^{\bullet})\arrow[d]\arrow[r] & \mathcal{B}_1^{\bullet}\arrow[d]
    \\
    \mathcal{B}_1^{\bullet}\arrow[r] & \mathcal{B}_2^{\bullet}.
\end{tikzcd}
\]
One will verify, using the extension of the above arguments that $f$ is NC for $\mathcal{M}_{-2}^{\bullet}.$
This implies the result for $\mathcal{B}_2^{\bullet},$ and we may continue inductively, noting the degree $-k$ stage defines $\mathcal{B}_k^{\bullet}$ via
\[
\begin{tikzcd}
    \mathrm{Sym}_{\mathcal{B}_{k-1}}^*(\mathcal{M}_{-k}^{\bullet}[k-1])\arrow[d]\arrow[r] & \mathcal{B}_{k-1}^{\bullet}\arrow[d]
    \\
    \mathcal{B}_{k-1}^{\bullet}\arrow[r] & \mathcal{B}_k^{\bullet}.
\end{tikzcd}
\]
Consequently, our so-obtained sequence
$\mathcal{O}_X\rightarrow \mathcal{B}_1^{\bullet}\rightarrow\ldots\rightarrow\mathcal{B}_k^{\bullet}\rightarrow\ldots\rightarrow \mathcal{B}^{\bullet},$
is constructed such that each stage satisfies relevant CKK-isomorphisms. 

We may then apply Proposition \ref{HofpDAlgTangentComplex}, and Proposition  \ref{Tangent Good Filtration}. Namely, returning to the statement of \autoref{theorem: DNLCKK}, proceed by assuming our morphism $f:X\rightarrow Y$ is a composition of a smooth submersion and a closed embedding, treating each case separately.
\begin{proposition}
\label{CKK: Smooth}
Consider \autoref{theorem: DNLCKK}. Suppose that $f:X\rightarrow Y$ is a smooth morphism of complex analytic manifolds. Then there is a $\pi_0$-equivalence of sheaves of spaces
$$\mathrm{Ckk}_{\mathcal{D}}(f;\mathcal{B}):\mathsf{Sol}_{X}(\mathcal{B})\xrightarrow{\simeq}\mathsf{Sol}_{\Sigma}(f^*\mathcal{B}).$$
\end{proposition}
\begin{proof}
We recall the proof given in \cite{Paugam2022}.
Assume that $f:\Sigma=X\times Z\rightarrow X$ is the natural morphism from a product manifold, where it is known (\cite{kashiwara1970algebraic}) $\D$-module pullback $f^*$ is exact given by the derived $\D$-module external tensor $\mathcal{O}_Z\boxtimes^L(-).$ 

In the case of a homotopically finitely $\D$-presented algebra $\mathcal{A}$, given by retracts of cellular objects by \autoref{Compact D algebras}, proving the result amounts to proving it for objects $\mathrm{Sym}^*(\mathcal{M}^{\bullet})$ for a compact $\D$-module $\mathcal{M}^{\bullet}.$
The CKK isomorphism on a ho-fp algebra then reduces to showing that it is an isomorphism for such free $\D$-algebras
on a given coherent $\D$-module. The result reduces to the linear case of a coherent
$\D$-module, that reduces to the free case by using a locally free resolution. 
Via Kashiwara's argument, each coherent $\mathcal{M}_i$ has a locally free resolution and in particular, we can assume
$$0\rightarrow \mathcal{N}_i\rightarrow \mathcal{F}_i\rightarrow \mathcal{M}_i\rightarrow 0,$$
where $\mathcal{F}_i$ is a direct sum of quotients $\mathcal{D}_X/\mathcal{D}_X\cdot P$ where $P$ is an operator of CK-type c.f \autoref{CK-type and NC}.
Thus, we want to show that we have isomorphisms
$H^k\big(\mathrm{Ckk}_{i}^{lin}(\mathcal{M}_i)\big):\mathcal{E}xt_{\D_X}^k(\mathcal{M}_i,\mathcal{O}_X)\rightarrow \mathcal{E}xt_{\mathcal{D}_Y}^k(\mathcal{M}_{i,X\rightarrow Y},\mathcal{O}),$ but reduce to considering $\mathcal{H}^k\big(\mathrm{Ckk}_i^{lin}(\mathcal{F})\big)$ are isomorphisms, for each $k.$
The result follows from an application of the relative de Rham theorem \cite{kashiwara1970algebraic}.
\end{proof}
Before considering the case of a closed-embedding, note the following.
\begin{proposition}
    Let $A_X$ be $\D_X$-afp and suppose $f:Y\to X$ is smooth. Then $\mathrm{Ckk}_{\D}(f;A_X)$ is an equivalence of locally finitely presented affine derived $\mathbb{C}$-schemes. If $A_X\simeq \mathrm{Sym}(M)$ for a perfect $\D$-module, then $\mathrm{Ckk}_{\D}(f;;A_X)$ is an equivalence if $f$ is non-characteristic for $M$, in the usual $\D$-module sense.
\end{proposition}
\begin{proof}
Since $A$ is $\D_X$-afp, by Proposition \ref{prop: PreserveDafp}, so is $A_Y.$ Then, by Proposition \ref{prop: D-afp means RSolfinite} both solution spaces are finitely presented. We look at the case of a single $A_X\simeq \mathrm{Sym}^{\otimes^!}(M),$ for a compact $\D_X$-module $M.$
Since we know that 
$$\mathbb{R}\mathrm{Sol}_{\D_X}(A_X)\simeq \mathsf{Spec}_X\big(\mathrm{Sym}(R\Gamma(X_{\DR},M^{\vee})\big),$$ 
which since $M$ is compact, it is perfect thus dualizable with perfect dual $M^{\vee},$ and since $X$ is proper, $R\Gamma(X_{\DR},M^{\vee})$ is perfect complex of $k$-vector spaces, we have that 
$$f^{-1}R\underline{\mathrm{Sol}}_{\D_X}(A_X)\simeq f^{-1}\mathsf{Spec}_{X}\big(\mathrm{Sym}(R\Gamma(M^{\vee})\big)\to \mathsf{Spec}_{Y}(f^{-1}\mathrm{Sym}(R\Gamma(M^{\vee}))\big)\simeq \mathsf{Spec}_{Y}\big(\mathrm{Sym}(f^*R\Gamma(M^{\vee}))\big), $$
is an equivalence if and only if 
$$f^*R\Gamma(X_{\DR},M^{\vee})\simeq R\Gamma(Y_{\DR},f_{D}^*M^{\vee}).$$
This is the usual CKK-isomorphism, which follows from the relative de Rham theorem \cite{KashiwaraOshima1977}. 
With more details, since $Z$ is $\D_X$-afp, we may write $Z\simeq\underset{n}{\mathrm{lim}}Z_n,$ with $Z_n\simeq \mathsf{Spec}_{X_{\DR}}(\tau^{\geq -n}\mathcal{O}_Z).$ Thus,
$$\Map{/X_{\DR}}(X_{\DR},\EQ)\simeq \underset{n}{\mathrm{lim}}\Map{/X_{\DR}}(X_{\DR},\mathsf{Spec}_{X_{\DR}}(\tau^{\geq -n}\mathcal{O}_Z)\big)\simeq \underset{n}{\mathrm{lim}}\mathbb{R}\mathrm{Sol}_X(Z_n).$$
Since $f^{-1}$ is continuous and commutes with limits, 

$$f^{-1}\mathbb{R}\mathrm{Sol}_X(\EQ)\simeq f^{-1}\underset{n}{\mathrm{lim}}\mathbb{R}\mathrm{Sol}_X(Z_n)\simeq \underset{n}{\mathrm{lim}}f^{-1}\mathbb{R}\mathrm{Sol}_X(Z_n).$$
By hypothesis $\tau^{\geq -n}\mathcal{O}_Z$ is compact and $f^*\tau^{\geq -n}\mathcal{O}_Z$ is compact, thus $f^*Z$ is $\D_Y$-afp and 
$f^{-1}\mathbb{R}\mathrm{Sol}_X(Z_n)\simeq \mathbb{R}\mathrm{Sol}_Y(f^*Z_n),$ so that
$$f^*\underset{n}{\mathrm{lim}} \tau^{\geq -n}\mathcal{O}_Z\simeq \underset{n}{\mathrm{lim}}f^*(\tau^{\geq -n}\mathcal{O}_Z),$$
so that 
$$\underset{n}{\mathrm{lim}}f^{-1}\mathbb{R}\mathrm{Sol}_X(Z_n)\simeq \underset{n}{\mathrm{lim}}\mathbb{R}\mathrm{Sol}_Y(f^*Z_n)\simeq \mathbb{R}\mathrm{Sol}_Y(f^*Z),$$
thus the comparison map $f^{-1}\mathbb{R}\mathrm{Sol}_X(\EQ)\to \mathbb{R}\mathrm{Sol}_Y(f^*Z)$ is an equivalence when $f$ is smooth and $Z$ is $\D_X$-afp.
\end{proof}
\begin{proposition}
\label{CKK: ClosedEmbedding}
Consider Theorem \ref{theorem: DNLCKK} and Proposition \ref{CKK: Smooth}. Assume that $f:X\rightarrow Y$ is a closed embedding of complex analytic manifolds. Then the induced morphism $\mathrm{Ckk}_{\mathcal{D}}(f;\mathcal{B})$ is an equivalence of sheaves of spaces.
\end{proposition}
Proposition \ref{CKK: ClosedEmbedding} follows from Proposition \ref{Compact D algebras} and Proposition \ref{Retract proposition} together with the
description of cellular $\D$-algebras as push-outs (\ref{eqn: Hopushout squares}), defined in terms of cofibrations, and the fact a cofibration is a retract of a (transfinite) composition of pushouts of generating cofibrations $\mathrm{Sym}(S^{n-1})\rightarrow \mathrm{Sym}({D^n})$ with $S^{n-1}$ and $D^n$ the sphere and disk complexes with $\D_X$ in degrees $n-1,n$ and $n$, respectively. In particular, they are acyclic complexes which have cohomology given by $\D_X$ itself. In this case, Theorem \ref{Thm: LinCKK} obviously holds for the case of closed embeddings.
Moreover, one works inductively on $d:=codim_Y(X),$ the codimension of $X$ in $Y.$ But through the usual reductions, the chain of sub-manifolds
$$X=X_0\subset X_1\subset\cdots X_{d-1}\subset X_d=Y,$$
reduce the statement to the codimension $d=1$ situation.

We will describe only the argument at the level of an $i$-cell, as part of the general inductive procedure (recall also (\ref{eqn: CotComplex})).

\begin{proposition}
\label{Char Properties}
    Consider an $i$-cell (\ref{eqn: Hopushout squares}). Denote by $\mathsf{S}_i:=\Free(\mathcal{M}_i^{\bullet})$ for each $i,$ and $_{\mathsf{S}_i}\mu(-)$ the functor (\ref{eqn: Microlocalization with Coefficients}) of twisted microlocalization for $\mathsf{S}_i[\mathcal{D}_X]$-modules.
    \begin{enumerate}
\item $\mathrm{Char}_{\mathcal{D}}(\mathsf{S}_i)\subseteq supp\big(\hspace{1mm}_{\mathsf{S}_{i+1}}\mu(\mathsf{S}_{i+1}\otimes_{\mathcal{O}_X}\mathcal{M}_i)\big)\cup supp\big(\hspace{1mm}_{\mathsf{S}_{i+1}}\mu \mathbb{L}_{c_i}\big) ;$

\item If $\mathcal{A}= lim \mathcal{A}_i$ then $\mathrm{Char}_{\mathcal{D}}(\mathcal{A})\simeq \underset{i\in I}{\mathrm{lim}}\hspace{1mm} supp\big(\hspace{1mm} _{\mathcal{A}}\mu(\mathbb{L}_{\mathcal{A}_i}\otimes_{\mathcal{A}i}^L\mathcal{A})\big);$

    \end{enumerate}
\end{proposition}

\begin{proof}
    We will only prove (1), noting that (2) follows from a more general result, which we do not prove; given $\mathcal{A}=lim_i \mathcal{A}_i\rightarrow \mathcal{B}=lim_i\mathcal{B}_i,$ then 
$$supp\big(\hspace{1mm}_{\mathcal{B}}\mu(\mathbb{L}_{\mathcal{B}/\mathcal{A}})\big)\simeq lim_i\hspace{1mm} supp\big(\hspace{1mm}_{\mathcal{B}_i}\mu(\mathbb{L}_{\mathcal{B}_i/\mathcal{A}_i}\otimes_{\mathcal{B}_i}^{L}\mathcal{B})\big).$$

By considering the fibre sequence of the cofibration $c_i$ one has that 
$\mathbb{L}_{\mathsf{S}_i}\otimes_{\mathsf{S}_i}\mathsf{S}_{i+1}\rightarrow \mathbb{L}_{\mathsf{S}_{i+1}}\rightarrow \mathbb{L}_{c_i}.$
This is equivalent to
$$\mathcal{M}_i^{\bullet}\otimes^!\mathsf{S}_{i+1}\rightarrow \mathsf{S}_{i+1}\otimes^!\mathcal{M}_{i+1}^{\bullet}\rightarrow \mathbb{L}_{\mathsf{S}_i/\mathsf{S}_{i+1}},$$
and therefore also gives rise to a sequence of $\mathsf{S}_{i+1}[\mathcal{E}_X]$-modules,
$$\scalemath{.90}{\pi^*(\mathcal{M}_i\otimes^!\mathsf{S}_{i+1})\otimes\pi^*\mathsf{S}_{i+1}[\mathcal{E}_X]\rightarrow \pi^*(\mathsf{S}_{i+1}\otimes^!\mathcal{M}_{i+1})\otimes \pi^*\mathsf{S}_{i+1}[\mathcal{E}_X]\rightarrow \pi^*\mathbb{L}_{c_i}\otimes\pi^*\mathsf{S}_{i+1}[\mathcal{E}_X]},$$
where $\pi:T^*X\rightarrow X.$ Taking supports gives the result.
\end{proof}
From (2) in \autoref{Char Properties}, we have the following.

\begin{corollary}
\label{Quasi-finite corollary}
    If $\mathcal{B}$ is quasi-finite $\mathrm{Char}_{\mathcal{D}}(\mathcal{B})=lim_i\hspace{1mm} \mathrm{supp}(_{\mathcal{B}_i}\mu(f_!^i\mathbb{L}_{\mathcal{B}_i})).$
\end{corollary}

\begin{corollary}
\label{CharPropCor1}
    Consider $\mathcal{B}$ as in Corollary \ref{Quasi-finite corollary} and assume $\mathcal{B}_i$ is of the form $\Free(\mathcal{M}_i^{\bullet})$ for a perfect $\D_X$-module. Then,
    $\mathsf{Sol}_X(\mathcal{B})$ is equivalently $\underset{i\in I}{\mathrm{lim}}\hspace{1mm}\mathbb{R}\mathsf{Sol}_{X,lin}(\mathcal{M}_i^{\bullet},\mathcal{O}_X),$ and $\mathrm{SS}\big(\underset{i\in I}{\mathrm{lim}}\mathsf{Sol}_{X,lin}(\mathcal{M}_i)\big)$ is contained in $\bigcup_{j} \underset{i\in I}{\mathrm{lim}} supp\big(\mathcal{H}^j(\mathcal{M}_i^{\bullet})\big).$
\end{corollary}
\begin{proof}
As
$\mathcal{B}=hocolim \Free(\mathcal{M}_i^{\bullet}),$
where each $\mathcal{M}_i^{\bullet}$ is coherent note $\mathcal{B}$ is homotopically finitely $\D$-presented. By \autoref{Ho-D-fp Proposition}
 $\mathsf{Sol}_X(\mathcal{B},\mathcal{O}_X)$ is determined by 
 $$\mathbb{R}\mathsf{Sol}_{X,lin}\big(\underset{i\in I}{\mathrm{colim}}\hspace{1mm} \Free(\mathcal{M}_i),\mathcal{O}_X\big)\simeq \underset{i\in I}{\mathrm{lim}} \mathbb{R}\mathsf{Sol}_{X,lin}(\mathcal{M}_i^{\bullet},\mathcal{O}_X).$$
Setting $f^i:\mathcal{B}_i\rightarrow \mathcal{B}$ the result follows since the functors $f_!^i=(-)\otimes_{\mathcal{B}_i}\mathcal{B},$ are exact.
\end{proof}

Assuming that $f$ is NC for the $(i-1)$-st cell, implies the result for the objects $\mathsf{S}_i$ and $\mathcal{A}_i.$ If $f:X\rightarrow Y$ is NC for $\mathsf{S}_{i+1}$ i.e.
$$f_{\pi}^{-1}\big(\mathrm{Char}_{\mathcal{D}}(\mathsf{S}_{i+1})\big)\cap \mathbb{R}\mathsf{Sol}(\mathcal{A})\times T_X^*Y\subset \mathbb{R}\mathsf{Sol}(\mathcal{A})\times X\times T_Y^*Y,$$
by \autoref{Char Properties}, (1) we get that $f$ is NC for $\mathcal{A}_{i+1},$ since it is NC for $\mathcal{A}_{i}$ by induction, and since it is NC for $\mathsf{S}_{i+1}$ it is NC for $\mathbb{L}_{c_i}.$

This follows by noticing for a given $i$-cell (\ref{eqn: Hopushout squares}), we have
$$\mathbb{L}_{\Free(\mathcal{M}_{i+1})/\Free(\mathcal{M}_i)}\otimes_{\Free(\mathcal{M}_{i+1})}\mathcal{A}_{i+1}\rightarrow \mathbb{L}_{\mathcal{A}_{i+1}/\mathcal{A}_i},$$
and similarly,
$\mathbb{L}_{\mathcal{A}_i/\mathcal{A}_{i+1}}\otimes\mathcal{A}_{i+1}\rightarrow \mathbb{L}_{\mathcal{A}_{i+1}}\rightarrow \mathbb{L}_{\mathcal{A}_{i+1}/\mathcal{A}_i}.$
Therefore, we get that
\begin{eqnarray*}
\mathbb{L}_{\mathcal{A}_{i+1}}&\simeq& Cone\big(\mathbb{L}_{\mathcal{A}_{i+1}/\mathcal{A}_i}\rightarrow \mathbb{L}_{\mathcal{A}_i}[1]\big)[-1]
\\
&\simeq& Cone\big(\mathbb{L}_{\Free(\mathcal{M}_{i+1})/\Free(\mathcal{M}_i)}\otimes_{\Free(\mathcal{M}_{i+1})} \mathcal{A}_{i+1}\rightarrow \mathbb{L}_{\mathcal{A}_i}[1]\big)[-1].
\end{eqnarray*}
Then $f$ is NC for $\mathcal{A}_{i+1}$ if we have
$$f_{\pi}^{-1}\big(\mathrm{supp}(_{\mathcal{A}_{i+1}}\mu\mathbb{L}_{\mathcal{A}_{i+1}} )\big)\cap \mathbb{R}\mathsf{Sol}(\mathcal{A}_{i+1})\times T_X^*Y\subset \mathbb{R}\mathsf{Sol}(\mathcal{A}_{i+1})\times X\times T_Y^*Y.$$
This will follow if this condition holds for supports of $\mathbb{L}_{\mathcal{A}_i}$ as well as for the relative cotangent complex of the cofibration $c_i$ between free $\D$-algebras.

Now it follows since $\mathrm{Char}(\mathcal{M}_i\otimes_{\mathcal{O}_X}\mathcal{M}_{i+1})=\mathrm{Char}(\mathcal{M}_i)+_Y\mathrm{Char}(\mathcal{M}_{i+1}),$ if $\mathcal{M}_i$ and $\mathcal{M}_{i+1}$ are a NC pair (c.f § \ref{ssec: Cauchy Problems}, equation (\ref{eqn: NC-pair})) that  $$\mathrm{Char}(\mathcal{M}_i\boxtimes\mathcal{M}_{i+1})=\mathrm{Char}(\mathcal{M}_i)\times \mathrm{Char}(\mathcal{M}_{i+1}),$$ and moreover, that 
$$\mathrm{Char}(\Delta^*(\mathcal{M}_i\boxtimes\mathcal{M}_{i+1})\big)=\Delta_d\Delta_{\pi}^{-1}\mathrm{Char}(\mathcal{M}_i\boxtimes\mathcal{M}_{i+1})=\mathrm{Char}(\mathcal{M}_i)\times_Y \mathrm{Char}(\mathcal{M}_{i+1}),$$ with $\Delta$ the diagonal. Moreover, if $\mathcal{M}_{i+1}$ and $\mathcal{M}_i$ are a NC pair so are $\Free(\mathcal{M}_{i+1})$ and $\mathcal{M}_i$. 

By noticing that (c.f. equation (\ref{eqn: Pushout cotangent diagram}))
$$ \mathbb{L}_{\mathsf{S}_i/\mathsf{S}_{i+1}}\simeq Cone(\mathcal{M}_i\otimes\mathsf{S}_{i+1}\rightarrow \mathsf{S}_{i+1}\otimes\mathcal{M}_{i+1}),$$
from \autoref{Retract proposition} we have induced morphisms of Cartesian diagrams:
\begin{equation}
\label{eqn: ckk cube}
\adjustbox{scale=.75}{
\begin{tikzcd}[row sep=scriptsize, column sep=scriptsize]
& f^{-1}\mathbb{R}\mathsf{Sol}_{Y}(\mathcal{A}_{i+1},\mathcal{O}_Y)\arrow[dl, "ckk(\mathcal{A}_{i+1})"] \arrow[rr,] \arrow[dd,] & & f^{-1}\mathbb{R}\mathsf{Sol}_{Y}(\mathcal{A}_{i},\mathcal{O}_Y) \arrow[dl, "\mathrm{Ckk}_f(\mathcal{A}_{i})"] \arrow[dd,] \\
\mathbb{R}\mathsf{Sol}_{X}(f_{\mathcal{D}}^*\mathcal{A}_{i+1},\mathcal{O}_X) \arrow[rr,] \arrow[dd,] & & \mathbb{R}\mathsf{Sol}_X(f_{\mathcal{D}}^*\mathcal{A}_i,\mathcal{O}_X)\\
& f^{-1}\mathbb{R}\mathsf{Sol}_{Y}(\mathsf{S}_{i+1},\mathcal{O}_Y)\arrow[dl,"\mathrm{Ckk}_f(\mathsf{S}_{i+1})"] \arrow[rr,] & & f^{-1}\mathbb{R}\mathsf{Sol}_{Y}(\mathsf{S}_{i},\mathcal{O}_Y)\arrow[dl,"\mathrm{Ckk}_f(\mathsf{S}_i)"] \\
\mathbb{R}\mathsf{Sol}_X(f_{\mathcal{D}}^*\mathsf{S}_{i+1},\mathcal{O}_X) \arrow[rr,] & & \mathbb{R}\mathsf{Sol}_X(f_{\mathcal{D}}^*\mathsf{S}_i,\mathcal{O}_X) \arrow[from=uu, crossing over]\\
\end{tikzcd}}
\end{equation}

 We have argued the NC-condition for free algebras which determines the NC-condition on each $\mathcal{A}_i$ if we have NC on the underlying $\D$-modules. By the inductive step, we have NC condition holds for $\mathcal{A}_i$. This is enough to establish $ckk(\mathcal{A}_{i+1})$ is an equivalence. Indeed, $\mathrm{Ckk}_f(\mathcal{A}_i)$ and obvious $\mathrm{Ckk}_f(\mathsf{S}_{i+1})$ and $\mathrm{Ckk}_f(\mathsf{S}_i)$ are equivalences.

Note also given $f_0:\mathcal{A}\rightarrow \mathcal{B}\simeq \mathrm{Sym}_{\mathcal{A}}(\mathcal{M}^{\bullet})$ for a $\mathcal{A}[\mathcal{D}_X]$-module $\mathcal{M}^{\bullet},$
$f_0$ is the tautological inclusion 
$\mathcal{A}\hookrightarrow \mathcal{B}=\mathcal{A}\oplus \mathcal{M}\oplus \mathrm{Sym}_{\mathcal{A}}^2(\mathcal{M})\oplus\ldots.$
It is clear that (c.f. equation \ref{eqn: Relative Cotangent k-1})
$$\mathbb{L}_{\mathrm{Sym}_{\mathcal{A}}(\mathcal{M}^{\bullet})/\mathcal{A}}\simeq \mathrm{Sym}_{\mathcal{A}}^*(\mathcal{M})\otimes_{\mathcal{A}}\mathcal{M},\hspace{2mm}\text{and }\hspace{1mm} \mathbb{L}_{\mathcal{A}/\mathrm{Sym}_{\mathcal{A}}(\mathcal{M})}\simeq\mathcal{M}^{\bullet}[1].$$

Even more simply, consider $\mathcal{A}$ to be the initial algebra $\mathcal{O}_X.$
Then if $\mathcal{B}$ is the free $\D$-algebra $\mathrm{Sym}_{\mathcal{O}_X}^*(\mathcal{M}^{\bullet})$ with $\mathcal{M}^{\bullet}\in \mathsf{Coh}(\D_X),$ the relative cotangent complexes are described simply by
$\mathbb{T}_{\mathcal{B}}=\mathbb{T}_{\mathrm{Sym}(\mathcal{M})}\simeq \mathcal{A}\otimes_{\mathcal{O}_X}\mathcal{M}^{\bullet}.$

\subsubsection{Locally for $\mathsf{S}_i$}
We elaborate the local resolution for $\mathsf{S}_i.$ Namely, (c.f. \autoref{Free Propagation Example}) we have
$\mathbb{R}\mathsf{Sol}\big(\mathcal{M}_i,\mathcal{O})\simeq \big[\mathcal{O}_{Y|X}[0]\xrightarrow{P}\mathcal{O}_{Y|X}\rightarrow 0\big],$ as a complex for which 
$$f^*(\mathcal{M}_i)\simeq \mathcal{D}_X^{\oplus m},\hspace{2mm} \mathrm{Sym}(f^*\mathcal{M}_i)\simeq \mathrm{Sym}(\mathcal{D}_X^{\oplus m}).$$ For each coherent $\mathcal{M}_i$ denote its pull-back $\D_X$-module
$\mathcal{M}_{i,X\rightarrow Y}.$

For simplicity, and without loss of generality, we may take $\mathcal{F}\simeq \mathcal{D}/\mathcal{D}\cdot P.$ Thus 
$0\rightarrow \mathcal{D}_X\rightarrow \mathcal{D}_X\xrightarrow{P}\mathcal{F}\rightarrow 0,$ is a resolution of length $1.$
This implies that
$$\begin{cases}
\mathcal{E}xt_{\mathcal{D}}^k(\mathcal{F},\mathcal{O}_X)=0,k\geq 2,
\\
\mathcal{E}xt_{\D_X}^1(\mathcal{F},\mathcal{O}_X)\simeq \mathcal{O}_X/Im(P),
\\
\mathcal{E}xt_{\D_X}^0(\mathcal{F},\mathcal{O}_X)\simeq ker(P).
\end{cases}$$

This is just \autoref{Thm: LinCKK}, that further gives rise to a commutative diagram with exact rows (since $\mathcal{E}xt^1(\mathcal{F},\mathcal{O})$ and $\mathcal{E}xt^1(\mathcal{F}_{X\rightarrow Y},\mathcal{O})$ vanish), given by 
\[
\adjustbox{scale=.75}{
\begin{tikzcd}
\mathcal{E}xt_{\D_X}^0(\mathcal{M}_i,\mathcal{O}_X)\arrow[d,"\alpha"] 
\arrow[r] &\mathcal{E}xt_{\D_X}^0(\mathcal{F}_i,\mathcal{O}_X)\arrow[d,"\beta"]\arrow[r] & \mathcal{E}xt_{\D_X}^0(\mathcal{M}_i,\mathcal{O}_X)\arrow[d,"\rho"]\arrow[r] & \mathcal{E}xt_{\D_X}^1(\mathcal{M}_i,\mathcal{O}_X)\arrow[d,"\gamma"]
\\
\mathcal{E}xt_{\mathcal{D}_Y}^0(Li^*\mathcal{M}_i,\mathcal{O}_Y)\arrow[r] &\mathcal{E}xt_{\mathcal{D}_Y}^0(Li^*\mathcal{F}_i,\mathcal{O}_Y)\arrow[r] & \mathcal{E}xt_{\mathcal{D}_Y}^0(Li^*\mathcal{N}_i,\mathcal{O}_Y)\arrow[r]& \mathcal{E}xt_{\mathcal{D}_Y}^1(Li^*\mathcal{M}_i,\mathcal{O}_Y)
\end{tikzcd}}
\]
We can extend the top row by the spaces $\mathcal{H}om_{\mathcal{D}_X-Alg}(\mathrm{Sym}(\mathcal{M}_i),\mathcal{O}_X)$ and similarly for the other columns. Such extensions are isomorphisms due to adjunction \ref{eqn: Free-forget AD-Mod/Alg Adjunction}, where we use \autoref{eqn: Non-lin CKK prop}.
One proceeds by standard arguments that $\beta$ is an isomorphism $\alpha,\rho$ are both injective and $\alpha$ is an isomorphism implies so is $\rho$ and $\gamma.$ Namely, for $k>1,$ there are isomorphisms
$\mathcal{E}xt_{\mathcal{D}}^k(\mathcal{M})\simeq \mathcal{E}xt_{\mathcal{D}}^{k-1}(\mathcal{N}),$ and further induction on $k$ leads to the conclusion that 
$\mathcal{H}^k
\big(\mathrm{Ckk}_i^{lin}(\mathcal{M})\big)$ are isomorphisms for all $k.$ This supplies the inductive step for a homotopically finitely presented $\D_X$-algebra at the levels of $i$-cells on symmetric algebras. One then proceeds to show it holds for $\mathcal{A}_i$ via 
\autoref{CharPropCor1} and \ref{Char Properties}.

Note that we might have used functors (§ \ref{sssec: Solutions with Coefficients}), directly interpreted via the cotangent and tangent $\D$-complexes in $\mathcal{B}$-modules.

 In fact if we have a semi-free $\mathcal{B}^{\bullet}$-module in $\D_X$-complexes and choose a given presentation
$$\mathcal{L}^{\bullet}\simeq \underset{\alpha}{\scalemath{1.05}{\mathrm{colim}}} \mathsf{Free}_{\mathcal{B}}(\mathcal{P}_{\alpha}^{\bullet}),$$
we observe that
\begin{equation*}
\mathcal{H}om_{\mathcal{B}[\mathcal{D}_X]}(\mathcal{L}^{\bullet},-)\simeq \mathcal{H}om_{\mathcal{B}[\mathcal{D}_X]}(\underset{\alpha}{\scalemath{1.05}{\mathrm{colim}}} \mathsf{Free}_{\mathcal{B}}(\mathcal{P}_{\alpha}^{\bullet}),-)
\simeq \underset{\alpha}{\mathrm{lim}} \mathcal{H}om_{\mathcal{B}[\mathcal{D}_X]}(\mathsf{Free}_{\mathcal{B}}(\mathcal{P}_{\alpha}^{\bullet}),-)
\simeq
\underset{\alpha}{\mathrm{lim}} \big(\mathcal{P}_{\alpha}^{\bullet}\otimes -\big),
\end{equation*}
by commutation of colimits into limits and that the relative tensor product on free modules is the forgetful functor. In this case, the remaining result follows directly from the fact that $\mathcal{P}_{\alpha}^{\bullet}$ are coherent as $\D_X$-modules, an applying the standard arguments. 

On the other hand, by assumption of being almost finitely $\D$-presented, we have that $\EQ$ is representable i.e. of the form $\mathsf{Spec}_{\D}(\mathcal{A})$ with $\mathcal{A}\in \mathsf{CAlg}_X(\D_X)^{afp}.$ Thus, for all $n$, we have $\tau_{\mathcal{D}}^{\geq -n}(\mathcal{A})\in\mathsf{CAlg}(\D_X)^{n-afp}.$ Thus, as a compact object, there exists a simplicial $\D$-algebra
$$\mathcal{B}_{\bullet}:\Delta^{op}\rightarrow \mathsf{CAlg}(\D_X),$$
where $\mathcal{B}_n\simeq \mathsf{Free}(\mathcal{M}_n),$ with $\mathcal{M}_n$ compact as a $\D$-module and with 
$\tau_{\mathcal{D}}^{\geq -n}(\mathcal{A})$ equivalent to a retract of $|\mathcal{B}_{\bullet}|.$
Since $X$ is smooth, $\mathsf{IndCoh}(X_{\DR})$ is compactly generated by $q_{\DR,*}^{\mathsf{IndCoh}}(E_n)$ for $E_n\in\mathsf{Coh}(X)$ as $\mathsf{IndCoh}(X)$ is compactly generated by $E_n$. In fact, $\mathsf{IndCoh}(X)\simeq \mathsf{Coh}(X)$ in this case.
Then, if $f:\Sigma\rightarrow X$ is
as above, then
looking at $f^{-1}\mathsf{Spec}(\mathcal{A})$, we see that
\begin{equation*}
\mathsf{Map}(Lf^*\mathcal{A},\mathcal{O})\simeq \mathrm{hocolim}_{i,n}\mathsf{Map}(Lf^*\tau_{\mathcal{D}}^{\geq -n}\mathcal{A},\tau_{\mathcal{D}}^{\geq -n}\mathcal{O})
\simeq \mathrm{hocolim}_{i,n}\mathsf{Map}(Lf^*\mathsf{Free}(M_n),\mathcal{O}),
\end{equation*}
which is readily seen to be equivalent to
$$\mathrm{hocolim}_{i,n}\mathsf{Maps}\big(\mathsf{Free}(f^*M_n),\mathcal{O})\simeq \mathsf{Spec}(f^*\mathcal{A}),$$
with $f^*\mathcal{A}$ the compact $\D$-algebra pullback.
\subsection{Extension to de Rham Spaces}
For smooth morphisms, we can address compaitiblity of mapping spaces.
\begin{proposition}
\label{prop: DAfp smooth CKK}
    Let $f:Y\to X$ be smooth and suppose $A$ is $\D_X$-afp. Then for each $T$-point $t:T\to Y_{\DR}$, with corresponding point $\widetilde{t}:T\to Y_{\DR}\to X_{\DR},$ via the map $f_{\DR}:Y_{\DR}\to X_{\DR},$ then 
    $$\mathrm{Maps}_{\mathsf{QCAlg}(Y_{\DR})}(A_Y,t_*\mathcal{O}_T)\simeq \mathrm{Maps}_{\mathsf{QCAlg}(X_{\DR})}(A,\widetilde{t}_*\mathcal{O}_T),$$
    is an equivalence.
\end{proposition}
In other words, pulling back $A$ to $Y$, then taking a $(T\to Y_{\DR})$-family of solutions, is equivalent to taking $(T\to X_{\DR})$-family of solutions and restricting to $Y.$
\begin{proof}
Since $f$ is smooth, by Proposition 
\ref{prop: PreserveDafp} $A_Y$ is $\D_Y$-afp. Thus by Proposition \ref{prop: D-afp means RSolfinite}, $\mathrm{RSol}_{\D}(A_Y)$ is a derived stack locally of finite presentation over $k.$

Moreover via the $(f_{\DR}^*,f_{DR,*})$ adjuntion, since $f_{DR,*}\circ t_*\simeq \widetilde{t}_*,$, we have
$\mathrm{Maps}_Y(A_Y,t_*\mathcal{O}_T)\simeq \mathrm{Maps}_X(A,f_{DR,*}t_*\mathcal{O}_T).$
Noting that smoothness implies $f_{\DR}^*$ is exact and Tor-independent, then for any connective $\M\in \mathsf{QCoh}(X_{\DR}),$
we have $f_{\DR}^*\M\simeq \mathcal{O}_{Y_{\DR}}\otimes_{f_{\DR}^{-1}\mathcal{O}_{X_{\DR}}}f_{\DR}^{-1}\M,$  is a flat-base change.
Applying this for $A_Y=f_{\DR}^*A$, via the adjunction we get 
$$\mathrm{Maps}_{/Y_{\DR}}(T,\mathsf{Spec}_{Y_{\DR}}(A_Y))\simeq \Map_{/X_{\DR}}(T,\mathsf{Spec}_{X_{\DR}}(A)).$$
The functorialities are compativle with compact presentations. Namely $A=\underset{n}{\mathrm{hocolim}} A_n,A_n\simeq \mathrm{Sym}(M_n).$ Then let $p_b:\mathsf{Spec}_{X_{\DR}}(A_n)\to \mathsf{Spec}_{X_{\DR}}(A)$ be the induced map. We obtain
$$L_{\underset{n}{\mathrm{hocolim}} A_n}^{alg}\simeq \underset{n}{\mathrm{hocolim}}\big(L_{A_n}^{alg}\otimes_{A_n}\underset{k}{\mathrm{hocolim}}A_k),$$ with $\mathrm{supp}(L_{\underset{n}{\mathrm{hocolim}}A_n}^{alg})$ given by $\cup_n\mathrm{supp}(L_{A_n}^{alg}).$
Now, since $\mathrm{supp}(L_A^{alg})\subset \mathrm{supp}(L_{\underset{n}{\mathrm{hocolim}}A_n}^{alg}),$ we see $\mathrm{Char}_{\D}(A)\subset \bigcup_n \mathrm{Char}_{D}(A_n).$ Since at each stage $M_n$ is perfect, $S_n:=\mathrm{supp}(M_n)\subset T^*X$ and we obtain that $\mathrm{Char}_{D}(\A)\subset \bigcup_n S_n.$
We note
$$L_{A_n/A}^{alg}\simeq Cofib\big(A_n\otimes_A L_A^{alg}\to A_n\otimes M_n),$$ as $L_{A_n}^{alg}\simeq A_n\otimes M_n.$ Then, if $\underset{k}{\mathrm{hocolim}} M_k\xrightarrow{\pi_n}M_n$ are the induced maps,
$L_{A_n/A}^{alg}\simeq A_n\otimes \mathrm{cofib}(\underset{k}{\mathrm{hocolim}} M_k\to M_n).$
Since filtered colimits in a stable presentable $\infty$-category commute with homotop cofibers, 
$$\mathrm{cofib}(\underset{k}{\mathrm{hocolim}} M_k\to M_n)\simeq \underset{k}{\mathrm{hocolim}} \mathrm{cofib}(M_k\to M_n).$$
Since each $\M_k\in \mathrm{Perf}(\D_X)$ is dualizable, with $\M_k^{\vee}=Hom_{\D_X}(M_k,\D_X),$ by dualizing
$$\underset{k}{\mathrm{hocolim}} M_k\xrightarrow{\pi_n}M_n\to \mathrm{cofib},$$
becomes $Fib(M_n^{\vee}\xrightarrow{\pi_n^{\vee}}holim_kM_k^{\vee}).$ If $M_k\simeq \mathrm{ind}(E_k)$ for vector bundles $E_k$ on $X$, the situation is as follows.
Since $ind$ is a left-adjoint, it preserves colimits and
$$\mathrm{ind}(\underset{k}{\mathrm{hocolim}} E_k)\simeq \underset{k}{\mathrm{hocolim}} \hspace{.3mm}\mathrm{ind}(E_k)\simeq \underset{k}{\mathrm{hocolim}} \M_k.$$
Then,
\begin{eqnarray*}
    \mathrm{cofib}(\underset{k}{\mathrm{hocolim}} \hspace{.3mm}\mathrm{ind}(E_k)\to \mathrm{ind}(E_n)) &\simeq& \underset{k}{\mathrm{hocolim}} \mathrm{cofib}(\mathrm{ind}(E_k)\to \mathrm{ind}(E_n))
    \\
    &\simeq& \underset{k}{\mathrm{hocolim}} ind\big(\mathrm{cofib}(E_k\to E_n)\big)
    \\
    &\simeq& ind \underset{k}{\mathrm{hocolim}} \mathrm{cofib}(E_k\to E_n).
\end{eqnarray*}
Letting $E_1\subset E_2\subset\cdots \subset E_n$ be an increasing filtration by subbundles, $\mathrm{cofib}(E_k\to E_n)\simeq E_n/E_k,$ and in fact
$$\mathrm{cofib}(\underset{k}{\mathrm{hocolim}} \hspace{.3mm}\mathrm{ind}(E_k)\to \mathrm{ind}(E_n))\simeq \mathrm{ind}(\underset{k}{\mathrm{hocolim}}(E_n/E_k)).$$
\end{proof}
Phrased in terms of pull-back operations, denoting $\Map_{/X_{\DR}}(\mathsf{Spec}_{X_{\DR}}A,-)\in \mathrm{PShf}\big(\mathrm{dAff}_{X_{\DR}}\big),$ Proposition \ref{prop: DAfp smooth CKK} states the presheaf of spaces sending $(u:U\to X_{\DR})$ to $\mathrm{Maps}_{\mathsf{QCAlg}(X_{\DR})}(A,u_*\mathcal{O}_U),$ then 
$$f^{-1}\Map_{/X_{\DR}}(\mathsf{Spec}_{X_{\DR}}A,-)\to \mathrm{Maps}_{/Y_{\DR}}\big(\mathsf{Spec}_{Y_{\DR}}(f_{\DR}^{\mathsf{QCAlg},*}A),-),$$
is an equivalence of functors $(\mathrm{dAff}_{/Y_{\DR}})^{op}\to \mathrm{Spc}.$

Suppose for the moment that $S$ is a derived $\mathbb{C}$-scheme.
The right $t$-structure on $\mathsf{IndCoh}(S_{\DR})$ is Zariski-local. In other words if $p:U\to S$ is Zariski cover, an object is (co)connective if and only if its restriction to $U$ is (co)connective \cite[Cor. 4.2.3]{IndCoh}. 
We remark that $i_{\DR,*}^{\mathsf{IndCoh}}$ is a left-adjoint of the functor $i_{\DR}^!$ that is $t$-exact. Furthermore, since we consider closed-embeddings $i:X\to Y$ where $Y$ is always smooth, the corresponding pull-back $i^!:\IC(Y)\to \IC(X)$ is of bounded cohomological amplitude. 
\begin{remark}
Let $i:X\to Y$ be a closed embedding with $Y$ smooth underiooved scheme. If $T$ is the formal completion of $Y$ along $X$, then $j_{\DR}:X_{\DR}\to T_{\DR}$ is an isomorphism.
Pullback for the sheaf theories $\mathsf{QCoh}$ and $\mathsf{IndCoh}$ along $j_{\DR}$ give equivalences of categories \cite[Sect.2.4.1]{GR14}.
\end{remark}

Let $\PS_{/Y_{\DR},f}^{\D-\text{fin}}$ denote the category of $\D$-finitary relative prestacks of non-characteristic presentation for $f.$ 
\begin{theorem}
\label{theorem: DAnCKK}
Let $f:X\rightarrow Y$ be a morphism of complex manifolds and consider the restriction of the induced morphism $f_{\DR}^*:\PS_{/Y_{\DR}}\rightarrow \PS_{/X_{\DR}}$ to $\PS_{/Y_{\DR},f}^{\D-\text{fin}}$ (denoted the same). Then the diagram
\[
\begin{tikzcd}[column sep=6em]
\PS_{/Y_{\DR},f}^{\D-\text{fin}} \arrow[d,"f_{\DR}^{\star}"] \arrow[r,"\RSol_Y(-)"] & \PS_{/Y_{\DR}}    \arrow[d,"f_{\DR}^{\star}"]
\\
\PS_{/X_{\DR}}\arrow[r,"\RS(-)"] & \PS_{/X_{\DR}}.
\end{tikzcd}
\]
is commutative.
\end{theorem}
\begin{proof}
The assumption that $\EQ$ is of non-characteristic presentation for $f$ means that in particular, $\EQ$ is locally of finite presentation 
$$\EQ\simeq \underset{i\in I}{\mathrm{colim}}\hspace{2mm}\EQ_i=\underset{i\in I}{\mathrm{colim}}\hspace{2mm} \mathsf{RSol}_Y(\mathcal{B}_i^{\bullet}),$$
where each $\mathsf{RSol}_Y(\mathcal{B}_i^{\bullet})$ is finitely $\D$-presented. By homotopically finitely $dR$-presentation (\ref{eqn: Ho dR-fp}), $\Map_{/X_{\DR}}(-,\underset{i\in I}{\mathrm{colim}})\simeq \underset{i\in I}{\mathrm{colim}} \Map_{X_{\DR}}(-,-),$ so get from \autoref{theorem: DNLCKK}, 
\begin{eqnarray*}
    \Map_{X_{\DR}}(X_{\DR},f_{\DR}^!\EQ)&\simeq& \Map_{X_{\DR}}(X_{\DR},f_{\DR}^!\underset{i\in I}{\mathrm{colim}} \EQ_i)
    \\
    &\simeq& \Map_{X_{\DR}}(X_{\DR},\underset{i\in I}{\mathrm{colim}} f^{-1}\mathsf{RSol}_Y(\mathcal{B}_i^{\bullet})\big)\\
    &\simeq& \Map_{X_{\DR}}(X_{\DR},hocolim \mathsf{RSol}_X(Lf^*\mathcal{B}_i^{\bullet})\big),
\end{eqnarray*}
while on the other hand 
\begin{eqnarray*}
    \Map_{Y_{\DR}}(Y_{\DR},\EQ)&=&\Map_{Y_{\DR}}(Y_{\DR},\underset{i\in I}{\mathrm{colim}} \EQ_i)
\\
&\simeq& \underset{i\in I}{\mathrm{colim}}\Map_{Y_{\DR}}(Y_{\DR},\EQ_i),
\\
&\simeq& \underset{i\in I}{\mathrm{colim}}\Map_{Y_{\DR}}(Y_{\DR},\mathbb{R}\mathsf{Sol}_Y(\mathcal{B}_i)
\end{eqnarray*}
Since $f$ is in particular non-characteristic 
(\autoref{NC for de Rhams}) and whose pull-back commutes with colimits we see that
\begin{eqnarray*}
    f^{-1}\Map_{Y_{\DR}}(Y_{\DR},\EQ)&\simeq& f^{-1} \underset{i\in I}{\mathrm{colim}} \Map_{Y_{\DR}}(Y_{\DR},\mathsf{RSol}_Y(\mathcal{B}_i\big)
    \\
    &\simeq& \underset{i\in I}{\mathrm{colim}}\Map_{X_{\DR}}(f_{\DR}^{-1}Y_{\DR},f^{-1}\mathsf{RSol}_Y(\mathcal{B}_i)\big)
    \\
&\simeq& \underset{i\in I}{\mathrm{colim}}\Map_{X_{\DR}}(X_{\DR},f^{-1}\mathsf{RSol}_Y(\mathcal{B}_i)\big).
\end{eqnarray*}
These expressions are equivalent by \autoref{theorem: DNLCKK}, property (2) of above.

Let $i:X\to Y$ be a closed-immersion of codimension $1.$ We proceed by induction on $c:=\mathrm{codim}_X(Y).$
For the $c=1$ case, let $X\hookrightarrow Y$ be a smooth hypersurface with ideal sheaf $\mathscr{I}_X$ locally generated by a single function $t.$ This case is proven by Proposition \ref{Codim1Prop}.  
Namely,
$$\RS(i_{\DR}^* A)(T) = \operatorname{Map}_{\mathsf{CAlg}(X_{\DR})}\big( i_{\DR}^* A,\, (X_{\DR}\times T)_*\mathcal{O}_{X_{\DR}\times T} \big),$$
which by adjunction is given by  $\operatorname{Map}_{\mathsf{CAlg}(Y_{\DR})}\big( A,\, (i_{\DR})_* ((X_{\DR}\times T)_*\mathcal{O}_{X_{\DR}\times T}) \big).$
We factor $f:X\hookrightarrow Y'\xrightarrow{p} Y$ into a closed-embedding and a smooth morphisms. Via the description of characteristics by smooth morphisms given by Proposition \ref{CharSmoothPB}, the smooth case is handled by Proposition \ref{prop: DAfp smooth CKK}.

Note that 
$$(i_{\DR})_* ((X_{\DR}\times T)_*\mathcal{O}_{X_{\DR}\times T}) \simeq (Y_{\DR}\times T)_*\mathcal{O}_{Y_{\DR}\times T} \otimes_{\mathcal{O}_{Y_{\DR}}}^{\mathbf{L}} \mathcal{O}_{X_{\DR}}.$$
Assume it holds for codimension $\leq c$ and set, locally on $Y$ 
$X:=\{y_1=\cdots=y_c=0\}$ and put $Y_k:=\{y_1=\ldots =y_k=0\}.$ We write $X=Y_c\hookrightarrow Y_{c-1}\hookrightarrow \cdots \hookrightarrow Y_1\hookrightarrow Y,$ with each $\iota_k:Y_k\hookrightarrow Y_{k-1}$ a hypersurface. 
Using that $Z$ is $\D$-afp and bounded, we reduce locally to the free case on a compact $\D$-module. In this case, we apply the $\D$-module theorem, upon repetead application of Proposition \ref{Codim1Prop} to get
$$\mathrm{Char}(Z_{Y_k}/Y_{k,\DR})=(\iota_k)_d\big((\iota_k)_{\pi}^{-1}\mathrm{Char}(Z_{Y_{k-1}}/Y_{k-1,\DR})\big),$$
for each $k=1,\ldots,c.$
\end{proof}

\begin{remark}
\normalfont
In \autoref{theorem: DAnCKK}, it is likely that in order to have ckk-equivalences for any two $\D$-finitary derived de Rham spaces $\EQ_1,\EQ_2\in \PS_{/Y_{\DR}}$ i.e.
$f^{-1}\RSol_Y(\EQ_1,\EQ_2)\simeq \RS(f_{\DR}^*\EQ_1,f_{\DR}^*\EQ_2),$
one requires additional conditions on the characteristic varieties and the zero section $T_Y^*Y$ e.g.
$$\big(\mathsf{Ch}_{Y_{\DR}}
(\EQ_1)\cap \mathsf{Ch}_{Y_{\DR}}(\EQ_2)\big)\cap \EQ_1\times \EQ_2\subset \EQ_1\times \EQ_2\times T_Y^*Y.$$
\end{remark}
We conclude by discussing the linearization of Theorem \ref{theorem: DAnCKK}. 

Assume \begin{equation}
    \label{eqn: ckk}
\RS(\EQ)\xrightarrow{ckk}\mathbb{R}\underline{\mathrm{Sol}}_{\Sigma}(\EQ_{\Sigma}),
\end{equation}
is an equivalence, where we set $\EQ_{\Sigma}:=f_{\DR}^*\EQ,$ for simplicity, where $f:\Sigma\rightarrow X.$
By functoriality of the tangent functor as a left-Kan extension to prestacks, we obtain a morphism 
\begin{equation}
\label{eqn: Tckk}
T(\mathrm{ckk}_{\Sigma}):T\RS(\EQ)\rightarrow T\RSol_{\Sigma}(\EQ_{\Sigma}).
\end{equation}
\begin{proposition}
\label{prop: Global non-lin implies global lin}
With the same homotopical finiteness hypothesis as above, assume the well-posedness of the generalized non-linear Cauchy problem. Then, for a given solution i.e. point $u:\star\rightarrow \RS(\EQ),$ the linearized (around $u$) Cauchy problem is well-posed.
\end{proposition}
\begin{proof}
Consider the morphism (\ref{eqn: Tckk}), and the pull-back diagram in $\PS_{/\Sigma_{\DR}}:$
\[
\begin{tikzcd}
T\RSol_{\Sigma}(\EQ_{\Sigma})^{\mathrm{ckk}}\arrow[d]\arrow[r] & T\RSol_{\Sigma}(\EQ_{\Sigma})\arrow[d]
\\
\RS(\EQ)\arrow[r,"\mathrm{ckk}"] & \RSol_{\Sigma}(\EQ_{\Sigma}).
\end{tikzcd}
\]
By the universal properties of $T$ as a left Kan extension: we have a diagram 
\begin{equation}
    \label{eqn: Tckk diagram}
\begin{tikzcd}
T\RS(\EQ) \arrow
[drr, "(\ref{eqn: Tckk})", bend left
,] \arrow
[ddr, bend right
, ] \arrow[dr, dotted
, "w_{\Sigma}" description] & & \\
& 
T\RSol_{\Sigma}(\EQ_{\Sigma})^{\mathrm{ckk}} \arrow[r]\arrow[d] & T\RSol_{\Sigma}(\EQ_{\Sigma})\arrow[d] \\
& \RS(\EQ)\arrow[r, "(\ref{eqn: ckk})"]  & \RSol_{\Sigma}(\EQ_{\Sigma})
\end{tikzcd}
\end{equation}
Morphism $w_{\Sigma}$ is defined by writing it via the description of Construction \ref{cons: Solutions}:
$$w_{\Sigma}:T\Sect_{X/S}(X,E)\times_{T\Sect(X,F)}^h T\Sect_{X/S}(X,E)$$
$$\rightarrow \RS(F_P=0)\times_{\mathbb{R}\underline{\mathrm{Sol}}_{\Sigma}(F_P|_{\Sigma}=0)}^hT\mathbb{R}\underline{\mathrm{Sol}}_{\Sigma}(F_P|_{\Sigma}=0).$$
Objects of on the left hand side are interpreted as: 
$$\{\text{Solutions of the linearized problem}\},$$ which are thus mapped to 
$$\{\text{Solutions of non-lin. problem satisfying initial conditions and linearized problem on } \Sigma\}.$$
Using the `heuristic' notation introduced before, it is
$$w_{\Sigma}:\RS(\mathrm{lin}(F_P)=0)\rightarrow \RS(F_P=0)\times_{\RSol_{\Sigma}(F_P|_{\Sigma}=0)}\RSol_{\Sigma}(\mathrm{lin}(F_P)|_{\Sigma}=0).$$
In other words, for every $T$-point $u_T:T\rightarrow \RS(\EQ),$ we have for every $T$-point $u_T^{0}:T\rightarrow \RSol_{\Sigma}(\EQ_{\Sigma}),$ that there exists a point $v_T$ satisfying 
$$\mathrm{lin}_{u_T}(P)(v_T)=0,\hspace{2mm} \mathrm{ckk}_{\Sigma}(v_T)=(u_T^{0}).$$
Thus, it is a solution to the linearized moduli problem and is subject to the initial conditions $u_T^0$ along $\Sigma.$
\end{proof}
Remark that one might consider higher tangent spaces to the derived stacks of solutions i.e. beyond first-order deformations. In this case, a typical $T$-parameterized point $u_T$ is a formal sum
$$u_T=\sum_{i\geq 0}\epsilon^iu_T^{(i)},$$
and in this case,
$$w_{\Sigma}(u_T):=\sum_{i\geq 0}\epsilon^{i}\cdot \mathrm{ckk}_{\Sigma}(u_T^{(i)})=\big(u_T,\mathrm{ckk}_{\Sigma}(u_T)+\epsilon\cdot \mathrm{ckk}_{\Sigma}(u_T^{(1)})+\epsilon^2\mathrm{ckk}_{\Sigma}(u_T^{(2)})+\cdots\big).$$

Let $E_1\rightarrow X$ and $E_2\rightarrow X$ be $S$-families of submersions, with $S$ a smooth stack. An $S$-parameterized family of non-linear PDEs $P:E_1\rightarrow E_2$ of order $\leq m$ is a morphism of stacks $F_P:\mathrm{Jets}_{X/S}^m(E_1)\rightarrow E_2$ in $\mathrm{DStk}_{/X}.$

Applying the tangent functor and restricting gives a morphism 
$$T(F_{P})|_{T(E_1/X)}:T(\mathrm{Jets}_{X/S}^k(E_1)/X)\rightarrow T(E_2/X).$$
\begin{remark}
An $S$-family of $m$-th order non-linear PDEs is \emph{elliptic} if: for any $s\in \mathrm{Sect}(X,E_1),$ the linear operator
$$T(F_{P})_{s}:\mathrm{Jets}_{X/S}^k(s^*T_{E_1/X})\simeq j^k(s)^*T(\mathrm{Jets}_{X/S}^k(E_1)/X)\rightarrow F_P(j_k(s))^*T(E_2/X),$$
is elliptic.
\end{remark}

\begin{example}
Given a submersion $\pi:E\rightarrow \Sigma,$ a first-order non-linear operator $F_P:\mathrm{Jets}_{S}^1(E/\Sigma)\rightarrow F$ is elliptic if its linear part, viewed as a bundle map $\mathrm{Hom}(\pi^*T\Sigma,T_{E/\Sigma})\rightarrow F,$ sends $\eta\in \pi^*T\Sigma$ to isomorphism $f_{\eta}:T_{E/\Sigma}\rightarrow F.$
\end{example}
We make use of standard estimates related to elliptic non-linear PDEs. For instance, if $P$ is elliptic of order $\leq m,$ then 
$$|\!|u|\!|_s\leq C_{P,K,s}\cdot |\!|P u|\!|_{s-m}+C_{P,K,N,s}|\!|u|\!|_{s-N},$$
for all solutions $u$ such that $\mathrm{supp}(u)\subset K,$ a compact subset of $X$ and for every $N<\infty.$

Consider $s\in S$ and a pair of sections $(\varphi_1,\varphi_2)$ of $E_1,E_2,$ as well as two $k$-th order non-linear operators $\mathsf{P}_{\mathsf{F}_i}$ for $i=1,2$. Suppose that 
$$\mathsf{P}_{\mathsf{F}_1}(\varphi_1)=\mathsf{P}_{\mathsf{F}_2}(\varphi_2)=\psi \in \mathrm{Sect}_{X/S}(X,V).$$
Then we have an elliptic non-linear partial differential system if the linearized operator, given by linearization at $(\varphi_1,\varphi_2)$:
\begin{equation}
    \label{eqn: Linearization for moduli}
\mathrm{lin}_{\varphi_1,\varphi_2}(P_1,P_2):\mathrm{Jets}_{X/S}^k\big(\varphi_1^*T(E_{1,s}/X_s)\big)\times_{X_s}\mathrm{Jets}_{X/S}^k\big(\varphi_2^*T(E_{2s}/X_s)\big)\rightarrow \psi^*T(V_s/X_s),
\end{equation}
is elliptic.

Note that (\ref{eqn: Linearization for moduli}) acts by
$$\mathrm{lin}_{\varphi_1,\varphi_2}(P_1,P_2)(\theta_1,\theta_2)\mapsto T(F_{P_1})_{\varphi_1}(\theta_1)-T(F_{P_2})_{\varphi_2}(\theta_2),$$
for $\theta_i\in \mathrm{Jets}_{X/S}^k(\varphi_i^*T(E_{i,s}/X_s)),$ for $i=1,2.$
As an example of these computations, consider  a first-order non-linear elliptic problem and $\mathbb{R}\underline{\mathrm{Sol}}_{S}(E/\Sigma).$ 
linearizing around $u_T$, a $T$-parameterized solution. Namely, for every $u_T\in \mathrm{Sol}_S(E\backslash \{0\}/\Sigma)(T),$ there is a first-order linear operator
$$D_{u_T}:\mathrm{Sect}(\Sigma,u_T^*T_{E/\Sigma})\rightarrow \mathrm{Sect}(\Sigma,u_T^*V).$$
The stack of solutions of this family of operators is simply $\mathrm{Sol}_{S}(T_{E/\Sigma}/\Sigma)$ i.e. $T\mathrm{Sol}.$ The symbol
$$\sigma(D_{u_T}):T^*\Sigma\rightarrow \mathrm{Hom}\big(u_T^*T_{E/\Sigma},u_T^*V\big),$$
which is just the linear part of $F_P$ pulled-back under $u_T$ (c.f. \ref{eqn: Symbol of Linearization of Non-linear}).
 \section{Linearization and the equivariant loop stacks}
 \label{sec: Derived Linearization and the Equivariant Loop Stack}
Now and for the remainder of this paper, we switch gears and describe the linearization complexes and their microlocalization in terms of derived loop spaces i.e. sheaves of commutative DG-algebras over $\Omega_X^{-\bullet}=\mathrm{Sym}^{\bullet}(\Omega_X[1])$, and derived Hodge stacks  \cite{BenZviNad}.
The main point here is to begin to study and translate results about non-linear \textsc{pde} in terms of their linearizations, over possibly singular spaces e.g. by assuming $X$ is a derived scheme locally of finite presentation. We leave applications and details for a sequel work, and include the following two sections as they pertain, broadly speaking, to the central theme of the paper. 

\begin{remark}
By making this assumption, we no longer have a sheaf of microdifferential operators. So, we must use de Rham spaces and a new method for describing microlocalization of tangent complexes as in Proposition \ref{Derived Characteristic Variety Lemma}.
\end{remark}
Loop spaces seem to provide a natural junction point for studying linearizations and microlocalization in a derived geometric setting. For example, singular supports of ind-coherent sheaves over quasi-smooth stacks $X$ are contained in a scheme of singularities, $\mathsf{Sing}(X)$ and if $X=\mathcal{L}Z$ where $Z$ is smooth, this is nothing but $T^*Z.$
\subsection{Derived Loop Stacks and Odd Tangents}
The derived loop stack $\mathcal{L}X$ of $X$ is given by the derived mapping stack, or the derived self-intersection with the diagonal,
$$\mathcal{L}X:=\mathbb{R}\underline{Hom}(S^1,X)\simeq X\times_{X\times X}^hX,$$
where $S^1$ is the simplical circle, viewed as a derived stack given by the delooping of the constant group stack.
Other variants include the unipotent loops $\mathcal{L}^uX:=\mathbb{R}\underline{Hom}(B\mathbb{G}_a,X)$ and formal loops $\hat{\mathcal{L}}X,$ given by the formal completion of $\mathcal{L}X$ at constant loops $X\subset \mathcal{L}X.$

The affinization map $S^1\rightarrow B\mathbb{G}_a$ defines a map $\mathcal{L}^uX\rightarrow \mathcal{L}X$ and $\mathbb{G}_a\rtimes \mathbb{G}_m$ acts naturally on unipotent loops.

The derived loop space is linearized to the odd tangent bundle $$\mathsf{T}[-1]X:=Spec_{X}\big(\mathrm{Sym}_{\mathcal{O}_X}(\mathbb{L}_X[1])\big),$$
which can be identified with the normal bundle to constant loops, and one has  
\begin{equation}
    \label{eqn: exp}
    exp:\mathsf{T}[-1]X\xrightarrow{\simeq} \mathcal{L}X.
\end{equation}

The odd tangent stack is a derived vector bundle over $X$, equipped with a natural
contracting $\mathbb{G}_m$-action.
Recall from \cite{GR17b}, as a formal moduli problems (see also §\ref{ssec: Notations and Conventions}),
\begin{equation}
\label{eqn: infSpec}
\mathsf{Spec}_{/X}^{\mathsf{inf}}(-):\mathsf{CAlg}\big(\mathsf{QCoh}(X)\big)\rightarrow \mathsf{FMP}_{X//X},
\end{equation}
sends $\mathcal{E}\in \mathsf{QCoh}(X)$ via the functor
$\mathrm{Sym}_{\mathsf{QCoh}(X)}:\mathsf{QCoh}(X)\rightarrow \mathsf{CAlg}\big(\mathsf{QCoh}(X)\big)$ to the formal moduli problem
\begin{equation}
    \label{eqn: FMP}
    \mathbb{V}(\mathcal{E}^{\bullet}):=\mathsf{Spec}_{/X}^{inf}\big(\mathrm{Sym}_{\mathsf{QCoh}(X)}(\mathcal{E}^{\vee})\big).
\end{equation}
\begin{definition}
\label{Definition: Dual Derived Vector Bundle}
\normalfont
Let $\mathsf{E}_X$ be a derived vector bundle, whose underlying sheaf is a compact object, concentrated in degrees $[a,b].$ Then, its \emph{dual} is defined by $\mathsf{E}_X^{\vee}:=\mathbb{V}(\mathcal{E}^{\vee}).$
\end{definition}
Objects (\ref{eqn: FMP}) behave well under shift and pullback e.g. $\mathsf{E}_X[n]:=\mathsf{Spec}_{/X}\big(\mathrm{Sym}^*(\mathcal{E}^{\vee}[-n])\big),$ with its weight-grading by $\mathbb{G}_m$-actions. 

If $\mathcal{E}$ is connective (concentrated in degrees $[0,\infty)$), one has an object $[\mathsf{E}_X/\mathbb{G}_m]\in \PS_{/X\times B\mathbb{G}_m},$ where the action is of weight $+1$ and shifting $\mathsf{E}_X$ to $\mathsf{E}_X[n]$ corresponds to weight-degree shearing down the weight $k$ by $nk.$ By `shearing' we mean the action by natural functors $[\![n]\!]$ for each $n\in \mathbb{Z}$ acting on graded complexes by shifting the $n$-th component by $[n].$ One calls the $n=-2$ functor i.e. $[\![-2]\!]$ the \emph{Tate shearing.}
It takes an ordinary graded vector space (i.e. a graded
DG-vector space concentrated in degree $0$) to a complex of vector spaces with weight
on the cohomology given by ohomological degree.

The Koszul dual and the sheared Koszul dual bundles are defined by
\begin{equation}
    \label{eqn: Koszul Dual and Sheared Koszul Dual}
\scalemath{.90}{\mathsf{E}_X^{\kappa}:=\mathbb{V}_X(\mathcal{E}^{\vee}[1])=\mathsf{Spec}_{/X}\big
(\mathrm{Sym}_X^*(\mathcal{E}[-1])\big),\hspace{2mm} \mathsf{E}_{X}^{\kappa}[\![-2]\!]=\mathsf{Spec}_{/X}(\mathrm{Sym}_{\mathcal{O}_X}(\mathcal{E}[1])\big).}
\end{equation}
Applying the construction of formal moduli (\ref{eqn: FMP}) as well as \autoref{Definition: Dual Derived Vector Bundle} to the context of loop spaces, we get 
$$\mathsf{T}[-1]X=\mathbb{V}_X(\mathbb{T}_X[-1]),$$
with projection $j:\mathbb{V}(\mathbb{T}_X[-1])\rightarrow X.$ The Koszul dual and sheared Koszul dual are given by 
$$\mathsf{T}^*[2]X=(\mathsf{T}[-1]X)^{\kappa}=\mathsf{Spec}_{/X}(\mathrm{Sym}_{\mathcal{O}_X}(\mathbb{T}_X[-2])\big),$$
the $2$-shifted cotangent stack and cotangent stack $(\mathsf{T}[-1]X)^{\kappa}[\![-2]\!]=\mathsf{T}^*X,$ respectively.

Notice that $\mathcal{L}X\xrightarrow{i}X$ is schematic and if $X$ is a smooth and proper reduced scheme over a field of characteristic zero, 
$$^{red}\mathcal{L}X\simeq ^{red}\big(X\times_{X\times X}X\big)\simeq ^{red}X\times_{^{red}X\times ^{red}X}\hspace{1mm} ^{red}X\simeq X\times_{X\times X}X\simeq X.$$

If we denote by $c:X\rightarrow \mathcal{L}X$ the inclusion of constant loops, then
$$c^*\mathbb{T}_{\mathcal{L}X/X}\simeq c^*i^*\mathbb{T}_{X/X\times X}\simeq \mathbb{T}_{X/X\times X}\simeq \mathbb{T}_X[-1].$$ 

\begin{proposition}
The natural map $X_{\DR}\rightarrow \mathcal{L}X_{\DR}$ is an equivalence, and there is a natural $S^1$-equivariant morphism of prestacks,
$\beta_X:\mathcal{L}X\rightarrow\mathcal{L}X_{\DR}\simeq X_{\DR},$
where the $S^1$-action on $X_{\DR}$ is taken to be the trivial one.
\end{proposition}
When $X$ is an ordinary affine scheme there is an equivalence of rings $\pi_0(\mathcal{O}_{\mathcal{L}X})\simeq \mathcal{O}_X,$ so the underlying ordinary scheme of $\mathcal{L}X$ is
just $X.$

\subsubsection{Singular Support Conditions}
\label{sssec: Singular Support Conditions}
We now elaborate the notion of (linearized) singular supports and microlocal correspondences (c.f. equation (\ref{eqn: Microlocal correspondence})) for spaces $\RS(\EQ)$ in the setting of (quasi-smooth) derived Artin stacks, following \cite{AriGai2015}. In particular, we extend (\ref{definition: 1-Microchar}) to this setting.
Recall to a quasi-smooth derived stack $X$ one may introduce a classical pre-stack of singularities $\mathsf{Sing}(X).$

\begin{definition}
\normalfont
Suppose that $Z$ is a generalized non-linear \textsc{pde} which is $\D$-quasi-smooth. Then its \emph{prestack of singularities} is 
$\mathsf{Sing}\big(\RS(\EQ)\big).$
\end{definition}
Explicitly, take the underlying classical pre-stack of the shifted cotangent space
\begin{equation}
\label{eqn: Sing}
\mathsf{Sing}\big(\RS(\EQ)\big):=\hspace{1mm} ^{cl}\big(\mathsf{T}^*[-1]\RS(\EQ)\big)\simeq \hspace{1mm}^{cl}Spec\big(\mathrm{Sym}_{\mathcal{O}}(\mathbb{T}_{\RS(\EQ)}[1])\big).
\end{equation}
There is an affine map $\mathsf{Sing}\big(\RS(\EQ)\big)\rightarrow \pi_0\big(\RS(\EQ)\big).$ Its zero section, denoted $\RS(\EQ)^{cl}=\{0\}$ is embedded as a closed sub-scheme.

Since $\mathcal{H}^0$ of the cotangent complex is the usual cotangent space, stacky-phenomena implies there exists non-vanishing $\mathcal{H}^{-1}$ and 
$$\mathsf{Sing}\big(\RS(\EQ)\big)\simeq Spec_{^{cl}X}\big(\mathrm{Sym}^*\mathcal{H}^1(\mathbb{L}_{\RS(\EQ)}^{\vee})\big).$$

When $\mathcal{H}^{1}=0,$ this indicates `stable' or `regular' phenomena.
In general, considering $s:*\rightarrow \RS/\mathcal{G},$ the behaviour $\mathbb{T}_{s}(\RS/\mathcal{G})$ changes qualitatively as $s$ moves around the moduli space. This is realized by the `jumping' of $H^{\bullet}\big(\mathbb{T}_{s}(\RS/\mathcal{G})\big).$ Indeed, this reflects the changes in the stabilizer of $s$ in $\mathcal{G}.$

\begin{example}
\label{ex: Sing of LocSys}
\normalfont 
Consider $\RS(\EQ)$ with $Z:=pt/G\times X_{\DR}.$ Then, for each point $(P,A)\in \RS(\EQ),$ one has that 
$\mathcal{H}^{-1}\big(\mathbb{L}_{\RS(\EQ),(P,A)}\big)$ describes the infinitesimal symmetries of $(P,A)$ i.e. elements $\sigma\in \Gamma\big(ad(P)\big)^A.$ In this case,
$$\mathsf{Sing}\big(\RS(\EQ)\big)\simeq \big\{(P,A,\sigma)| (P,A)\in \RS(\EQ),\sigma \text{ an infinitesimal symmetry}\big\}.$$
\end{example}

\subsubsection{}
Given a closed conic subset $\Lambda$ of $\mathsf{Sing}(X)$ there is a subcategory $\mathsf{IndCoh}_{\Lambda}^!(X)$ consisting of those ind-coherent sheaves with singular support inside $\Lambda$ and if $Z$ is a dg-Artin stack with smooth atlas $U\rightarrow Z,$ for each $\Lambda\subset \mathsf{Sing}(\EQ)$ is defined such that $\mathsf{Sing}(U)=\mathsf{Sing}(\EQ)\times_Z U.$ 
Such a conic subset is said to be a \emph{support condition}. Note that $\IC_{\Lambda}^!(X)$ is compactly generated by $\mathsf{Coh}_{\Lambda}(X)$, the full subcategory of coherent sheaves with the same restriction to singular support. 
To $\mathcal{F}^{\bullet}\in \mathsf{IndCoh}^!(X)$ there is a closed conic involutive subset 
$\mathsf{SS}(\mathcal{F}^{\bullet})\subset \mathsf{Sing}(X),$ recording the failure of $\mathcal{F}^{\bullet}$ to be perfect. In other words, it is a measure of the codirections of smoothings of $X$ in which $\mathcal{F}^{\bullet}$ is obstructed from extending as a coherent complex (see also §§ \ref{sssec: Interpreting SS via Prop} below).

Singular support have the properties (c.f with § \ref{sec: Microlocal Analysis and Sheaf Propagation}) that: $\mathsf{SS}(\mathcal{F}^{\bullet})\cap \hspace{2mm} ^{cl}X= \mathrm{supp}(\mathcal{F}^{\bullet}),$ and $\mathcal{F}^{\bullet}\in \mathsf{Perf}(X)$ if and only if $\mathsf{SS}(\mathcal{F}^{\bullet})$ is the contained in the zero section.
In the case $Z$ is a derived Artin stack,
$$\IC_{\Lambda}(\EQ)=\big\{\mathcal{F}^{\bullet}\in \IC(\EQ)|\mathsf{SS}(p^!\mathcal{F}^{\bullet})\subset \Lambda\times_Z U\subset \mathsf{Sing}(\EQ)\times_Z U,\forall p:U\xrightarrow{\text{smooth}}Z\big\}.$$

    Given a $\D$-finitary space $Z$, put
    $\mathsf{SS}_{\varphi}(\EQ):=\mathsf{SS}\big(\mathbb{T}_{\RS(\EQ),\varphi}^{\ell}\big).$

If $f:X\rightarrow Y$ is affine schematic morphism of derived stacks, one has a microlocal correspondence
\begin{equation}
\label{eqn: Derived Microlocal Diagram}
\mathsf{Sing}(X)\xleftarrow{f_d}X\times_Y \mathsf{Sing}(Y)\xrightarrow{f_{\pi}}\mathsf{Sing}(Y).
\end{equation}

Given two support conditions $\Lambda_X$ on $X$ and $\Lambda_Y$ on $Y$, using (\ref{eqn: Derived Microlocal Diagram}) we put
$f^{-1}(\Lambda_Y):=f_d\big(f_{\pi}^{-1}(\Lambda_Y)\big),$ and $f(\Lambda_X):=f_{\pi}\big(f_d^{-1}(\Lambda_X)\big).$

\begin{proposition}
\label{prop: V-Sing}
    Let $f:X\rightarrow Y$ be a morphism of quasi-smooth derived Artin stacks. If $f$ is proper, then $f_{\pi}$ is proper and if $f$ is quasi-smooth, then $f_d$ is a closed immersion. 
Thus, for a quasi-smooth $f$ and a subset $V\subset \mathsf{Sing}(Y)$ there exists a support condition 
$V_X:=f_d(V\times_Y X)\subset\mathsf{Sing}(X).$
\end{proposition}
Singular support conditions behave in the expected way for push-forward and pull-backs along schematic and quasi-compact morphisms \cite[\S.~7.4]{AriGai2015}.

A $k$-point is thus a pair $p=(z;\xi)$ with $z\in Z(k)$ and $\xi\in H^{-1}\big(\mathbb{T}_z^{*}(\EQ)\big).$ We can observe what happens when this point belongs to $\mathsf{SS}(\mathcal{F})$ for a given $\mathcal{F}\in\mathsf{Coh}(\EQ).$ 

This characterization comes from representing $Z$ as a complete intersection by smooth schemes $U,V$ and letting $\xi$ be a cotangent to $V$ at $pt$ with $f$ a function on $V$ such that $f(pt)=0$ and $df=\xi.$ Let $Z'$ be the scheme cut out by restriction of $f$ to $U$, with corresponding embedding $i:Z\rightarrow Z'.$

For quasi-smooth stacks $X$, the degree $1$ part of the tangent complex, or the degree $-1$ part of the cotangent
complex, controls the singularities of $X$. 
\begin{proposition}
\label{Sing and LX Prop}
Consider $Z,U,V$ and $i:Z\rightarrow Z'$. Then the element $(z,\xi)$ belongs to $\mathsf{SS}(\mathcal{F})$ if and only if $i_*\mathcal{F}$ is not perfect on a Zariski neighborhood of $z$. 
Moreover, if $X$ is a smooth classical scheme, then $\mathsf{Sing}(\mathcal{L}X)\simeq T^*X$, and $\mathsf{Sing}(\mathcal{L}X)\rightarrow\hspace{1mm} ^{cl}\mathcal{L}X$ identifies with the ordinary cotangent bundle of $X.$
\end{proposition}
Proposition \ref{Sing and LX Prop} is proven in \cite{AriGai2015}. Via $\pi:\mathsf{Sing}(X)\rightarrow ^{cl}X,$ 
the fibers $\pi^{-1}(x)$ over $x\in ^{cl}X$ are the vector spaces $H^{-1}(\mathbb{L}_{X,x})$ specifying how singular $X$ is at $x$. In particular, the fiber  over $x$ is a point if and only if $X$ is smooth at $x,$ and in this case $\pi$ is an equivalence. The claim follows from applying the definition via (\ref{eqn: exp}).
Now if $(z,\xi)$ is not in the singular support, then it also not in the singular support of $\mathsf{SS}(f_*^{\IC}\mathcal{F})$ and so on a Zariski neighborhood of $z$, one has $\mathsf{SS}(f_*^{\IC}(\mathcal{F}))\subset \{0\}.$

We have an analog of Example \ref{Lagrangian Relation Example}.
\begin{example}
\normalfont 
Given a quasi-smooth morphism $f:X\rightarrow Y$ a Lagrangian relation on $f$ is $\Lambda_f:=f_{\pi}\big(X\times_{\mathsf{Sing}(X)}\big(\mathsf{Sing}(Y)\times_Y X\big))\subset \mathsf{Sing}(Y).$ If $f$ is proper and $\mathcal{E}\in \mathsf{Perf}(X),$ one has that $f_*^{\mathsf{Perf}}(\mathcal{E})$ is coherent with $\mathsf{SS}\big(f_*^{\mathsf{Perf}}(\mathcal{E})\big)$ contained in $\Lambda_f.$
\end{example}

\subsubsection{Singular Supports and Microcharacteristics}
\label{sssec: Interpreting SS via Prop}
We relate $\mathsf{SS}(\mathcal{F})$ to the propagation interpretation of microsupports § \ref{sec: Microlocal Analysis and Sheaf Propagation} we have used throughout this work and describe a possible analog of the property of a morphism being non-microcharacteristic in this context.
Suppose that for a geometric point $\eta:Spec(k)\rightarrow X$ we can present $X$ via functions $(f_1,\ldots,f_n):Y\rightarrow \A^n,$ from an affine smooth variety $Y$ as in the iterated pull-back diagram
\[
\begin{tikzcd}
    X\arrow[d]\arrow[r] & \overline{X}\arrow[d]\arrow[r] & Y
\arrow[d]
\\
\{0\in \A^n\}\arrow[r] & \mathbb{A}^1\times \{0\in \A^{n-1}\}\arrow[r] & \A^n
\end{tikzcd}
\]
The covector $df_1$ may be thought of as a section of $\mathsf{Sing}(X),$ and if $U$ is a neighbourhood of $\eta,$ then 
$df_1|_{U}\notin \mathrm{supp}(\mathcal{F}|_U)\subset \mathsf{Sing}(X)|_U,$
if and only if $\mathcal{F}$ is contained in the smallest thick subcategory of sheaves on $X$ generated by pullbacks from $\overline{X}.$ Roughly speaking, this is the case when $\mathcal{F}$ extends in the $f_1$-direction near $\eta$ i.e. $\mathcal{F}$ propagates in the codirection $df_1$ in the neighbourhood $U.$

Now, suppose that $f:X\rightarrow Y$ is a quasi-smooth morphism of quasi-smooth derived stacks, and fix a closed conical subset $V\subset \mathsf{Sing}(Y),$ with ideal sheaf $\mathcal{I}_V.$ Consider a $\D$-finitary derived non-linear \textsc{pde} $\J(E)$ over $Y$. We want to formulate non-characteristic conditions via the diagram (\ref{eqn: Derived Microlocal Diagram}).

On the one hand, we have a sequence of sheaves
$$\mathbb{L}_Y\rightarrow H^1(\mathbb{L}_Y[-1])\rightarrow I_V/I_V^2\rightarrow \bigoplus_jI_V^j/I_V^{j+1},$$
which define the conormal space $\mathsf{N}_{V/\mathsf{Sing}(Y)}=Spec(\mathrm{Sym}(\mathcal{I}_V/\mathcal{I}_V^2)),$ as well as the normal cone $C^{\mathsf{N}}=Spec(\bigoplus_j\mathcal{I}_V^j/\mathcal{I}_V^{j+1}).$ On the other hand we have a diagram of spaces that define a composite map $\alpha,$
\begin{equation}
\label{eqn: Sing-MicroChar Diagram}
\begin{tikzcd}
\alpha:C^{\mathsf{N}}\arrow[r] & \mathsf{N}_{V/\mathsf{Sing}(Y)}\arrow[r] & \mathsf{Sing}(Y)\arrow[r] & \mathsf{T}^*Y.
\end{tikzcd}
\end{equation}

Fix $\mathcal{F}\in \IC(Y)$ and put $\mathsf{SS}(\mathcal{F})\cap V$ to be the image under the composite map (\ref{eqn: Sing-MicroChar Diagram}) of the intersection of $V$ (corresponding to the normal cone) and $\mathsf{SS}(\mathcal{F}),$ simply denoted $\mathsf{SS}(\mathcal{F})\cap V.$

\begin{definition}
The morphism $f$ is said to be \emph{non-microcharacteristic} for $\mathcal{F}$ along $V$, if
\begin{equation}
    \label{eqn: SS Non-micro-char}
f_d\circ f^{-1}\big(\mathsf{SS}(\mathcal{F})\cap V\big)|_{\mathsf{Sing}(X)\subset \mathsf{T}^*X}\subset \{0\}_X.
\end{equation}
\end{definition}

\subsubsection{Sheaves on loops and Koszul duality}
Let $\mathsf{Perf}_X(\mathcal{L}X)$ denote the small, stable $\infty$-subcategory of $\mathsf{QCoh}(\mathcal{L}X)$ whose quasi-coherent pushforward $\tau_*^{\mathsf{QCoh}}$ along the canonical map $\tau:\mathcal{L}X\rightarrow X$ are perfect. Note the structure sheaf $\mathcal{O}(\mathcal{L}X)\simeq \mathrm{Sym}_{\mathcal{O}_X}^{\bullet}(\mathbb{L}_X[1])\big)$ is perfect over $\mathcal{O}_X$.

\begin{proposition}[{\cite[Cor.5.3]{BenZviNad}}]
\label{LoopSpacesConnectionsThm}
There is a canonical equivalence of small, stable $\infty$-categories
$\mathsf{Perf}_X(\mathcal{L}X)_{per}^{B\mathbb{G}_a \rtimes \mathbb{G}_m}\simeq \mathsf{Perf}(\D_X).$
\end{proposition}

\begin{remark}
Theorem \ref{LoopSpacesConnectionsThm} holds for schemes, but extends for more general spaces e.g. QCA stacks. In the latter case, by avoiding mentioning the notion of \emph{renormalized} or un-renormalized categories of (ind-)coherent sheaves e.g. $\IC(X)^{\omega B\mathbb{G}_a},$ we mention the identifications
$\IC(\mathsf{T}[-1]X)^{\omega B\mathbb{G}_a}$ with a certain category of $\Q(X)^{\mathbb{G}_m}$-module objects over a monad given by tensoring with $\mathbb{R}\mathcal{H}om(\mathcal{O}_X,\mathcal{O}_X)\otimes_{\mathcal{O}_X}(-),$ taken in the category of modules over the (coconnective) mixed de Rham algebra of $X$ (see \cite[Theorem 4.9]{BenZviNad}). Via Tate-unshearing, the monad agrees with tensor by $Rees(\D_X).$
\end{remark}
Filtered Koszul duality preserves supports \cite[Theorem 3.4.9]{ChenFiltered}. Left and right filtered $\D$-modules on QCA stacks $\mathsf{Mod}^{fil}(\D_{\mathcal{X}})$ and $\mathsf{Mod}^{fil}(\D_{\mathcal{X}}^{op}),$ are defined according to \emph{loc.cit.} 
\begin{proposition}
\label{KoszulDualStacks}
Let $\mathcal{F}\in\mathsf{IndCoh}(\mathcal{L}X)^{B\mathbb{G}_a\rtimes\mathbb{G}_m}$ with singular support $\mathsf{SS}(\mathcal{F}).$
Ind-coherent Koszul duality for QCA stacks $\mathcal{X}$ provides an equivalence
$\kappa:\mathsf{IndCoh}\big(\hat{\mathbb{T}}_{\mathcal{X}}[-1]\big)^{\mathbb{G}_m}\xrightarrow{\simeq}\mathsf{QCoh}\big(\mathbb{T}_{\mathcal{X}}^*\big)^{\mathbb{G}_m}.$ If $\mathcal{F}$ is an ind-coherent sheaf, then $\mathsf{SS}(\mathcal{F})$ and the micro-support of the Koszul dual complex $\mu\mathrm{supp}(\kappa\mathcal{F})$ coincide. 
\end{proposition}
We do not prove Theorem \ref{KoszulDualStacks} here, but say a few words about its method.
Namely, the Koszul duality for QCA stacks $\mathcal{X}$ sit in the diagram of equivalences,

\begin{equation}
    \label{eqn: KoszulDiagram}
\begin{tikzcd}
\mathsf{IndCoh}\big(\hat{\mathbb{T}}_{\mathcal{X}}[-1]\big)^{grTate}\arrow[d,"\simeq"] & \mathsf{IndCoh}\big(\hat{\mathbb{T}}_{\mathcal{X}}[-1]\big)^{B\mathbb{G}_a^{gr}}\arrow[d,"\simeq"]\arrow[l] \arrow[r] & \mathsf{IndCoh}\big(\hat{\mathbb{T}}_{\mathcal{X}}[-1]\big)^{\mathbb{G}_m}\arrow[d,"\simeq"]
\\
\mathsf{Mod}(\D_{\mathcal{X}})& \arrow[l] \mathsf{Mod}^{fil}(\D_{\mathcal{X}})\arrow[r,"gr"] & \mathsf{QCoh}\big(\mathbb{T}_{\mathcal{X}}^*\big)^{\mathbb{G}_m}.
\end{tikzcd}
\end{equation}

Diagram (\ref{eqn: KoszulDiagram}) is compatible with schematic pullback, induction and restriction functors, which follows from Proposition \ref{HodgeCrystalsProperties}. If $f:\mathcal{X}\rightarrow \mathcal{Y}$ is proper, by passing to left adjoints, the equivalences are compatible with $f_*^{\IC}$. 

Remark we use the sheared Koszul dualities (\ref{eqn: Koszul Dual and Sheared Koszul Dual}).
Moreover, the compact objects in $\mathsf{Mod}^{fil}(\D_{\mathcal{X}})$ are generated
by inductions from $\mathsf{Coh}(\mathcal{X})$ while $\mathsf{Coh}(\hat{\mathbb{T}}_{\mathcal{X}}[-1])$ is generated by objects pushed forward from $\mathsf{Coh}(\mathcal{X}).$ This claim follows from the identification of categories $\mathsf{Shv}_{\mathcal{Z}}(\mathcal{X}),$ of sheaves set-theoretically supported on a closed-substack $\mathcal{Z}$ and sheaves on its completion $\mathsf{Shv}(\widehat{\mathcal{Z}}),$ applied to $\mathcal{X}$ and $\mathbb{T}_{\mathcal{X}}[-1]$ i.e. $\widehat{\mathbb{T}}_{\mathcal{X}}[-1].$  

\begin{remark}
    If $X\rightarrow Z$ is a morphism with $Z$ smooth and affine, we have $\IC^!(X_{\DR})\simeq \IC_X^!(Z_{\DR})\simeq fib\big(\IC(Z_{\DR})\rightarrow \IC\big((Z\backslash X)_{\DR}\big)\big).$
\end{remark}

The result then follows by noticing $\mathsf{Coh}^{fil}(\D_{\mathcal{X}})$ consists of filtered $\D_{\mathcal{X}}$-modules $\mathcal{M}$ such that for all smooth atlases $p:U\rightarrow \mathcal{X}$ the pull-backs $p^!\mathcal{M}$ are coherent as filtered $\mathcal{D}_U$-modules.

One uses \autoref{KoszulDualStacks} in a practical manner to compute supports via descent along atlases. Namely, if $\mathcal{M}$ is a filtered coherent $\D_X$-module its support is computed as a limit of supports over the atlases; for our algebraic stack $\mathcal{X}$ with a smooth surjection from an affine space $p:U\rightarrow \mathcal{X}$, we set
$\mathrm{supp}(\mathcal{M}):=p\big(\mathrm{supp}(p^!\mathcal{M})\big),$
and note that for $\mathcal{M}^{\bullet}$ we get
\begin{equation*}
    \mathrm{supp}(\mathcal{M}^{\bullet})\bigcup_{n\in \mathbb{Z}}supp\big(\mathcal{H}^n(\mathcal{M}^{\bullet})\big)
    \simeq \bigcup_{n\in\mathbb{Z}} \underset{p:U\rightarrow \mathcal{X},\text{smooth}}{\mathrm{lim}} p\big(\mathrm{supp}(\mathcal{H}^n\big(p^!\mathcal{M})\big)
    \simeq \bigcup_{n\in\mathbb{Z}} \underset{p:U\rightarrow \mathcal{X},\text{smooth}}{\mathrm{lim}}p\big(\mathrm{supp}(p^*\mathcal{H}^n\big(\mathcal{M}^{\bullet})\big).
    \end{equation*}
    By Proposition \ref{KoszulDualStacks}, it agrees with
    $\underset{p:U\rightarrow \mathcal{X},\text{smooth}}{\mathrm{lim}} \mathsf{SS}(p^!Gr\mathcal{M})$ and is further identified with 
    $$\mathsf{SS}(Gr\mathcal{M})\simeq \underset{p:U\rightarrow \mathcal{X},\text{smooth}}{\mathrm{lim}} \mathsf{SS}(Grp^!\mathcal{M}),$$
as $!$-pullback along smooth morphisms coincide with $*$-back which commute with taking cohomologies.

\subsection{}
\label{ssec: Proof of Theorem D}
We now state the result giving a passage for viewing linearized $\D$-spaces of solutions over $X_{\DR}$, as sheaves over $\mathcal{L}X.$ In what follows, given a $\D$-module $\mathcal{M}$, understand the abusive notation $\widetilde{\mathcal{M}}$ to refer to the status of $\mathcal{M}$ over $\mathcal{L}X$ via (\ref{KoszulDualStacks}) or (\ref{LoopSpacesConnectionsThm}). See \ref{sssec: Longer Example} below for an explicit description.
\begin{theorem}
\label{Thm C}
Suppose that $X$ is a smooth scheme and consider
a $\D$-algebraic non-linear \textsc{pde} $\mathcal{I}\rightarrow \mathcal{A}\xrightarrow{p}\mathcal{B}$ defined on sections of a trivial bundle $E\rightarrow X,$ which is differentially generated and whose solution stack is homotopically almost finitely $\D$-presented.

For every globally defined classical solution $\varphi$, there is a distinguished triangle 
$\widetilde{\mathbb{T}}_{\mathcal{B}/\mathcal{A},\varphi}\rightarrow \widetilde{\mathbb{T}}_{\mathcal{B},\varphi}\rightarrow \widetilde{p_!\mathbb{T}}_{\mathcal{A}/\mathcal{I},\varphi}\xrightarrow{+1}$ in $\mathsf{IndCoh}(\mathbb{T}_Y[-1])^{S^1},$ independent of $\varphi.$
\end{theorem}
 
\begin{proof}
Using our fixed horizontal section $\varphi:X\rightarrow \mathsf{Sol}(\mathcal{B}
)$, form the pullback $\mathbb{T}_{\mathsf{Sol}_X(\mathcal{B})}\in \mathsf{IndCoh}\big(\mathsf{Sol}(\mathcal{B})\times_X X_{\DR}\big)$, where the fibre product $\mathsf{Sol}(\mathcal{B})\times_XX_{\DR}$ is taken in the $\infty$-category $\mathsf{dStk}_X(\D_X)$ of derived $\D_X$-stacks over $X$. A computation via the Koszul duality relation in \autoref{LoopSpacesConnectionsThm} (this version of Koszul duality is sufficient as $X$ is smooth), they correspond to $X$-perfect equivariant sheaves on $\mathcal{L}X$ shifted in degree by $-2,$ that we denoted simply by installing a $\widetilde{(-)}$ overhead.

By \autoref{LoopSpacesConnectionsThm},
homotopy cokernels are preserved, or that distinguished triangles in the corresponding homotopy categories of $\infty$-categories are sent to distinguished triangles. So, it follows that the cofiber sequence
$p_!\mathbb{L}_{\mathcal{A}/\mathcal{I}}\rightarrow \mathbb{L}_{\mathcal{B}/\mathcal{I}}\rightarrow \mathbb{L}_{\mathcal{B}/\mathcal{A}}\xrightarrow{+1}$ in $\hspace{1mm}\mathsf{Mod}_{\D_X}\big(\mathsf{Sol}(\mathcal{B})\big),$ and the functoriality of the $!$-pullbacks along $\varphi$, we have by duality a homotopy kernel sequence in $\IC(X_{\DR}).$
By our finiteness assumptions on $\mathsf{Sol}(\mathcal{B})$ its cotangent tangent $\D$-complexes are written locally as a bounded complex of free and finite rank $H^0(\mathcal{A}/\mathcal{I})[\mathcal{D}_X]$-modules, built out of symmetric algebras on underlying compact $\D$-modules. In this case, we see by \autoref{Tangent Good Filtration} that $\mathbb{T}_{\mathcal{B}}$ is a bounded complex of free and finite rank modules. Under pull-back gives us a complex of coherent $\D_X$-modules.
It is possible to write a locally free resolution of our derived linearization by coherent $\D_X$-modules, which will be bounded complex whose terms are free and finite rank Koszul duals.
\end{proof}

Furthermore, identifying $\mathsf{Mod}_{\D_X}(\mathsf{Sol}(\mathcal{B}))\simeq \IC(\mathsf{Sol}_X(\mathcal{B})\times_X X_{\DR})$ via pull-back along the infinite jet of a solution, the Koszul dual complex $\widetilde{\mathbb{T}}_{\mathcal{B},\varphi},$ is an $\Omega_X^{-\bullet}$-module.

Note that in this case when $\mathcal{B}$ is cofibrantly replaced i.e. $Q\mathcal{B}\rightarrow \mathcal{B}$ there naturally arises the diagram
\begin{equation}
\label{Cotangent Diagram}
\begin{tikzcd}
& & \mathcal{B}^{\bullet}\otimes_{\mathcal{B}_{k-1}}\mathcal{M}_{-k}^{\bullet}[k]
\arrow[d]
\\
\mathcal{B}^{\bullet}\otimes_{\mathcal{B}_{k-1}}\mathbb{L}_{\mathcal{B}_{k-1}}\arrow[d]\arrow[r] & \mathbb{L}_{\mathcal{B}}\arrow[d]\arrow[r] & \mathbb{L}_{\mathcal{B}/\mathcal{B}_{k-1}}\arrow[d]
\\
\mathcal{B}^{\bullet}\otimes_{\mathcal{B}_{k}}\mathbb{L}_{\mathcal{B}_{k}}\arrow[d]\arrow[r] & \mathbb{L}_{\mathcal{B}}\arrow[r] & \mathbb{L}_{\mathcal{B}/\mathcal{B}_{k}}
\\
\mathcal{B}\otimes_{\mathcal{B}_{k-1}}\mathcal{M}_{-k}[k] & & 
\end{tikzcd}
\end{equation}
which is useful in computing the cohomologies of $\mathbb{L}_{\mathcal{B}}$ as well as $\widetilde{\mathbb{L}}_{\mathcal{B},\varphi}.$

\begin{corollary}
    The singular supports ind-coherent sheaves satisfy the triangle inequality
$\textbf{supp}\big(\widetilde{\mathbb{T}}_{\mathcal{B},\varphi}\big)\subset \textbf{supp}\big(\widetilde{\mathbb{T}}_{\mathcal{B}/\mathcal{A},\varphi})\cup\textbf{supp}\big(\widetilde{p_!\mathbb{T}}_{\mathcal{A}/\mathcal{I},\varphi}\big).$
Moreover, by forgetting the $B\mathbb{G}_a$-action, this is equivalent to the usual triangle inequality for microsupports $\mu\mathrm{supp}$ of $\D$-modules in the $\infty$-category $\mathsf{Mod}(\D_X).$ In particular, we recover the classical inequality (c.f. § \ref{sec: Microlocal Analysis and Sheaf Propagation}) in the corresponding homotopy category $h\mathsf{Mod}(\D_X).$
\end{corollary}

\section{Microlocalization and the Hodge stack}
\label{sec: Derived Microlocalization and the Hodge Stack}
In this section, we describe a version of microlocalization sufficient for our purposes. In particular, it plays the role of endowing derived objects with a suitable notion of a good filtration.

The idea we follow produces an $\mathbb{A}^1$-family of prestacks $X\times\mathbb{A}^1\rightarrow Y_{scaled}$ from an object $X\rightarrow Y\in \mathsf{FMP}_{X/}$ where $Y_{scaled}$ as constructed in \cite[Chapter 9]{GR17b}. It is left-lax equivariant with respect to $\mathbb{A}^1$ acting by multiplication on itself. Equivariance with respect to $\mathbb{G}_m\subset \mathbb{A}^1$ implies all fibers $Y_{t}$ of $Y_{scaled}$ at $t\neq 0\in \mathbb{A}^1$ are canonically isomorphic to $X.$ The fiber at $0\in \mathbb{A}^1$ identifies with the formal version of the total space of the normal bundle $\mathsf{T}(X/Y)[1]$ of $X$ in $Y$.  

\subsubsection{Filtrations and gradings}
A graded object in quasi-coherent sheaves on a derived scheme $X$ can be defined by $\Q(X)^{gr}:=\Q([X/B\mathbb{G}_m]),$ while filtered modules are given by $\Q(\AG).$ Extending these ideas to prestacks, put
\begin{equation}
    \label{eqn: Gr-fil stacks}
    \PS^{gr}:=\PS_{/B\mathbb{G}_m},\hspace{2mm}\text{ and }\hspace{3mm} \PS^{fil}:=\PS_{/\AG}.
    \end{equation}

Since $\mathbb{G}_m$ acts on the formal group $\hat{\mathbb{G}}_a$ and on the formal circle $\hat{S}^1:=B\hat{\mathbb{G}}_a$ as derived stacks, the formal circle is a group stack acting by translation and we have a group stack $\mathcal{H}:=\hat{S}^1 \rtimes \mathbb{G}_m$. Then, $\Q(B\mathcal{H})$ is naturally equivalent as a symmetric monoidal $\infty$-category to the $\infty$-category of graded mixed complexes $\epsilon-\mathsf{dg}_k^{gr}.$
We can make sense of $\epsilon-\PS^{gr}$ as the $\infty$-category of prestacks over $B\mathcal{H}.$ Such `graded-mixed prestacks' are not required further in this work, but might be related to the spaces in Propositions \ref{FilPrestackProp1} and \ref{FilPrestackProp2} below.

Considering the canonical maps
$0:B\mathbb{G}_m\rightarrow \AG,$ and $1:Spec(k)=pt\rightarrow \AG,$
we can consider for a given object
$X\in \PS$, its base-change along $0$, denoted by $Gr(X).$ Similarly, base-change along $1$ defines an underlying-stack object $Un(X).$

\begin{proposition}
\label{FilPrestackProp1} 
Let $X\in \PS^{fil}.$ The formation of associated graded, and underlying stack objects are compatible operations. That is, there is a canonical pull-back diagram
\[
\begin{tikzcd}
    Gr(X)\arrow[d,"Gr(c)"] \arrow[r,"0"] & X\arrow[d,"c"] & \arrow[l,"1", labels=above] Un(X)\arrow[d,"Un(c)"]
    \\
    B\mathbb{G}_m\arrow[r,"0"] &\AG& \arrow[l,"1", labels=above] Spec(k).
\end{tikzcd}
\]
\end{proposition}
Suppose that $X$ is a prestack with a $\mathbb{G}_m$-action. Then we may consider its quotient stack $[X/\mathbb{G}_m]$ together with a canonical map $\varrho$ to $B\mathbb{G}_m.$

\begin{proposition}
\label{FilPrestackProp2} 
Let $X\in \PS_{laft}$ be endowed with a $\mathbb{G}_m$-action. Then there exists an object $X^{fil}\in \PS^{fil}$ defined by the pull-back diagram
\[
\begin{tikzcd}
X^{fil}:=X/\mathbb{G}_m\times_{B\mathbb{G}_m}\AG \arrow[d] \arrow[r] & \XG \arrow[d,"\varrho"]
    \\
    \AG \arrow[r] & B\mathbb{G}_m
\end{tikzcd}
\]
which is again laft. Furthermore, $Gr(X^{fil})$ is the stacky quotient by $\mathbb{G}_m,$ and
$X^{fil}|_{\mathbb{A}^1}\simeq \mathbb{A}^1\times X.$ The underlying object of $X^{fil}$ coincides with $X.$
\end{proposition}
\begin{proof}
Notice from Proposition \ref{FilPrestackProp1} there is a pull-back diagram,
\[
\begin{tikzcd}
Gr(X^{fil})=\XG\arrow[d]\arrow[r] & X^{fil}\arrow[d,"c"]& \arrow[l] Un(X^{fil})\arrow[d]
\\
B\mathbb{G}_m\arrow[r,"0"] & \AG& \arrow[l,"1", labels=above] Spec(k)
\end{tikzcd}
\]
The remaining claims are established by noticing the construction of the associated graded and underlying-object prestacks are formed via base-change and that we have a Cartesian square
\[
\begin{tikzcd}
    \mathbb{A}^1\arrow[d]\arrow[r] & Spec(k)\arrow[d]
    \\
    \AG\arrow[r] & B\mathbb{G}_m.
\end{tikzcd}
\]
\end{proof}
By Propositions \ref{FilPrestackProp1} and \ref{FilPrestackProp2}, we can define graded and filtered lifts of a prestack $X$. They will be objects over $pt/\mathbb{G}_m$ whose pull-back to $pt\rightarrow pt/\mathbb{G}_m$ is equivalent to $X$, and objects over $\AG$ whose pull-back $pt\simeq \mathbb{A}^1-\{0\}/\mathbb{G}_m\rightarrow \AG,$ agree with $X$.

    This makes sense in greater generality e.g. for quasi-coherent sheaves of categories. The main example we have in mind, is the category $\IC(X_{\DR})$ for a laft-prestack $X,$ whose filtered lift is $\IC(X_{\Hod})$ and whose graded lift is $\IC(\mathsf{T}[1]X)\simeq \IC
(\mathsf{T}^*X).$

\subsection{Microlocalization with derived bases}
Suppose now that $X$ be a derived stack and consider a nil-isomorphism $X\rightarrow Y$ where $Y\in \PS_{\mathrm{laft-def}}$ (a formal moduli problem under $X$ \cite[§ 5.1.3]{GR17b}).  

To such an object, define a formal deformation
to the normal bundle $\widehat{\mathsf{N}}_{X/Y}.$ It is a prestack over $[\mathbb{A}^1/\mathbb{G}_m]$ with generic fiber $Y$ and special fiber a formal moduli problem under $X$. The latter is give by a colimit of split square-zero extensions by symmetric powers to the normal bundle i.e. $\mathrm{Sym}^n(\mathsf{T}_{X/Y}[1]),$ for each $n.$ 
Denote the substacks $[\mathbb{G}_m/\mathbb{G}_m]\simeq *$ by $[1]$ and $[\{0\}/\mathbb{G}_m]\simeq B\mathbb{G}_m$ by $[0].$

Recalling Proposition \ref{FilPrestackProp1} and Proposition \ref{FilPrestackProp2}, we have the following.
\begin{definition}
\label{Definition: Hodge Prestack}
The \emph{Hodge prestack} $X_{\Hod}$ of $X$ is the deformation to the normal bundle construction of the moduli problem $X\rightarrow X_{\DR}$, described by a prestack
$$\widehat{\mathsf{N}}_{X/X_{\DR}}\in \PS_{/(\mathbb{A}^1/\mathbb{G}_m)}\simeq \mathsf{FMP}_{X\times \mathbb{A}^1/ / X_{\DR}\times\mathbb{A}^1}:=\big(\mathsf{FMP}_{/X_{\DR}\times \mathbb{A}^1}\big)_{X\times \mathbb{A}^1/}.$$
\end{definition}
In particular, $X_{\Hod}$ is a filtered prestack with generic fiber (fiber over $[1]$) given by $X_{\Hod,[1]}\simeq X_{\DR}$ and special fiber (fiber over $[0]$) the formal moduli problem under $X\times[\mathbb{A}^1/\mathbb{G}_m]$ denoted
\begin{equation}
    \label{eqn: Hodge Structure Map}
i:X\times[\mathbb{A}^1/\mathbb{G}_m]\rightarrow X_{\Hod}.
\end{equation}
Special fiber is the so-called Dolbeault stack $X_{Dol}$ over $B\mathbb{G}_m,$ with natural $\mathbb{G}_m$-action.
In other words, it is the Dolbeault degeneration of $X_{\DR}$ to (total space) of $\mathsf{T}[1]X.$

The morphism (\ref{eqn: Hodge Structure Map}) induces forgetul functors,
\begin{equation}
\label{eqn: Hodge Forgetful Functors}
i_{\IC}^!:\IC(X_{\Hod})\rightarrow \IC(X\times\AG),
\hspace{2mm} \text{ and }
i_{\Q}^*:\Q(X_{\Hod})\rightarrow \Q(X\times\AG).
\end{equation}

\begin{proposition}
Functors \emph{(\ref{eqn: Hodge Forgetful Functors})} are compatible with side-changing equivalences $\Upsilon_{X_{\Hod}}$ and $\Upsilon_{X\times \AG}$ in the sense that
\[
\begin{tikzcd}
\Q(X_{\Hod})\arrow[d,"\Upsilon_{X_{\Hod}}"] \arrow[r,"i_{\Q}^*"] & \Q(X\times[\mathbb{A}^1/\mathbb{G}_m])\arrow[d,"\Upsilon_{X\times[\mathbb{A}^1/\mathbb{G}_m]}"]
\\
\IC(X_{\Hod})\arrow[r,"i_{\IC}^!"] & \IC(X\times[\mathbb{A}^1/\mathbb{G}_m]),
\end{tikzcd}
\]
is commutative.
\end{proposition}
The map (\ref{eqn: Hodge Structure Map}) also defines a canonical morphism $p_{X,\Hod}:X\rightarrow X_{\Hod}$ for which one has that
\begin{equation}
    \label{eqn: ICHdg-pb}
    (p_{\Hod})_{\IC}^!:\IC(X_{\Hod})\rightarrow \IC(X),
\end{equation}
admits a left-adjoint, denoted by 
\begin{equation}
    \label{eqn: HdgInd}
    (p_{\Hod})_{*}^{\IC}:\IC(X)\rightarrow \IC(X_{\Hod}).
\end{equation}
Objects in the essential image of functor (\ref{eqn: HdgInd}) are said to be \emph{induced}. Composition defines an $\infty$-monad
\begin{equation}
\label{eqn: Hodge Monad}
\mathbf{oblv}^{fil}\circ \mathrm{ind}^{fil}(-):=(p_{\Hod})_{\IC}^!\circ(p_{\Hod})_{*}^{\IC}:\IC(X)\rightarrow \IC(X).
\end{equation}
The action of the monad (\ref{eqn: Hodge Monad}) should be interpreted as tensoring with the algebra of filtered differential operators and we want to think of (\ref{eqn: Hodge Monad}) as the \emph{filtered} lift of the action of the monad $\mathbf{oblv}\circ ind$. Via $\Upsilon:\Q(X)\rightarrow \IC(X),$
define
$$(p_{\Hod})_*^{\Q}:=\Upsilon_X^{-1}\circ (p_{\Hod})_*^{\IC}\circ\Upsilon_X:\Q(X)\rightarrow \Q(X_{\Hod}).$$

\begin{proposition}
\label{HodgeCrystalsProperties}
Let $f:X\rightarrow Y$ be a morphism of laft-def prestacks. Then there is an induced morphism $f_{\Hod}:X_{\Hod}\rightarrow Y_{\Hod}$ and it is:
\begin{itemize}
  \item Compatible with $\mathsf{IndCoh}^!$

  \item Compatible with side-changing operations,

  \item Compatible with the functors (\ref{eqn: Hodge Forgetful Functors}) induced by (\ref{eqn: Hodge Structure Map}).

  \end{itemize}
  Moreover if $f$ is a map of inf-schemes, there is a push-forward functor $f_*$ compatible with induction (the left-adjoint $i_*$ to forgetful $i^!$) and if $f$ is proper, the functors $f_*$ and $f^!$ are adjoint i.e.
  $$(f_{\Hod})_*^{\IC}:\IC(X_{\Hod})\rightarrow \IC(Y_{\Hod}),$$
  is such that $(f_{\Hod})_*^{\IC}\circ \mathrm{ind}_X^{fil}\simeq \mathrm{ind}_Y^{fil}\circ f_*^{\IC}.$
\end{proposition}
The compatibilities in Proposition \ref{HodgeCrystalsProperties} amount to the commutativity diagrams
\[
\begin{tikzcd}
  \IC(Y_{\Hod})\arrow[d,"f_{\Hod}^!"] \arrow[r,"p_{\Hod}^!"] & \IC(Y)\arrow[d,"f^!"]
  \\
  \IC(X_{\Hod})\arrow[r,"p_{\Hod}^!"] & \IC(X)
\end{tikzcd}\hspace{5mm} \begin{tikzcd}
  \Q(Y_{\Hod})\arrow[d,"f_{\Hod}^*"] \arrow[r,"p_{\Hod}^*"] & \Q(Y)\arrow[d,"f^*"]
  \\
  \Q(X_{\Hod})\arrow[r,"p_{\Hod}^*"] & \Q(X)
\end{tikzcd}
\]

\begin{proposition}
    If $f:X\rightarrow Y$ is a morphism of (ind-)inf schemes and $f$ is (ind-)proper, then $f_{\Hod}:X_{\Hod}\rightarrow Y_{\Hod}$ is (ind-)proper.
\end{proposition}
\subsubsection{Filtered crystals}
Consider the cocartesian fibration associated to the functor $\mathsf{IndCoh}^!\circ (-)_{\Hod}$, defined on derived affines $U$ and right Kan extended via Yoneda to all prestacks $\EuScript{Y}$ along smooth maps from $U$ to $\EuScript{Y}$: 
\begin{equation}
    \label{eqn: Right filtered crystals}
\C_{fil}^{r}(\EuScript{Y}):=\IC^!(\EuScript{Y}_{\Hod})\simeq \underset{U\rightarrow \EuScript{Y}\in \mathsf{dAff}_{/\EuScript{Y}}^{op}}{lim}\IC^!(U_{\Hod}).
\end{equation}
More precisely \eqref{eqn: Right filtered crystals} is defined by Kan extension from affines to a functor $\underline{\C}_{fil}^r:\mathsf{IndInfSch}_{indprop}\rightarrow \mathsf{Cat}_{\infty},$ with $\mathsf{IndInfSch}$ the $\infty$-category of ind-inf schemes with ind-proper maps.

Left filtered-crystals are defined similarly, in terms of the $*$-pullback operation on crystals
\begin{equation}
    \label{eqn: Left filtered crystals}
\C_{fil}^{\ell}(\EuScript{Y}):=\Q^*(\EuScript{Y}_{\Hod})\simeq \underset{U\rightarrow \EuScript{Y}\in \mathsf{dAff}_{/\EuScript{Y}}^{op}}{lim}\Q^*(U_{\Hod}).
\end{equation}

Consider the morphism $\mathbf{1}:X_{\Hod}\rightarrow X_{\DR},$ and its induced map
$$Un:\IC^!(X_{\Hod})\rightarrow \IC^!(X_{\DR}).$$ 
Passing to special and generic fibers gives rise to functorialities of $\Q(\AG)$-module $(\infty,1)$-categories for (\ref{eqn: Left filtered crystals}), 
\begin{equation}
\label{eqn: LCrysDiagram}
\adjustbox{scale=.90}{
\begin{tikzcd}
& \arrow[dl,"Un", labels=above left] \scalemath{.90}{\C_{fil}^{\ell}(X)} \arrow[dr,"Gr"] \arrow[d,"res"] & 
\\
\equalto{\scalemath{.90}{\C_{fil}^{\ell}(X)\otimes_{\Q(\AG)}\Q([\mathbb{G}_m/\mathbb{G}_m])}}{\Q(X_{\DR})}& \scalemath{.90}{\Q(X)^{fil}} & \equalto{\scalemath{.90}{\C_{fil}^{\ell}(X)\otimes_{\Q(\AG)}\Q([\{0\}/\mathbb{G}_m])}}{\Q(\mathsf{T}^*X/\mathbb{G}_m)}
\end{tikzcd}}
\end{equation}
 and similarly for and (\ref{eqn: Right filtered crystals}),
\begin{equation}
\label{eqn: RCrysDiagram}
\adjustbox{scale=.90}{
\begin{tikzcd}
& \arrow[dl,"Un", labels=above left] \scalemath{.87}{\C_{fil}^{r}(X)} \arrow[dr,"Gr"] \arrow[d,"res"] & 
\\
\equalto{\scalemath{.87}{\C_{fil}^{r}(X)\otimes_{\IC(\AG)}\IC([\mathbb{G}_m/\mathbb{G}_m])}}{\IC(X_{\DR})} & \scalemath{.87}{\IC(X)^{fil}} & \equalto{\scalemath{.87}{\C_{fil}^{r}(X)\otimes_{\IC(\AG)}\IC([\{0\}/\mathbb{G}_m])}}{\IC(\mathsf{T}^*X/\mathbb{G}_m)}
\end{tikzcd}}
\end{equation}

Apply the Hodge stack construction to the derived loop stack.
\begin{proposition}
\label{LXHodge}
    Let $X$ be a finite type derived scheme. Then the Hodge stack of $\mathcal{L}X$ is the $\AG$-family of derived prestacks $\mathcal{L}X_{\Hod},$ such that
$\mathcal{L}X_{\Hod}\times_{[\mathbb{A}^1/\mathbb{G}_m]}[1]\simeq \mathbf{1}^*(\mathcal{L}X_{\Hod})\simeq X_{\DR},$ and $\mathcal{L}X_{\Hod}\times_{[\mathbb{A}^1/\mathbb{G}_m]}[0]\simeq \mathbf{0}^*(\mathcal{L}X_{\Hod})\simeq \mathcal{L}X.$
\end{proposition}
Using Proposition \ref{LXHodge} and apply the functors $\Q$ and $\IC^!$ to $\mathcal{L}X_{\Hod},$ one has a family of sheaves of categories with special and generic fibers,
\begin{equation}
\label{eqn: IC Loop Hodge}
\IC^!(\mathcal{L}X_{\Hod})_{\{1\}}\simeq \mathsf{Crys}^r(X),\hspace{5mm} \IC^!(\mathcal{L}X_{\Hod})_{\{0\}}\simeq \IC(\mathcal{L}X),
\end{equation}
as well as $\Q(\mathcal{L}X_{\Hod})_{\{1\}}\simeq\Q(X_{\DR})\simeq \C^{\ell}(X),$ and $\Q(\mathcal{L}X_{\Hod})_{\{0\}}\simeq \Q(\mathcal{L}X).$

Insertion of generic and special fibers functors are given by
$$
\mathbf{un}(-):=\IC^!(\mathcal{L}X_{\Hod})\rightarrow \IC^!(X_{\DR}),
\hspace{2mm}
\mathbf{Gr}(-):\IC^!(\mathcal{L}X_{\Hod})\rightarrow \IC^{gr}(\mathcal{L}X).$$

\begin{proposition}
\label{SymL_X-module Proposition}
    The $\infty$-functor $Gr$ factors through $\mathrm{Sym}_{\mathcal{O}_X}^*(\mathbb{L}_X[1])-\IC^!(X).$
\end{proposition}
\begin{proof}
It suffices to notice that there exists a fully-faithful $\infty$-functor
$$\mathrm{Sym}_{\mathcal{O}_X}^*(\mathbb{L}_X[1])-\IC^{gr}(X)\rightarrow \IC^{gr}\big(\mathbb{V}(\mathbb{T}_X[-1])\big),$$
which coincides under (\ref{eqn: exp}) with $\IC^{gr}(\mathcal{L}X).$
\end{proof}
By composition of these functors with the graded Koszul dual functor
$$\IC(\mathbb{V}\big(\mathcal{E}^{\bullet})\big)\xrightarrow{\simeq}\mathsf{Pro}\big(\mathsf{Coh}(\mathbb{V}(\mathcal{E}^{\bullet})^{\kappa})\big),$$
which in our case is given by 
$\IC\big(\mathbb{V}(\mathbb{T}_{X}[-1])\big)\xrightarrow{\simeq}\Q\big(\mathsf{T}_X^*\big),$
since the Tate-shearing of the Koszul dual formal prestack (c.f. \ref{eqn: FMP})
$\mathsf{Spec}_{/X}^{\mathsf{inf}}(\mathrm{Sym}_X^*(\mathbb{L}_X[1]\big),$
is indeed the total cotangent stack of $X$,
\begin{eqnarray*}
\mathsf{T}_X^*&=&\mathsf{T}^*[0]X\simeq\mathbb{V}_X\big(\mathbb{L}_X[0]\big)
\\
&\simeq& \mathsf{Spec}_{/X}^{\mathsf{inf}}\big(\mathrm{Sym}_X^*(\mathbb{L}_X)[2][-2]\big)
\\
&\simeq&
\mathsf{Spec}_{/X}^{\mathsf{inf}}\big(\mathrm{Sym}_X^*(\mathbb{L}_X)\big).
\end{eqnarray*}
Denote the mapping space of sections of a perfect complex $\mathcal{E}^{\bullet}$ underlying a formal prestack $\mathbb{V}_X(\mathcal{E}^{\bullet})$ by 
$\mathsf{Maps}_{\Q(X)}(\mathcal{O}_X,\mathcal{E}^{\bullet})\simeq\Gamma\big(X,\mathbb{V}_X(\mathcal{E}^{\bullet})\big).$

\begin{example}
\label{sssec: Longer Example}
\normalfont Consider \autoref{Free Propagation Example}. 
We recall that Tate shearing $\mathfrak{t}(-):=(-)[\![-2]\!]$ puts complexes in degree zero in weight-degree $(k,2k).$ Moreover, we have
$$|u|=|\mathfrak{t}(t)|=(2,1),\hspace{5mm} |t|=|\mathfrak{t}^{-1}(u)|=(0,1),\hspace{3mm} |u|=2.$$

If a space lies in cohomological degree $-n$, an application of $t^n$ places it in degree $0$, while applying $u^n$ will place it in degree $+2n.$ \autoref{KoszulDualStacks} exchanges $\mathcal{D}_X^{fil}$ with $\mathcal{O}_{X\subset \mathcal{L}X},$ as in (\ref{eqn: D-fil}) where it is explicitly given by $\mathcal{D}_X^{fil}=\bigoplus_{k\geq 0}\mathcal{D}_X^{\leq k}\cdot t^k,$ with Rees parameter $t$ of weight $1.$
Pass to the Rees construction and implement a shift by the order of $P$, say $k$, gives
$$\mathcal{A}^{fil}=\mathrm{Sym}^*\big(Cone(\mathcal{D}_X^{fil}[k]\rightarrow \mathcal{D}_X^{fil})\big).$$
Under assumptions (\ref{Assumptions}), the equivalences in \autoref{Free Propagation Example}
induce equivalences of $\D$-spaces of solutions 
by adjunction (\ref{eqn: Free-forget AD-Mod/Alg Adjunction}). In particular we have $\mathcal{M}_P^{fil}\simeq Cone(\mathcal{D}_X^{fil}[k]\rightarrow \mathcal{D}_X^{fil})$ which by Tate shearing and \autoref{KoszulDualStacks} gives a sheaf on $\mathcal{L}X,$
$$Cone\big(\mathcal{O}_{X\subset\mathcal{L}X}[-2k]\rightarrow \mathcal{O}_{X\subset\mathcal{L}X}\big).$$
In the derived category of sheaves on $\mathcal{L}X$ taking a free resolution by $\Omega_X^{-\bullet}$-modules implies operator $P$ determines 
$P^{fil}[k]$ in
$$\mathcal{E}xt_{\Omega_X^{-\bullet}}^{2k}(\mathcal{O}_X,\mathcal{O}_X)\simeq \mathbb{R}\mathcal{H}om_{\Omega_{X}^{-\bullet}}(\mathcal{O}_{X\subset\mathcal{L}X},\mathcal{O}_{X\subset \mathcal{L}X}[-2k]),$$ computed 
via a Spencer resolution of constant loops as sn $(\Omega_X^{-\bullet},\mathcal{O}_X)$-bimodule algebra. We end up with $P^{fil}[k]_{\kappa}$ as a shifted morphism $\mathrm{Sym}(\Omega_X^1)\rightarrow \mathcal{O}_{X\subset\mathcal{L}X},$ (a shifted symbol map). For instance, take $X=\mathbb{A}^1$ so $\Omega_{\mathbb{A}^1}^{-\bullet}\simeq \Omega_{\A}^1[-1]$ and the resolution is the surjection $\Omega_{\mathbb{A}^1}^1[-1]\rightarrow \mathcal{O}.$
\end{example}
Combining \autoref{Free Propagation Example} with \autoref{sssec: Longer Example}, we see that passing from free $\D_X$-algebras to loop stacks remembers a shift of the symbol of the operator, and via the assumptions (\ref{Assumptions}) in \autoref{Free Propagation Example}, propagation criteria is translated to a (non)-vanishing condition on the shifted symbol (c.f. the description of $\mathrm{Char}(P)$ in \autoref{Char Zero Locus}).

\subsubsection{Linearization}
Consider a generalized PDE $\EQ\subset q_{\DR,*}E$, where $E$ is laft admitting deformation theory relative to a derived scheme $X$. Suppose $\EQ$ is $\D$-finitary. By Proposition \ref{Derived Analytic Relative Linearization Complex}, we have a functor
$$\mathsf{IndCoh}\big(\EQ\times_X X_{\DR}\big)\rightarrow \mathsf{IndCoh}\big(\J(E)\times_X X_{\DR}\big),$$
by the induced morphism from our $\D$-algebraic PDE acting as the inclusion of the critical points (i.e. solution locus) into the ambient space of derived jets.
Adapting diagrams (\ref{eqn: LCrysDiagram}) and (\ref{eqn: RCrysDiagram}), to our situation, and pulling-back along a chosen classical solution $\varphi$ gives
\begin{equation}
    \label{eqn: ICPDEandHodge}
    \adjustbox{scale=.90}{
    \begin{tikzcd}
        & \arrow[dl] \scalemath{.90}{\IC(\EQ\times_X X_{\Hod})}\arrow[d] \arrow[dr] & 
        \\
    \scalemath{.90}{\IC(\EQ)\otimes_{\Q(X)}\IC(X_{\DR})} \arrow[d,"(\varphi\times id)^*"] & \arrow[dl] \scalemath{.90}{\IC(X_{\Hod})}\arrow[dr] & \scalemath{.90}{\IC(\EQ)\otimes_{\IC(X)}\IC(\mathsf{T}^*X/\mathbb{G}_m)}\arrow[d,"(\varphi\times id)^*"]
    \\
    \scalemath{.90}{\IC(X_{\DR})} & & \scalemath{.90}{\IC(\mathsf{T}^*X/\mathbb{G}_m)}.
    \end{tikzcd}}
\end{equation}
Under these considerations with our finiteness assumptions on $\EQ,$ we have one hand the functor
$$\IC(\EQ\times_X X_{\DR}\big)\rightarrow \IC(\EQ\times_X X_{\Hod})\simeq \IC(\EQ)\otimes \IC(X_{\Hod}),
$$
while on the other hand since $\EQ$ is $\D$-finitary, 
by pull-back along the fixed classical $U$-solution $\varphi_U$, as in §\ref{ssec: The pull-back along solutions}, we get a functor
\begin{equation*}
\IC(U\times X_{\DR})\simeq \IC(U)\otimes \IC(X_{\DR})
\simeq \Q(U)\otimes \Q(X_{\DR})\simeq \Q(U\times X_{\DR}).
\end{equation*}
The diagram (\ref{eqn: ICPDEandHodge}) supplies a functor
$$\IC(\EQ)\otimes \IC(X_{\Hod})\rightarrow \IC(X_{\Hod}),$$
and the image of the $\D$ cotangent complex under this sequence of functors, as in \autoref{Linearization D-Module Proposition}, is to be denoted by 
$$\mathsf{T}(\EQ)_{\varphi^{\mathrm{cl}}}^{fil}\in \IC(X_{\Hod}).$$

Since compact objects of $\IC(U\times X_{\DR})$ are bounded, it follows they are coherent and their lift to $\IC(U\times X_{\Hod})$ via $Id_U\otimes 1^*,$ with $Id_U:\IC(U)\rightarrow \IC(U)$ are also bounded.
The notation $(-)^{fil}$ emphasizes this is object lives over $X_{\Hod},$ thus corresponds roughly, to a filtered $\D$-module. 

Since coherent objects in $\IC(X_{\DR})$ are spanned by $\mathrm{ind}(\mathcal{M})$ for $\mathcal{M}\in \mathsf{Coh}(X),$ one has that $\IC(X_{\Hod})$ has the same property via $\mathrm{ind}^{fil}.$ 

If we are in the affine case of $\D$-geometry, by \autoref{Main Theorem: Comparison}, we would have objects of finite presentation admitting good filtrations as objects of the form $\mathcal{M}^{fil}$ over $X_{\Hod}.$  

Consider now the essential image of the functor $Gr$ of passing to the special fiber of the Hodge stack of $X$, and the image of the linearization complex under which we denote by
\begin{equation}
    \label{eqn: Derived Microlocalization}
\Rmu(\mathsf{T}(\EQ)_{\varphi^{\mathrm{cl}}}):=Gr\big(\mathsf{T}(\EQ)_{\varphi^{\mathrm{cl}}}^{fil}\big).
\end{equation}
In particular, for the `empty' equation $id:\J(E)\rightarrow \J(E),$ we have an object $$\Rmu\big(\mathsf{T}(\J(E))_{\varphi^{\mathrm{cl}}})\in \mathsf{Perf}(\mathsf{T}^*X/\mathbb{G}_m).$$
\begin{definition}
\normalfont The \emph{derived microlocalization}
of the linearization complex in \autoref{Linearization D-Module Proposition} is given by the support 
of the complex (\ref{eqn: Derived Microlocalization}). Denote it by 
$Ch(\mathsf{Jet}_{\DR}^{\infty}(E))$, 
and call it the \emph{characteristic variety} of $p_{dR*}E\in \PS_{/X_{\DR}}.$
\end{definition}
Formation of derived microlocalizations respects the functoriality of relative tangent $\D$-complexes.
\begin{proposition}
\label{InducedFilteredTangents}
Consider \autoref{Derived Analytic Relative Linearization Complex} and a $\D$-finitary derived algebraic non-linear \textsc{pde} $\EQ$ over a derived scheme $X$, locally of finite presentation. Then there is an induced morphism of derived microlocalizations in the $\infty$-category graded sheaves on $\mathsf{T}^*X$ 
$$\alpha^{fil}:\hspace{1mm}\Rmu(\mathsf{T}(\EQ)_{\varphi^{\mathrm{cl}}})\rightarrow \hspace{1mm} \Rmu\big(i^*\mathsf{T}(\J(E))_{\varphi^{\mathrm{cl}}}).$$
\end{proposition}
\autoref{InducedFilteredTangents} is familiar when our derived objects are merely classical.
\begin{corollary}
If $\EQ$ is an affine derived $\D$-space which is moreover classical as a derived prestack. Then $\alpha^{fil}$ induces a morphism
$$H_{\D}^0(\alpha^{fil}):\mu\big(\mathcal{H}_{\mathcal{D}}^0(\EuScript{T}_{\EQ,\varphi}^{\ell})\rightarrow \mu\big(H_{\D}^0(i^*\EuScript{T}_{\J(E),\varphi}^{\ell}\big),$$
which agrees under \autoref{Higher Symmetry Result} with the universal linearization of \autoref{Linearization Lemma}.
\end{corollary}

\subsection{$\D$-Fredholmness}
\label{ssec: D-Fredholm}
The theory of linear Fredholm operators is often used to study
nonlinear elliptic problems. Nonlinear operators are called Fredholm operators if
the corresponding linearized operators satisfy this property. It is translated into the setting of derived geometry as follows.

\begin{definition}
\label{Definition: D-Fredholm}
\normalfont Suppose that $X$ is a derived scheme which is locally
of finite presentation, quasi-compact and quasi-separated. Let $\EQ\rightarrow \J(E)$ be$\D$-finitary. It is said to be \emph{$\D$-Fredholm} if, for every solution $\varphi$ and any factorization of the canonical map $a:X\rightarrow Spec(k)$ into an open embedding $j:X\hookrightarrow \overline{X},$ followed by a proper and quasi-smooth map, the object
$j_*^{\IC}\big(\varphi^*\mathbb{T}_{\EQ}^{\ell}\big)^{fil}$ is a compact object of $\mathsf{IndCoh}(\overline{X}_{\Hod}).$ 
\end{definition}
Suppose that $\mathcal{E}_1^{\bullet},\mathcal{E}_2^{\bullet}$ are perfect complexes on $X$ such that we have an equivalence
$$\varphi^*\mathbb{T}_{\EQ}^{\ell}\simeq Cone\big(q_{\DR,*}^{\IC}\mathcal{E}_1^{\bullet}\xrightarrow{\mathsf{P}}q_{\DR,*}^{\IC}\mathcal{E}_2^{\bullet}\big)\in \IC(X_{\DR}).$$
In other words, $\mathsf{P}$ is a morphism of crystals thought of as a linear differential operator (\ref{eqn: Derived Differential Operator}) from $\mathcal{E}_1^{\bullet}$ to $\mathcal{E}_2^{\bullet},$
as in \autoref{Derived Differential Operator Proposition}, since by
adjunction $\mathsf{P}$ is equivalently
$$\mathsf{P}:\mathcal{E}_1^{\bullet}\rightarrow p^!q_{\DR,*}\mathcal{E}_2^{\bullet}\simeq \mathbf{obvl}\circ \mathrm{ind}(\mathcal{E}_2^{\bullet}).$$
\begin{remark}
\normalfont
By \cite{GR14} the action of the $\infty$-monad $\mathbf{obvl}\circ ind$ is tensoring with the algebra of differential operators on $X$. This is literally the case when such an object 
$\mathcal{D}_X\in \mathsf{AssocAlg}\big(\mathsf{QCoh}(X\times X)\big),$
exists as a sheaf supported on the diagonal.
\end{remark}
Since $\mathcal{E}_1^{\bullet}$ is perfect (it is compact in $\mathsf{QCoh}(X)$) there is a factorization of $\mathsf{P}$ along the $k$-th formal neighbourhood of the diagonal, for some $k$ that we will call the order of $\mathsf{P}.$
Again, if there exists an sheaf $\D_X$-defined on $X$, this means $\mathsf{P}:\mathcal{E}_1^{\bullet}\rightarrow \mathcal{D}_X^{\leq k}\otimes\mathcal{E}_2^{\bullet}.$

Suppose that there exists a lift of the morphism $\mathsf{P}$ to the Hodge stack of $X$ i.e. we have a corresponding morphism of filtered crystals
$$\mathsf{P}^{\Hod}:p_{\Hod,*}^{\IC}\mathcal{E}_1^{\bullet}\rightarrow p_{\Hod.*}^{\IC}\mathcal{E}_2^{\bullet},$$
which degenerates to $\mathsf{P}$ under the passage to the generic fiber.
Similarly reasoning gives, under suitable finiteness hypothesis on $X$, the existence of some object $\mathcal{D}_X^{fil}$ in the $\infty$-category 
\begin{equation}
\label{eqn: D-fil} \scalemath{.95}{\mathsf{AssocAlg}\big(\Q(X\times X)^{fil}\big)=\mathsf{AssocAlg}\big(\Q(X\times\AG\times_{\AG}X\times\AG)\big)},
\end{equation}
supported on the diagonal of $X.$ Classically, $\mathcal{D}_X^{fil}$ is identified with $Rees(\D_X).$

Generally, we have an associated morphism 
$$\mathsf{P}^{\Hod}(-k):p_{\Hod,*}^{\IC}\mathcal{E}_1^{\bullet}(-k)\rightarrow p_{\Hod*}^{\IC}\mathcal{E}_2^{\bullet}.$$

Notice that 
$$Cone\big(\mathsf{P}^{\Hod}(-k):p_{\Hod,*}^{\IC}\mathcal{E}_1^{\bullet}(-k)\rightarrow p_{\Hod*}^{\IC}\mathcal{E}_2^{\bullet}\big),$$
defines a coherent filtration on $\varphi^*\mathbb{T}_{\EQ}^{\ell}$ in the sense that 
$$\scalemath{.90}{\mathbf{Un}\big(Cone\big(\mathsf{P}^{\Hod}(-k):p_{\Hod,*}^{\IC}(\mathcal{E}_1^{\bullet}(-k)\rightarrow p_{\Hod*}^{\IC}\mathcal{E}_2^{\bullet}\big)\big)\simeq Cone\big(q_{\DR,*}^{\IC}\mathcal{E}_1^{\bullet}\xrightarrow{\mathsf{P}}q_{\DR,*}^{\IC}\mathcal{E}_2^{\bullet}\big)},$$
is an equivalence.

Letting $j:X\hookrightarrow \overline{X}$ be as before and observe that being $\D$-Fredholm implies that 
$j_*\big(\varphi^*\mathbb{T}_{\EQ}^{\ell}\big)\simeq j_*\big(Cone(\mathsf{P})\big),$
is compact.
Consider the canonical map $a:X\rightarrow Spec(k)$, as well as
$\overline{a}:\overline{X}\rightarrow Spec(k),$ and its factorization through the chosen quasi-smooth compactification $j$. There are equivalences
$$
a_*\big(\varphi^*\mathbb{T}_{\EQ}^{\ell}\big)\simeq\overline{a}_!\big(j_*(\varphi^*\mathbb{T}_{\EQ}^{\ell})\big)
\simeq\overline{a}_*\big(j_*(Cone(\mathsf{P})\big)
\simeq Cone\big((\overline{a}_*(\mathsf{P})\big).$$
The latter complex is a compact object in $\mathsf{Vect}_{\mathbf{k}}$ given by the cone of 
$\overline{a}_*(\mathsf{P}):\overline{a}_*(\mathcal{E}_1^{\bullet}\otimes \omega_X)[d]\rightarrow \overline{a}_*(\mathcal{E}_2^{\bullet}\otimes\omega_X)[d].$
\begin{proposition}
\label{Derived Microlocal Proposition}
Let $f:X\rightarrow Y$ be a morphism of quasi-smooth derived stacks where $\EQ$ is $\D$-Fredholm over $Y$.
There is a micro-local diagram (c.f. (\ref{eqn: Derived Microlocal Diagram})) for cotangent derived stacks, 
$$\mathsf{T}^*X\xleftarrow{q}\mathsf{T}^*Y\times_Y X\xrightarrow{p}\mathsf{T}^*Y,$$
such that the morphism $p$ is representable, proper and quasi smooth. In particular, $Gr(\EuScript{T}_{\EQ})^{fil}$ defines a sheaf on $\EQ\times \mathsf{T}^*Y.$
\end{proposition}
By Proposition \ref{Derived Microlocal Proposition}, we have that there exists a left-adjoint $p_!$ to the functor of quasi-coherent pullback $p_{\mathsf{QC}}^*.$
Denoting the projection $\pi:\mathsf{T}^*Y\times_YX\rightarrow X,$ set the relative dimension of $f:X\rightarrow Y$ to be $d_{X/Y}.$
Then for any $\mathcal{E}^{\bullet}\in \mathsf{Perf}\big(\mathsf{T}^*Y\times_Y X),$ one may compute the $!$-push-forward of perfect complexes
$p_!^{\mathsf{Perf}}(\mathcal{E}^{\bullet})\simeq p_*^{\mathsf{QC}}(\mathcal{E}^{\bullet}\otimes \pi^*\omega_{X/Y}[d_{X/Y}]).$

Assume that $\varphi^*\mathbb{T}_{\EQ}^{\ell}$ has a lift to the Hodge stack, that we denote by $\mathbb{T}_{\EQ,\varphi}^{\ell,\Hod}$. Recall that this means there is an equivalence of objects by taking the underlying object $\infty$-functor $\mathbf{un}$ of a filtration,
$$\varphi^*\mathbb{T}_{\EQ}^{\ell}\simeq \mathbf{un}\big(\mathbb{T}_{\EQ,\varphi}^{\ell,\Hod}\big).$$

By our assumptions on $\EQ$ we have that pull-back along the canonical map $0:B\mathbb{G}_m\simeq \{0\}/\mathbb{G}_m\rightarrow \mathbb{A}^1/\mathbb{G}_m$ induces an associated graded $\infty$-functor $Gr$ and whats more, we have that $Gr\big(\mathbb{T}_{\EQ,\varphi}^{\ell,\Hod}\big)\in \mathsf{Perf}\big(\mathsf{T}^*X\big).$
In other words, it is perfect as a graded $\mathrm{Sym}_{\mathcal{O}_X}^*(\mathbb{L}_X)$-module, by \autoref{SymL_X-module Proposition}.

\begin{proposition}
\label{MainProposition}
Consider Proposition \ref{Derived Microlocal Proposition}.
For a quasi-smooth compactification $j:X\hookrightarrow \overline{X}$, we have that $j_*^{\mathsf{Perf}}\big(Gr(\mathbb{T}_{\EQ,\varphi}^{\ell,\Hod}\big
)\in \mathsf{Perf}(\mathsf{T}^*\overline{X}),$ and its characteristic variety along the infinite jet of a solution $j_{\infty}(\varphi)$, is the support of its microlocalization,
$$Ch(\EQ,j_{\infty}(\varphi)):=supp\big(Gr\mathbb{T}_{\EQ,\varphi}^{\ell,\Hod}\big)\subset \mathsf{T}^*X.$$
This defines a closed set in $\mathsf{T}^*\overline{X},$ with 
$q^{-1}\big(Ch_{\varphi}(\EQ)\big)$ closed in $\mathsf{T}^*X\times_Y X.$
\end{proposition}
Proposition \ref{MainProposition} can be used to introduce non-characteristic conditions in this setting. Indeed, $f:Z\rightarrow X$ is non-characteristic for $\EQ$ if $q$ is proper on the closed subspace $p^{-1}\big(Ch_{\varphi}(\EQ)\big).$

Now, using the natural cofiber sequence of functors $\mathrm{ind}^{\leq k-1}\rightarrow \mathrm{ind}^{\leq k}$ which gives a functor of tensoring with $\mathrm{Sym}_{\mathcal{O}_X}^{k}(\mathbb{T}_X)$, we obtain an extension of $\mathsf{P}$ above to a morphism denoted
$$Gr_k(\mathsf{P}):\mathcal{E}_1^{\bullet}\rightarrow \mathrm{Sym}_{\mathcal{O}_X}^{k}(\mathbb{T}_X)\otimes\mathcal{E}_2^{\bullet}.$$

This morphism is extended by the above mention deformation to a morphism of graded $\mathrm{Sym}_{\mathcal{O}_X}^*(\mathbb{T}_X)$-modules over $X.$ Denote the associated monad of tensoring with this symmetric algebra by $\eta^*\eta_*(-).$
We induced a morphism 
$$\eta^*\eta_*Gr_k(\mathsf{P}):\eta^*\eta_*\mathcal{E}_1^{\bullet}\rightarrow \eta^*\eta_*\mathcal{E}_2^{\bullet}.$$
\begin{proposition}
\label{prop: Perfect module on cotangent stack}
The push-forward $j_*Cone\big(\eta^*\eta_*Gr_k(\mathsf{P}))$ is a perfect module over the cotangent stack $\mathsf{T}^*\overline{X}$.
\end{proposition}
\begin{proof}
    It follows from noticing that $Cone(\eta^*\eta_*Gr_k(\mathsf{P}))$ is compact object in sheaves on $\mathsf{T}^*X$ and from the assumption of $\D$-Fredholmness on $\EQ.$
\end{proof}

\begin{theorem}
\label{thm: Main Theorem}
Suppose that $X$ is a derived scheme with a quasi-smooth compactification $j:X\hookrightarrow \overline{X}.$ 
Suppose that $\EQ$ is a generalized \textsc{pde} over $X$ which is $\D$-finitary (in particular, $\D$-Fredholm).
Then, the derived micro-linearization module 
$\Rmu(\mathsf{T}(\EQ)_{\varphi^{\mathrm{cl}}})$ is a perfect complex on $\mathsf{T}^*X$, for every solution $\varphi^{cl},$ and pull-back along the zero-section $s$ and push-forward along $a:X\rightarrow Spec(k),$ gives a perfect object
$$\mathbb{R}a_*\big(s^*\hspace{1mm} \Rmu(\mathsf{T}(\EQ))_{\varphi^{\mathrm{cl}}})\otimes \omega_X)[d]\in \mathsf{Vect}_k.$$
\end{theorem}
\begin{proof}
    From Proposition \ref{prop: Perfect module on cotangent stack} note that by considering the zero section morphism 
$s:X\rightarrow \mathsf{T}^*X$ and the pull-back along with the push-forward along $X\rightarrow Spec(k)$ i.e.
$$\mathsf{Perf}(\mathsf{T}^*X)\xrightarrow{s^*}\mathsf{Perf}(X)\xrightarrow{a_*}\mathsf{Perf}(\mathbf{k}),$$
it sends 
\begin{equation}
\label{eqn: FredholmSymbol}
Cone\big(\eta^*\eta_*Gr_k(\mathsf{P})\big)\mapsto \Gamma\big(X;s^*Cone\big(\eta^*\eta_*Gr_k(\mathsf{P})\big)\otimes \omega_X\big)[d_X],
\end{equation}
as required.
\end{proof}
Use Theorem \ref{thm: Main Theorem} and Theorem \ref{thm: RSol is laft-def} to obtain similar statements for solution stacks. For a given $U$-parameterized solution $\varphi_U$ as in (\ref{ssec: The pull-back along solutions}), consider the induced projection map $q_U:U\times X_{\DR}\rightarrow U\times \{pt\}\simeq U,$ and the corresponding functor $q_{U*}$.

From universal property of pro-cotangent and their dual ind-coherent tangent complexes associated to mapping prestacks in $\PS_{/X_{\DR}}$ one has an identification of the ind-coherent sheaf $\mathbb{T}[\RS(\EQ)]_{\varphi}$ with $ (q_U)_*\EuScript{T}(\EQ)_{\varphi}$ (see Remark \ref{RSol Tangent Remark}) and Proposition \ref{T and D-T equivalence}.

\begin{corollary}
\label{Main Theorem Corollary}
Consider $\EQ$ before and assume that $X$ is compact. Then 
$$\mathsf{SS}\big(\mathbb{T}[\RS(\EQ)]_{\varphi}\big)\subset \{0\}_X.$$
\end{corollary}

Notice that if $X=\mathcal{L}Z$ with $Z$ smooth, then the Koszul dual description as in § \ref{sec: Derived Linearization and the Equivariant Loop Stack} together with the interpretation of $\mathsf{SS}$ given in §§
\ref{sssec: Interpreting SS via Prop} allows us to, roughly speaking, view this as a statement on the micro-support in terms of $T^*Z$, since $\mathsf{Sing}(\mathcal{L}Z)\simeq T^*Z.$

Corollary \ref{Main Theorem Corollary} can be considered in the wider context of non-microcharacteristic morphisms as in §§ \ref{sssec: Singular Support Conditions}. In particular, if $f:Z\rightarrow X$ is a quasi-smooth morphism with $V\subset \mathsf{Sing}(X),$ then it is said to be non-microcharacteristic for $\EQ$ along $V$ if (c.f. (\ref{eqn: SS Non-micro-char})),
$$f_d\circ f_{\pi}^{-1}\big[\mathsf{SS}\big(\hspace{1mm} s^* \hspace{.5mm} \Rmu(\EuScript{T}_{\EuScript{Y},\varphi})\big)\cap V\big]|_{\mathsf{Sing}(\EQ)}\subset \{0\}_Z,$$
via the diagram (\ref{eqn: Derived Microlocal Diagram}). Discussing this and its relation to ind-coherent sheaf propagation for non-linear \textsc{pde} with prescribed support conditions e.g. $\mathsf{SS}\subset \Lambda,$ will appear in a future work.

\appendix

\section{Higher sheaf theory}
We collect our terminology and basic facts regarding what it means have a theory of sheaves on a given $\infty$-site $(\mathsf{C},\tau)$ equipped with a class of $1$-morphisms  with specified property $P,$ satisfying some natural conditions (what defines a $(\infty,n)$-geometric context, but we dont need this generality here).

It is understood as lax symmetric monoidal $\infty$-functor of $(\infty,2)$-categories 
$$\underline{\mathsf{Shv}}(-):\mathsf{Corr}\big(\mathsf{Stk}(\mathsf{C};\tau)\big)\rightarrow (\infty,2)\mathsf{Cat}_{\mathsf{pres,cont}}^{\mathsf{st}},$$
whose $(\infty,1)$-categorical truncation is $\mathsf{Shv}^{Cat}.$ The use of correspondences and functoriality of $\underline{\mathsf{Shv}}$ naturally encodes base-change for $f_*^{\mathsf{Shv}}$ and $f_{\mathsf{Shv}}^!.$ In other words, there exists adjoint pairs of functors 
$$\big(F_!^{\mathsf{Shv}},F_{\mathsf{Shv}}^!\big),\hspace{5mm} \big(F_{\mathsf{Shv}}^*,F_*^{\mathsf{Shv}}\big),$$
together with the usual monoidal structures. 

One considers variants
$\mathrm{Corr}\big(\PS_{\mathbf{Q}}^{\mathbf{P}}\big)$ and $\mathrm{Corr}\big(\mathrm{Stk}(\mathcal{C};\tau;\mathbf{p})_{\mathbf{Q}}^{\mathbf{P}}\big),$
where $\mathbf{P}$ are a class of `properties' on morphisms defining a colimit representation $\EuScript{Y}$ of our objects, and $\mathbf{Q}$ are `properties' on the objects $\EuScript{Y}$ themselves.

\begin{example}
\normalfont A typical example is: $\mathbf{P}$ taken to be ``closed-embedding'', and $\mathbf{Q}$ to be laft. Then we have that
$\mathrm{Corr}\big(\PS_{\mathrm{laft}}^{\mathbf{cl-embed}}\big),$ is the $(\infty,2)$-category of correspondences in pseudo-ind-schemes \cite{GR17b}.
\end{example}
We do not use correspondences in this work, so consider $\mathsf{Shv}$ defined on $\PS$ or $\mathsf{Stk}.$ The main theories we consider are $\underline{\Q}^*,\underline{\IC}^!,\underline{\mathsf{Crys}}^{r,!}$ or $\underline{\mathsf{Crys}}^{\ell,!},$ referring to \cite{GR17a,GR17b} and \cite{GR14}.

In practice, one defines a theory of sheaves on prestacks by first defining it over affines and Kan-extending along the Yoneda embedding, further restricting to those with certain properties $\mathbf{P},\mathbf{Q},$ if needed. That is,
\begin{equation}
    \label{eqn: F-Object Functor}
\underline{\mathsf{Shv}}_{\PS_{\mathbf{Q}}^{\mathbf{P}}}^!:=\mathrm{Res}\big(\mathrm{Kan}_{j}^{\mathrm{R}}(\underline{\mathsf{Shv}}_{\mathsf{Aff}}^!)\big):\big(\PS_{\mathbf{Q}}^{\mathbf{P}}\big)^{\mathrm{op}}\rightarrow \mathsf{Cat}_{\mathsf{pres},\mathsf{cont}}^{\infty,\mathsf{st}},\end{equation}
where $\mathrm{Res}$ is the restriction along the functor $\PS_{\mathbf{Q}}^{\mathbf{P}}\hookrightarrow \PS.$ Right Kan-extension is along the Yoneda, $j$. It can be calculated explicitly: if $\EuScript{Y}\in \PS_{\mathbf{Q}}^{\mathbf{P}}$, represented as $\EuScript{Y}\simeq \mathrm{colim} Z_i$ then
$$\underline{\mathsf{Shv}}^!(\EuScript{Y})\simeq \mathrm{colim} \underline{\mathsf{Shv}}^!(Z_i),$$
taken with respect to the diagram formed
by functors $f_{\mathsf{Shv}}^!.$

\section{$\D$-Prestacks via Chiral Koszul Duality}
\label{sec: D-Prestacks via Chiarl Koszul Duality}
In this Appendix we discuss the main objects in this paper--$\D$-prestacks, defined via $Ran_X$ and looking at commutative chiral algebras supported on the diagonal. In this way, we also produce an important class of formal derived prestacks via Chiral Koszul duality \cite{FraGai}, similar to what is done in \cite{Hinich} which we believe may be of independent interest.

\subsection{Formal $\D_X$-moduli from Chiral Koszul duality}
\label{ssec: Formal D-Moduli from Chiral Koszul Duality}
Let $X$ be a smooth affine $\mathbb{C}$-variety. Chiral Koszul duality \cite{FraGai} exploits a certain natural duality between Lie algebra objects in the symmetric monoidal $(\infty,1)$-category $\mathsf{IndCoh}(RanX_{\DR}),$ with respect to the so-called chiral structure $\otimes^{ch}$ and commutative co-algebras.

This category is a homotopy limit in the $\infty$-category of stable $\infty$-categories of all $\IC(X_{\DR}^I)$ formed by the diagram of functors $\Delta(\pi)_{\IC}^!$ induced by the embeddings $\Delta(\pi):X^J\hookrightarrow X^I$ given for each surjection $\pi:I\rightarrow J$ of finite sets.

Said differently,
\begin{equation}
\label{eqn: IndCoh Ran}
\IC(Ran X_{\DR})=\underset{I\in fSet}{holim}\hspace{1mm} \IC(X_{\DR}^I)\simeq Lax^{\circ}\big(\underline{\IC}_{\underline{Ran}_{X_{\DR}}}^!\big),
\end{equation}
where the functor $\underline{\IC}_{\underline{Ran}_{X_{\DR}}}^!$ in (\ref{eqn: IndCoh Ran}) is defined by the composition $\IC\circ (-)_{\DR}\circ \underline{Ran}_X$ where $\underline{Ran}_X$ is the diagram given for each finite set $I$ by $\underline{Ran}_X(I):=X^I.$ More explicitly,
$$\underline{\IC}_{\underline{Ran}_{X_{\DR}}}^!:fSet^{op}\rightarrow \mathsf{Cat}_{st}^{\infty},\hspace{2mm} I\mapsto \IC(X_{\DR}^I),$$ and sends a morphism $\pi:I\rightarrow J$ of finite sets to the functor $\Delta(\pi)_{\DR}^!$ of ind-coherent pull-back on $\Delta(\pi)_{\DR}:X_{\DR}^J\rightarrow X_{\DR}^I.$

There are compatible functors for each $I$ with left adjoints,
\begin{equation}
    \label{eqn:D Ran pull-back and adjoint functor}
\Delta_I^!:\mathsf{Mod}\big(\mathcal{D}(Ran_X)\big)\rightleftarrows \mathsf{Mod}\big(\D_{X^I}\big):\Delta_{I*}
\end{equation}

When $I=\{\mathrm{pt}\},$ we obtain a functor
$\Delta_*^{\mathrm{main}}:\mathsf{Mod}(\D_X)\rightarrow \mathsf{Mod}\big(\mathcal{D}(Ran_X)\big),$ and objects in the essential image are said to be supported along the main diagonal.
More precisely, for $| I|=1$, the functors (\ref{eqn:D Ran pull-back and adjoint functor}) define inverse equivalences of  $\mathsf{Mod}\big(\mathcal{D}_X\big)$ with $\mathsf{Mod}_{X}\big(\mathcal{D}(Ran_X)\big),$ by Kashiwara's theorem. That is $\D_X$-modules are those $\D$-modules on $RanX$ supported on the main diagonal.

The operation of disjoint union of subsets 
$union:\mathrm{Ran}(X)\times \mathrm{Ran}(X)\rightarrow \mathrm{Ran}(X)$ and of the inclusion of disjoint pairs 
$j_{disj}:\big(\mathrm{Ran}(X)\times \mathrm{Ran}(X)\big)_{disj}\rightarrow \mathrm{Ran}(X)\times \mathrm{Ran}(X),$
induces two monoidal structures,
$$\otimes^{\star}=union_*\circ (-\boxtimes -),\hspace{3mm} \otimes^{ch}:=union_*\circ j_*\circ j^*(-\boxtimes-).$$

There is an identity functor induced by the natural transformation of monoidal structures $\otimes^{\star}\rightarrow \otimes^{ch},$
$$Id:\big(\mathsf{Mod}
(\mathcal{D}(Ran_X),\otimes^{\star})\rightarrow \big(\mathsf{Mod}(\mathcal{D}(Ran_X),\otimes^{ch}\big).$$
It is lax symmetric monoidal. Thus, for any operad $\mathcal{O},$ or cooperad $\mathcal{P}$, there are forgetful $\infty$-functors
$$\mathcal{O}\mathsf{-Alg}^{ch}(Ran_X)\xrightarrow{\mathbf{oblv}_{\mathcal{O}}^{ch\rightarrow \star}}\mathcal{O}\mathsf{-Alg}^{\star}(Ran_X),\hspace{2mm} \mathcal{P}\mathsf{-coAlg}^*(Ran_X)\xrightarrow{\mathbf{oblv}_{\mathcal{P}}^{\star\rightarrow ch}}\mathcal{P}-\mathsf{coAlg}^{ch}(
Ran_X).$$
\begin{example}
  Consider the Lie-operad $\mathcal{O}:=\mathcal{L}ie$, one defines chiral Lie algebras on $X.$ A chiral Lie algebra $\mathcal{C}$ is said to be \emph{commutative} its bracket $[-,-]$ vanishes on $\mathbf{oblv}^{ch\rightarrow \star}(\mathcal{C}).$ 
\end{example}
These objects were introduced and studied in detail for $X$ a smooth projective curve in \cite{BD2004}. There is an equivalence 
\begin{equation}
    \label{eqn: CAlg and CommChiral}
\mathsf{CAlg}_X(\D_X):=\mathsf{CAlg}\big(\mathsf{IndCoh}(X_{\DR}),\otimes^!\big)\simeq \mathsf{LieAlg}^{ch}(X)_{comm}.
\end{equation}

A factorization
structure on $\mathcal{E}$ over $Ran_X$ is the data of compatible equivalences between $\mathcal{E}^I$ and $\boxtimes_{i\in I}\mathcal{E}^{(\{i\})}$ when restricted to the compliment of the big diagonal in $X^I$ for each $I$. 

\begin{proposition}
\label{prop: Chiral Koszul Duality}
Let $X$ be a smooth affine variety. Chiral Koszul duality is a pair of adjoint $\infty$-functors compatible with restriction to main diagonals $X\subset \mathrm{Ran}(X):$
\[
\begin{tikzcd}
CE^{\otimes^{ch}}:\mathsf{ChiralAlg}(Ran_X)\arrow[d,shift left=.7ex]\arrow[r,shift left=.5ex] & \arrow[l,shift left=.7ex]\arrow[d,shift left=.7ex]\mathsf{CommCoAlg}^{ch}(Ran_X):Prim^{ch}[-1]
\\
CE_{X}^{\otimes^{ch}}:\mathsf{ChiralAlg}(X) \arrow[u,shift left=.7ex]\arrow[r,shift left=.7ex] & \mathsf{FAlg}(X)\arrow[l,shift left=.7ex]\arrow[u,shift left=.7ex]:Prim_X^{ch}[-1]
\end{tikzcd}
\]
commuting in compatible manner with all forgetful functors to $\mathsf{IndCoh}(Ran_{X_{\DR}}).$  
\end{proposition}
\begin{proof}
See \cite[Proposition 6.3.6 - 6.4.2]{FraGai}.

\end{proof}

We have set $CE_X^{\otimes^{ch}}:=CE^{\otimes^{ch}}|_{X},$ and similarly for $\mathrm{Prim}_X^{ch}[-1].$

Consider commutative factorization algebras with the natural factorization analogs of dualizing, structure and constant sheaves $\omega_{RanX},\mathcal{O}_{Ran X}$ and $\mathbb{C}_{Ran X}$ and omit reference to $Ran X$.

A morphism of (commutative) factorization algebras $\mathcal{A}\rightarrow \mathcal{B}$ is said to be $\D_X$-\emph{small} if it is a finite composition of morphisms that are $\D_X$-\emph{generating}, where $f:\mathcal{A}\rightarrow\mathcal{B}$ is $\D$-generating if there exists some integer $n$ for which 
\[
\begin{tikzcd}
\mathcal{A}\arrow[d]\arrow[r,"f"] & \mathcal{B}\arrow[d]
    \\
    \mathbb{C}\arrow[r] & \mathbb{C}\oplus \mathbb{C}[n]
\end{tikzcd}
\]
is a pull-back diagram. A factorization $\D$-module is \emph{small} if its canonical augmentation map to $\mathbb{C}$ is $\D_X$-small.
Define a full subcategory spanned by $\D_X$-small factorization algebras $\mathcal{A}$ over $X$, with the additional property that $H_{\D}^0(\mathcal{A})$, its $0$-th cohomology is a nilpotent $\D_X$-algebra as in\autoref{Nilpotent D-Algebras}. Denote these objects by $\mathsf{FAlg}_{\mathcal{D}}^!(X)_{sm}.$

\begin{definition}
    \label{defn: Factorization Deformation Functor}
    \normalfont A functor $F:\mathsf{FAlg}_{\mathcal{D}}^!(X)_{sm}\rightarrow \mathsf{Spc},$ is called a \emph{$\mathcal{D}^!$-factorization moduli functor} if $F(\mathbb{C})$ is contractible and $F$ sends pull-back squares along small morphisms in $\mathsf{FAlg}_{\mathcal{D}}^!(X)_{sm}$ to pull-back squares in $\mathsf{Spc}$.
\end{definition}
\autoref{defn: Factorization Deformation Functor} means for Cartesian diagrams in $\mathsf{FAlg}_{\mathcal{D}}(X),$
\[
\begin{tikzcd}
    \mathcal{A}'\arrow[d]\arrow[r] & \mathcal{B}'\arrow[d,"f"]
    \\
    \mathcal{A}\arrow[r] & \mathcal{B}
\end{tikzcd}
\]
with $f$ a $\D_X$-small morphism, and $H_{\D}^0(\mathcal{A})\rightarrow H_{\D}^0(\mathcal{B})$ a surjection of nilpotent $\D_X$-algebras, 
one has an equivalence $F(\mathcal{A}')\simeq F(\mathcal{A})\times_{F(\mathcal{B})}F(\mathcal{B}').$ 

Denote this functor category by $\mathsf{Moduli}_{\mathcal{D}^!}^{fact}(X).$ 
One can ask when such an object $F$ is representable, leading to factorization analogs of the derived deformation theory philosophy. 

One also introduces \emph{coherent} factorization $\D$-modules, defined by objects in the oplax-limit of the functor
$$\underline{\mathsf{Coh}}_*^{X^{fSet}}:I\in fSet\mapsto \mathsf{Coh}(\D_{X^I}^{op}).$$
In other words, a coherent (lax) $\D$-module is a (lax) module $\EuScript{E}=\big(\EuScript{E}^{(I)}\big)_{I\in fSet}$ such that each $\EuScript{E}^{(I)}$ is a coherent $\D_{X^I}^{\mathrm{op}}$-module, for ever $I\in fSet.$

Consider \autoref{Formal derived D-stack} and introduce for any of the categories mentioned above e.g. $\mathsf{CohFAlg}_{\mathcal{D}}(X)$ or $\mathsf{FAlg}_{\mathcal{D}}(X)$ or their strict variants, together with all $\D$-small $1$-morphisms, the corresponding functor categories, whose functors satisfy \autoref{Formal derived D-stack}.

For instance, put $\mathsf{CohFAlg}_{\mathcal{D}}(X)^{\mathrm{sm}},$ and consider
$$\mathsf{Moduli}_{\mathcal{D}}^{coh}(X)\subseteq \mathsf{Fun}\big(\mathsf{CohFAlg}_{\mathcal{D}}(X)_{sm}^{op},\mathsf{Spc}\big).$$
Moreover let us consider the full-subcategory $\mathsf{CommCohFAlg}$ those coherent lax factorization $\D$-modules which are additionally commutative.

We are interested in classes of representable functors
$$F:\big(\mathsf{CommCohFAlg}_{\mathcal{D}}(X)^{\mathrm{sm}}\big)^{\mathrm{op}}\rightarrow \mathsf{Spc},$$
which, by extension induce representable functors
$\EuScript{Y}:\mathsf{CohFAlg}_{\mathcal{D}}^{\mathrm{sm}}\rightarrow \mathsf{Spc}.$

We call them \emph{formal (coherent) factorization stacks} and by \autoref{prop: Chiral Koszul Duality} they correspond to (co)representable functors
\begin{equation}
\label{eqn: Coherent chiral functors}
\EuScript{Y}^{ch}:\mathsf{CommChiralAlg}_{\mathcal{D}}(X)^{\mathrm{sm}}\rightarrow \mathsf{Spc}.
\end{equation}
\begin{definition}
\label{defn: Coherent chiral functors definition}
\normalfont A co-representable functor of the form (\ref{eqn: Coherent chiral functors}) is a \emph{formal chiral moduli functor}.
\end{definition}
A formal chiral moduli functor is \emph{coherent} if it takes values on coherent chiral algebras.
If $\mathcal{A}$ is a lax chiral commutative coalgebra, the associated formal coherent chiral moduli problem as in \autoref{defn: Coherent chiral functors definition} is (c.f. \ref{eqn: Coherent chiral functors})),
$$\underline{\mathsf{Spf}}_{\mathcal{D}}(\mathcal{A}):\mathsf{CommCoAlg}_{\mathcal{D}}^{ch}(X)\rightarrow \mathsf{Spc},\hspace{1mm}\mathcal{B}\mapsto\underline{\mathsf{Maps}}_{\mathsf{CommCoAlg}}(\mathcal{B},\mathcal{A}).$$

\begin{proposition}
\label{prop: Chiral Moduli Proposition}
Given a representable coherent chiral moduli problem $\EuScript{Y}^{ch}\simeq \underline{\mathsf{Spf}}_{\mathcal{D}}(\mathcal{A}),$ its formal shifted tangent complex is 
$\mathbb{T}_{\EuScript{Y}^{ch}}\simeq LPrim[-1](\mathcal{A}).$ 
\end{proposition}
These discussions amount to equivalences of $\infty$-categories:
$$\Psi_{\mathcal{D}}^{ch}(-):\mathsf{Chiral}_{\mathcal{D}}(X)\rightarrow \mathsf{Moduli}^{\mathsf{FAlg}_{\mathcal{D}}^!(X)}\subseteq\mathsf{Fun}\big(\mathsf{FAlg}_{\mathcal{D}}^!(X)^{\mathrm{aug,small}},\mathsf{Spc}\big),$$
with
$$\Psi_{\mathcal{D}}^{ch}(\mathcal{L})(\mathcal{A}):=\mathrm{Maps}_{\mathsf{Chiral}_{\mathcal{D}}(X)}\big(Prim^{ch}[-1]\mathcal{A},\mathcal{L}\big).$$
Using the left-adjoint funtor $CE^{ch},$ one has
$\Psi:\mathsf{Chiral}(X)\rightarrow \mathrm{Fun}^{(\infty,1)}\big(\mathsf{FAlg}_X^!,\mathsf{Spc}\big),$
given by the composition
$$\mathsf{Chiral}_X\xrightarrow{j^{ch}}\mathsf{Fun}\big(\mathsf{Chiral}(X)^{op},\mathsf{Spc}\big)\xrightarrow{\circ CE}\mathsf{Fun}\big(\mathsf{FAlg}^!(X),\mathsf{Spc}\big),$$
where $j^{ch}$ is the Yoneda embedding for $\mathsf{Chiral}(X)$ i.e. $\mathcal{C}\mapsto j^{ch}(\mathcal{C}):=\mathrm{Hom}_{\mathsf{Chiral}(X)^{op}}(\mathcal{C},-).$

First, we notice that since the functor $\Psi$ commutes with limits and accessible colimits, by the adjoint functor theorem  \cite{Lur17}, there exists a left-adjoint $\Phi:=\Psi^{\mathrm{L}}.$ We need to show there is an equivalence of functors
 $\Phi\circ j^{fact}\simeq (CE^{ch})^{op},$ on the full-subcategory $\big(\mathsf{FAlg}^!(X)^{small}\big)^{\mathrm{op}}.$
 
Fix $\mathcal{B}\in \mathsf{FAlg}^!(X)^{small}$ and $\mathcal{C}\in \mathsf{Chiral}(X).$ Then we have 
\begin{eqnarray*}
\mathrm{Hom}_{\mathsf{Chiral}(X)}\big(\Phi\big(j^{fact}(\mathcal{B})\big),\mathcal{C}\big)&=&\mathrm{Hom}_{\mathrm{Moduli}^{fact}}\big(j^{fact}(\mathcal{B}),\Psi(\mathcal{C})\big)
\\
&=&\mathrm{Hom}_{\mathsf{Fun}(\mathsf{FAlg}^{\mathrm{small}},\mathsf{Spc})}\big(j^{fact}(\mathcal{B}),\Psi(\mathcal{C})\big)
\\
&=&\Psi(\mathcal{C})(\mathcal{B})
\\
&=&\mathrm{Hom}_{\mathsf{Chiral}(X)^{op}}\big(\mathcal{C},CE(\mathcal{B})\big)
\\
&=&\mathrm{Hom}_{\mathsf{Chiral}(X)}\big(CE^{op}(\mathcal{B}),\mathcal{C}\big).
\end{eqnarray*}
Moreover, the equivalence
$
j^{fact}(\mathcal{B})\rightarrow \Psi\circ \Phi\big(j^{fact}(\mathcal{B})\big)
\simeq \Psi\big(CE^{\mathrm{op}}(\mathcal{B})\big),$ is given by
\begin{equation*}
\mathrm{Hom}_{\mathsf{Chiral}(X)^{op}}\big(CE(\mathcal{B}),CE(-)\big)
\simeq \mathrm{Hom}_{\mathsf{FAlg}
(X)}\big(\mathcal{B},LPrim^{ch}[-1]CE^{ch}(-)\big).
\end{equation*}

By \autoref{prop: Chiral Moduli Proposition}, such functors admit an interpretation by uniquely defined (up to homotopy equivalence) coalgebras of distributions at a point.

Consider a commutative derived $\D$-algebra $\mathcal{A},$ together with the natural maps of $n$-fold products $\beta_n:j_*j^*(\mathcal{A}\boxtimes\cdots\boxtimes \mathcal{A})\rightarrow \Delta_*(\mathcal{A}\otimes \cdots\otimes\mathcal{A}),$ it follows that we have fibration sequence by extending $\beta_n$ on the left by $\mathcal{A}\boxtimes\cdots\boxtimes\mathcal{A}.$ 
\begin{proposition}
 \label{prop: D-Prestack Proposition}
There is a functor $\mathsf{CAlg}_{X}(\D_X)\rightarrow \mathsf{FAlg}_X(\D_X)_{comm}$.
\end{proposition}
By \autoref{prop: Chiral Moduli Proposition} and \autoref{prop: D-Prestack Proposition}, there is a well-defined full-subcategory of formal factorization/chiral moduli functors determined by restriction along $\mathsf{CAlg}_X(\D_X)\rightarrow \mathsf{FAlg}_X(\D_X).$
They recover the $\D_X$-prestacks of \autoref{Definition: D-PreStack} and are formal if representable by a co-commutative coalgebra object in $\IC(RanX_{\DR})$.
\begin{example}
    \label{Formal derived D-stack}
\normalfont Suppose that $\mathcal{L}^{Ran}$ is a $\mathrm{Lie}^*$-algebra. Then by \autoref{prop: Chiral Koszul Duality} and \autoref{prop: Chiral Moduli Proposition} there is a corresponding formal derived stack
$\mathcal{Y}_{\mathcal{L}}:=\underline{\mathsf{Spec}}_{\mathcal{D}}\big(CE^{\bullet,\otimes^{\star}}(\mathcal{L}^{\circ}[-1])\big),$
associated to the co-commutative co-algebra $\mathrm{Sym}^{\bullet,\otimes^{\star}}\big(\mathcal{L}^{\circ}[-1]\big),$ with $(-)^{\circ}$ a suitable duality (\ref{Defin: Inner Verdier Duality}).
A morphism of formal derived $\D$-pre stacks $f:\mathcal{Y}_{\mathcal{L}_1}\rightarrow \mathcal{Y}_{\mathcal{L}_2}$ is then understood as a morphism of underlying co-algebras\footnote{This can actually be taken as a definition of morphisms of $\mathsf{Lie}_{\infty}^*$-algebras.}
$$\mathrm{Sym}^{\bullet,\otimes^{\star}}\big(\mathcal{L}_1^{\circ}[-1]\big)\rightarrow \mathrm{Sym}^{\bullet,\otimes^{\star}}\big(\mathcal{L}_2^{\circ}[-1]\big),\hspace{2mm}\text{ such that }d_{CE}^2\circ f=f\circ d_{CE}^1.$$
\end{example}
Moreover, suppose that $Spec_{\D_X}(\mathcal{A}^{\ell})$ is an affine $\D_X$-space and $\mathcal{L}$ is a DG Lie-algebroid object in $\mathcal{A}^{\ell}[\mathcal{D}_X]$-modules. Then one may define $CE(\mathcal{L}):=Sym\big(\mathcal{L}^{\circ}[-1]\big)$ via the dual (\ref{Defin: Inner Verdier Duality}). The associated (formal) derived $\D_X$-prestack is  
$$\mathcal{Y}_{\mathcal{L}}:=\mathsf{Spec}_{\D_X}\big(CE_{\mathcal{D}}(\mathcal{L})\big)=\mathsf{Spec}_{\D_X}\big(Sym_{\mathcal{A}}^*(\mathcal{L}^{\circ}[-1])\big).$$

\section{Examples}
As promised, we give some examples of derived algebraic pre-stacks in $\D$-geometry (the Euler-Lagrange, Koszul-Tate and Batalin-Vilkovisky $\D$-spaces). An extensive study is given in \cite{KSY2}. 

\begin{example}
\label{Derived EL PDE Example}

Forming a natural quotient by the $\D$-algebra of intinite jets by a suitable $\D$-ideal $\mathcal{I}_{EL}$ generated by the usual Euler-Lagrange system one defines the Euler-Lagrange $\D$-prestack $Spec_{\D_X}(\mathcal{A}_{EL})$. If $\mathcal{N}_{\mathcal{S}}=\emptyset,$ is a corresponding empty space of Noether symmetries, a cofibrant model for this space is $\mathrm{Sym}_{\mathcal{A}}(\Theta_{\mathcal{A}}[1]).$ The derived Euler-Lagrange prestack is
$\mathbb{R}\mathbf{Sol}_{\D_X}(\mathcal{A}_{EL})=\mathbb{R}\mathsf{Spec}_{\D}\big(\mathrm{Sym}_{\mathcal{A}}(\Theta_{\mathcal{A}}[1])\big),$ which models polynomial functions on the cotangent space $\mathsf{T}^*[-1]\mathsf{Spec}_{\D}(\mathcal{A}).$ In this situation, 
$\pi_0(\mathbb{R}\mathbf{Sol}_{\mathcal{D}}(\mathcal{A}_{EL})\big)\simeq \mathsf{Sol}(\mathcal{A}_{EL})$. 
\end{example}
Example 
\ref{Derived EL PDE Example} is an instance of a more general construction: given a quotient $\D$-algebra defined by an equation ideal 
$$\mathcal{O}(\mathrm{Jet}^{\infty}(E))\hookrightarrow \underbrace{\mathcal{O}(\mathrm{Jet}^{\infty}(E))\otimes_{\mathcal{O}_X}\mathrm{Sym}_{\mathcal{O}}^*(\mathcal{V}^{\bullet})}_{\mathcal{A}_{KT}}\xrightarrow{\simeq}\mathcal{O}(\mathrm{Jet}^{\infty}(E))/\big<\mathsf{P}\big>,\hspace{1mm} \mathsf{P}\in \mathcal{O}(\mathrm{Jet}^{\infty}E),$$
this resolution defines a derived $\D_X$-scheme
$$\mathcal{X}_{\mathrm{KT}}:=\mathsf{Spec}_{\D_X}(\mathcal{A}_{\mathrm{KT}}):\cdga_{\D_X,\mathcal{A}/}\rightarrow SSets.$$

It may be extended to the BV-formalism in the presence of gauge symmetries i.e. $\mathcal{N}_{\mathcal{S}}\neq 0.$ In this case, there may be non-trivial higher homotopies (non-trivial Noether identities\footnote{There are no non-trivial Noether identities if 
    $H_{\D}^k(\mathcal{A}_{\mathrm{EL}})=0,\hspace{2mm} k\geq 1.$}) e.g.  $H_{\D}^1\big(\mathcal{A}_{EL}\big)\cong \pi_1\big(Spec_{\mathcal{D}}(\mathcal{A}_{EL})\big).$
The following is treated in greater detail in \cite{KSY2}.
\begin{example}
\label{BV-Example}
Considering the maximal Lie algebroid of Noether symmetries $\mathcal{N}_{\mathcal{S}}$ the output of the BV-quantization \cite{CG16,CG16b}, for $\D$-spaces is roughly a derived BV $\D$-space \cite{Paugam2014} given by a quotient stack
$$\mathcal{X}_{BV}=\mathbb{R}\mathsf{Spec}_{\D}(\mathcal{A}_{BV}^{\bullet})\simeq [\![\mathbb{R}\mathsf{Spec}_{\D}(\mathcal{A}_{EL})/\mathcal{N}_{\mathcal{S}}^r]\!],$$
interpreted as the homotopic leaf space of the foliation induced by the action of Noether gauge
symmetries on the derived Euler-Lagrange space, with $\mathcal{A}_{BV}$ a bi-graded $\D$-algebra\footnote{Its generators are what physicists call fields, anti-fields, ghosts, antighosts etc.}.
If $\mathcal{N}_{\mathcal{S}}$ is projective $\mathrm{Sym}_{\mathcal{A}}\big(\mathcal{N}_{\mathcal{S}}[2]\oplus Cone(i_{d\mathcal{S}})\big)$ resolves $\mathcal{A}_{EL}.$ One says there are irreducible gauge symmetries. Generally, one resolves $\mathcal{N}_{\mathcal{S}}$ in $\DG(\mathcal{A})$ and forms the free algebra resolution (generating space of Noether identities).
\end{example}

\bibliographystyle{amsplain}
\bibliography{Bibliography}

@article{PTVV13,
  author     = {T. Pantev and B. Toën and M. Vaquié and G. Vezzosi},
  title      = {Shifted symplectic structures},
  journal    = {Publ. Math. Inst. Hautes Études Sci.},
  volume     = {117},
  number     = {1},
  pages      = {271--328},
  year       = {2013},
  note  ={\href{http://arxiv.org/abs/1111.3209}{arXiv:1111.3209}}
}

@article{Kry,
  title={Cohomological aspects of the Hamiltonian formalism: steps towards quantum observability},
  author={Kryczka, J.A},
  journal={Doctoral dissertation (PhD), University of Angers, Angers, France},
  volume={},
  number={},
  pages={},
  year={2021},
  publisher={(University of Angers, Angers, France}
}

@article{KSY2,
  author     = {Kryczka, J. and Sheshmani, A. and Yau, S.T.},
  title      = {Derived Moduli Spaces of Nonlinear {PDEs}: Variational Tricomplex and {BV} Formalism},
  journal    = {preprint},
  eprint     = {2406.16825},
  year       = {2024}, 
  note = {\href{https://arxiv.org/abs/2406.16825}{arXiv:2406.16825}}
}

@book{GR17b,
  author     = {Gaitsgory, D. and Rozenblyum, N.},
  title      = {A study in derived algebraic geometry, {Vol.~2}: Deformations, Lie theory and formal geometry},
  publisher  = {Amer. Math. Soc.},
  address    = {Providence, RI},
  year       = {2017}
}

@article{CPTVV17,
  author     = {D. Calaque and T. Pantev and B. Toën and M. Vaquié and G. Vezzosi},
  title      = {Shifted {Poisson} structures and deformation quantization},
  journal    = {J. Topol.},
  volume     = {10},
  number     = {2},
  pages      = {483--584},
  year       = {2017},
  note = {\href{https://arxiv.org/abs/1506.03699}{arXiv:1506.03699}},
}

@article{Kashiwara1984,
  author  = { Kashiwara, M.},
  title   = {The {Riemann-Hilbert} problem for holonomic systems},
  journal = {Publ. Res. Inst. Math. Sci.},
  volume  = {20},
  number  = {2},
  pages   = {319--365},
  year    = {1984}
}

@book{Ri,
  author = {Ritt, J. F.},
  title = {Differential algebra},
  publisher = {American Mathematical Society},
  series = {Colloquium Publications},
  volume = {33},
  year = {1950}
}

@article{PortaYuPRIMS,
  author  = {Porta, M. and Yu, T. Y.},
  title   = {Derived {H}om spaces in rigid analytic geometry},
  journal = {Publ. Res. Inst. Math. Sci.},
  volume  = {57},
  number  = {3--4},
  year    = {2021},
  pages   = {921--958},
  note = {\href{https://arxiv.org/abs/1801.07730}{arXiv:1801.07730}}, 
}

@article{HolsteinPorta2025,
  author  = {Holstein, J. and Porta, M.},
  title   = {Analytification of mapping stacks},
  journal = {Algebr. Geom.},
  volume  = {12},
  number  = {1},
  year    = {2025},
  note = {\href{https://arxiv.org/abs/1812.09300}{arXiv:1812.09300}}
}

@article{toenvaquie2007,
  author     = {To\"en, B. and Vaqui\'e, M.},
  title      = {Moduli of objects in dg-categories},
  journal    = {Ann. Sci. \'Ec. Norm. Sup\'er. (4)},
  volume     = {40},
  number     = {3},
  pages      = {387--444},
  year       = {2007}, 
  note ={\href{https://arxiv.org/abs/math/0503269}{arXiv:math/0503269}},
}

@article{TV2,
  author    = {To\"en, B. and Vezzosi, G.},
  title     = {Homotopical algebraic geometry {II}: Geometric stacks and applications},
  journal   = {Memoirs of the American Mathematical Society},
  volume    = {208},
  number    = {902},
  pages     = {x+223},
  year      = {2010},
  publisher = {American Mathematical Society},
  note ={\href{https://arxiv.org/abs/2312.02815}{arXiv:2312.02815}}, 
}

@incollection{SimpsonHodgeFiltrationNonabelian,
  author    = {Simpson, C. T.},
  title     = {The {H}odge filtration on nonabelian cohomology},
  booktitle = {Algebraic Geometry: Santa Cruz 1995},
  series    = {Proceedings of Symposia in Pure Mathematics},
  volume    = {62},
  pages     = {217--281},
  publisher = {American Mathematical Society},
  year      = {1997},
  eprint    = {alg-geom/9604005},
  archivePrefix = {arXiv}
}

@inproceedings{SiNonabelian,
  author    = {Simpson, C. T.},
  title     = {Nonabelian {H}odge theory},
  booktitle = {Proceedings of the International Congress of Mathematicians, Vol.\ I (Kyoto, 1990)},
  pages     = {747--756},
  year      = {1990}
}

@article{ElliotYoo2018,
  author  = {Elliott, C. and Yoo, P.},
  title   = {Geometric {L}anglands twists of $\mathcal{N}=4$ gauge theory from derived algebraic geometry},
  journal = {Adv.Theor. Math. Phys.},
  volume  = {22},
  number  = {3},
  pages   = {615--708},
  year    = {2018}, 
  note    ={\href{https://arxiv.org/abs/1507.03048}{arXiv:1507.03048}},
}

@book{Lur09,
  author    = {Lurie, J.},
  title     = {Higher Topos Theory},
  series    = {Annals of Mathematics Studies},
  volume    = {170},
  publisher = {Princeton University Press},
  year      = {2009}
}

@article{LurieDAGV,
  author    = {Lurie, J.},
  title     = {{DAG V}: Structured Spaces},
  year      = {2011},
  note      = {Preprint}
}

@article{KMMP,
  author  = {Kern, D. and Mann, {\'E}t. and Manolache, C. and Picciotto, R.},
  title   = {Derived moduli of sections and push-forwards},
  journal = {Selecta Math.},
  volume  = {31},
  number  = {2},
  pages   = {1--46},
  year    = {2025}, 
  note ={\href{https://arxiv.org/abs/2210.11386}{arXiv:2210.11386}}
}

@article{GR14,
  title        = {Crystals and {$\mathcal{D}$}-modules},
  author       = {Gaitsgory, D. and Rozenblyum, N.},
  journal      = {Pure Appl. Math. Quart.},
  volume       = {10},
  number       = {1},
  pages        = {57--155},
  year         = {2014},
  publisher    = {International Press of Boston}, 
  note = {\href{https://arxiv.org/abs/1111.2087}{arXiv:1111.2087}}
}

@book{GR17a,
  author    = {Gaitsgory, D. and Rozenblyum, N.},
  title     = {A study in derived algebraic geometry, Vol.~1},
  publisher = {Amer. Math. Soc.},
  year      = {2017}
}

@book{KashiwaraMicro,
  author    = {Kashiwara, M.},
  title     = {Systems of microdifferential equations},
  volume    = {34},
  series = {Progress in Mathematics},
  publisher = {Birkh\"auser},
  year      = {1983}
}

@misc{kashiwara1970algebraic,
  author       = {Kashiwara, M.},
  title        = {Algebraic Study of Systems of Partial Differential Equations},
  howpublished = {Thesis, Tokyo University (1970), translated by A. D'Agnolo and J.-P. Schneiders},
  journal      = {M\'em. Soc. Math. Fr.},
  volume       = {63},
  year         = {1995},
}

@book{KashiwaraSchapira1990,
  author    = {Kashiwara, M. and Schapira, P.},
  title     = {Sheaves on Manifolds},
  series    = {Grundlehren der Mathematischen Wissenschaften},
  volume    = {292},
  publisher = {Springer-Verlag},
  year      = {1990},
}

@article{KhanRydh2025,
  author  = {Khan, A. and Rydh, D.},
  title   = {Virtual Cartier divisors and blow-ups},
  journal = {Sel. Math. (N.S)},
  volume  = {31},
  number  = {4},
  pages   = {67},
  year    = {2025},
  doi     = {10.1007/s00029-025-01060-7}
}

@article{AriGai2015,
  author    = {Arinkin, D. and Gaitsgory, D.},
  title     = {Singular support of coherent sheaves and the geometric Langlands conjecture},
  journal   = {Sel. Math. (N.S.)},
  volume    = {21},
  number    = {1},
  pages     = {1--199},
  year      = {2015},
  publisher = {Springer},
  doi       = {10.1007/s00029-014-0167-5},
   note  ={\href{http://arxiv.org/abs/1201.6343}{arXiv:1201.6343}}
}

@article{Gabber1981,
  author  = {Gabber, O.},
  title   = {The integrability of the characteristic variety},
  journal = {Amer. J. Math.},
  volume  = {103},
  year    = {1981},
  pages   = {445--468}
}

@article{KS82,
  author  = {Kashiwara, M. and Schapira, P.},
  title   = {Micro-support des faisceaux: applications aux modules diff{\'e}rentiels},
  journal = {C. R. Acad. Sci. Paris S{\'e}rie I Math},
  volume  = {295},
  year    = {1982},
  pages   = {487--490}
}

@book{KS85,
  author    = {Kashiwara, M. and Schapira, P.},
  title     = {Microlocal study of sheaves},
  series    = {Astérisque},
  number    = {128},
  publisher = {Société Mathématique de France},
  year      = {1985}
}

@article{KashiwaraOshima1977,
  author = {Kashiwara, M. and Oshima, T.},
  title = {Systems of differential equations with regular singularities and their boundary value problems},
  journal = {Ann. Math.},
  volume = {106},
  number = {1},
  pages = {145--200},
  year = {1977}
}

@article{Kry2026,
  author  = {Kryczka, J.},
  title   = {Conservation laws and cohomological intersections},
  journal = {J. Geom. Phys.},
volume  = {225},
  number  = {105763},
  year    = {2026},
  url     = { https://doi.org/10.1016/j.geomphys.2026.105763.}, 
  note = {\href{https://doi.org/10.1016/j.geomphys.2026.105763}{https://doi.org/10.1016/j.geomphys.2026.105763}}, 
}

@book{Krasilshchik1986,
  author = {Krasil'shchik, J. and Lychagin, V. and Vinogradov, A.},
  title = {Geometry of jet spaces and non-linear partial differential equations},
  series = {Adv. Stud. Cont. Math.},
  volume = {1},
  year = {1986},
  publisher = {Gordon and Breach Science Publishers},
  address = {New York}
}

@book{Vinogradov2001,
  author    = {Vinogradov, A.M.},
  title     = {Cohomological Analysis of Partial Differential Equations and Secondary Calculus},
  publisher = {Amer. Math. Soc.},
  year      = {2001},
}

@article{schneiders1994,
  author  = {Schneiders, J.P},
  title   = {An Introduction to {$\mathscr{D}$}-Modules},
  journal = {Bull. Soc. Roy. Li.},
  volume  = {63},
  pages   = {223--295},
  year    = {1994},
}

@book{schapira19941,
  author    = {Schapira, P. and Schneiders, J-P.},
  title     = {Elliptic pairs I. Relative finiteness and duality},
  series    = {Astérisque},
  volume    = {224},
  publisher = {Soc. Math. Fr.},
  year      = {1994},
}

@book{Schapira19942,
  author    = {Schapira, P. and Schneiders, J-P.},
  title     = {Elliptic pairs II. Euler class and relative index theorem},
  series    = {Astérisque},
  volume    = {224},
  publisher = {Soc. Math. Fr.},
  year      = {1994},
}

@book{Lurie2018,
  author    = {Lurie, J.},
  title     = {Spectral algebraic geometry},
  year      = {2018},
  note      = {Preprint}
}

@article{KapustinWitten2007,
    author        = {Kapustin, A. and Witten, E.},
    title         = {Electric-magnetic duality and the geometric {Langlands} program},
    journal       = {Commun. Number Theory Phys.},
    volume        = {1},
    number        = {1},
    pages         = {1--236},
    year          = {2007}, 
    note = {\href{https://arxiv.org/abs/hep-th/0604151}{hep-th/0604151}},
}

@misc{IndCoh,
  title={Ind-coherent sheaves},
  author={Gaitsgory, D.},
  year={2012},
  eprint={1105.4857},
  archivePrefix={arXiv},
  primaryClass={math.AG},
  url={https://arxiv.org/abs/1105.4857},
   note = {\href{https://arxiv.org/abs/1105.4857}{arXiv:1105.4857}}, 
}

@misc{QCoh,
  title={Notes on Geometric Langlands: Quasi-coherent sheaves on stacks},
  author={Gaitsgory, D.},
  url={https://people.mpim-bonn.mpg.de/gaitsgde/GL/QCohtext.pdf},
   note = {\href{https://people.mpim-bonn.mpg.de/gaitsgde/GL/QCohtext.pdf}{https://people.mpim-bonn.mpg.de/gaitsgde/GL/QCohtext.pdf}} 
}

@book{Cartan1945,
  author    = {Cartan, E.},
  title     = {Les syst\`emes diff\'erentiels ext\'erieurs et leurs applications g\'eom\'etriques},
  series    = {Act. Sci. Ind.},
  volume    = {994},
  publisher = {Hermann},
  address   = {Paris},
  year      = {1945},
  language  = {French},
}

@article{Paugam2022,
  author  = {Paugam, F
  F.},
  title   = {A note on the nonlinear derived {C}auchy problem},
  journal = {arXiv preprint arXiv:2211.04860},
  year    = {2022},
   note = {\href{https://arxiv.org/abs/2211.04860}{arXiv:2211.04860}}
}

@incollection{Sato1973,
  author    = {Sato, M.},
  title     = {Hyperfunctions and pseudodifferential equations},
  booktitle = {Hyperfunctions and Pseudo-differential Equations},
  series    = {Lecture Notes in Mathematics},
  volume    = {287},
  pages     = {265--529},
  publisher = {Springer},
  year      = {1973}
}

@book{Paugam2014,
  author    = {Paugam, F.},
  title     = {Towards the mathematics of quantum field theory},
  volume    = {59},
  publisher = {Springer Science \& Business Media},
  year      = {2014}
}

@article{JubinSchapira2016,
  author  = {Jubin, B. and Schapira, P.},
  title   = {Sheaves and {$\mathscr{D}$}-modules on Lorentzian manifolds},
  journal = {Lett. Math. Phys.},
  volume  = {106},
  pages   = {607--648},
  year    = {2016},
  publisher = {Springer}
}

@phdthesis{Quillen1964,
  author = {Quillen, D.},
  title = {Formal properties of over-determined systems of linear partial differential equations},
  school = {Harvard University},
  year = {1964}
}

@article{Goldschmidt1967,
  author = {Goldschmidt, H.},
  title = {Integrability criteria for systems of nonlinear partial differential equations},
  journal = {J. Diff. Geom.},
  volume = {1},
  number = {3-4},
  year = {1967},
  pages = {269--307}
}

@article{PortaYuDerivedGAGA2019,
  author    = {Porta, M. and Yu, T.Y.},
  title     = {{GAGA} theorems in derived complex geometry},
  journal   = {J. Alg. Geom.},
  volume    = {28},
  number    = {3},
  pages     = {519--565},
  year      = {2019},
  mrnumber  = {3959070}
}

@misc{Lur17,
  author    = {Lurie, J.},
  title     = {Higher Algebra},
  note      = {Preprint, August 2012},
  year      = {2012},
  howpublished = {\url{http://www.math.harvard.edu/~lurie/papers/HA.pdf}}
}

@article{KryQS2026,
  author    = {Kryczka, J.},
  title     = {Formal integrability for quasi-smooth PDEs and overconvergent geometry},
  year      = {2026},
  note      = {Preprint}
}

@article{Kuranishi1957,
  author  = {Kuranishi, M.},
  title   = {On E. Cartan’s prolongation theorem of exterior differential systems},
  journal = {Amer. J. Math.},
  volume  = {79},
  number  = {1},
  pages   = {1--47},
  year    = {1957},
}

@book{Malgrange2005,
  author    = { Malgrange, B.},
  title     = {Syst\`emes diff\'erentiels involutifs},
  series    = {Panoramas et Synth\`eses},
  year      = {2005},
  address   = {Paris},
  language  = {French},
}

@article{Sp,
  author  = {Spencer, D.C.},
  title   = {De Rham theorems and Neumann decompositions associated with linear partial differential equations},
  journal = {Ann. Four.},
  volume  = {14},
  number  = {1},
  year    = {1964}
}

@article{Sp2,
  author  = {Spencer, D. C.},
  title   = {Overdetermined systems of linear partial differential equations},
  journal = {Bull. Amer. Math. Soc.},
  volume  = {75},
  number  = {2},
  pages   = {179--239},
  year    = {1969}
}

@book{Kahler,
  author    = {K{\"a}hler, Erich},
  title     = {Einf{\"u}hrung in die Theorie der Systeme von Differentialgleichungen},
  publisher = {B.G. Teubner},
  address   = {Leipzig},
  year      = {1934}
}

@article{Gold1967,
  author  = {Goldschmidt, H.},
  title   = {Existence theorems for analytic linear partial differential equations},
  journal = {Ann. Math.},
  pages   = {246--270},
  year    = {1967},
  publisher = {JSTOR}
}

@article{Lewy1957,
  title={An example of a smooth linear partial differential equation without solution},
  author={Lewy, Hans},
  journal={Ann. Math.},
  volume={66},
  number={1},
  pages={155--158},
  year={1957},
  publisher={JSTOR}
}

@article{Ishii,
  author  = {Ishii, T},
  title   = {On propagation of regular singularities for solutions of nonlinear partial differential equations},
  journal = {Jo. Fac. Sci. Univ. Tokyo. Sect. IA, Math.},
  volume  = {37},
  number  = {2},
  pages   = {377--424},
  year    = {1990}
}

@book{Hormander,
  author    = {H{\"o}rmander, L.},
  title     = {The Analysis of Linear Partial Differential Operators I: Distribution Theory and Fourier Analysis},
  series    = {Grund. Math. Wiss.},
  volume    = {256},
  publisher = {Springer-Verlag},
  address   = {Berlin},
  year      = {1983}
}

@inproceedings{Bonylocal,
  author    = {Bony, J.-M.},
  title     = {Localisation et propagation des singularit{\'e}s pour les {\'e}quations non lin{\'e}aires},
  booktitle = {J. {\'E}q. D{\'e}riv. Part. (St. Jean-de-Monts, 1978)},
  year      = {1978}
}

@article{Nirenberg,
  author    = {Nirenberg, L.},
  title     = {An abstract form of the nonlinear {C}auchy--{K}owalewski theorem},
  journal   = {J. Diff. Geom.},
  volume    = {6},
  number    = {4},
  pages     = {561--576},
  year      = {1972},
  publisher = {Lehigh University}
}

@article{DriGai2013,
    author = {Drinfeld, V. and Gaitsgory, D.},
    title = {On Some Finiteness Questions for Algebraic Stacks},
    journal = {Geom. Funct. An.},
    volume = {23},
    number = {1},
    pages = {149--294},
    year = {2013},
    publisher = {Springer}
}

@article{BenZviNad,
    author = {Ben-Zvi, D. and Nadler, D.},
    title = {Loop Spaces and Connections},
    journal = {J. Top.},
    volume = {5},
    number = {2},
    pages = {377--430},
    year = {2012},
    publisher = {Oxford University Press},
    doi = {10.1112/jtopol/jts011}
}

@incollection{Laurent1985,
  author    = {Laurent, Y},
  title     = {Th{\'e}orie de la deuxi{\`e}me microlocalisation dans le domaine complexe},
  booktitle = {Microlocal Analysis},
  series    = {Prog. Math.},
  volume    = {Microlocal Analysis},
  publisher = {Birkh{\"a}user},
  year      = {1985}
}

@article{Zerner,
  author  = {Zerner, M.},
  title   = {Domaines d'holomorphie des fonctions v{\'e}rifiant une {\'e}quation aux d{\'e}riv{\'e}es partielles},
  journal = {Compt. Rend. Sci. Math.},
  volume  = {272},
  pages   = {1646--1648},
  year    = {1971}
}

@article{Hamada,
  author    = {Hamada, Y.},
  title     = {The singularities of the solutions of the {C}auchy problem},
  journal   = {Pub. RIMS.},
  volume    = {5},
  number    = {1},
  pages     = {21--40},
  year      = {1969},
  publisher = {RIMS, Kyoto University}
}

@article{BonySymbol,
  author    = {Bony, J.-M.},
  title     = {Calcul symbolique et propagation des singularit{\'e}s pour les {\'e}quations aux d{\'e}riv{\'e}es partielles non lin{\'e}aires},
  journal   = {Ann. Sci. l'{\'E}c. norm. sup.},
  volume    = {14},
  number    = {2},
  pages     = {209--246},
  year      = {1981}
}

@article{Kobayashi,
  author    = {Kobayashi, T},
  title     = {Singular solutions and prolongation of holomorphic solutions to nonlinear differential equations},
  journal   = {Publ. R.I.M.S},
  volume    = {34},
  number    = {1},
  pages     = {43--63},
  year      = {1998},
  publisher = {RIMS, Kyoto University}
}

@misc{Gai11,
  author    = {Gaitsgory, D.},
  title     = {Ind-coherent sheaves},
  howpublished = {arXiv preprint},
  eprint    = {1105.4857},
  year      = {2011}
}

@article{ChenFiltered,
    author = {Chen, H.},
    title = {Categorical Cyclic Homology and Filtered {$\mathcal{D}$}-Modules on Stacks: {K}oszul Duality},
    journal = {arXiv preprint},
    eprint = {2301.06949},
    year = {2023},
    archivePrefix = {arXiv},
    primaryClass = {math.AG},
    note = {\href{https://arxiv.org/abs/2301.06949}{arXiv:2301.06949}}
    
}

@article{SGA4I,
  title={SGA 4 EXPOSE I},
  author={Grothendieck, A and Verdier, JL},
  journal={Theorie des Topos et Cohomologie Etale des Schemas. Seminaire de Geometrie Algebrique du Bois-Marie 1963-1964 (SGA 4): Tome 1},
  volume={269},
  pages={1},
  year={2006},
  publisher={Springer}
}

@article{KS1,
  author     = {Kryczka, J. and Sheshmani, A.},
  title      = {The {$\mathscr{D}_X$}-Geometric Hilbert Scheme -- Part {I}: Involutivity and Stability},
  journal    = {preprint},
  eprint     = {2507.07937},
  year       = {2025}, 
  note = {\href{https://arxiv.org/abs/2507.07937}{arXiv:2507.07937}}
}

@article{KS2,
  author     = {Kryczka, J. and Sheshmani, A.},
  title      = {The {$\mathscr{D}_X$}-Geometric Hilbert Scheme -- Part {II}: Hilbert and Quot DG-Schemes},
  journal    = {preprint},
  eprint     = {2411.02387},
  year       = {2024}, 
  note = {\href{https://arxiv.org/abs/2411.02387}{arXiv:2411.02387}}
}

@book{BD2004,
  author    = {Beilinson, A. and Drinfeld, V.},
  title     = {Chiral Algebras},
  series    = {Amer. Math. Soc.},
  volume    = {51},
  publisher = {Amer. Math. Soci.},
  year      = {2004},
}

@article{dBPP18,
  author    = {di Brino, G. and Pi{\v{s}}talo, D. and Poncin, N.},
  title     = {Koszul--Tate resolutions as cofibrant replacements of algebras over differential operators},
  journal   = {Journal of Homotopy and Related Structures},
  volume    = {13},
  number    = {4},
  pages     = {793--846},
  year      = {2018}
}

@book{CG16,
  author    = {Costello, K. and Gwilliam, O.},
  title     = {Factorization algebras in quantum field theory, Vol. 1},
  publisher = {Cambridge University Press},
  year      = {2016}
}

@article{Diss,
  author  = {Pištalo, D. and Poncin, N.},
  title   = {On Koszul--Tate resolutions and Sullivan models},
  journal = {Diss. Math.},
  volume  = {531},
  pages   = {1--71},
  year    = {2018}
}

@book{CG16b,
  author    = {Costello, K. and Gwilliam, O.},
  title     = {Factorization algebras in quantum field theory, Vol. 2},
  year      = {2016}
}

@Article{GwPa18,
  author  = { Gwilliam, O. and Pavlov, D.},
  title   = {Enhancing the filtered derived category},
  journal = {Journal of Pure and Applied Algebra},
  volume  = {222},
  number  = {11},
  year    = {2018},
  pages   = {3621--3674},
  doi     = {10.1016/j.jpaa.2018.01.004},
  eprint  = {arXiv:1602.01515},
  note    = {Also published online April 2, 2018},
}

@article{PavSch,
  title        = {Homotopy Theory of Symmetric Powers},
  author       = {Pavlov, D. and Scholbach, J.},
  journal      = {Homology, Homotopy and Applications},
  volume       = {20},
  number       = {1},
  pages        = {359--397},
  year         = {2018},
  doi          = {10.4310/HHA.2018.v20.n1.a20},
}

@misc{CS3,
  author       = {Clausen, D. and Scholze, P.},
  title        = {Lectures on Complex Geometry and Condensed Mathematics},
  howpublished = {\url{https://people.mpim-bonn.mpg.de/scholze/Complex.pdf}},
  note         = {Online lecture notes},
  year         = {2022},
}

@article{Tor,
  author       = {Alisa Govzmann and Damjan Pi{\v{s}}talo and Norbert Poncin},
  title        = {The {Tor} Spectral Sequence and Flat Morphisms in Homotopical {D}-Geometry},
  journal      = {Preprint},
  year         = {2022},
  eprint       = {2212.07912},
  archiveprefix = {arXiv},
  primaryclass = {math.AT},
  url          = {https://arxiv.org/abs/2212.07912}
}

@article{Bla01,
  author    = {Blander, B.A.},
  title     = {Local projective model structures on simplicial presheaves},
  journal   = {K-Theory},
  volume    = {24},
  number    = {3},
  pages     = {283--301},
  year      = {2001}
}

@book{Hirschhorn2003,
  author    = {Hirschhorn, P.S.},
  title     = {Model Categories and Their Localizations},
  publisher = {Amer. Math. Soc.},
  series    = {Math. Surv. Mon.},
  volume    = {99},
  year      = {2003}
}

@misc{Cam,
  author       = {Rodr\'iguez Camargo, J.E.},
  title        = {The analytic de Rham stack in rigid geometry},
  year         = {2024},
  eprint       = {2401.07738},
  archivePrefix = {arXiv},
  primaryClass = {math.AG},
  note         = {Preprint},
   note = {\href{https://arxiv.org/abs/2401.07738}{arXiv:2401.07738}}
}

@article{TV02,
  author = {To{\"e}n, B. and Vezzosi, G.},
  title = {Homotopical algebraic geometry I, Topos theory},
  journal = {Adv. Math.},
  volume = {193},
  number = {2},
  year = {2005},
  pages = {257--372}
}

@book{Saunders1989,
  author    = {Saunders, D.},
  title     = {The Geometry of Jet Bundles},
  volume    = {142},
  series    = {Lond. Math. Soc. Lect. Notes},
  publisher = {Cambridge University Press},
  year      = {1989},
  isbn      = {9780521369480},
  doi       = {10.1017/CBO9780511526411}
}

@book{HT,
  author = {Hotta, R. and Tanisaki, T.},
  title = {$\mathcal{D}$-modules, perverse sheaves, and representation theory},
  series = {Progress in Mathematics},
  volume = {236},
  publisher = {Birkhäuser},
  year = {2007}
}

@book{Kock2010,
  author    = {Kock, A.},
  title     = {Synthetic Geometry of Manifolds},
  volume    = {180},
  series    = {Camb. Tr. Math.},
  publisher = {Cambridge University Press},
  year      = {2010},
  doi       = {10.1017/CBO9780511691690}
}

@book{Kock1979,
  editor    = {Kock, A.},
  title     = {Topos Theoretic Methods in Geometry},
  volume    = {30},
  series    = {Various Publications Series},
  publisher = {Matematisk Institut, Aarhus University},
  address   = {Aarhus},
  year      = {1979},

}

@article{FraGai,
  title={Chiral Koszul duality},
  author={Francis, J. and Gaitsgory, D.},
  journal={Sel. Math.},
  volume={18},
  number={1},
  pages={27--87},
  year={2012},
  publisher={Springer}
}

@article{Hinich,
  title={DG coalgebras as formal stacks},
  author={Hinich, V.},
  journal={Journal of pure and applied algebra},
  volume={162},
  number={2-3},
  pages={209--250},
  year={2001},
  publisher={Elsevier}
}

@article{Sim09,
  author    = {Simpson, C.},
  title     = {Geometricity of the {Hodge} {filtration on the $\infty$-stack of perfect complexes over $X_{dR}$}},
  journal   = {Moscow Mathematical Journal},
  volume    = {9},
  number    = {3},
  pages     = {665--721},
  year      = {2009}, 
  note = {\href{https://arxiv.org/abs/math/0510269}{arXiv:math/0510269}}
}
\end{document}